\documentclass[11pt]{amsart}

\usepackage{epigamath}

\usepackage[english]{babel}

\numberwithin{equation}{section}

\usepackage[shortlabels]{enumitem}

\usepackage{amsmath}
\usepackage{amscd}
\usepackage{amssymb}
\usepackage{rotating}
\usepackage{amsfonts}
\usepackage{amsthm}
\usepackage{youngtab}

\usepackage{tikz}
\usetikzlibrary{arrows,shapes,positioning}
\usetikzlibrary{decorations.markings}
\tikzstyle arrowstyle=[scale=1]

\usepackage[all]{xy}

\usepackage{url}

\newtheorem{theorem}{Theorem}[section]
\newtheorem{proposition}[theorem]{Proposition}
\newtheorem{lemma}[theorem]{Lemma}
\newtheorem{corollary}[theorem]{Corollary}
\newtheorem{conjecture}[theorem]{Conjecture}

\theoremstyle{remark}
\newtheorem{remark}[theorem]{Remark}
\newtheorem{example}[theorem]{Example}
\newtheorem{question}[theorem]{Question}
\newtheorem{algorithm}[theorem]{Algorithm}

\theoremstyle{definition}
\newtheorem{definition}[theorem]{Definition}
\newtheorem{notation-conventions}[theorem]{Notation and conventions}

\setlist[enumerate,1]{label={\rm(\arabic*)}, ref={\rm\arabic*}} 
\newlist{i-enumerate}{enumerate}{2}
\setlist[i-enumerate,1]{label={\rm(\roman*)}, ref={\rm\roman*}}

\tikzstyle directed=[postaction={decorate,decoration={markings,
    mark=at position .65 with {\arrow[arrowstyle]{stealth}}}}]

\newcommand{\nn}{\nonumber}

\allowdisplaybreaks[1]

\newcommand\threeDYoung[1]{
\begin{tikzpicture}[x=(220:0.6cm), y=(-40:0.6cm), z=(90:0.42cm)]

\foreach \m [count=\y] in  #1{
  \foreach \n [count=\x] in \m {
  \ifnum \n>0
      \foreach \z in {1,...,\n}{
        \draw [fill=gray!30] (\x+1,\y,\z) -- (\x+1,\y+1,\z) -- (\x+1, \y+1, \z-1) -- (\x+1, \y, \z-1) -- cycle;
        \draw [fill=gray!40] (\x,\y+1,\z) -- (\x+1,\y+1,\z) -- (\x+1, \y+1, \z-1) -- (\x, \y+1, \z-1) -- cycle;
        \draw [fill=gray!5] (\x,\y,\z)   -- (\x+1,\y,\z)   -- (\x+1, \y+1, \z)   -- (\x, \y+1, \z) -- cycle;  
      }
 \fi
 }
} 
\end{tikzpicture}
}

\DeclareMathOperator{\Adj}{Adj}
\DeclareMathOperator{\ch}{ch}
\DeclareMathOperator{\Char}{char}
\DeclareMathOperator{\Hilb}{Hilb}

\DeclareMathOperator{\Ext}{Ext}
\DeclareMathOperator{\GL}{GL}
\DeclareMathOperator{\glo}{glo}
\DeclareMathOperator{\Grass}{Grass}
\DeclareMathOperator{\HE}{HE}
\DeclareMathOperator{\Jac}{Jac}
\DeclareMathOperator{\Proj}{Proj}
\DeclareMathOperator{\pyr}{pyr}
\DeclareMathOperator{\rank}{rank}
\DeclareMathOperator{\Spec}{Spec}
\DeclareMathOperator{\Sym}{Sym}
\DeclareMathOperator{\Td}{Td}
\DeclareMathOperator{\Tor}{Tor}

\newcommand{\longhookrightarrow}{\lhook\joinrel\longrightarrow}
\newcommand{\longtwoheadrightarrow}{\relbar\joinrel\twoheadrightarrow}
\newcommand{\lowsim}{\vbox to 0pt{\vss\hbox{$\scriptstyle\sim$}\vskip-2pt}}

\newcommand{\supth}[1]{\ensuremath{#1^{\mathrm{th}}}}

\EpigaVolumeYear{9}{2025} \EpigaArticleNr{15} \ReceivedOn{January 5, 2024}
\InFinalFormOn{December 31, 2024}
\AcceptedOn{February 7, 2025}

\title{On singular Hilbert schemes of points: Local structures and tautological sheaves}
\titlemark{On singular Hilbert schemes of points: Local structures and tautological sheaves}
  
\author{Xiaowen Hu}
\address{School of Sciences, Great Bay University, Dongguan, P.\,R.~China}
\email{huxw06@gmail.com}

\authormark{X.~Hu}

\AbstractInEnglish{We show an intrinsic version of Thomason's fixed-point theorem. Then we determine the local structure of the Hilbert scheme of at most $7$ points in $\mathbb{A}^3$. In particular, we show that in these cases, the points with the same extra dimension have the same singularity type. Using these results, we compute the equivariant Hilbert functions at the singularities and verify a conjecture of Zhou on the Euler characteristics of tautological sheaves on Hilbert schemes of points on $\mathbb{P}^3$ for at most $6$ points.}

\MSCclass{14C05, 14C40, 13D40}
\KeyWords{Hilbert scheme, localization theorem, Haiman coordinates, tautological sheaf}

\acknowledgement{This work is supported by NSFC 11701579, NSFC 12371063, and  Guangzhou Science and Technology Programme 202201010793.}

\begin{document}



\maketitle

\begin{prelims}

\DisplayAbstractInEnglish

\bigskip

\DisplayKeyWords

\medskip

\DisplayMSCclass

\end{prelims}


\newpage

\setcounter{tocdepth}{1}

\tableofcontents


\section{Introduction}

The study of the cohomology of various tautological sheaves on Hilbert schemes of points starts from \cite{Got90}. 
There are many important works, and we just mention a few of them: \cite{GS93,EGL01,Sca09,WZ14,Kru18}. These works all concern Hilbert schemes of points on smooth surfaces, except for \cite{WZ14}, where $\Hilb^3(\mathbb{P}^3)$ is discussed. The involved Hilbert schemes are all smooth.

The Hilbert schemes of  points on higher-dimensional varieties are in general singular. There is a  conjecture on the tautological sheaves proposed by Jian Zhou (see \cite{WZ14})\footnote{This conjecture in \cite{WZ14} was proposed by Zhou in a talk on Hilbert schemes at a conference held at the Chinese Academy of Sciences, and Wang joined the work later.}, as we recall in the following. 
Let $\Bbbk$ be an algebraically closed field of characteristic zero and $X$ be a smooth projective scheme over $\Bbbk$.
Let $\mathcal{Z}$ be the universal subscheme over $\Hilb^n(X)$: 
\[
	\xymatrix{
	\mathcal{Z} \ar[d]^{\pi} \ar[r]^{f} & X \\
	\Hilb^n(X)\rlap{.} & 
	}
	\]
Let $\mathcal{E}$ be a locally free sheaf on $X$. 
The \emph{tautological sheaf}  associated with $\mathcal{E}$ is 
\[
\mathcal{E}^{[n]}:=\pi_* f^* \mathcal{E}.
\]
Then $\mathcal{E}^{[n]}$ is a locally free sheaf on $\Hilb^n(X)$. For every length $n$ closed subscheme $Z$ of $X$, let $[Z]$ denote the point of $\Hilb^n(X)$ representing $Z$; then
\[
\mathcal{E}^{[n]}|_{[Z]}\cong H^{0}(Z,\mathcal{E}).
\]
Moreover, let $\Lambda_u(\mathcal{E})$ be the polynomial of a formal variable $u$ with $K$-theory elements as coefficients:   
\[
 \Lambda_u(\mathcal{E}):=\sum_{i=0}^{\mathrm{rank}\ \mathcal{E}}u^i \wedge^i \mathcal{E}.
 \]
For two locally free sheaves $\mathcal{F}$ and $\mathcal{G}$ on a projective scheme $Y$ over  $\Bbbk$, define
\begin{gather*}
 \chi(\mathcal{F},\mathcal{G})=\sum_{i=0}^{\infty}(-1)^i \dim_{\Bbbk} \Ext_{\mathcal{O}_Y}^i(\mathcal{F},\mathcal{G}).
 \end{gather*}
We will also use the above notation for vector bundles associated with locally free sheaves.
Now we are ready to recall \cite[Conjecture 1]{WZ14} in a slightly generalized form.

\begin{conjecture}
\label{conj-intro-eularchar-global}
Suppose $\dim X\geq 2$. Then for line bundles $K,L$ on $X$,
\begin{eqnarray}\label{eq-intro-eularchar-global}
1+\sum_{n=1}^{\infty}\chi\left(\Lambda_{-v}K^{[n]},\Lambda_{-u}L^{[n]}\right)Q^n=
\exp \left(\sum_{n=1}^{\infty}\chi(\Lambda_{-v^n}K, \Lambda_{-u^n}L)\frac{Q^n}{n}\right).
\end{eqnarray}
\end{conjecture}

In particular, taking coefficients of lower-degree terms of $u$ and $v$, we deduce from (\ref{eq-intro-eularchar-global}) 
\[
 1+\sum_{n=1}^{\infty}\chi\left(\mathcal{O}_{X^{[n]}}\right)Q^n=(1-Q)^{-\chi(\mathcal{O}_X)}
\] 
and
\[
\sum_{n=1}^{\infty}\chi\left(L^{[n]}\right)Q^n=(1-Q)^{-\chi(\mathcal{O}_X)}\chi(L)Q.
\]

\begin{remark}\label{rem:history_line_bundle_case}
  In \cite{WZ14}  the corresponding conjecture is stated  with the assumption $K=L$, and it is shown to be true in \textit{op.\ cit.}~when  $\dim X=2$. Their proof also works  for any line bundles $K$ and $L$. They also studied an equivariant version of this conjecture, for $X=\mathbb{C}^3$ and $n=3$ (\textit{i.e.}, the equality induced by the coefficients of $Q^3$). For $v=0$ and still $\dim X=2$, Conjecture~\ref{conj-intro-eularchar-global} is a consequence of \cite[Theorem~2.4.5]{Sca09}. The complete dimension $2$ case, together with some sheaf version enhancement, is proven by Krug in  \cite{Kru18}. 
The conjecture fails for curves, as observed in \cite[Section~6.1]{Kru18}. See also \cite{Wan16} and \cite{OpS23} for related problems on curves.
\end{remark}

\begin{remark}\label{rem:history_vector_bundle_case}
If  $K$ and $L$ are allowed to be vector bundles of arbitrary ranks,  in \cite[Section~2.3]{WZ14}  it is conjectured that Formula (\ref{eq-intro-eularchar-global}) remains true if $K=L$. In dimension $2$, Krug showed in \cite[Section~6.3]{Kru18} that this does not hold and gave a formula in a special case. We showed in \cite{Hu21} the existence of a universal formula in every dimension, but an explicit formula like the right-hand side of (\ref{eq-intro-eularchar-global})  is still missing even in dimension $2$.
\end{remark}

As main motivation of this paper, we study the validity of Conjecture~\ref{conj-intro-eularchar-global} for smooth toric 3-folds $X$, where the involved Hilbert schemes are singular.

Our main tool is Thomason's Lefschetz fixed-point theorem for singular schemes with torus actions (see also \cite[Theorem 6.4]{Tho86} and \cite[th\'eor\`eme 3.5]{Tho92}). In fact, we prove a Lefschetz fixed-point theorem for such schemes with reduced isolated fixed points, without assuming a global equivariant embedding. Roughly speaking, for such schemes (and, more generally, algebraic spaces) $X$ and a locally free $T$-sheaf on $X$, we have the following. 

\begin{proposition}[=Corollary~\ref{cor-Thomason-reducedisolatefixed-1} and Equation~(\ref{eq-Thomason-reducedisolatefixed-2})]\label{prop-intro-Thomason-reducedisolatefixed-1}
\begin{equation}\label{eq-intro-Thomason-reducedisolatefixed-2}
\sum (-1)^i H^i(X,\mathcal{F})=\sum_{x\in X^T} (\mathcal{F}_x/\mathfrak{m}_x \mathcal{F}_x)\cdot H\left(\widehat{\mathcal{O}}_{X,x};\mathbf{t}\right),
\end{equation}
where $H(\widehat{\mathcal{O}}_{X,x};\mathbf{t})$ is the equivariant Hilbert function of the completed local ring of\, $X$ at the fixed points $x$. 
\end{proposition}

This revised formula makes Thomason's theorem more computable.
In Section~\ref{sec:tautological-sheaves}, we outline how (\ref{eq-intro-Thomason-reducedisolatefixed-2}) can in principle be used to compute the Euler characteristics of tautological sheaves on $\Hilb^{\Phi}(\mathbb{P}^r)$ for any Hilbert polynomial $\Phi$.
Conjecture~\ref{conj-intro-eularchar-global} is reduced to a conjecture on equivariant Hilbert functions
\begin{eqnarray}\label{eq-intro-conj-equi-Hilbert-functions}
&&\sum_{\lambda\in \mathscr{P}_r}\left(Q^{|\lambda|} H(A_{\lambda};\theta_1,\dots,\theta_r)
\prod_{\mathbf{i}=(i_1,...,i_r)\in \lambda}\left(1-u \theta_1^{i_1}\cdots\theta_r^{i_r}\right)\left(1-v \theta_1^{-i_1}\cdots\theta_r^{-i_r}\right)\right)\notag\\
&&=\exp\left(\sum_{n=1}^{\infty}\frac{(1-u^n )(1-v^n)Q^n}{n\left(1-\theta_1^n\right)\cdots\left(1-\theta_r^n\right)}\right),
\end{eqnarray} 
where $r\geq 2$ and $\mathscr{P}_r$ is the set of $r$-dimensional partitions and  $H(A_{\lambda};\theta_1,\dots,\theta_r)$ is the equivariant Hilbert function of the completed local ring of $\Hilb^{|\lambda|}(\mathbb{A}^r)$ at the monomial ideal $I_{\lambda}$ associated with the partition $\lambda$; for the notation, see Sections~\ref{sec:Monomial-ideals-of-finite-colength} and~\ref{sec:Haiman-equations}. 
The equivariant Hilbert functions seem far from being effectively computable in general, while a natural expansion of the right-hand side of (\ref{eq-intro-conj-equi-Hilbert-functions}) into terms parametrized by higher-dimensional (at least $3$) partitions does not seem to exist in combinatorics for the time being. The work of Wang--Zhou \cite{WZ14} essentially solved the case $r=2$.

Therefore, to apply this formula, we need to study the local structure of Hilbert schemes of points on~$\mathbb{A}^r$. In this paper, we only study $\mathbb{A}^3$. The tool is Haiman's equations; see \cite{Hai98} (see also \cite{Hui06}). We make some algebraic manipulations on Haiman's equations so that we can compute the equivariant Hilbert functions. 

Conjecture~\ref{conj-existence-equivariant-iso-exdim6}, mentioned in Theorem~\ref{thm-intro-verify-conj}, predicts the equivariant local structure at points defined by non-Borel monomial ideals of colength $7$ and extra dimension $6$. The equivariant local structures at  the points defined by  Borel ideals of colength at most $6$, and the non-equivariant local structures at  the points  for all monomial ideals of colength at most $7$, are solved in this paper. In this process, we find some interesting phenomena. 

\begin{proposition}[= Proposition~\ref{prop--localstructure-extradim6}]\label{prop-intro-localstructure-extradim6}
Let $z$ be a point on $\Hilb^n(\mathbb{A}^3)$. For $n\leq 7$, if the embedded dimension at $z$ is $3n+6$, then there exist an open neighborhood $U$ of $z$ and an open immersion  $U\hookrightarrow \widehat{G}(2,6)\times \mathbb{A}^{3n-9}$.
\end{proposition}

When $n=4$, this is a classical result of S.~Katz \cite{Kat94}.
For  points with embedded dimension $3n+8$, we also find  such \emph{similarity} phenomena (see Section~\ref{sec:3D-1321} and Appendix~\ref{sec:3D-2311}). For an extensive discussion, we refer the reader to Section~\ref{sec:conjectures}. In summary, although the local structures of Hilbert schemes of points are in general very bad (see \cite{Jel20}), there seem to be unexpected patterns.

As a result, we can show that the Hilbert schemes of at most $7$  points on smooth 3-folds have certain good local properties.

\begin{theorem}\label{thm-intro-local-property-lessorequal-7points}
Let $X$ be the smooth quasi-projective 3-fold. Then $\Hilb^n(X)$ is normal, Gorenstein for $n\leq 7$, and has only rational singularities for $n\leq 6$.
\end{theorem}

We expect that $\Hilb^7(X)$ also has only rational singularities; see the end of Section~\ref{sec:local-properties-hilb-schemes} for some discussions. In Section~\ref{sec:McKay-correspondence}, we consider Conjecture~\ref{conj-intro-eularchar-global} from the viewpoint of the McKay correspondence. This provides us motivation to consider the rationality of the above-mentioned singularities.

Our results on the local structures enable us to compute the equivariant Hilbert functions. Consequently, we verify Conjecture~\ref{conj-intro-eularchar-global} for $n\leq 6$ and equivariant line bundles $K$ and $L$ on toric 3-folds. In \cite{Hu21}, using degenerations of Hilbert schemes, we extend this result to all smooth proper 3-folds.

\begin{theorem}\label{thm-intro-verify-conj}
Conjecture~\ref{conj-intro-eularchar-global} modulo $Q^7$ holds for smooth proper toric 3-folds $X$ and equivariant line bundles $K,L$ on $X$. Assume that Conjecture~\ref{conj-existence-equivariant-iso-exdim6} is true; then Conjecture~\ref{conj-intro-eularchar-global} modulo $Q^8$ holds for smooth proper toric 3-folds $X$ and equivariant line bundles $K,L$ on $X$.
\end{theorem}

The difficulty that prevents us from extending Theorem~\ref{thm-intro-verify-conj} to the 7-point case, \textit{i.e.}, the validity of Conjecture~\ref{conj-intro-eularchar-global} modulo $Q^8$, is that our determination of the local structure of $\Hilb^7(\mathbb{A}^3)$ at the torus-fixed points corresponding to non-Borel ideals is \emph{indirect}, and presently we are not able to find an explicit \emph{equivariant} isomorphism. This difficulty appears already for $\Hilb^6(\mathbb{A}^3)$, where we find an explicit equivariant isomorphism in Appendix~\ref{sec:change-var-1311} by brute force.\\

The structure of this paper is as follows:
\begin{enumerate}
	\item In Section~\ref{sec:thomason-localization-theorem}, we prove a Thomason-type localization theorem without assuming a global equivariant embedding into a regular scheme. This enables us to apply this theorem to smooth proper toric varieties. Moreover, we recall the notion of equivariant Hilbert functions and the properties of these functions  that we will use.
	\item Section~\ref{sec:equi-embeddings-in-Grass} explains the framework needed to apply the localization theorem to compute the equivariant Euler characteristic of tautological sheaves on Hilbert schemes. We introduce some notions and notation for later use.
	\item Section~\ref{sec:local-equations-Hilbert-schemes} is the most technical part of this paper. We recall Haiman's defining equations of Hilbert schemes of points and the notions of Borel and non-Borel ideals. Then we give an explicit superpotential function for $3$-dimensional pyramids; we consider the pyramid partitions as a window to get a glimpse of the structure of Hilbert schemes of points in higher dimensions.  In Section~\ref{sec:3D-131}, we explain an algorithm and an elementary trick of manipulations of Haiman's equations at the unique singular fixed point of $\Hilb^5(\mathbb{A}^3)$. Then we use it to simplify the Haiman ideal for singular points corresponding to Borel ideals. In Section~\ref{sec:3D-1311}, we explain the difficulty in getting the equivariant local structure at the singular points corresponding to non-Borel monomial ideals. We conclude this section with a series of observations and conjectures on the local structure of $\Hilb^n(\mathbb{A}^3)$. Recently, Jelisiejew, Ramkumar, and  Sammartano made remarkable progress on these conjectures in \cite{JRS24}.
	\item In Section~\ref{sec:euler-char-taut-sheaves}, we compute the equivariant Hilbert functions of the local rings at the fixed singular points and prove Theorem~\ref{thm-intro-verify-conj}. We explain  Conjecture~\ref{conj-intro-eularchar-global} as a McKay correspondence, and using Riemann--Roch for stacks, we give evidence for it.
	\item In Section~\ref{sec:local-properties-hilb-schemes}, we study the local properties of $\Hilb^n(X)$ for $n\leq 7$ using the results and tricks in Section~\ref{sec:local-equations-Hilbert-schemes}.
	\item The appendices contain some tedious details. Extremely complicated is Appendix~\ref{sec:change-var-1311}, where we obtain the equivariant local structure at the unique non-Borel point of $\Hilb^6(\mathbb{A}^3)$ by the trick of Section~\ref{sec:3D-131}. In principle, one can solve Conjecture~\ref{conj-existence-equivariant-iso-exdim6} in a similar way. But we hope that there is a more conceptual approach.
\end{enumerate}

\begin{notation-conventions}\label{nota-conv:intro}\leavevmode 
\begin{enumerate}
	\item The results in Sections~\ref{sec:thomason-localization-theorem} and~\ref{sec:equi-embeddings-in-Grass} hold over an arbitrary field $\Bbbk$. Starting from Section~\ref{sec:Haiman-equations}, we assume $\Char(\Bbbk)=0$.
	\item Most of the notation is defined in Sections~\ref{sec:equivariant-Hilbertfunction},~\ref{sec:Monomial-ideals-of-finite-colength}, and~\ref{sec:Haiman-equations}.
	\item $\widehat{G}(2,6)$ stands for the cone of the Grassmannian $G(2,6)$ in $\mathbb{P}^{14}$ by the Pl\"ucker embedding.
\end{enumerate}
\end{notation-conventions}

\subsection*{Accompanying files} The Macaulay2 codes implementing Algorithm~\ref{alg-step0-Haiman}, the computation for Proposition~\ref{prop-Hilbertfunction-1321}, and the computation in the proof of Proposition~\ref{prop-verify-local}, together with some other pertinent accompanying files, can be found at
\url{https://github.com/huxw06/Hilbert-scheme-of-points}.\\

\subsection*{Acknowledgments}
I am grateful to Jian Zhou for sharing his insights on Conjecture~\ref{conj-intro-eularchar-global}.
I thank Zhilan Wang,  Hanlong Fang, Lei Song, Dun Liang, Feng Qu, and Yun Shi for helpful discussions. I thank Yongqiang Zhao for his interesting comment. I am indebted to the authors of Macaulay2 \cite{GS-Mac} and its packages for providing such a wonderful tool. I thank the referees for their helpful suggestions, especially those about Lemma~\ref{lem-transitive-open-nbd} and Remark~\ref{rem:proof_from_Krug_paper}.  I thank also Mark Wibrow, for I used his codes\footnote{https://tex.stackexchange.com/questions/145137/tikz-plane-partitions-with-labeled-faces} to draw the $3$-dimensional partition graphs in latex.

\section{Thomason's localization theorem and equivariant Hilbert functions}\label{sec:thomason-localization-theorem}
The purpose of this section is twofold. Firstly, we will recall Thomason's localization theorem, and in the case that the fixed locus consists of reduced isolated points, we  express the formula in terms of equivariant Hilbert functions. Secondly, in this restricted case, we remove the assumption of the existence of a global equivariant embedding into a regular algebraic space. 

Throughout this section, all schemes and algebraic spaces are over a base field $\Bbbk$. Let $G$ be a group scheme over $\Bbbk$, and let $X$ be a separated $\Bbbk$-algebraic space with an action of $G$, that is, a $\Bbbk$-morphism $\tau\colon G\times_{\Bbbk}X\rightarrow X$ which satisfies the usual requirements of group actions when $\tau$  is regarded as a functor on the category of $\Bbbk$-schemes. The fixed-point subfunctor $X^G$ of $X$ is defined as $X^G(Y)=X(Y)^G$ for any $\Bbbk$-scheme $Y$. By \cite[Proof of Theorem 7.1]{Mil17}, the functor $X^G$ is represented by a closed subspace of $X$, still denoted by~$X^G$.

Thomason's theorem is stated for diagonalizable group schemes $G$ (over a general base scheme $S$).  For our purpose, we consider only split tori $T$ over $\Bbbk$, \textit{i.e.}, $T\cong (\mathbb{G}_m^{r})_{\Bbbk}$ for some $r\geq 1$. Then $T\cong \Spec(\Bbbk[M])$,
where~$M$ is the group of characters of $T$ and $\Bbbk[M]$ is the group algebra associated with $M$. We can identify~$M$ with the standard lattice $\mathbb{Z}^r$ in $\mathbb{R}^r$. The representation ring $R(T)$ of $T$ is $\mathbb{Z}[M]$.
\subsection{Localization theorem}

For a regular $T$-space $Z$ over $\Bbbk$ and a coherent $T$-sheaf $\mathcal{F}$ on $Z$, there always exist a locally free $T$-sheaf~$\mathcal{P}$ and a $T$-equivariant surjection $\mathcal{P}\rightarrow \mathcal{F}$ (see \cite[Lemma 5.6]{Tho83}). So there is a $T$-equivariant locally free resolution $\mathcal{P}_{\centerdot}\rightarrow \mathcal{F}$, and we can define
\[
\Tor^{\mathcal{O}_Z}_i(\mathcal{F},\mathcal{G})=\mbox{the $\supth{i}$ homology of}\ 
\mathcal{P}_{\centerdot} \otimes_{\mathcal{O}_Z}\mathcal{G}\in K_0(T,Z),
\]
which is independent of the choice of $\mathcal{P}_{\centerdot}$. By \cite[Proposition 3.1]{Tho92}, $Z^T$ is a regular subspace. On each connected component $Z'$ of $Z^T$, the conormal sheaf $\mathcal{N}$ is locally free of constant rank, and 
\[
\Lambda_{-1}\mathcal{N}=\sum (-1)^i \Tor_i^{\mathcal{O}_Z}\left(\mathcal{O}_{Z^T},\mathcal{O}_{Z^T}\right)
\]
is invertible in $K_0(T,Z')_{(0)}$, the localization of $K_0(T,Z')$ as a module over $R(T)$ at the zero ideal of~$R(T)$. Now we are ready to state Thomason's localization theorem (see \cite[Theorem 6.4]{Tho86} and \cite[th\'eor\`eme~3.5]{Tho92}) in our setting.

\begin{theorem}[Thomason]\label{thm-Thomason-2}
Let $\Bbbk$ be a field and $T$ a split torus over $\Bbbk$. Let  $X$ be a  proper algebraic space over $\Bbbk$ with a  $T$-action, $Z$ a regular and  proper algebraic space over $\Bbbk$ with a $T$-action, and $j\colon X\rightarrow Z$ a $T$-equivariant closed immersion. Let $\mathcal{F}$ be a coherent $T$-sheaf over $X$. Let $\mathcal{N}$ be the conormal sheaf of\, $Z^T$ in $Z$. Then in the localized representation ring $R(T)_{(0)}$, we have the equality
\begin{equation}\label{eq-thomason-3}
\sum (-1)^i H^i(X,\mathcal{F})=\sum (-1)^k H^k\left(X^T,
\frac{\sum (-1)^i \Tor_i^{\mathcal{O}_Z}(j_* F,\mathcal{O}_{Z^T})}{
	\Lambda_{-1}\mathcal{N}
}\right).
\end{equation}
\end{theorem}

We are going to prove a theorem of a similar form, assuming that all the connected components of $X^T$ are reduced isolated $\Bbbk$-points and $\mathcal{F}$ is locally free, while not assuming a global embedding of $X$ into a regular $T$-space. The assumption of the local freeness of $\mathcal{F}$ is not essential.

Suppose one of the connected components of $X^T$ is a reduced isolated point $x\in X(\Bbbk)$. Thanks to \cite[Theorem 19.1]{AHD19}, there exists a $T$-equivariant \'etale morphism $\psi\colon (\Spec(A),w)\rightarrow (X,x)$ such that~$w$ is a $\Bbbk$-point of $\Spec(A)$ and $w$ is fixed by $T$.
Since $\Spec(A)^T=\Spec(A)\otimes_X X^T$, $w$ is a connected component of $\Spec(A)^T$; in particular, this means that this connected component is reduced. Let $\mathfrak{m}_w$ be the maximal ideal of $A$ corresponding to $w$. Then $\mathfrak{m}_w$ is a representation of the algebraic group $T$, and $\mathfrak{m}_w\rightarrow \mathfrak{m}_w/\mathfrak{m}_w^2$ is $T$-equivariant. Let $d=\dim_{\Bbbk}\mathfrak{m}_w/\mathfrak{m}_w^2$, \textit{i.e.}, the embedding dimension of $A$ at $w$. Recall that any representation of $T$ is decomposable (see \cite[Theorem 12.12]{Mil17}). So there exist $f_1,\dots,f_d\in \mathfrak{m}_w$ which are $T$-semi-variant and such that $\bar{f}_1,\dots,\bar{f}_d$ generate $\mathfrak{m}_w/\mathfrak{m}_w^2$. Suppose $T$ acts on $f_i$ by weight $w_i$ for $1\leq i\leq r$. 
Let $B=\Bbbk[Y_1,\dots,Y_d]$, and equip $B$ with a $T$-action by assigning the action on $Y_i$ by weight $w_i$. Then the homomorphism
\[
\varphi\colon B\longrightarrow A,\quad Y_i\longmapsto f_i,\quad 1\leq i\leq d
\]
is $T$-equivariant. Denote by $\phi\colon \Spec(A)\rightarrow \Spec(B)$ the morphism associated with $\varphi$. Let $\widehat{A}_w=\varprojlim_{i}A/\mathfrak{m}_w^i$ and $\widehat{B}=\Bbbk[[Y_1,\dots,Y_d]]$. The induced homomorphism
\[
\widehat{\varphi}\colon \widehat{B}\longrightarrow \widehat{A}_w
\]
is surjective. Let $\widehat{I}=\ker(\widehat{\varphi})$.

\begin{lemma}\label{lem-generatedbyeigenpolynomials-1}
The kernel $\widehat{I}$ is generated by $T$-semi-invariant polynomials; i.e, $\widehat{I}=(g_1,\dots,g_r)$, where the $g_i$ are polynomials of\, $Y_1,\dots,Y_d$, and  $T$ acts on $g_i$ by a certain weight $v_i$, for $1\leq i\leq r$. Moreover, $\deg g_i\geq 2$ for $1\leq i\leq r$.
\end{lemma}

Note that  there is in general not an action of the algebraic group $T$ on $\widehat{I}$, so we cannot use the decomposability theorem for representations of $T$.

\begin{proof}
By the Weierstrass preparation theorem (see \cite[Section~VIII.3, $\mathrm{n}^{\circ}$ 8, Proposition 6]{Bou65}), $\widehat{I}$ is generated by polynomials, say $h_1,\dots,h_m$. 
Each $h_i$ lies in a finite-dimensional sub-representation of $\Bbbk[Y_1,\dots,Y_d]$. So we have
$h_i=\sum_{a\in \mathcal{A}} h_{i,a}$ with the $h_{i,a}\in \Bbbk[Y_1,\dots,Y_d]$ such that $T(\bar{k})$ acts on $h_{i,a}$ in $\overline{\Bbbk}[Y_1,\dots,Y_d]$ by a weight~$v_a$ and the $v_a$ are pairwise distinct for $a\in \mathcal{A}$ with $\mathcal{A}$  a finite set of indices. 
It suffices to show $h_{i,a}\in \widehat{I}$ for each $a\in \mathcal{A}$. If $\Bbbk$ is an infinite field, this is obvious via the Vandermonde determinant. If $\Bbbk$ is finite, let $\widehat{J}=\widehat{I}+\sum_{a\in A}(h_{i,a})$. Then $\Bbbk[[Y_1,\dots,Y_d]]/\widehat{I}\rightarrow \Bbbk[[Y_1,\dots,Y_d]]/\widehat{J}$ is an isomorphism after base change to $\overline{\Bbbk}$ so is itself an isomorphism. Hence $\widehat{J}=\widehat{I}$ and thus $h_{i,a}\in \widehat{I}$ for $a\in \mathcal{A}$.
\end{proof}

\begin{lemma}\label{lem-nonzero-weights-1}
For $1\leq i\leq d$, the weight $w_i$ of the $T$-action on $Y_i$ is not zero.
\end{lemma}

\begin{proof}
Without loss of generality, suppose that $T$ acts on $Y_1$ trivially. There is a surjection
\begin{equation*}
\Bbbk[[Y_1,\dots,Y_d]]/\widehat{I}\longrightarrow \Bbbk[[Y_1,\dots,Y_d]]/\left(\widehat{I}+(Y_2,\dots,Y_d)\right)
\cong\Bbbk[[Y_1]]/\left(\left(\widehat{I}+(Y_2,\dots,Y_d)\right)\cap \Bbbk[[Y_1]]\right)
\end{equation*}
and thus a closed embedding
\[
\Spec\left(\Bbbk[[Y_1]]/\left(\left(\widehat{I}+(Y_2,\dots,Y_d)\right)\cap \Bbbk[[Y_1]]\right)\right)\longhookrightarrow \Spec\left(\widehat{A}_{w}\right),
\]
where the left-hand side is equal to $\Spec(\Bbbk[[Y_1]]/(Y_1^k))$ for some $k\geq 1$ and is $T$-fixed.  Since $\widehat{I}=(g_1,\dots,g_{r})$, we have
\[
\left(\widehat{I}+(Y_2,\dots,Y_d)\right)\cap \Bbbk[[Y_1]]=\left(g_1(Y_1,0,\dots,0),\dots,g_{r}(Y_1,0,\dots,0)\right).
\]
But no $g_i$ has a constant or linear term, so $\left(g_1(Y_1,0,\dots,0),\dots,g_{r}(Y_1,0,\dots,0)\right)\subset (Y_1^2)$, so $\Spec(A)^T$ is not reduced at $w$. This gives a contradiction because $X^T$ is reduced at $x$.
\end{proof}

Let $I$ be the ideal $(g_1,\dots,g_r)$ of $B$. Then $\widehat{B/I}$, the completion of $B/I$ at the ideal $(Y_1,\dots,Y_d)$, is isomorphic to $\widehat{B}/I\cong \widehat{A}_w$. So $\overline{\varphi}=\varphi/I\colon B/I\rightarrow A$ is \'etale at $w$. 

\begin{definition}
We call the 5-tuple $(\psi,\phi,A,B,I)$ an \emph{equivariant chart at $x$}.
\end{definition} 

In the above, we see that an equivariant chart exists for every reduced isolated component of $X^T$ which is a $\Bbbk$-point. Note that even when $X$ is a scheme, an equivariant  chart may not exist in the Zariski topology, \textit{i.e.}, if one demands that $\psi$ be an open immersion.
By abuse of notation, we also denote the point of $\Spec(B)$ corresponding to the maximal ideal $(Y_1,\dots,Y_d)$  by $x$. So we can speak of the conormal bundle $\mathcal{N}_x$  of $x$ (\textit{i.e.}, the cotangent space of $\Spec(B)$ at $x$)   in $\Spec(B)$. By Lemma~\ref{lem-nonzero-weights-1}, $\Spec(B)^T$ is reduced at $x$, and the weights of $\mathcal{N}_x$ are nonzero. 

The following lemma says that a locally free $T$-sheaf $\mathcal{F}$ on an equivariant chart is determined by its fiber at the fixed point.

\begin{lemma}\label{lem-locallfree-Tsheaf-formation}
Let $W=\Spec(A)$ be an affine $\Bbbk$-scheme with a $T$-action. Let $\mathfrak{m}$ be a $T$-invariant maximal ideal of\, $A$, and suppose that $\Bbbk\rightarrow A/\mathfrak{m}$ is an isomorphism. Let $\mathcal{F}$ be a locally free $T$-sheaf of finite rank on $W$. Let $V=\mathcal{F}/\mathfrak{m}\mathcal{F}$, the fiber of\, $\mathcal{F}$ at the closed point. 
 Then there exists a $T$-invariant open subset $U$ of\, $X$ containing $x$ such that there is an isomorphism of\, $T$-sheaves on $U$
\[
\mathcal{F}|_U\cong V\otimes_{\Bbbk} \mathcal{O}_U.
\]
\end{lemma}

\begin{proof}
Let $s=\rank \mathcal{F}$. Let $f_1,\dots,f_s$ be a basis of $\mathcal{F}/\mathfrak{m}\mathcal{F}$ such that $f_i$ is a $T$-eigenvector with weight $w_i$ for $1\leq i\leq s$. Consider the surjection
\[
\pi\colon\Gamma(X,\mathcal{F})\longrightarrow \mathcal{F}/\mathfrak{m}\mathcal{F}.
\]
For any affine $\Bbbk$-scheme $\Spec(R)$, there is an action of $T(R)$ on
\[
 \Gamma(X\times_{\Bbbk}\Spec(R),\mathcal{F}\otimes_{\Bbbk}R)=\Gamma(X,\mathcal{F})\otimes_{\Bbbk}R
 \]  
 and 
\[
\mathcal{F}\otimes_{\Bbbk}R/(\mathfrak{m}\mathcal{F}\otimes_{\Bbbk}R)=(\mathcal{F}/\mathfrak{m}\mathcal{F})\otimes_{\Bbbk}R 
\]
which is functorial in $R$ and such that $\pi$ is equivariant. So $\pi$ is a surjection of representations of the algebraic group $T$.  
Any representation of $T$ is completely decomposable (see \cite[Theorem 12.12]{Mil17}). So there exist $\tilde{f}_1,\dots,\tilde{f}_s\in \Gamma(X,\mathcal{F})$ such that $\pi(\tilde{f}_i)=f_i$ and $\tilde{f}_i$ is a $T$-eigen section with weight $w_i$. Then the morphism of $\mathcal{O}_X$-modules
\[
\psi\colon V\otimes_{\Bbbk} \mathcal{O}_X\longrightarrow \mathcal{F}
\]
induced by $f_i\mapsto \tilde{f}_i$ is an isomorphism in an open neighborhood of $x$. Moreover, by checking the actions of $T(A)$, 
one sees that $\psi$ is $T$-equivariant. So the open subset where $\psi$ is an isomorphism is $T$-invariant.
\end{proof}

Recall the concentration theorem \cite[Theorem 2.1]{Tho92}: The pushforward induced by the closed immersion $\iota\colon X^T\hookrightarrow X$, 
\begin{equation}\label{eq-concentration-theorem}
\iota_*\colon G(T,X^T)_{(0)}\longrightarrow G(T,X)_{(0)}, 
\end{equation}
is an isomorphism.

\begin{theorem}\label{thm-Thomason-reducedisolatefixed-0}
Let $\Bbbk$ be a field and $T$ a split torus over $\Bbbk$. Let  $X$ be a  proper algebraic space over $\Bbbk$ with a $T$-action. Let $\mathcal{F}$ be a locally free $T$-sheaf over $X$.  Suppose that all the connected components of\, $X^T$ are reduced isolated $\Bbbk$-points, denoted by $x_1,\dots,x_n$. 
For each fixed point $x_k$, let $(\psi_k,\phi_k, A_k,B_k,I_k)$ be an equivariant chart at~$x_k$,  let $N_k$ the cotangent space of\, $x_k$ in $\Spec(B_k)$, and let $d_k$ be the embedding dimension of\, $X$ at $x_k$. Let $\mathcal{F}/\mathfrak{m}_{k}\mathcal{F}$ be the fiber of\, $\mathcal{F}$ at $x_k$.
Then
\begin{equation}\label{eq-localizationtheorem-withoutembedding-0}
(\iota_*)^{-1}(\mathcal{F})=\sum_{k=1}^n \left(\mathcal{F}/\mathfrak{m}_{k}\mathcal{F}\cdot \frac{\sum_{i=0}^{d_k}(-1)^i\Tor_i^{B_k}\left(B_k/I_k,\kappa(x_k)\right)}{\Lambda_{-1}N_k}\right),
\end{equation}
where the term in the sum corresponding to $x_k$ is regarded as a sheaf supported at $x_k\in X^T$.
\end{theorem}

\begin{proof}
By the concentration theorem, it suffices to show 
\begin{equation}\label{eq-localizationtheorem-withoutembedding-0-1}
	\mathcal{F}=\iota_*\sum_{k=1}^n \left(\mathcal{F}/\mathfrak{m}_{x_k}\mathcal{F}\cdot \frac{\sum_{i=0}^{d_k}(-1)^i\Tor_i^{B_k}\left(B_k/I_k,\kappa(x_k)\right)}{\Lambda_{-1}N_k}\right).
\end{equation}
Let $U_k=\Spec(A_k)\backslash(\Spec(A_k)^T-\{w_k\})$. Then $U_k$ is an open subscheme of $\Spec(A_k)$, and $U_k^T=w_k$. We still denote the composition $U_k\hookrightarrow A_k\xrightarrow{\psi_k}X$  by $X$. Then we have the following cartesian diagram:
\[
\xymatrix{
	X^T \ar@{=}[r] \ar@{^{(}->}[d]_{\iota'} & X^T \ar@{^{(}->}[d]^{\iota} \\
	\bigsqcup_{k=1}^n U_k \ar[r]^>>>>>>{\sqcup \psi_k} & X\rlap{.}
}
\]
Set $\psi=\sqcup \psi_k$; thus $\psi$ is an \'etale morphism. We have the following commutative diagram of localized $K$-theory:
\[
\xymatrix{
	& G(T,X^T)_{(0)} \ar[ld]_{\iota'_*} \ar[rd]^{\iota_*} & \\
	\oplus_{k=1}^n G(T,U_k)_{(0)} && G(T,X)_{(0)}\rlap{.} \ar[ll]_{\psi^*}
	}
\]
By the concentration theorem, both $\iota_*$ and $\iota'_*$ are isomorphisms, so $\psi^*$ is an isomorphism. So it suffices to show (\ref{eq-localizationtheorem-withoutembedding-0-1}) after replacing $X$ by $\bigsqcup_{k=1}^n U_k$. Then it suffices to prove the theorem for $X=U_k$. Denote the composite $U_k\hookrightarrow \Spec(A_k)\rightarrow \Spec(B_k)$ by $\overline{\phi}_k$. Denote the closed immersion $\Spec(B_k/I_k)\hookrightarrow \Spec(B_k)$ by~$j_k$. By Lemma~\ref{lem-nonzero-weights-1}, $\Spec(B_k)^T=\Spec(B_k/I_k)=x_k$. So we have a cartesian diagram
\[
\xymatrix{
	w_k \ar@{=}[r] \ar@{^{(}->}[d]_{\iota'_k} &  x_k \ar@{=}[r] \ar@{^{(}->}[d]_{\iota''_k} & x_k \ar@{^{(}->}[d]_{\iota'''_k}\\
	U_k \ar[r]^>>>>>{\overline{\phi}_k} & \Spec(B_k/I_k) \ar@{^{(}->}[r]^{j_k} & \Spec(B_k)\rlap{.}
}
\]
By Lemma~\ref{lem-locallfree-Tsheaf-formation}, 
\[
\mathcal{F}|_{U_k}\cong \overline{\phi}_k^* \left(\left(\mathcal{F}/\mathfrak{m}_{x_k}\mathcal{F}\right)\otimes_{\Bbbk} \mathcal{O}_{B_k/I_k}\right).
\]
Applying the concentration  to $U_k$ and to $\Spec(B_k/I_k)$, we obtain that the localized map
\[
\overline{\phi}_k^*\colon G\left(T,\Spec(B_k/I_k)\right)_{(0)}\longrightarrow G(T,U_k)_{(0)}
\]
is an isomorphism. So we are reduced to showing (\ref{eq-localizationtheorem-withoutembedding-0-1}) for $X=\Spec(B_k/I_k)$ and $\mathcal{F}=(\mathcal{F}/\mathfrak{m}_{x_k}\mathcal{F})\otimes_{\Bbbk} \mathcal{O}_{B_k/I_k}$.

Finally, we have the commutative diagram
\[
\xymatrix{
	& G(T,x_k)_{(0)} \ar[ld]_{\iota''_{k*}} \ar[dr]^{\iota'''_{k*}}\\
	G\left(T,\Spec(B_k/I_k)\right)_{(0)} \ar[rr]^{j_{k*}} && G\left(T,\Spec(B_k)\right)_{(0)}\rlap{.}
}
\]
By the concentration theorem, $\iota''_{k*}$ and $\iota'''_{k*}$ are isomorphisms, so $j_{k*}$ is an isomorphism. So it suffices to show
\[
\left(\mathcal{F}/\mathfrak{m}_{x_k}\mathcal{F}\right)\otimes_{\Bbbk} j_{k*}\mathcal{O}_{B_k/I_k}
=\iota'''_*\left(\mathcal{F}/\mathfrak{m}_{x_k}\mathcal{F}\cdot \frac{\sum_{i=0}^{d_k}(-1)^i\Tor_i^{B_k}\left(B_k/I_k,\kappa(x_k)\right)}{\Lambda_{-1}N_k}\right)
\]
in $G\left(T,\Spec(B_k)\right)_{(0)}$, or equivalently
\[
(\iota'''_*)^{-1}\left(j_{k*}\mathcal{O}_{B_k/I_k}\right)=\frac{\sum_{i=0}^{d_k}(-1)^i\Tor_i^{B_k}\left(B_k/I_k,\kappa(x_k)\right)}{\Lambda_{-1}N_k}.
\]
But this is the localization theorem for  regular schemes; see  \cite[Lemma 3.3]{Tho92}. So we have completed the proof.
\end{proof}

\begin{corollary}\label{cor-Thomason-reducedisolatefixed-1}
With the assumption and notation of Theorem~\ref{thm-Thomason-reducedisolatefixed-0}, we have
\begin{equation}\label{eq-localizationtheorem-withoutembedding-1}
\sum_{i=0}^{\infty}(-1)^i H^i(X,\mathcal{F})=\sum_{k=1}^n \left(\mathcal{F}/\mathfrak{m}_{k}\mathcal{F}\cdot \frac{\sum_{i=0}^{d_k}(-1)^i\Tor_i^{B_k}\left(B_k/I_k,\kappa(x_k)\right)}{\Lambda_{-1}N_k}\right)
\end{equation}
and
\begin{equation}\label{eq-localizationtheorem-withoutembedding-2}
\sum_{i=0}^{\infty}(-1)^i H^i(X,\mathcal{F})=\sum_{k=1}^n \left(\mathcal{F}/\mathfrak{m}_{k}\mathcal{F}\cdot \frac{\sum_{i=0}^{d_k}(-1)^i\Tor_i^{\widehat{B_k}}\left(
\left(\widehat{A_k}\right)_{w_k},\kappa(x_k)\right)}{\Lambda_{-1}N_k}\right)
\end{equation}
in $R(T)_{(0)}$.
\end{corollary}

\begin{proof}
Pushing both sides of (\ref{eq-localizationtheorem-withoutembedding-0}) to $\Spec(\Bbbk)$, we obtain the equality in $R(T)_{(0)}$. 

For (\ref{eq-localizationtheorem-withoutembedding-2}), recall that $(\widehat{A_k})_{w_k}\cong \widehat{B_k/I_k}$. Taking an equivariant resolution of $B_k/I_k$ and base change to $\widehat{B_k}$, we obtain
\[
\Tor_i^{\widehat{B_k}}\left(
\left(\widehat{A_k}\right)_{w_k},\kappa(x_k)\right)=\Tor_i^{B_k}\left(B_k/I_k,\kappa(x_k)\right).
\]
Then (\ref{eq-localizationtheorem-withoutembedding-2}) follows from (\ref{eq-localizationtheorem-withoutembedding-1}).
\end{proof}

\begin{remark}
The advantage of Formula (\ref{eq-localizationtheorem-withoutembedding-2}) is that the right-hand side depends  only on the completed local ring $\widehat{\mathcal{O}}_{X,x}$ at the fixed points $x\in X^T$. When $X$ represents a moduli functor, we do not need to find a global embedding of $X$ but only need to solve the corresponding formal deformation problem at the fixed points.
\end{remark}

\begin{remark}
Although in Theorem~\ref{thm-Thomason-reducedisolatefixed-0} and Corollary~\ref{cor-Thomason-reducedisolatefixed-1}, we do not assume the existence of an embedding of a (Zariski) local chart at a fixed point into a regular algebraic space $Z$ over $\Bbbk$, in practice to show that the isolated (which is comparatively easy to verify set-theoretically) fixed points of $X$ are reduced,  it turns out to be convenient to find such a local embedding satisfying that the fixed points of $Z$ are isolated since the fixed locus of a regular $T$-space is  regular; see \cite[Proposition 3.1]{Tho92}. See Section~\ref{sec:equi-embeddings-in-Grass} for examples.
\end{remark}

\begin{remark}
In practice, the \'etale morphism $\psi$ in the data of an equivariant chart $(\psi,\phi,A,B,I)$  can  usually be taken as an open immersion. This is the case if there is an equivariant embedding of $X$ into a geometrically normal $\Bbbk$-scheme $Z$ since by Sumihiro's theorem \cite[Corollary 3.11]{Sum75}, such $Z$ can be covered by $T$-invariant affine open subsets. Moreover, when $X$ is a normal projective scheme over $\Bbbk$ with a $T$-action, an equivariant immersion of $X$ into a projective space always exists (see \cite[Corollary 1.6 and Proposition~1.7]{MFK94}).
\end{remark}

Let us introduce some notation for our forthcoming computations. Let $\mathbf{t}=(t_1,\dots,t_r)$ be the generic point of $T\cong \mathbb{G}_{m}^{r}$; one can also regard it as a formal symbol. For a character $w=(a_1,\dots,a_r)\in M$, the trace of $\mathbf{t}$ is $\mathbf{t}^{w}=t_1^{a_1}\cdots t_r^{a_r}$. We use this formal product $\mathbf{t}^{w}$ to represent a $1$-dimensional representation in the representation ring $R(T)$. Thus direct sums (resp.\ tensor products) of representations correspond to sums  (resp.\ products) of polynomials of $t_1,\dots,t_r$.

\begin{example}\label{example-Thomason-localization}
Let $Y\subset \mathbb{P}^n$ be the union of the $n+1$ coordinate lines with the reduced subscheme structure. For $l\in \mathbb{Z}$, equip $\mathcal{O}(l)$ with the $T=\mathbb{G}_m^{n+1}$-linearization induced by the $T$-action on the module $\Bbbk[X_0,\dots,X_n](l)$. From (\ref{eq-localizationtheorem-withoutembedding-2}) we have
\[
\chi(Y,\mathcal{O}(l))=
\sum_{i=0}^n \left(t_i^{l}\cdot\frac{1-\prod_{j\neq i}\frac{t_j}{t_i}}{\prod_{j\neq i}\left(1-\frac{t_j}{t_i}\right)}\right).
\]
One can easily check this result by a d\'evissage on the components of $X$.

The results in this section can be generalized to diagonalized groups. In this example, suppose that~$\Bbbk$ contains $\mathbb{Q}(\zeta_{n+1})$. Let $G=(\mu_{n+1})_{\Bbbk}$  be the constant group scheme $\mathbb{Z}/(n+1)\mathbb{Z}$ that acts on $\mathbb{P}^n$ by cyclicly permuting the coordinates $X_0,\dots,X_{n+1}$. This induces an action on $Y$. But with this action, $Y$ has no fixed points. Write $\mathbb{X}(G)=\mathbb{Z}[\lambda]/(\lambda^{n+1}-1)$. Then by the localization theorem, $\chi(Y,\mathcal{O}(l))$ vanishes in $\mathbb{X}(G)_{\rho}$, where $\rho$ is the prime ideal generated by $\lambda^n+\lambda^{n-1}+\dots+1$. One can check this by observing that $\chi(Y,\mathcal{O}(l))$ is a direct sum of copies of regular representations of $\mathbb{Z}/(n+1)\mathbb{Z}$.
\end{example}

\subsection{Equivariant Hilbert functions}\label{sec:equivariant-Hilbertfunction}
\begin{definition}\label{def-equi-Hilb-function}
Let $S=\Bbbk[x_1,\dots,x_d]$ or $\Bbbk[[x_1,\dots,x_d]]$. Suppose $T$ acts on $x_i$ by weight $w_i$, for $1\leq i\leq d$, and extend the $T$-action to the ring $\Bbbk[x_1,\dots,x_d]$, and to $\Bbbk[[x_1,\dots,x_d]]$ continuously. Suppose $w_i\neq 0$ for $1\leq i\leq d$. For an $T$-invariant ideal $I$ of $S$ and the associated quotient $R=S/I$, we define the equivariant Hilbert function of $R$ to be
\begin{equation}\label{eq-def-equi-Hilb-function}
H(R;\mathbf{t})=\frac{\sum_{i=0}^{d}(-1)^i\Tor_i^{S}(R,k)}{\prod_{j=1}^d\left(1-\mathbf{t}^{w_j}\right)}.
\end{equation}
\end{definition}

Then we rewrite (\ref{eq-localizationtheorem-withoutembedding-2}) as
\begin{equation}\label{eq-Thomason-reducedisolatefixed-2}
\sum (-1)^i H^i(X,\mathcal{F})=\sum_{x\in X^T} (\mathcal{F}_x/\mathfrak{m}_x \mathcal{F}_x)\cdot H(\widehat{\mathcal{O}}_{X,x};\mathbf{t}).
\end{equation}

By the Weierstrass preparation theorem and Lemma~\ref{lem-generatedbyeigenpolynomials-1}, for $S=\Bbbk[[x_1,\dots,x_d]]$ and a $T$-invariant ideal~$I$, there exists a $T$-invariant $I_0$ of $S_0=\Bbbk[x_1,\dots,x_d]$ such that $S/I=S_0/I_0\otimes_{S_0} S$. Let $R_0=S_0/I_0$. By the flatness of $S_0\rightarrow S$, we have $\Tor_i^{S}(R,\Bbbk)=\Tor_i^{S_0}(R_0,\Bbbk)$. So we only need  to study the case $S=\Bbbk[x_1,\dots,x_d]$.

A $T$-equivariant $S$-module is equivalent to an $M$-graded $S$-module (recall that $M$ is the character group of $T$). Regarding $R$ as an $M$-graded $S$-module, the numerator of (\ref{eq-def-equi-Hilb-function}) is no other than the \emph{$K$-polynomial} of~$R$ (see \cite[Definition 8.32]{MS05}). So when the weights $w_1,\dots,w_d$ are all positive, or more generally lie in a convex sector of $\mathbb{R}^r\supset\mathbb{Z}^r=M$, our equivariant Hilbert function $H(R;\mathbf{t})$ coincides with the \emph{multigraded Hilbert function} of $R$ (see \cite[Theorem 8.20]{MS05}). This justifies the name.\\

In the following, we give a brief survey on the computation of $H(S/I;\mathbf{t})$.  By definition, one may compute $\Tor_i^S(S/I,\Bbbk)$ by finding a $T$-equivariant resolution, or equivalently a multigraded resolution of $S/I$. But this is rather inefficient because one needs to compute a Gr\"obner basis many times. Another way is using the Koszul resolution of $\Bbbk$ as an $S$-module. This time it is difficult to compute the homology. A much more efficient algorithm is given in \cite{BS92}.

In the first step, we apply the following theorem (see \cite[Theorem 8.36]{MS05}), which  generalizes  a famous theorem of Macaulay to general, not necessarily positive, multigraded modules.

\begin{proposition}\label{prop-hilbertseries-gbbasis}
We have $
H(S/I;\mathbf{t})=H(S/\operatorname{in}_{<}(I);\mathbf{t})
$.
\end{proposition}

This reduces the computation of $H(S/I;\mathbf{t})$ to the case of monomial ideals. 
Then one can compute $H(S/\operatorname{in}_{<}(I);\mathbf{t})$ by using a resolution of the monomial ideal $\operatorname{in}_{<}(I)$, \textit{e.g.}, the Taylor complex. Alternatively, we will use the following lemma. 

\begin{lemma}\label{lem-key}
Let $J_{m}=(f_1,\dots,f_{m-1},f_m)$, where the $f_i$ are eigen-polynomials under the $T$-action. Denote the weight of\, $f_i$ by $\mathbf{w}(f_i)$. 
Let $J_{m-1}=(f_1,\dots,f_{m-1})$. Then
\begin{equation}\label{eq-Kpolynomial-reduction}
\sum_{i=0}^d(-1)^i\Tor_i(S/J_m,\Bbbk)=\sum_{i=0}^d(-1)^i\Tor_i(S/J_{m-1},\Bbbk)
- \mathbf{t}^{\mathbf{w}(f_m)}\sum_{i=0}^d(-1)^i\Tor_i\left(S/\left(J_{m-1}:(f_m)\right),\Bbbk\right).
\end{equation}
\end{lemma}

\begin{proof}
We have exact sequences
\[
0\longrightarrow S/(J_{m-1}\cap (f_m))\longrightarrow S/J_{m-1}\oplus S/(f_m)\longrightarrow R/J_{m}\longrightarrow 0 
\]
and
\[
0\longrightarrow S/(J_{m-1}\cap (f_m)\colon (f_m))\xrightarrow{\times f_m} S/(J_{m-1}\cap (f_m))\longrightarrow S/(f_m)\longrightarrow 0. 
\]
By the definition of quotient ideals, $\left(J_{m-1}\cap (f_m):(f_m)\right)=\left(J_{m-1}:(f_m)\right)$. Now (\ref{eq-Kpolynomial-reduction}) follows from the additivity of~$\sum_{i=0}^d(-1)^i\Tor_i(\cdot,\Bbbk)$.
\end{proof}
When $f_1,\dots,f_m$ are monomials, both $J_{m-1}$ and $\left(J_{m-1}:(f_m)\right)$ are monomial ideals generated by at most $m-1$ monomials. So we can compute $H(S/J_m;\mathbf{t})$ recursively.

Finally, we recall a theorem of Stanley \cite[Theorem 4.4]{Sta78}. Let $R$ be the ring in Definition~\ref{def-equi-Hilb-function}.

\begin{theorem}[Stanley]\label{thm-Stanley-Gorenstein}
Let $R$ be a ring as in Definition~\ref{def-equi-Hilb-function}, and assume that $R$ is Cohen--Macaulay of Krull dimension $d$ and the weights $w_i$ lie in a strictly convex cone. Then $R$ is Gorenstein if and only if there is a multi-index $\alpha\in \mathbb{Z}^r$ such that
\begin{equation}\label{eq-thm-Stanley}
	H\left(R;t_1^{-1},\dots,t_r^{-1}\right)=(-1)^d \mathbf{t}^{\alpha} H(R;t_1,\dots,t_r).
\end{equation}
\end{theorem}

Stanley's theorem is stated for graded algebras with positive degrees in $\mathbb{Z}$. When the weights $w_i$ lie in a strictly convex cone, we can find a subtorus of $T$ such that the induced gradings on the variables are strictly positive.

\section{Equivariant embeddings in Grassmannians and fixed points}\label{sec:equi-embeddings-in-Grass}
\subsection{Equivariant embeddings}
Let $V$ be the $(r+1)$-dimensional vector space over $\Bbbk$ spanned by $X_0,\dots,X_r$.
Let $\mathbb{P}^r=\Proj(\Bbbk[X_0,\dots,X_r])$. For a graded ideal $I$ of $\Bbbk[X_0,\dots,X_r]$, we denote by $\tilde{I}$ the sheaf of ideals associated with $I$. We say that $I$ is $m$-regular, and also that $\tilde{I}$ is $m$-regular,  if $H^i(\mathbb{P}^r,\tilde{I}(m-i))=0$ for $i>0$.  For a polynomial $\Phi(z)\in \mathbb{Q}[z]$ which takes integer values for $z\in \mathbb{Z}$, set
\[
 \sigma(\Phi)=\inf\{m: \mathcal{I}_Z\ \mbox{is $m$-regular for every closed subscheme $Z\subset \mathbb{P}^{r}$ with Hilbert polynomial $\Phi$}
 \}.
 \] 
 For a $\Bbbk$-vector space $W$, denote by $\Grass(n,W)$ the moduli scheme of dimension $n$ quotients
 of $W$. Then for any $d\geq \sigma(\Phi)$, there is a closed embedding
\[
\alpha\colon \Hilb^{\Phi}(\mathbb{P}^r)\longhookrightarrow 
\Grass\left(\Phi(d),\Sym^d(V)\right),
\]
\[
\alpha(Z)=\left[H^0(\mathbb{P}^r,\mathcal{O}(d))\longtwoheadrightarrow H^0(\mathbb{P}^r,\mathcal{O}_Z(d))\right].
\]

Let $\mathbb{G}_{m,\Bbbk}^{r+1}$ act on $X_0,\dots,X_r$ by $(t_0,\dots,t_r)\cdot X_i=t_i X_i$. Let $T\cong \mathbb{G}_{m,\Bbbk}^{r}$ be the subtorus of $\mathbb{G}_{m,\Bbbk}^{r+1}$ defined by $t_0\cdots t_r=1$.   There are induced actions of $T$ on $\mathbb{P}^r$, and thus on $\Hilb^{\Phi}(\mathbb{P}^r)$ and $\Grass(\Phi(d),\Sym^d(V))$, rendering $\alpha$  $T$-equivariant.

\begin{lemma}\label{lem-fixedloci-Grassmannian}
The fixed loci of the induced $T$-action on $\Grass(n,\Sym^d(V))$ consist of reduced isolated $\Bbbk$-points. Let $S_d$ be the set of monomials in $X_0,\dots,X_r$ of degree $d$. Then the fixed points of\, $\Grass(n,\Sym^d(V))$ correspond bijectively to the  subsets of\, $S_d$ of cardinality $\binom{d+r}{r}-n$.
\end{lemma}

\begin{proof} Since $\Grass(n,\Sym^d(V))$ is smooth, the fixed loci are regular (see \cite[Proposition 1.3]{Ive72} and \cite[Proposition 3.1]{Tho92}). So it suffices to show that the fixed points are isolated $\Bbbk$-points.  The monomials in $X_0,\dots,X_r$ of degree $d$ form a basis of $\Sym^d(V)$. The weights of the $T$-action on these monomials are pairwise distinct.  Let $W$ be a $T$-invariant subspace of $\Sym^d(V)$. By the complete decomposability of the representations of $T$, the space $W$ is spanned by a subset of $S_d$. So the fixed points are isolated $\Bbbk$-points.  The second statement also follows.
\end{proof}

\begin{corollary}\label{cor-fixedloci-Hilb-Projspace}
The fixed loci of the induced $T$-action on $\Hilb^{\Phi}(\mathbb{P}^r)$ consist of reduced isolated  $\Bbbk$-points.
\end{corollary}

\begin{proof}
This follows from Lemma~\ref{lem-fixedloci-Grassmannian} by using the embedding $\Hilb^{\Phi}(\mathbb{P}^r)^T\hookrightarrow \Grass(\Phi(d),\Sym^d(V))^T$, where $d\geq \sigma(\Phi)$.
\end{proof}

Now let $Y$ be a smooth proper toric variety of dimension $r$. Then $Y$ contains $T=\mathbb{G}_{m,\Bbbk}^{r}$ as a dense open subset, and $T$ acts on $Y$ in a natural way. This induces a $T$-action on $\Hilb^n(Y)$, the Hilbert scheme parametrizing length $n$ closed subschemes on $Y$.

\begin{proposition}\label{prop-equivariant-immersion-toric}\leavevmode
\begin{i-enumerate}
	\item\label{peit-1} If\, $Y$ is projective, there exists a $T$-equivariant closed immersion of\, $\Hilb^n(Y)$ into a smooth and projective $\Bbbk$-scheme.
	\item\label{peit-2} The fixed loci of the $T$-action on $\Hilb^n(Y)$ consist of reduced isolated $\Bbbk$-points.
\end{i-enumerate}
\end{proposition}

\begin{proof}
\eqref{peit-1}~ By \cite[Corollary 1.6]{MFK94}, there exists a $T$-equivariant immersion of $Y$ into a projective space $\mathbb{P}^{N}$ with a $T$-action. This $T$-action induces a $T$-equivariant closed immersion of $\Hilb^n(\mathbb{P}^{N})$ into a certain Grassmannian. Precomposing this immersion with the $T$-equivariant $\Hilb^n(Y)\hookrightarrow \Hilb^n(\mathbb{P}^{N})$, we are done. But note that with this $T$-action, the fixed loci on $\mathbb{P}^{N}$ may not be isolated.

\eqref{peit-2}~ Let $Z$ be a $T$-fixed length $n$ closed subscheme of $Y$. Then $Z=\sqcup_{i=1}^{k}Z_i$, where $Z_i$ is a $T$-fixed length $n_i$ subscheme supported at a $T$-fixed point $y_i$ of $Y$, satisfying $\sum_{i=1}^k n_i=n$.
   The $T$-fixed points of $Y$ form a finite set of $\Bbbk$-points, and each fixed point $y$ has a $T$-invariant open neighborhood $U_y$ such that there is a $T$\nobreakdash-equivariant open immersion $U_y\rightarrow \mathbb{A}^r_{\Bbbk}$.  There is a $T$-equivariant open immersion $\Hilb^{n_i}(U_{y_i})\hookrightarrow \Hilb^{n_i}(\mathbb{P}^r)$. Then Corollary~\ref{cor-fixedloci-Hilb-Projspace} implies that the fixed loci of $\Hilb^{n_i}(U_{y_i})$ consist of reduced $\Bbbk$-points. Then each $Z_i$ is a reduced $\Bbbk$-point of $\Hilb^{n_i}(U_{y_i})$. So $Z$ is a $\Bbbk$-point. Moreover, since $[Z]\in \Hilb^n(Y)$ shares a common \'etale neighborhood  with $[Z_1]\times\cdots\times [Z_k]\in\prod_i^k \Hilb^{n_i}(U_{y_i})$, there is an equivariant isomorphism
   \[
   \widehat{\mathcal{O}}_{\Hilb^n(Y),[Z]}\cong \widehat{\bigotimes}_{i=1}^{k} \widehat{\mathcal{O}}_{\Hilb^{n_i}(U_{y_i}),[Z_i]}; 
   \]
   hence $[Z]$ is a reduced isolated fixed point.
\end{proof}

\begin{remark}
The embedding of $\Hilb^{\Phi}(\mathbb{P}^r)$ in Grassmannians can be explicitly defined as a determinantal scheme of a homomorphism of two vector bundles on the Grassmannians. 
We refer the reader to \cite{Got78,IK99,Bay82} and \cite[Section~4]{HS04} for an account of various embeddings with explicit equations. In practice, these global embeddings are too complicated for computing the equivariant Hilbert functions, unless one can find a uniform projective resolution of the local rings at the fixed points $x$ of $\Hilb^{\Phi}(\mathbb{P}^r)$, or at least a uniform way to describe the generators of the initial ideal of $I_x$. In Section~\ref{sec:local-equations-Hilbert-schemes}, we will take another approach for constant Hilbert polynomials $\Phi(z)\equiv n$.
\end{remark}

\subsection{Saturated monomial ideals}
Recall that a \emph{saturated ideal} of $\Bbbk[X_0,\dots,X_r]$ is a homogeneous ideal $I$ satisfying that $s\in I$ if for each $0\leq i\leq r$, there exists an $m\geq 0$ such that $X_i^m s\in I$. There is a one-one  correspondence between the closed subschemes of $\mathbb{P}^r$ and the saturated ideals of $k[X_0,\dots,X_r]$.

The $T$-fixed points of $\Hilb(\mathbb{P}^r)$ correspond one-one to saturated ideals of $k[X_0,\dots,X_r]$ generated by a finite set of monomials in $X_0,\dots,X_r$.

\begin{definition}
Let $I$ be a monomial ideal of $\Bbbk[X_1,\dots,X_r]$. The minimal monomial generators $(X^{\alpha})_{\alpha\in \mathcal{A}}$ of $I$ are unique, where $\alpha\in \mathbb{Z}_{\geq 0}^{r}$ and $\mathcal{A}$ is a finite set of indices.  The \emph{affine monomial datum of dimension $r$} associated with $I$ is the set $\mathcal{A}_I=\mathcal{A}$, regarded as a subset of lattice points in $\mathbb{Z}_{\geq 0}^{r}$. 
\end{definition}

\begin{definition}
A \emph{projective monomial datum of dimension $r$} is an $(r+1)$-tuple $(\mathcal{A}_0,\dots,\mathcal{A}_r)$, where $\mathcal{A}_i$ is an affine monomial datum of dimension $r$.
\end{definition}

Let $I\subset \Bbbk[X_0,\dots,X_r]$ be a saturated monomial ideal. For $0\leq i\leq r$, the localized ideal $I_{X_i}$ of $\Bbbk[\frac{X_0}{X_i},\dots,\frac{X_{i-1}}{X_i},\frac{X_{i+1}}{X_i},\dots,\frac{X_r}{X_i}]$ is a monomial ideal. Let $\mathcal{A}_i$ be the affine monomial datum associated with $I_{X_i}$. Then we call 
\[
\mathcal{P}_I=(\mathcal{A}_0,\dots,\mathcal{A}_r)
\]
the \emph{projective monomial datum} associated with $I$.

\begin{lemma}
The assignment $I\mapsto \mathcal{P}_I$ is a bijection from the set of the saturated monomial ideals of\, $\Bbbk[X_0,\dots,X_r]$ to the set of projective monomial data of dimension $r$.
\end{lemma}

\begin{proof} 
Let $\mathcal{P}_I=(\mathcal{A}_0,\dots,\mathcal{A}_r)$ be a projective monomial datum. For each $\mathcal{A}_i$, choose $m_i$ sufficiently large such that
\[
g_{\alpha}:=X_i^{m_i}\cdot \left(\frac{X_0}{X_i}\right)^{a_0}\cdots \left(\frac{X_{i-1}}{X_i}\right)^{a_{i-1}}\cdot 
\left(\frac{X_{i+1}}{X_i}\right)^{a_{i+1}}\cdots \left(\frac{X_{r}}{X_i}\right)^{a_{r}}\in \Bbbk[X_0,\dots,X_r]
\]
for all $\alpha=(a_0,\dots,a_{i-1},a_{i+1},\dots,a_{r})\in \mathcal{A}_i$. Let $J_{\mathcal{P}}$ be the ideal of $\Bbbk[X_0,\dots,X_r]$ generated by the $g_{\alpha}$, where $\alpha$ runs over $\mathcal{A}_0,\dots,\mathcal{A}_r$. Let $I_{\mathcal{P}}$ be the saturation of $J_{\mathcal{P}}$. Then the saturated monomial ideal $I_{\mathcal{P}}$ is  independent of the choice of $\{m_i\}_{0\leq i\leq r}$, and the assignment $\mathcal{P}\mapsto I_{\mathcal{P}}$ is an inverse to $I\mapsto \mathcal{P}_I$.
\end{proof}

\subsection{Monomial ideals of finite colength}\label{sec:Monomial-ideals-of-finite-colength}

For $\alpha=(a_1,\dots,a_r)$ and $\beta=(b_1,\dots,b_r)$ in $\mathbb{Z}^r$, we say $\alpha\leq \beta$ if $a_i\leq b_i$ for $1\leq i\leq r$. In this paper, an \emph{$r$-dimensional partition} of $n$ is a set $\lambda\subset \mathbb{Z}_{\geq 0}^r$ with $|\lambda|=n$ satisfying that if $\beta\in \lambda$ and $\alpha\leq \beta$, then $\alpha\in \lambda$. Thus, a 2-dimensional partition is  a partition in the usual sense, and a $3$-dimensional partition is usually called a plane partition.

If $I$ is a monomial ideal of $\Bbbk[X_1,\dots,X_r]$ with finite colength $n$, there is a unique $r$-dimensional partition~$\lambda$ of $n$ such that $\mathcal{A}_I$ is the set of minimal lattices of the complement $\mathbb{Z}_{\geq 0}^r\backslash \lambda$. We denote this $\lambda$ by $\lambda_I$. The set of monomials $\{X^\beta\}_{\beta\in \lambda_I}$ is a $\Bbbk$-basis of $\Bbbk[X_1,\dots,X_r]/I$. 
The map $\lambda\mapsto \lambda_I$ is a bijection between the monomial ideals of $\Bbbk[X_1,\dots,X_r]$ with finite colength $n$ and the $r$-dimensional partitions of $n$. We denote the inverse map by $\lambda\mapsto I_{\lambda}$.

We can present  $r$-dimensional partitions graphically. To each lattice point $\mathbf{i}=(i_1,\dots,i_r)\in \mathbb{Z}^r$, we assign a \emph{box} 
\[
B_{\mathbf{i}}=\{(x_1,\dots,x_r)\in \mathbb{R}^r\mid i_k\leq x_k\leq i_k+1 \mbox{for}\ 1\leq k\leq r\}.
\]
Then $B_{\lambda}:=\bigcup_{\mathbf{i}\in \lambda}B_{\mathbf{i}}$ is a graphical presentation of $\lambda$. For example, for the monomial ideal
\[
I=\left(X_1^3, X_1^2 X_2, X_1 X_3, X_2^2,X_2 X_3, X_3^2\right),
\]
the $3$-dimensional partition 
\begin{equation}\label{eq-partition-132}
\lambda_I=\{(0,0,0),(1,0,0),(2,0,0),(0,1,0),(1,1,0),(0,0,1)\}
\end{equation}
 is presented as
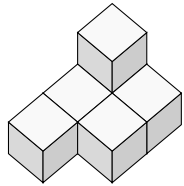
\begin{center}
\begin{tikzpicture}[x=(220:0.6cm), y=(-40:0.6cm), z=(90:0.42cm)]

\foreach \m [count=\y] in {{2,1,1},{1,1}}{
  \foreach \n [count=\x] in \m {
  \ifnum \n>0
      \foreach \z in {1,...,\n}{
        \draw [fill=gray!30] (\x+1,\y,\z) -- (\x+1,\y+1,\z) -- (\x+1, \y+1, \z-1) -- (\x+1, \y, \z-1) -- cycle;
        \draw [fill=gray!40] (\x,\y+1,\z) -- (\x+1,\y+1,\z) -- (\x+1, \y+1, \z-1) -- (\x, \y+1, \z-1) -- cycle;
        \draw [fill=gray!5] (\x,\y,\z)   -- (\x+1,\y,\z)   -- (\x+1, \y+1, \z)   -- (\x, \y+1, \z) -- cycle;  
      }
 \fi
 }
}    

\end{tikzpicture}
\end{center}

We also need  a compact way to present $r$-dimensional partitions for $r\geq 2$. 
If $\lambda$ is a $2$-dimensional partition of $n$, let 
\[
\lambda_{i}=\{a\in \mathbb{Z}\mid (a,i)\in \lambda\}.
\]
Then $(\lambda_0,\lambda_1,\dots)$ is a partition of $n$ in the usual sense. We will present a $2$-dimensional partition in this way. Note that $\lambda_0\geq \lambda_1\geq \cdots$.

If $\lambda$ is a $3$-dimensional partition of $n$, let
\[
\lambda_{i}=\{(a,b)\in \mathbb{Z}^2\mid (a,b,i)\in \lambda_I\}.
\]
Then $\lambda_i$ is a $2$-dimensional partition. Then we  present $\lambda$ by an ascending chain of usual partitions
\[
\dots\subset \lambda_1\subset \lambda_0.
\]
For example, the $3$-dimensional partition (\ref{eq-partition-132}) is  presented compactly as
\[
(1)\subset (3,2).
\]

\subsection{Tautological sheaves}\label{sec:tautological-sheaves}
Let $X$ be a projective scheme over $\Bbbk$ with a given polarization. Let $\Phi\in \mathbb{Q}[z]$ be a polynomial that takes integer values for $z\in \mathbb{Z}$. Let $\Hilb^{\Phi}(X)$ be the Hilbert scheme that parametrizes closed subschemes of $X$ with Hilbert polynomial $\Phi$. Consider the diagram
\[
\xymatrix{
	\mathcal{Z} \ar[r]^{f} \ar[d]_{\pi} & X \\
	\Hilb^{\Phi}(X)\rlap{,} & 
}
\]
where $\mathcal{Z}$ is the universal subscheme of $X\times \Hilb^{\Phi}(X)$. Let  $\mathscr{F}$ be a locally free sheaf on $X$. Since $\pi$ is proper and flat, $\mathrm{R} \pi_* f^*\mathscr{F}$ has a  finite Tor-amplitude. Let $\mathscr{P}^{\bullet}$ be a  complex consisting of finitely many locally free sheaves on $\Hilb^{\Phi}(X)$ that is a representative of $\mathrm{R}\pi_* f^*\mathscr{F}$, and define  
\[
\mathscr{F}^{[\Phi]}:=\sum_i (-1)^i\left[\mathscr{P}^i\right]\in K^0\left(\Hilb^{\Phi}(X) \right).
\]
By abuse of terminology, we call $\mathscr{F}^{[\Phi]}$ the \emph{tautological sheaf associated with} $\mathscr{F}$ (when $\Phi$ is understood). It depends only on the class of $\mathscr{F}$ in $K^0(X)$. We are interested in the Euler characteristics defined as $K$-theoretical pushforwards
\[
\chi\left(\mathscr{F}^{[\Phi]}\right)=\pi_{\circ*}\left(\mathscr{F}^{[\Phi]}\right)
\]
and more generally
\begin{equation}\label{eq-eulercharacteristics-tautological}
\chi\left(\wedge^p\mathscr{F}^{[\Phi]},\wedge^q\mathscr{G}^{[\Phi]}\right)
=\pi_{\circ*}\left(\left(\wedge^p\mathscr{F}^{[\Phi]}\right)^{\vee}\otimes \wedge^q\mathscr{G}^{[\Phi]}\right)
\end{equation}
into $K^0(\Spec\ \Bbbk)\cong \mathbb{Z}$, where $\pi_{\circ}$ denotes the structure morphism to $\Spec(\Bbbk)$.\\ 

Now suppose that $(X,\Phi)$ is in one of the following two situations:
\begin{enumerate}
    \item\label{sit1} $X=\mathbb{P}^n$ and $\Phi$ is arbitrary, or
    \item\label{sit2} $X$ is a  smooth proper toric variety and $\Phi=n\in \mathbb{Z}_{\geq 0}$.
\end{enumerate}
 Let $T$ be the open dense torus contained in $X$. By Corollary~\ref{cor-fixedloci-Hilb-Projspace} and Proposition~\ref{prop-equivariant-immersion-toric}, respectively, the fixed points of the $T$ action on $\Hilb^{\Phi}(X)$ are reduced isolated $\Bbbk$-points. We denote the fixed points by $w_1,\dots,w_k$. Let $Z_i$ be the closed subscheme of $X$ represented by $w_i$.

 \begin{proposition}
 Let $A_i$ be the completed local ring of\, $\Hilb^{\Phi}(X)$ at $w_i$ for $1\leq i\leq k$. Let $\mathscr{F}$ and $\mathscr{G}$ be  locally free $T$-sheaves. Then we have an equality in $R(T)$
 \begin{equation}\label{eq-eulercharacteristics-localization-0}
    \chi\left(\wedge^p\mathscr{F}^{[\Phi]},\wedge^q\mathscr{G}^{[\Phi]}\right)
    =\sum_{i=1}^{k}\chi\left(\wedge^p\mathscr{F}|_{Z_i},\wedge^q\mathscr{G}|_{Z_i}\right)H(A_i;\mathbf{t}).
 \end{equation}
 \end{proposition}

 \begin{proof}
  Using an ample invertible $T$-sheaf on $X$, we can take the above $\mathscr{P}^\bullet$ to be a complex of locally free $T$-sheaves, and similarly for such a complex $\mathscr{Q}^\bullet$ for $\mathscr{G}$. Then 
applying (\ref{eq-Thomason-reducedisolatefixed-2}) to  $\mathscr{P}^{\bullet}$ and $\mathscr{Q}^{\bullet}$, and using the base change theorem, we obtain  Formula \eqref{eq-eulercharacteristics-localization-0}.
\end{proof}

 Moreover, each $Z_i$ is a $T$-scheme with reduced isolated fixed points, as it is a subscheme of $X$ which has this property. One can thus compute $\chi(\wedge^p\mathscr{F}|_{Z_i},\wedge^q\mathscr{G}|_{Z_i})$ by using (\ref{eq-Thomason-reducedisolatefixed-2}) again. 

 Summarizing, in principle one can compute (\ref{eq-eulercharacteristics-tautological}) by using (\ref{eq-Thomason-reducedisolatefixed-2}) ``$k+1$'' times. In the rest of this paper, we are concerned with situation~\eqref{sit2}. Then the factor $\chi(\wedge^p\mathscr{F}|_{Z_i},\wedge^q\mathscr{G}|_{Z_i})$ is easily computed directly. Our main task is to compute the equivariant Hilbert series.

\section{Local equations of  Hilbert schemes}\label{sec:local-equations-Hilbert-schemes}
In this section, we will study the local rings of $\Hilb^n(\mathbb{A}^3)$ at the closed subschemes defined by monomial ideals of colength $n$. We first introduce some notions and terminology. 
From now on in this paper, we assume the base field $\Bbbk$ has $\Char(\Bbbk)=0$ unless otherwise stated.

For an ideal $I$ of $\Bbbk[X_1,\dots,X_r]$ of colength $n$, we denote the subscheme $\Spec(\Bbbk[X_1,\dots,X_r]/I)$ by $Z_I$. 

The main component of $\Hilb^n(\mathbb{A}^r)$ is the component whose general point parametrizes $n$ distinct points. In general, $\Hilb^n(\mathbb{A}^r)$ may be highly singular, \textit{e.g.}, nonreduced or reducible (see \cite{Iar72,Jel20}). But we recall the following theorem.

\begin{theorem}\label{thm-monomialideal-smoothable}
The points on $\Hilb^n(\mathbb{A}^r)$ corresponding to monomial ideals lie on the main component.
\end{theorem}

This is \cite[Lemma 18.10]{MS05} in characteristic zero, and \cite[Proposition 4.15]{CEVV09} in positive characteristic. It follows that $\Hilb^n(\mathbb{A}^r)$ is smooth at a monomial ideal $Z_I$ if and only if the dimension of the tangent space of $\Hilb^n(\mathbb{A}^r)$ at $Z_I$ is equal to $rn$. 

\begin{definition}
For an ideal $I$ of $\Bbbk[X_1,\dots,X_r]$ of colength $n$, we define the \emph{extra dimension at}  $I$ to be
\[
\mathrm{extra.dim}_{Z_I} \Hilb^n(\mathbb{A}^r):=\dim_{\Bbbk}T_{Z_I} \Hilb^n(\mathbb{A}^r)-rn.
\]
\end{definition}

At this moment, the reader can just think of the extra dimension as the simplest way to measure the singularity. We will see later that the local structure of $\Hilb^n(\mathbb{A}^r)$ seems related to the extra dimension in unexpected ways.

Up to now, we only made use of the torus action on $\Hilb^n(\mathbb{P}^r)$ or $\Hilb^n(\mathbb{A}^r)$. But a larger group, $\GL(r)$, also acts on them, which will also turn out to be useful. A notion related to this action is that of the \emph{Borel fixed ideals}; see \cite[Section~2.1]{MS05}. An ideal $I$ of $\Bbbk[X_1,\dots,X_r]$ is called \emph{Borel fixed} if it is fixed by the Borel subgroup $\mathbb{B}(r)$ (\textit{i.e.}, the subgroup of nonsingular upper-triangular matrices) in $\GL(r)$.  Our later use of Borel fixed ideals is based on the following consequence of the Borel fixed-point theorem.

\begin{lemma}\label{lem-transitive-open-nbd}
Let $I_1,\dots,I_q$ be the set of Borel fixed ideals of colength $n$ of\, $\Bbbk[X_1,\dots,X_r]$. For each $1\leq p\leq q$, let $U_p$ be an open neighborhood of\, $Z_{I_p}$ in $\Hilb^{n}(\mathbb{A}^n_{\Bbbk})$. Then for every ideal $J$ of colength $n$ of\, $\Bbbk[X_1,\dots,X_r]$, there exists a $g\in\mathbb{B}(r)(\Bbbk)$ such that $g\cdot Z_{J}$ lies in $U_p$ for some $1\leq p\leq q$. 
\end{lemma}

\begin{proof}
  Let $\Hilb^{n}(\mathbb{A}^r_{\Bbbk})_0$ be the fiber of the Hilbert--Chow  morphism $\rho\colon \Hilb^{n}(\mathbb{A}^r_{\Bbbk})\rightarrow (\mathbb{A}^r_{\Bbbk})^{(n)}$ over $0^{n}$. It parametrizes the length $n$ subschemes of $\mathbb{A}^r_{\Bbbk}$ supported at $0$. Note that $\Hilb^{n}(\mathbb{A}^r_{\Bbbk})_0$ is also a fiber of the Hilbert--Chow morphism $\rho\colon \Hilb^{n}(\mathbb{P}^r_{\Bbbk})\rightarrow (\mathbb{P}^r_{\Bbbk})^{(n)}$, so it is projective.

We first consider  the case where $Z_J$ is supported at 0, \textit{i.e.}, $Z_J$ lies in $\Hilb^{n}(\mathbb{A}^r_{\Bbbk})_0$. 
 The group $\GL(r)$ fixes~$0$, so it acts on $\Hilb^{n}(\mathbb{A}^r_{\Bbbk})_0$. 
The action morphism $\alpha\colon\mathbb{B}(r)\times \Hilb^{n}(\mathbb{A}^r_{\Bbbk})_0\rightarrow \Hilb^{n}(\mathbb{A}^r_{\Bbbk})_0$ is smooth and thus is an open map. Let $U=\bigcup_{p=1}^{q}\alpha(\mathbb{B}(r)\times U_q)$, and let $V$ be the complement  $\Hilb^{n}(\mathbb{A}^r_{\Bbbk})_0\setminus U$. Then $V$ is projective. By the Borel fixed-point theorem, see \cite[Th\'eor\`eme 2]{God61}, if $V$ is nonempty, then $V$ has a $\Bbbk$-point fixed by the solvable group $\mathbb{B}(r)$, and we have a contradiction. So $V$ is empty and 
$U$ contains $\Hilb^{n}(\mathbb{A}^r_{\Bbbk})_0$. 
Moreover, since $\rho$ is  proper, $U\supset \rho^{-1}(W)$ for some open neighborhood $W$ of $0^n\in (\mathbb{A}^r_{\Bbbk})^{(n)}$.

Now let $J$ be an arbitrary ideal of colength $n$. The support of $Z_J$ consists of finitely many closed points, say $z_1,\dots,z_m$. There exists an open neighborhood $W'$ of $0$ in $\mathbb{A}^r$ such that $(W')^n$  maps into $W$ under the quotient map $(\mathbb{A}^r_{\Bbbk})^n\rightarrow (\mathbb{A}^r_{\Bbbk})^{(n)}$. For each $i$ with $1\leq i\leq m$, there exists a nonempty open subset $T_i$ of the diagonal group $\mathbb{G}_m^r\subset \mathbb{B}(r)$ such that $T_i\cdot z_i\subset W'$. Then $T_0:=\bigcap_{i=1}^m T_i$ is nonempty and open, and  $T_0\cdot z_i\subset W'$ for $1\leq i\leq m$. Since $\rho$ is $\GL(r)$-equivariant, it follows that $T_0\cdot Z_J$ lies in $\rho^{-1}(W)\subset U$, and the proof is complete.
\end{proof}

The Borel subgroup depends on the choice of the order of the variables $X_1,\dots,X_r$; just now the default choice was $X_1\prec\dots\prec X_r$. Now we introduce an order-independent notion.

\begin{definition}\label{def-Borelideal}
We say that an ideal $I$ of $\Bbbk[X_1,\dots,X_r]$ is a \emph{Borel ideal} if one of the following two equivalent conditions holds:
\begin{i-enumerate}
	\item\label{dB-1} There exists a total order on $X_1,\dots,X_r$ such that $I$ is fixed by the corresponding Borel group. 
	\item\label{dB-2} The ideal $I$ is  monomial, and there exists a total order $i_1\prec\dots\prec i_r$, where $\{i_1,\dots,i_r\}=\{1,\dots,r\}$, such that if $j\in \{1,\dots,r\}$ and  $f\in I$ is any monomial divisible by $X_j$, then $f\cdot \frac{X_i}{X_j}\in I$ for all $i\prec j$.
\end{i-enumerate}
We say that a monomial ideal $I$ is \emph{non-Borel} if $I$ is not a Borel ideal. 
\end{definition}

The equivalence of~\eqref{dB-1} and~\eqref{dB-2} is \cite[Proposition 2.3]{MS05}.

\begin{example}
In $\Bbbk[X_1,X_2,X_3]$, the ideal $(X_1^3, X_1^2 X_2, X_1 X_3, X_2^2,X_2 X_3, X_3^2)$ is a Borel ideal, while $$(X_1^3, X_1X_2, X_1 X_3, X_2^3,X_2 X_3, X_3^2)$$ is non-Borel.
\end{example}

\begin{example}\label{example-list-Borelideals}
The ideals corresponding to the partitions in Sections~\ref{sec:3D-121}--\ref{sec:3D-132} and Appendix~\ref{sec:3D-141}--\ref{sec:3D-232} are the only Borel ideals of colength at most $7$.
\end{example}

Later on we will make some nonlinear changes of variables. To check that they are isomorphisms, we introduce the following notion.

\begin{lemma}\label{lem-unipotent-homomorphism}
Let $\varphi\colon \Bbbk[x_1,\dots,x_n]\rightarrow \Bbbk[y_1,\dots,y_n]$ be a ring homomorphism.  Let $f_i=\varphi(x_i)$ for $1\leq i\leq n$. Assume that there is a permutation $(i_1,\dots,i_n)$ of\, $(1,\dots,n)$ such that $f_{i_j}-y_{i_j}$ is a polynomial in $x_{i_{j+1}},\dots,x_{i_{n}}$ for $1\leq j\leq n$. Then $\varphi$ is an isomorphism.
\end{lemma}

\begin{definition}\label{def-unipotent}
A homomorphism $\varphi\colon \Bbbk[x_1,\dots,x_n]\rightarrow \Bbbk[y_1,\dots,y_n]$ satisfying the assumption of Lemma~\ref{lem-unipotent-homomorphism} is called a \emph{unipotent isomorphism}.
\end{definition}

From now on, our base field will be $\Bbbk\supset\mathbb{Q}$. 

\subsection{Haiman equations}\label{sec:Haiman-equations}
We recall the explicit description by Haiman of the local defining equations of $\Hilb^n(\mathbb{A}^r)$ and some consequences of it.

Fix a natural number $r\geq 1$. Let $e_1=(1,0,\dots,0)$, \dots, $e_r=(0,\dots,0,1)$ be the standard basis of $\mathbb{R}^r$.  Let~$\lambda$ be the set of lattice points in an $r$-dimensional partition. The Haiman coordinates are $c_{i}^{j}$, where $i\in \lambda$, $j\in \mathbb{Z}_{\geq 0}^{r}$, subject to the relations
\begin{equation}\label{eq-Haiman-0}
c_{i}^{j}=\begin{cases}
1 & \mbox{if}\  i=j\in \lambda,\\
0& \mbox{if}\ i,j\in \lambda\ \mbox{and}\ i\neq j, 
\end{cases}
\end{equation}
and for $1\leq b\leq r$,
\begin{equation}\label{eq-Haiman-1}
	c_{i}^{j+e_b}=\sum_{k\in \lambda} c_{k}^{j}c_{i}^{k+e_b},\quad i\in \lambda,\ j\not\in \lambda.
\end{equation}

For the following theorem, we refer the reader to \cite{Hai98} and \cite{Hui06}.

\begin{theorem}\label{thm-fundamental-haiman}
Let
\[
\tilde{R}_{\lambda}=\Bbbk\left[\left\{c_{i}^j\right\}_{i\in \lambda,j\in \mathbb{Z}_{\geq 0}^r}\right],
\]
let $\tilde{\mathcal{H}}_{\lambda}$ be the ideal of\, $\tilde{R}_{\lambda}$ generated by  Equations \eqref{eq-Haiman-0} and \eqref{eq-Haiman-1}, and let  
\[
\tilde{A}_{\lambda}=\tilde{R}_{\lambda}/\tilde{\mathcal{H}}_{\lambda}.
\]
Then
\begin{i-enumerate}
 	\item\label{tfh-1} $\Spec(\tilde{A}_{\lambda})$ is isomorphic to an affine  neighborhood $U_{\lambda}$ of\, $I_{\lambda}$ in $\Hilb^{|\lambda|}(\mathbb{A}^r)$;  
 	\item\label{tfh-2} for an ideal $J$ of\, $\Bbbk[X_1,\dots,X_r]$, $Z_J$ lies in $U_\lambda$ if and only if the monomials $\{\mathbf{X}^{\alpha}\}_{\alpha\in \lambda}$ form a basis of\, $\Bbbk[X_1,\dots,X_r]/J$;
 	\item\label{tfh-3}  for any $\Bbbk$-algebra $K$, a $K$-point $(c_{i}^j)_{i\in \lambda,j\in \mathbb{Z}_{\geq 0}^r}$ of\, $\Spec(A_{\lambda})$ corresponds to the ideal  of\, $K[X_1,\dots,X_r]$ generated by
\[
X^j-\sum_{i\in \lambda}c_{i}^j X^i\quad \mbox{for}\ j\in \mathbb{Z}_{\geq 0}^r.
\]
 \end{i-enumerate} 
\end{theorem}

While $\tilde{R}_{\lambda}$ is not finitely generated, it is not hard to see that the ring $\tilde{A}_{\lambda}$ is finitely generated over $\Bbbk$. Let us recall an explicit finite presentation of $\tilde{A}_{\lambda}$.

\begin{definition}
Let $\lambda$ be a $r$-dimensional partition. The \emph{glove} of $\lambda$ is
\begin{equation*}
\glo(\lambda)=\{i\in \mathbb{Z}_{\geq 0}^r\mid i\not\in \lambda,\ 
\mbox{and at least one element of}\ \{i-e_1,\dots,i-e_r\}\ \mbox{lies in}\ \lambda\}.
\end{equation*}
\end{definition}

\begin{definition}
Let $S$ be a subset of $\mathbb{Z}_{\geq 0}^r$.  An unordered pair $\{i,j\}$ of two distinct elements of $S$ is called an \emph{adjacent pair in $S$} if $i-j$ is equal to $e_a$ or $e_a-e_b$ for some $a,b\in\{1,\dots,r\}$.  The set of adjacent pairs in $S$ is denoted by $\Adj(S)$.
\end{definition}

\begin{definition}
Let $\lambda$ be an $r$-dimensional partition. For an adjacent pair $\{i,j\}$ of $\glo(\lambda)$, we define the set of \emph{Haiman equations} associated with $\{i,j\}$, denoted by $\HE(i,j)$, in the following way: 
\begin{i-enumerate}
	\item If $i-j=e_b$ and  $1\leq b\leq r$,
	\begin{equation}\label{eq-Huibregtse-1}
	\HE(i,j)=\left\{c_{l}^{j+e_b}-\sum_{k\in \lambda} c_{k}^{j}c_{l}^{k+e_b}\mid l\in \lambda\right\}.
	\end{equation}
	\item If  $i-j=e_a-e_b$ and  $1\leq a\neq b\leq r$,
	\begin{equation}\label{eq-Huibregtse-2}
	\HE(i,j)=\left\{\sum_{k\in \lambda} c_{k}^{j}c_{l}^{k+e_a}-\sum_{k\in \lambda} c_{k}^{i}c_{l}^{k+e_b}\mid l\in \lambda\right\}.
	\end{equation}
\end{i-enumerate}
\end{definition}

Denote by $\mu$ the glove of $\lambda$. Set
\[
R_{\lambda}=\Bbbk\left[c_{i}^j\right]_{i\in \lambda,j\in \mu}.
\]
Let $\Adj(\mu)$ be the set of adjacent pairs in $\mu$.  Let $\mathcal{H}_{\lambda}$ be the ideal of $R_{\lambda}$ generated by the equations in
\begin{equation}\label{eq-Huibregtse-total}
\bigcup_{\{i,j\}\in \Adj(\mu)}\HE(i,j).
\end{equation}
Let $A_{\lambda}=R_{\lambda}/\mathcal{H}_{\lambda}$.  
The $\mathbb{G}_m^r$-action on $\mathbb{A}^r$ induces actions on $R_{\lambda}$ and $\mathcal{H}_{\lambda}$, and thus on $A_{\lambda}$. We denote the associated equivariant Hilbert function of the completed local ring of $A_{\lambda}$ at $I_{\lambda}$ by $H(A_{\lambda};\mathbf{t})$.

\begin{proposition}
The obvious homomorphism $R_{\lambda}\rightarrow \tilde{R}_{\lambda}$ induces an isomorphism $A_{\lambda}\cong \tilde{A}_{\lambda}$.
\end{proposition}

One can find a proof in  \cite[Sections~5--7]{Hui06}.

\begin{corollary}\label{cor-description-cotangentspace}
The set $\{c_{\ell}^i\}_{\ell\in \lambda,i\in \mu}$ modulo the equivalence relations generated by the relations
\begin{equation*}\label{eq-description-cotangentspace}
\begin{cases}
c_{\ell+e_a}^{i+e_a}\sim c_{\ell}^i&  \mbox{for}\ 1\leq a\leq r\ \mbox{satisfying that}\ \ell+e_a\in \lambda\ \mbox{and}\ i+e_a\in \mu,\\[.5ex]
c_{\ell+e_a-e_b}^{i+e_a-e_b}\sim c_{\ell}^i&  \mbox{for}\ 1\leq a\neq b\leq r\ \mbox{satisfying that}\ \ell+e_a, \ell+e_a-e_b \in \lambda\ \mbox{and}\ i+e_a-e_b\in \mu,\\[.5ex]
c_{\ell}^{i}\sim 0& \mbox{if}\ i-e_a\in \mu\ \mbox{and}\ \ell-e_a\not\in \mathbb{Z}_{\geq 0}^r\ \mbox{for some}\ 1\leq a\leq r,
\end{cases}
\end{equation*}
forms an equivariant basis of the cotangent space of\, $\Spec(A_{\lambda})$ at $0$. Here $\sim 0$ means deleting this equivalence of coordinates from $\{c_{\ell}^j\}_{\ell\in \lambda,j\in \mu}$.
\end{corollary}

This is a direct consequence of (\ref{eq-Huibregtse-1}) and (\ref{eq-Huibregtse-2}).  One can deduce from this description the smoothness of $\Hilb^n(\mathbb{A}^2)$ (see \cite{Hai98} or \cite[Section~18.2]{MS05}). In Section~\ref{sec:pyramids}, we will get a graphical impression of how the smoothness fails in higher dimensions.

\begin{lemma}\label{lem-lower-dim-partition}
Let $\iota\colon \mathbb{Z}^{r-1}\rightarrow \mathbb{Z}^{r}$ be the embedding $x\mapsto (x,0)$. In this way, for an $(r-1)$-dimensional partition $\lambda$, $\iota(\lambda)$ is an $r$-dimensional partition. Then
\[
\Spec(A_{\iota(\lambda)})\cong \mathbb{A}^{|\lambda|}\times \Spec(A_{\lambda}).
\]
\end{lemma}

\begin{proof}
 Let $K$ be an arbitrary $\Bbbk$-algebra. Consider the ideals $I$ of $K[X_1,\dots,X_r]$ such that $\{\mathbf{X}^i\}_{i\in \iota(\lambda)}$ freely generates $K[X_1,\dots,X_r]/I$. For such $I$, there are unique elements $c_i(I)\in K$ such that
 \[
 X_r\equiv \sum_{i=(i_1,\dots,i_{r-1})\in \lambda}c_i(I) X_1^{i_1}\cdots X_{r-1}^{i_{r-1}} \mod I. 
 \]
 Let $J=I\cap K[X_1,\dots,X_{r-1}]$. Then there is a canonical isomorphism
 \[
 K[X_1,\dots,X_{r-1}]/J\cong K[X_1,\dots,X_{r}]/I.
 \] 
 This gives a bijection of the sets of $K$-points
 \begin{align*}
   \Spec(A_{\iota(\lambda)})(K)&\longrightarrow K^{|\lambda|}\times \Spec(A_{\lambda})(K),\\
I&\longmapsto \left(\left(c_i(I)\right)_{i\in \lambda},J\right).
  \end{align*}
This bijection is functorial in $K$. So we have the wanted isomorphism of schemes.
\end{proof}

\subsection{Pyramids}\label{sec:pyramids}
In this section, we study the Haiman ideal corresponding to a special type of $3$-dimensional partitions, called \emph{pyramids}.  Let $\pyr_r(n)$ be the $r$-dimensional partition
\[
\{(a_1,\dots,a_r)\in \mathbb{Z}_{\geq 0}^r\mid a_1+\dots+a_r\leq n-1\}.
\]
The glove of  $\pyr_r(n)$ is
\[
\glo\left(\pyr_r(n)\right)=\{(a_1,\dots,a_r)\in \mathbb{Z}_{\geq 0}^r\mid  a_1+\dots+a_r=n
\}.
\]
The monomial ideal corresponding to $\pyr_3(n)$ is
\[
I_{\pyr_3(n)}=\sum_{a+b+c=n}(X_1^{a}X_2^{b}X_3^{c}).
\]
For example, $\pyr_3(4)$ is graphically presented as
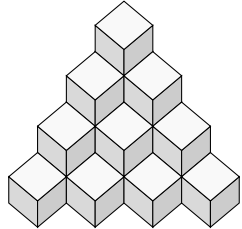
\begin{center}
\begin{tikzpicture}[x=(220:0.5cm), y=(-40:0.5cm), z=(90:0.35cm)]

\foreach \m [count=\y] in {{4,3,2,1},{3,2,1},{2,1},{1}}{
  \foreach \n [count=\x] in \m {
  \ifnum \n>0
      \foreach \z in {1,...,\n}{
        \draw [fill=gray!30] (\x+1,\y,\z) -- (\x+1,\y+1,\z) -- (\x+1, \y+1, \z-1) -- (\x+1, \y, \z-1) -- cycle;
        \draw [fill=gray!40] (\x,\y+1,\z) -- (\x+1,\y+1,\z) -- (\x+1, \y+1, \z-1) -- (\x, \y+1, \z-1) -- cycle;
        \draw [fill=gray!5] (\x,\y,\z)   -- (\x+1,\y,\z)   -- (\x+1, \y+1, \z)   -- (\x, \y+1, \z) -- cycle;  
      }
 \fi
 }
}    
\end{tikzpicture}
\end{center}

Define a polynomial in $\Bbbk[\{c_{i}^j\}_{|i|=n-1,|j|=n}]$
\begin{equation}\label{eq-3D-pyramid-superpotential}
	F_{\pyr_3(n)}=-\sum_{|i|=|j|=|k|=n-1}c_{j}^{i+(0,0,1)}c_{k}^{j+(1,0,0)}c_{i}^{k+(0,1,0)}
+\sum_{|i|=|j|=|k|=n-1}c_{j}^{i+(0,0,1)}c_{k}^{j+(0,1,0)}c_{i}^{k+(1,0,0)}.
\end{equation}

\begin{proposition}\label{prop-potential-pyramid}
There is an isomorphism
\begin{equation}
	\Bbbk\left[\left\{c_{i}^j\right\}_{|i|=n-1,|j|=n}\right]\Big/ \Jac\left(F_{\pyr_3(n)}\right)\overset{\lowsim}\longrightarrow A_{\pyr_3(n)}.
\end{equation}
\end{proposition}

\begin{proof}
As a convention, we define $c_i^j=0$ for $i\in \mathbb{Z}^3\setminus \mathbb{Z}^3_{\geq 0}$.  The Haiman equations (\ref{eq-Haiman-0}) and (\ref{eq-Haiman-1}) remain valid.  In dimension $3$, Equations (\ref{eq-Haiman-1}) read
\begin{subequations}\label{eq-Haiman-1-3D}
\begin{eqnarray}
c_{i}^{j+(1,0,0)}=\sum_{k\in \lambda} c_{k}^{j}c_{i}^{k+(1,0,0)},\quad i\in \lambda,\ j\not\in \lambda,\label{subeqHaiman-1a}\\
c_{i}^{j+(0,1,0)}=\sum_{k\in \lambda} c_{k}^{j}c_{i}^{k+(0,1,0)},\quad i\in \lambda,\ j\not\in \lambda,\label{subeq-Haiman-1b}\\
c_{i}^{j+(0,0,1)}=\sum_{k\in \lambda} c_{k}^{j}c_{i}^{k+(0,0,1)},\quad i\in \lambda,\ j\not\in \lambda.\label{subeq-Haiman-1c}
\end{eqnarray}
\end{subequations}

Suppose $i,j\in \mathbb{Z}_{\geq 0}^3$ satisfy $|i|=|j|=n-1$. Applying (\ref{subeqHaiman-1a}) to $c_{i}^{j+(0,1,0)+(1,0,0)}$, we obtain
\[
c_{i}^{j+(0,1,0)+(1,0,0)}=\sum_{k\in \lambda} c_{k}^{j+(0,1,0)}c_{i}^{k+(1,0,0)}.
\]
Then applying (\ref{eq-Haiman-0}), we get
\begin{equation}\label{eq-potential-proof-pyramid-1}
	c_{i}^{j+(0,1,0)+(1,0,0)}=
	 c_{i-(1,0,0)}^{j+(0,1,0)}+ \sum_{|k|=n-1} c_{k}^{j+(0,1,0)}c_{i}^{k+(1,0,0)}.
\end{equation}
Similarly, applying (\ref{subeq-Haiman-1b}) and (\ref{eq-Haiman-0}) to $c_{i}^{j+(0,1,0)+(1,0,0)}$,  we get
\begin{align}\label{eq-potential-proof-pyramid-2}
c_{i}^{j+(0,1,0)+(1,0,0)}&=\sum_{k\in \lambda} c_{k}^{j+(1,0,0)}c_{i}^{k+(0,1,0)}\notag\\
&= c_{i-(0,1,0)}^{j+(1,0,0)}+ \sum_{|k|=n-1} c_{k}^{j+(1,0,0)}c_{i}^{k+(0,1,0)}.
\end{align}
Equating the right-hand sides of (\ref{eq-potential-proof-pyramid-1}) and (\ref{eq-potential-proof-pyramid-2}), we obtain
\begin{eqnarray}\label{eq-potential-proof-pyramid-3}
c_{i-(1,0,0)}^{j+(0,1,0)}-c_{i-(0,1,0)}^{j+(1,0,0)}
= \sum_{|k|=n-1} c_{k}^{j+(1,0,0)}c_{i}^{k+(0,1,0)}-
\sum_{|k|=n-1} c_{k}^{j+(0,1,0)}c_{i}^{k+(1,0,0)}.
\end{eqnarray}
Similarly, we have
\begin{eqnarray}\label{eq-potential-proof-pyramid-4}
c_{i-(0,0,1)}^{j+(1,0,0)}-c_{i-(1,0,0)}^{j+(0,0,1)}
= \sum_{|k|=n-1} c_{k}^{j+(0,0,1)}c_{i}^{k+(1,0,0)}
-\sum_{|k|=n-1} c_{k}^{j+(1,0,0)}c_{i}^{k+(0,0,1)},
\end{eqnarray}
\begin{eqnarray}\label{eq-potential-proof-pyramid-5}
c_{i-(0,1,0)}^{j+(0,0,1)}-c_{i-(0,0,1)}^{j+(0,1,0)}
=\sum_{|k|=n-1} c_{k}^{j+(0,1,0)}c_{i}^{k+(0,0,1)}
- \sum_{|k|=n-1} c_{k}^{j+(0,0,1)}c_{i}^{k+(0,1,0)}.
\end{eqnarray}
Suppose $|i|=n-1$, $|j|=n$, and $j\geq (1,0,0)$.
By (\ref{eq-potential-proof-pyramid-3}), (\ref{eq-potential-proof-pyramid-4}), and (\ref{eq-potential-proof-pyramid-5}), we have
\begin{eqnarray}\label{eq-pyramid-4}
0&=& \left(c_{i}^{j}-c_{i+(-1,1,0)}^{j+(-1,1,0)}\right)
+\left(c_{i+(-1,1,0)}^{j+(-1,1,0)}-c_{i+(-1,0,1)}^{j+(-1,0,1)}\right)
+\left(c_{i+(-1,0,1)}^{j+(-1,0,1)}-c_{i}^{j}\right)\notag\\
&=& \left(-\sum_{|k|=n-1} c_{k}^{j}c_{i+(0,1,0)}^{k+(0,1,0)}
+\sum_{|k|=n-1} c_{k}^{j+(-1,1,0)}c_{i+(0,1,0)}^{k+(1,0,0)}\right)\notag\\
&&+\left(-\sum_{|k|=n-1} c_{k}^{j+(-1,1,0)}c_{i+(-1,1,1)}^{k+(0,0,1)}
+\sum_{|k|=n-1} c_{k}^{j+(-1,0,1)}c_{i+(-1,1,1)}^{k+(0,1,0)}\right)\notag\\
&&+\left(-\sum_{|k|=n-1} c_{k}^{j+(-1,0,1)}c_{i+(0,0,1)}^{k+(1,0,0)}
+\sum_{|k|=n-1} c_{k}^{j}c_{i+(0,0,1)}^{k+(0,0,1)}\right).
\end{eqnarray}
Similarly, for $|i|=n-1$, $|j|=n$, and $j\geq (0,1,0)$, we have
\begin{eqnarray}\label{eq-pyramid-5}
0&=&\left(-\sum_{|k|=n-1} c_{k}^{j}c_{i+(0,0,1)}^{k+(0,0,1)}
+\sum_{|k|=n-1} c_{k}^{j+(0,-1,1)}c_{i+(0,0,1)}^{k+(0,1,0)}\right)\notag\\
&&+\left(-\sum_{|k|=n-1} c_{k}^{j+(0,-1,1)}c_{i+(1,-1,1)}^{k+(1,0,0)}
+\sum_{|k|=n-1} c_{k}^{j+(1,-1,0)}c_{i+(1,-1,1)}^{k+(0,0,1)}\right)\notag\\
&&+\left(-\sum_{|k|=n-1} c_{k}^{j+(1,-1,0)}c_{i+(1,0,0)}^{k+(0,1,0)}
+\sum_{|k|=n-1} c_{k}^{j}c_{i+(1,0,0)}^{k+(1,0,0)}\right),
\end{eqnarray}
and for $|i|=n-1$, $|j|=n$, and $j\geq (0,0,1)$, we have
\begin{eqnarray}\label{eq-pyramid-6}
0&=&\left(-\sum_{|k|=n-1} c_{k}^{j}c_{i+(1,0,0)}^{k+(1,0,0)}
+\sum_{|k|=n-1} c_{k}^{j+(1,0,-1)}c_{i+(1,0,0)}^{k+(0,0,1)}\right)\notag\\
&&+\left(-\sum_{|k|=n-1} c_{k}^{j+(1,0,-1)}c_{i+(1,1,-1)}^{k+(0,1,0)}
+\sum_{|k|=n-1} c_{k}^{j+(0,1,-1)}c_{i+(1,1,-1)}^{k+(1,0,0)}\right)\notag\\
&&+\left(-\sum_{|k|=n-1} c_{k}^{j+(0,1,-1)}c_{i+(0,1,0)}^{k+(0,0,1)}
+\sum_{|k|=n-1} c_{k}^{j}c_{i+(0,1,0)}^{k+(0,1,0)}\right). 
\end{eqnarray}
Note that
\begin{eqnarray*}
\mbox{RHS of (\ref{eq-pyramid-4})}=\frac{\partial F_{\pyr_3(n)}}{\partial c_{j-(1,0,0)}^{i+(0,1,1)}},\quad
\mbox{RHS of (\ref{eq-pyramid-5})}=\frac{\partial F_{\pyr_3(n)}}{\partial c_{j-(0,1,0)}^{i+(1,0,1)}},\quad
\mbox{RHS of (\ref{eq-pyramid-6})}=\frac{\partial F_{\pyr_3(n)}}{\partial c_{j-(0,0,1)}^{i+(1,1,0)}}.
\end{eqnarray*}
So $\Jac(F_{\pyr_3(n)})$ is contained in the Haiman ideal. We are left to show that they are equal.

Let $|i|=n-2$ and $|j|=n$.  Apply (\ref{eq-potential-proof-pyramid-3}) successively to $c_{i}^{j}$, $c_{i+(1,-1,0)}^{j+(1,-1,0)}$, $c_{i+(2,-2,0)}^{j+(2,-2,0)}$, and so on. If the subscript runs out of $\mathbb{Z}_{\geq 0}^3$ faster than the superscript, we obtain a formula expressing $c_{i}^j$ as a quadratic polynomial in $\{c_{a}^b\}_{|a|=n-1,|b|=n}$. By (\ref{eq-potential-proof-pyramid-3})--(\ref{eq-potential-proof-pyramid-5}), there are six possible directions to do this. Since $|j|>|i|$, at least one direction has this property. In fact, write $i=(i_1,i_2,i_3)$ and $j=(j_1,j_2,j_3)$; when $j_1>i_1$, we have 
\begin{equation}\label{eq-express-Haiman-coordinate-subscript-n-2a}
 c_{i}^{j}=
 -\sum_{l=0}^{i_1}\sum_{|k|=n-1} c_{k}^{j+(-l,l,0)}c_{i+(-l,l+1,0)}^{k+(0,1,0)}
+\sum_{l=0}^{i_1}\sum_{|k|=n-1} c_{k}^{j+(-l-1,l+1,0)}c_{i+(-l,l+1,0)}^{k+(1,0,0)},
\end{equation}
and 
\begin{equation}\label{eq-express-Haiman-coordinate-subscript-n-2b}
c_{i}^{j}= 
-\sum_{l=0}^{i_1}\sum_{|k|=n-1} c_{k}^{j+(-l,0,l)}c_{i+(-l,0,l+1)}^{k+(0,0,1)}
+\sum_{l=0}^{i_1}\sum_{|k|=n-1} c_{k}^{j+(-l-1,0,l+1)}c_{i+(-l,0,l+1)}^{k+(1,0,0)}.
\end{equation}
Similarly, when $j_2>i_2$, we have
\begin{equation}\label{eq-express-Haiman-coordinate-subscript-n-2c}
c_{i}^{j}=
-\sum_{l=0}^{i_2}\sum_{|k|=n-1} c_{k}^{j+(0,-l,l)}c_{i+(0,-l,l+1)}^{k+(0,0,1)}
+\sum_{l=0}^{i_2}\sum_{|k|=n-1} c_{k}^{j+(0,-l-1,l+1)}c_{i+(0,-l,l+1)}^{k+(0,1,0)},
\end{equation}
\begin{equation}\label{eq-express-Haiman-coordinate-subscript-n-2d}
c_{i}^{j}=
 -\sum_{l=0}^{i_2}\sum_{|k|=n-1} c_{k}^{j+(l,-l,0)}c_{i+(l+1,-l,0)}^{k+(1,0,0)}
+\sum_{l=0}^{i_2}\sum_{|k|=n-1} c_{k}^{j+(l+1,-l-1,0)}c_{i+(l+1,-l,0)}^{k+(0,1,0)},
\end{equation}
and when $j_3>i_3$,
\begin{equation}\label{eq-express-Haiman-coordinate-subscript-n-2e}
c_{i}^{j}=
-\sum_{l=0}^{i_3}\sum_{|k|=n-1} c_{k}^{j+(l,0,-l)}c_{i+(l+1,0,-l)}^{k+(1,0,0)}
+\sum_{l=0}^{i_3}\sum_{|k|=n-1} c_{k}^{j+(l+1,0,-l-1)}c_{i+(l+1,0,-l)}^{k+(0,0,1)},
\end{equation}
\begin{equation}\label{eq-express-Haiman-coordinate-subscript-n-2f}
c_{i}^{j}=
 -\sum_{l=0}^{i_3}\sum_{|k|=n-1} c_{k}^{j+(0,l,-l)}c_{i+(0,l+1,-l)}^{k+(0,1,0)}
+\sum_{l=0}^{i_3}\sum_{|k|=n-1} c_{k}^{j+(0,l+1,-l-1)}c_{i+(0,l+1,-l)}^{k+(0,0,1)}.
\end{equation}
When a Haiman coordinate of the form $c_{i}^{j}$ with $|i|=n-2$ and $|j|=n$
has different expressions via (\ref{eq-express-Haiman-coordinate-subscript-n-2a})--(\ref{eq-express-Haiman-coordinate-subscript-n-2f}), we obtain a quadratic relation in the ring $\Bbbk[\{c_{a}^b\}_{|a|=n-1,|b|=n}]$. These relations are the same as the relations (\ref{eq-pyramid-4})--(\ref{eq-pyramid-6}), up to linear combinations. This can be seen from the following type of graphs of the glove of $\pyr_3(n)$ (this example is the glove of $\pyr_3(4)$):
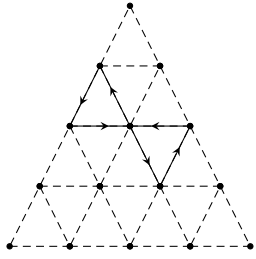
\begin{center}
\begin{tikzpicture}[x=(0:0.8cm), y=(90:0.8cm)]
\foreach \m [count=\y] in {{4,3,2,1},{3.5,2.5,1.5},{3,2},{2.5}}{
  	\foreach \n [count=\x] in \m {
        \draw[densely dashed] (\n,\y) -- (\n+1,\y) -- (\n+0.5,\y+1)  -- cycle;
 		\filldraw (\n,\y) circle (1pt);
 		\filldraw (\n+1,\y) circle (1pt);
 		\filldraw (\n+0.5,\y+1) circle (1pt);
 }
}    
 	\draw[directed] (2,3) --(3,3);
 	\draw[directed] (3,3) --(2.5,4);
 	\draw[directed] (2.5,4) --(2,3);

 	\draw[directed] (3,3) --(3.5,2);
 	\draw[directed] (3.5,2) --(4,3);
 	\draw[directed] (4,3) --(3,3);
\end{tikzpicture}
\end{center}
That is to say, an expression in (\ref{eq-express-Haiman-coordinate-subscript-n-2a})--(\ref{eq-express-Haiman-coordinate-subscript-n-2f}) is induced by a path from $j$ to the boundary of the above graph. Comparisons between paths can be decomposed into the relations induced by the triangles
\begin{tikzpicture}[x=(0:0.5cm), y=(90:0.5cm)]
	\draw[directed] (2,3) --(3,3);
 	\draw[directed] (3,3) --(2.5,4);
 	\draw[directed] (2.5,4) --(2,3);
\end{tikzpicture},
and the \emph{upside-down} triangles \begin{tikzpicture}[x=(0:0.5cm), y=(90:0.5cm)]
	\draw[directed] (3,3) --(3.5,2);
 	\draw[directed] (3.5,2) --(4,3);
 	\draw[directed] (4,3) --(3,3);
\end{tikzpicture}.
The relations induced by the triangles \begin{tikzpicture}[x=(0:0.5cm), y=(90:0.5cm)]
	\draw[directed] (2,3) --(3,3);
 	\draw[directed] (3,3) --(2.5,4);
 	\draw[directed] (2.5,4) --(2,3);
\end{tikzpicture} are no other than (\ref{eq-pyramid-4})--(\ref{eq-pyramid-6}), while the relations induced by the upside-down triangles turn out to be trivial:
\begin{eqnarray}\label{eq-pyramid-10}
0&=& \left(c_{i}^{j}-c_{i+(1,-1,0)}^{j+(1,-1,0)}\right) +\left(c_{i+(1,-1,0)}^{j+(1,-1,0)}-c_{i+(1,0,-1)}^{j+(1,0,-1)}\right) +\left(c_{i+(1,0,-1)}^{j+(1,0,-1)}-c_{i}^{j}\right)\notag\\ &=& \left(-\sum_{|k|=n-1} c_{k}^{j}c_{i+(1,0,0)}^{k+(1,0,0)} +\sum_{|k|=n-1} c_{k}^{j+(1,-1,0)}c_{i+(1,0,0)}^{k+(0,1,0)}\right)\notag\\ &&+\left(-\sum_{|k|=n-1} c_{k}^{j+(1,-1,0)}c_{i+(1,0,0)}^{k+(0,1,0)} +\sum_{|k|=n-1} c_{k}^{j+(1,0,-1)}c_{i+(1,0,0)}^{k+(0,0,1)}\right)\notag\\ &&+\left(-\sum_{|k|=n-1} c_{k}^{j+(1,0,-1)}c_{i+(1,0,0)}^{k+(0,0,1)} +\sum_{|k|=n-1} c_{k}^{j}c_{i+(1,0,0)}^{k+(1,0,0)}\right).
\end{eqnarray}
The same procedure applies to a Haiman coordinate $c_{i}^{j}$ for $|i|\leq n-3$ and $|j|=|n|$,
which results in expressions for $c_{i}^{j}$ as a quadratic polynomial that has homogeneous degree $1$ in $\{c_{a}^{b}\}_{|a|=n-1,|b|=n}$ and homogeneous degree $1$ in $\{c_{a}^{b}\}_{|a|=|i|+1,|b|=n}$; thus inductively we obtain expressions for $c_{i}^{j}$ as a homogeneous polynomial of degree $n-|i|$ in $\{c_{a}^{b}\}_{|a|=n-1,|b|=n}$. We need only to show that the triangles
\begin{tikzpicture}[x=(0:0.5cm), y=(90:0.5cm)] 
\draw[directed] (2,3) --(3,3); 
\draw[directed] (3,3) --(2.5,4); 
\draw[directed] (2.5,4) --(2,3);
\end{tikzpicture} do not induce new relations. 
For $|i|=n-3$, modulo (\ref{eq-pyramid-4})--(\ref{eq-pyramid-6}),  
we have
\begin{align*}
& \hspace*{-5.2cm} \left(-\sum_{|k|=n-1} c_{k}^{j}c_{i+(0,1,0)}^{k+(0,1,0)}
+\sum_{|k|=n-1} c_{k}^{j+(-1,1,0)}c_{i+(0,1,0)}^{k+(1,0,0)}\right)\notag\\
&\hspace*{-5.2cm} +\left(-\sum_{|k|=n-1} c_{k}^{j+(-1,1,0)}c_{i+(-1,1,1)}^{k+(0,0,1)} 
+\sum_{|k|=n-1} c_{k}^{j+(-1,0,1)}c_{i+(-1,1,1)}^{k+(0,1,0)}\right)\notag\\
&\hspace*{-5.2cm} +\left(-\sum_{|k|=n-1}c_{k}^{j+(-1,0,1)}c_{i+(0,0,1)}^{k+(1,0,0)}
+\sum_{|k|=n-1} c_{k}^{j}c_{i+(0,0,1)}^{k+(0,0,1)}\right)
\end{align*}
\begin{eqnarray*}
&=&\sum_{|k|=n-1} c_{k}^{j}\left(c_{i+(0,0,1)}^{k+(0,0,1)}-c_{i+(0,1,0)}^{k+(0,1,0)}\right)
+\sum_{|k|=n-1} c_{k}^{j+(-1,0,1)}\left(c_{i+(-1,1,1)}^{k+(0,1,0)}-c_{i+(0,0,1)}^{k+(1,0,0)}\right)\\
&&+\sum_{|k|=n-1} c_{k}^{j+(-1,1,0)}\left(c_{i+(0,1,0)}^{k+(1,0,0)}-c_{i+(-1,1,1)}^{k+(0,0,1)}\right)\\
&=&\sum_{|k|=n-1} c_{k}^{j}\left(-\sum_{|l|=n-1} c_{l}^{k+(0,0,1)}c_{i+(0,1,1)}^{l+(0,1,0)}
+\sum_{|l|=n-1} c_{l}^{k+(0,1,0)}c_{i+(0,1,1)}^{l+(0,0,1)}\right)\\
&&+\sum_{|k|=n-1} c_{k}^{j+(-1,0,1)}
\left(-\sum_{|l|=n-1} c_{l}^{k+(0,1,0)}c_{i+(0,1,1)}^{l+(1,0,0)}
+\sum_{|l|=n-1} c_{l}^{k+(1,0,0)}c_{i+(0,1,1)}^{l+(0,1,0)}\right)\\
&&+\sum_{|k|=n-1} c_{k}^{j+(-1,1,0)}
\left(-\sum_{|l|=n-1} c_{l}^{k+(1,0,0)}c_{i+(0,1,1)}^{l+(0,0,1)}
+\sum_{|l|=n-1} c_{l}^{k+(0,0,1)}c_{i+(0,1,1)}^{l+(1,0,0)}\right)\\
&=& \sum_{|l|=n-1}c_{i+(0,1,1)}^{l+(0,1,0)}
\left(-\sum_{|k|=n-1} c_{k}^{j}c_{l}^{k+(0,0,1)}
+\sum_{|k|=n-1} c_{k}^{j+(-1,0,1)}c_{l}^{k+(1,0,0)}\right)\\
&&+\sum_{|l|=n-1} c_{i+(0,1,1)}^{l+(0,0,1)}\left(
-\sum_{|k|=n-1} c_{k}^{j+(-1,1,0)}c_{l}^{k+(1,0,0)}
+\sum_{|k|=n-1} c_{k}^{j}c_{l}^{k+(0,1,0)}\right)\\
&&+\sum_{|l|=n-1}c_{i+(0,1,1)}^{l+(1,0,0)}\left(
-\sum_{|k|=n-1} c_{k}^{j+(-1,0,1)}c_{l}^{k+(0,1,0)}
+\sum_{|k|=n-1} c_{k}^{j+(-1,1,0)}c_{l}^{k+(0,0,1)}\right)\\
&\begin{subarray}{c} 
\mbox{modulo (\ref{eq-pyramid-4})--(\ref{eq-pyramid-6})}\\ 
\equiv\end{subarray}& \sum_{|l|=n-1}c_{i+(0,1,1)}^{l+(0,1,0)}
\left(\right.
-\sum_{|k|=n-1} c_{k}^{j}c_{l+(0,1,-1)}^{k+(0,1,0)}
+\sum_{|k|=n-1} c_{k}^{j+(-1,1,0)}c_{l+(0,1,-1)}^{k+(1,0,0)}\\
&&-\sum_{|k|=n-1} c_{k}^{j+(-1,1,0)}c_{l+(-1,1,0)}^{k+(0,0,1)}
+\sum_{|k|=n-1} c_{k}^{j+(-1,0,1)}c_{l+(-1,1,0)}^{k+(0,1,0)}
\left.\right)\\
&&+\sum_{|l|=n-1} c_{i+(0,1,1)}^{l+(0,0,1)}\left(
-\sum_{|k|=n-1} c_{k}^{j+(-1,1,0)}c_{l}^{k+(1,0,0)}
+\sum_{|k|=n-1} c_{k}^{j}c_{l}^{k+(0,1,0)}\right)\\
&&+\sum_{|l|=n-1}c_{i+(0,1,1)}^{l+(1,0,0)}\left(
-\sum_{|k|=n-1} c_{k}^{j+(-1,0,1)}c_{l}^{k+(0,1,0)}
+\sum_{|k|=n-1} c_{k}^{j+(-1,1,0)}c_{l}^{k+(0,0,1)}\right)\\
&=&0.
\end{eqnarray*}

By an induction on $|i|$ from $|i|=n-3$ to $|i|=1$,  modulo (\ref{eq-pyramid-4})--(\ref{eq-pyramid-6}), there are no new relations. The proof is thus completed.
\end{proof}

\begin{remark}
Using deformation theory, modulo a conjecture, Hsu showed in \cite{Hsu16} that the formal completion of $\Spec(A_{\pyr_{3}(n)})$ at $0$ is a critical locus in the tangent space, without giving an explicit form of the superpotential.
\end{remark}

\begin{remark}\label{rem418}
An impression we get from the above proof is that the nonlinear relations of the Haiman coordinates come from the configuration of the glove. In dimension $3$, it is closely related to the triangles in a certain plane projection of the glove. One can compare this to the dimension $2$ case, where there is no room for such triangles, so the resulting Haiman neighborhoods are smooth.
\end{remark}

\begin{remark}
From Remark~\ref{rem418}, one naturally expects that the pyramid ideal $I_{\pyr_3(n)}$ should correspond to the most singular point in $\Hilb^{|\pyr_3(n)|}(\mathbb{A}^3)$, while its equations are the simplest within the singularities of $\Hilb^{|\pyr_3(n)|}(\mathbb{A}^3)$. An unsolved problem is to determine the smallest $n$ such that $\Hilb^{n}(\mathbb{A}^3)$ is reducible. Up to now, it is only known  that $12\leq n\leq 77$ (see \cite{Iar84,DJNT17}). Does the first reducible singularity of $\Hilb^{n}(\mathbb{A}^3)$ for $n\geq 12$  arise  from a pyramid?
\end{remark}

\begin{remark}
The potential $F_{\pyr_3(n)}$ can be slightly simplified by a linear change of variables that  eliminates the freedom of moving the support of the subscheme defined by the pyramid ideal. We do not do this at this stage for it will break the symmetry of $F_{\pyr_3(n)}$.
\end{remark}

\subsection{Examples of local equations at Borel ideals}

\subsubsection{\texorpdfstring{$((1)\subset (2,1))$}{((1)) in (2,1))}}\label{sec:3D-121}
Let $\lambda_{121}$ be the 3D partition $((1)\subset (2,1))$. The corresponding diagram is 
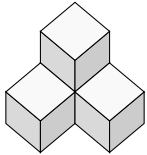
\begin{center}
\begin{tikzpicture}[x=(220:0.6cm), y=(-40:0.6cm), z=(90:0.42cm)]

\foreach \m [count=\y] in {{2,1},{1}}{
  \foreach \n [count=\x] in \m {
  \ifnum \n>0
      \foreach \z in {1,...,\n}{
        \draw [fill=gray!30] (\x+1,\y,\z) -- (\x+1,\y+1,\z) -- (\x+1, \y+1, \z-1) -- (\x+1, \y, \z-1) -- cycle;
        \draw [fill=gray!40] (\x,\y+1,\z) -- (\x+1,\y+1,\z) -- (\x+1, \y+1, \z-1) -- (\x, \y+1, \z-1) -- cycle;
        \draw [fill=gray!5] (\x,\y,\z)   -- (\x+1,\y,\z)   -- (\x+1, \y+1, \z)   -- (\x, \y+1, \z) -- cycle;  
      }
 \fi
 }
}    
\end{tikzpicture}
\end{center}
and the corresponding monomial ideal is
\[
I_{\lambda_{121}}=\left(X_1^2,X_2^2,X_3^2, X_1 X_2, X_1X_3,X_2 X_3\right).
\]
This is the pyramid $\pyr_3(2)$. For clarity, we make the change of variables
\begin{alignat}{6}\label{eq-step0-change-varialbes-A121}
 c_{100}^{200}&=a, &\quad c_{100}^{110}&=b, &\quad  c_{100}^{101}&=c, &\quad c_{100}^{020}&=d, &\quad c_{100}^{011}&=e, &\quad  c_{100}^{002}&=f,\nn\\
c_{010}^{200}&=g, &\quad c_{010}^{110}&=h, &\quad  c_{010}^{101}&=i, &\quad c_{010}^{020}&=j, &\quad c_{010}^{011}&=k, &\quad c_{010}^{002}&=l,\\
c_{001}^{200}&=m, &\quad c_{001}^{110}&=n, &\quad c_{001}^{101}&=o, &\quad c_{001}^{020}&=p, &\quad c_{001}^{011}&=q, &\quad c_{001}^{002}&=r,\nn
\end{alignat}
and set
\begin{eqnarray*}
F_{121}(a,\ldots,r)&=&-c d g + b e g + b c h - a e h + e h^2 - b^2 i + a d i - d h i -  e g j + b i j + d g k - b h k \\
 &&- c e m + b f m - e k m + d l m + 
 c^2 n - a f n + f h n - k^2 n - b l n + j l n - b c o \\
 &&+ a e o - 
 d i o + b k o + f n o - e o^2 - f g p + c i p + i k p - h l p + 
 l o p + e g q - c h q \\
 && - i j q + h k q - f m q - l n q + c o q - 
 k o q + i q^2 + e m r - c n r + k n r - i p r.
\end{eqnarray*}
Then Proposition~\ref{prop-potential-pyramid} in this case reads
\[
A_{\lambda_{121}}=\Bbbk[a,\dots,r]\Big/\left(\frac{\partial F_{121}}{\partial a},\dots,\frac{\partial F_{121}}{\partial r}\right)
\]
(see also, \textit{e.g.}, \cite[Appendix]{Kat94} and \cite[Section~4]{Ste03} for such presentations). To cancel the freedom of moving the support, we make a change of variables
\begin{alignat*}{3}
	a&\longmapsto a+h+2o,&\quad j&\longmapsto j+2b+q,&\quad r&\longmapsto r+c+2k,\\
	q&\longmapsto b+q,&\quad c&\longmapsto k+c,&\quad h&\longmapsto o+h.
\end{alignat*}
Then we obtain
\[
A_{\lambda_{121}}\cong \Bbbk[a, c, d, e, f, g, h, i, j,  l, m, n,  p, q, r]/J_{121}\otimes \Bbbk[b,k,o],
\]
where 
\begin{eqnarray}\label{eq-ideal-J121}
J_{121}&=& (e m-c n-i p,\ e g-l n+i q,\ e h-d i+f n,\ a n+g p+m q,\notag\\
&& \hphantom{(}c g-a i-l m,\ d m+j n-h p,\ c d+e j+f p,\ f g+h l+i r,\notag\\
&& \hphantom{(}d l-f q+e r,\ d g+h q+n r,\ -c h-i j-f m,\ -a e-l p-c q,\notag\\
&&\hphantom{(}-a h-g j+m r,\ -a d+j q+p r,\ -a f+j l-c r).
\end{eqnarray}
Finally, the map
\begin{alignat}{5}\label{eq-variablechange-plucker-A121}
 a&\longmapsto -p_{1,2},&\quad c&\longmapsto p_{2,4},&\quad d&\longmapsto -p_{0,3},&\quad e&\longmapsto {p}_{3,4},&\quad f&\longmapsto -p_{0,4},\notag\\
g&\longmapsto p_{1,5},&\quad h&\longmapsto p_{0,5},&\quad i&\longmapsto p_{4,5},&\quad j&\longmapsto p_{0,2},&\quad l&\longmapsto p_{1,4},\\
m&\longmapsto p_{2,5},&\quad  n&\longmapsto p_{3,5},&\quad p&\longmapsto -p_{2,3},&\quad q&\longmapsto p_{1,3},&\quad r&\longmapsto p_{0,1}\notag
\end{alignat}
transforms $J_{121}$ into  the Pl\"ucker ideal for the Grassmannian $G(2,6)$:
\begin{equation}\label{eq-pluckerideal}
\begin{split}
(&{p}_{3,4}{p}_{2,5}-{p}_{2,4}{p}_{3,5}+{p}_{2,3}{p}_{4,5},\
{p}_{3,4}{p}_{1,5}-{p}_{1,4}{p}_{3,5}+{p}_{1,3}{p}_{4,5},\
{p}_{2,4}{p}_{1,5}-{p}_{1,4}{p}_{2,5}+{p}_{1,2}{p}_{4,5},\\
&{p}_{2,3}{p}_{1,5}-{p}_{1,3}{p}_{2,5}+{p}_{1,2}{p}_{3,5},\
{p}_{3,4}{p}_{0,5}-{p}_{0,4}{p}_{3,5}+{p}_{0,3}{p}_{4,5},\
{p}_{2,4}{p}_{0,5}-{p}_{0,4}{p}_{2,5}+{p}_{0,2}{p}_{4,5},\\
&{p}_{1,4}{p}_{0,5}-{p}_{0,4}{p}_{1,5}+{p}_{0,1}{p}_{4,5},\
{p}_{2,3}{p}_{0,5}-{p}_{0,3}{p}_{2,5}+{p}_{0,2}{p}_{3,5},\
{p}_{1,3}{p}_{0,5}-{p}_{0,3}{p}_{1,5}+{p}_{0,1}{p}_{3,5},\\
&{p}_{1,2}{p}_{0,5}-{p}_{0,2}{p}_{1,5}+{p}_{0,1}{p}_{2,5},\
{p}_{2,3}{p}_{1,4}-{p}_{1,3}{p}_{2,4}+{p}_{1,2}{p}_{3,4},\
{p}_{2,3}{p}_{0,4}-{p}_{0,3}{p}_{2,4}+{p}_{0,2}{p}_{3,4},\\
&{p}_{1,3}{p}_{0,4}-{p}_{0,3}{p}_{1,4}+{p}_{0,1}{p}_{3,4},\
{p}_{1,2}{p}_{0,4}-{p}_{0,2}{p}_{1,4}+{p}_{0,1}{p}_{2,4},\
{p}_{1,2}{p}_{0,3}-{p}_{0,2}{p}_{1,3}+{p}_{0,1}{p}_{2,3}).
\end{split}
\end{equation}
Denote by $\widehat{G}(2,6)$ the cone of $G(2,6)$ in $\mathbb{P}^{14}$. 
So $\Spec(A_{\lambda_{121}})$ is isomorphic to the product $\mathbb{A}^3\times\widehat{G}(2,6)$. This is \cite[Theorem 1.6(3)]{Kat94}.

\subsubsection{\texorpdfstring{$\left((1)\subset (3,1)\right)$}{((1) in (3,1))}}\label{sec:3D-131}
Let $\lambda_{131}$ be the 3D partition $\left((1)\subset (3,1)\right)$. It corresponds to the 3D diagram and the monomial ideal
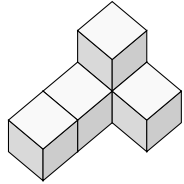
\begin{center}
\begin{tikzpicture}[x=(220:0.6cm), y=(-40:0.6cm), z=(90:0.42cm)]

\foreach \m [count=\y] in {{2,1,1},{1}}{
  \foreach \n [count=\x] in \m {
  \ifnum \n>0
      \foreach \z in {1,...,\n}{
        \draw [fill=gray!30] (\x+1,\y,\z) -- (\x+1,\y+1,\z) -- (\x+1, \y+1, \z-1) -- (\x+1, \y, \z-1) -- cycle;
        \draw [fill=gray!40] (\x,\y+1,\z) -- (\x+1,\y+1,\z) -- (\x+1, \y+1, \z-1) -- (\x, \y+1, \z-1) -- cycle;
        \draw [fill=gray!5] (\x,\y,\z)   -- (\x+1,\y,\z)   -- (\x+1, \y+1, \z)   -- (\x, \y+1, \z) -- cycle;  
      }
 \fi
 }
}    

\end{tikzpicture}
\end{center}
\[
I_{\lambda_{131}}=\left(X_1^3, X_1X_2, X_1 X_3, X_2^2,X_2 X_3, X_3^2\right).
\]
In this subsubsection, we give a quite detailed account of our method of simplifying the Haiman equations. There are 40 variables in $R_{\lambda_{131}}$. First we use the following algorithm to diminish the number of generators of the $\Bbbk$-algebra $A_{\lambda_{131}}$. 

\begin{algorithm}\label{alg-step0-Haiman}
Suppose given a $3$-dimensional partition $\lambda$.
\begin{enumerate}
	\item Find the glove $\mu$ of $\lambda$ and the minimal lattices $\mathrm{min}(\mu)$ of $\mu$.
	\item Find the adjoint pairs in $\mu$.
	\item Define monomials $c_{i}^j$ for $i\in \lambda,j\in \mu$, and the ring $R_{\lambda}=\Bbbk[c_{i}^j]_{i\in \lambda,j\in \mu}$. 
	\item Define the Haiman equations (\ref{eq-Huibregtse-2}).
	\item Run the \emph{simple elimination} of the Haiman equations for the Haiman coordinates $c_{i}^j$, where $i\in \lambda$ and $j\in \mu\backslash \mathrm{min}(\mu)$. Here by a simple elimination, we mean that if there is an equation of the form
	\[
	ax-f(y,z,\dots)=0, 
	\]
	where $0\neq a\in \mathbb{Q}$ and $f(y,z,\dots)$ is a polynomial of variables other than $x$, then we eliminate the variable $x$ and replace the appearance of $x$ in the other equations by $f(y,z,\dots)$.
	\item Run the simple elimination for all the remaining Haiman coordinates, including those not eliminated  in the previous step and the coordinates $c_{i}^j$, where $i\in \lambda$ and $j\in \mathrm{min}(\mu)$.
	\item Finally, for the convenience of finding further simplifications, we reindex the remaining variables by $x_1$, $x_2,\dots$ and return this ring.
\end{enumerate}
\end{algorithm}

We implement this algorithm in Macaulay2. Since our base field $\Bbbk$ contains $\mathbb{Q}$, the result is valid in $\Bbbk$.\footnote{In fact, this is probably valid even over $\mathbb{Z}$. We have not checked this. The following  steps of change of variables are over $\mathbb{Z}$.} There are  21 remaining coordinates:
\begin{alignat*}{6}
c_{1,0,0}^{1,1,0} &\longmapsto  x_{1},&\quad c_{1,0,0}^{1,0,1} &\longmapsto x_{2},&\quad c_{1,0,0}^{3,0,0} &\longmapsto  x_{3},&\quad 
c_{2,0,0}^{1,1,0} &\longmapsto x_{4},&\quad c_{2,0,0}^{1,0,1} &\longmapsto  x_{5},&\quad c_{2,0,0}^{3,0,0} &\longmapsto x_{6},\\
c_{2,0,0}^{0,2,0} &\longmapsto  x_{7},&\quad c_{2,0,0}^{0,1,1} &\longmapsto x_{8},&\quad c_{2,0,0}^{0,0,2} &\longmapsto  x_{9},&\quad 
 c_{0,1,0}^{1,1,0} &\longmapsto x_{10},&\quad c_{0,1,0}^{1,0,1} &\longmapsto  x_{11},&\quad c_{0,1,0}^{3,0,0} &\longmapsto x_{12},\\
 c_{0,1,0}^{0,2,0} &\longmapsto  x_{13},&\quad c_{0,1,0}^{0,1,1} &\longmapsto x_{14},&\quad c_{0,1,0}^{0,0,2} &\longmapsto  x_{15},&\quad 
  c_{0,0,1}^{1,1,0} &\longmapsto x_{16},&\quad c_{0,0,1}^{1,0,1} &\longmapsto  x_{17},&\quad c_{0,0,1}^{3,0,0} &\longmapsto x_{18},\\
 c_{0,0,1}^{0,2,0} &\longmapsto  x_{19},&\quad c_{0,0,1}^{0,1,1} &\longmapsto x_{20},&\quad c_{0,0,1}^{0,0,2} &\longmapsto  x_{21}.&&&&&
  \end{alignat*}
Thus $A_{\lambda_{131}}$ is generated by $x_1,\dots,x_{21}$. We denote the resulting equations for $x_1,\dots,x_{21}$ by $\mathcal{H}'_{\lambda_{131}}$. We have
\[
A_{\lambda_{131}}\cong \Bbbk[x_1,\dots,x_{21}]/\mathcal{H}'_{\lambda_{131}}.
\]
We call the above procedure \textit{Step} 0.

The equations in $\mathcal{H}'_{\lambda_{131}}$ are complicated; they may have degrees up to $5$ and are long. We select a subset of them.

\begin{definition}
For a polynomial $f$,  we denote by $\mathrm{minDegree}(f)$  the lowest degree of the nonzero monomials in $f$. 
\end{definition}

The  equations of minDegree $2$ of $\mathcal{H}'_{\lambda}$ are
\begin{eqnarray*}
&&-x_{4}x_{5}x_{10}+x_{4}^{2}x_{11}-x_{5}^{2}x_{16}+x_{4}x_{5}x_{17}+x_{8}x_{10}-x_{7}x_{11}+x_{9}x_{16}-x_{8}x_{17},\\
&&x_{4}x_{10}^{2}+x_{4}^{2}x_{12}+x_{5}x_{10}x_{16}+x_{5}x_{16}x_{17}-x_{4}x_{17}^{2}+x_{1}x_{10}-x_{7}x_{12}+x_{2}x_{16}+x_{14}x_{16}-x_{1}x_{17}\\
&&\quad -x_{10}x_{20}+x_{17}x_{20}-x_{16}x_{21},\\
&&x_{4}x_{10}x_{16}+x_{4}x_{16}x_{17}+x_{4}^{2}x_{18}+x_{1}x_{16}-x_{13}x_{16}-x_{7}x_{18}+x_{10}x_{19}-x_{17}x_{19}+x_{16}x_{20},\\
&&x_{4}x_{10}x_{11}+x_{4}x_{5}x_{12}+x_{4}x_{11}x_{17}+x_{1}x_{11}-x_{8}x_{12}+x_{15}x_{16}-x_{11}x_{20},\\
&&x_{5}x_{10}x_{16}+x_{5}x_{16}x_{17}+x_{4}x_{5}x_{18}+x_{2}x_{16}-x_{14}x_{16}-x_{8}x_{18}+x_{11}x_{19},\\
&&x_{4}x_{5}x_{10}-x_{4}^{2}x_{11}+x_{5}^{2}x_{16}-x_{4}x_{5}x_{17}-x_{8}x_{10}+x_{7}x_{11}-x_{9}x_{16}+x_{8}x_{17},\\
&&x_{4}x_{10}x_{11}+x_{4}x_{5}x_{12}+x_{4}x_{11}x_{17}+x_{1}x_{11}-x_{8}x_{12}+x_{15}x_{16}-x_{11}x_{20},\\
&&x_{5}x_{10}x_{16}+x_{5}x_{16}x_{17}+x_{4}x_{5}x_{18}+x_{2}x_{16}-x_{14}x_{16}-x_{8}x_{18}+x_{11}x_{19},\\
&&x_{5}x_{10}x_{11}+x_{5}^{2}x_{12}+x_{5}x_{11}x_{17}+x_{2}x_{11}-x_{9}x_{12}+x_{11}x_{14}-x_{10}x_{15}+x_{15}x_{17}-x_{11}x_{21},\\
&&-x_{5}x_{10}^{2}+x_{4}x_{10}x_{11}+x_{4}x_{11}x_{17}+x_{5}x_{17}^{2}+x_{5}^{2}x_{18}-x_{2}x_{10}+x_{1}x_{11}-x_{11}x_{13}+x_{10}x_{14}\\
&&\quad+x_{2}x_{17}-x_{14}x_{17}-x_{9}x_{18}+x_{11}x_{20},\\
&&x_{6}x_{10}x_{16}-x_{10}^{2}x_{16}-x_{4}x_{12}x_{16}-x_{11}x_{16}^{2}+x_{6}x_{16}x_{17}-x_{10}x_{16}x_{17}-x_{16}x_{17}^{2}-x_{4}x_{10}x_{18}\\
&&\quad-x_{5}x_{16}x_{18}-x_{4}x_{17}x_{18}+x_{3}x_{16}-x_{1}x_{18}+x_{12}x_{19}+x_{18}x_{20},\\
&&x_{6}x_{10}x_{11}-x_{10}^{2}x_{11}-x_{5}x_{10}x_{12}-x_{4}x_{11}x_{12}-x_{11}^{2}x_{16}+x_{6}x_{11}x_{17}-x_{10}x_{11}x_{17}-x_{5}x_{12}x_{17}\\
&&\quad-x_{11}x_{17}^{2}-x_{5}x_{11}x_{18}+x_{3}x_{11}-x_{2}x_{12}+x_{12}x_{14}+x_{15}x_{18},\\
&&-x_{6}x_{10}^{2}+x_{10}^{3}+2\,x_{4}x_{10}x_{12}+x_{10}x_{11}x_{16}-x_{11}x_{16}x_{17}+x_{6}x_{17}^{2}-x_{17}^{3}-2\,x_{5}x_{17}x_{18}-x_{3}x_{10}\\
&&\quad+x_{1}x_{12}-x_{12}x_{13}+x_{3}x_{17}-x_{2}x_{18}-x_{14}x_{18}+x_{12}x_{20}+x_{18}x_{21},\\
&&-x_{4}x_{10}^{2}-x_{4}^{2}x_{12}-2\,x_{5}x_{10}x_{16}-2\,x_{5}x_{16}x_{17}+x_{4}x_{17}^{2}-x_{4}x_{5}x_{18}-x_{1}x_{10}+x_{7}x_{12}-2\,x_{2}x_{16}\\
&&\quad+x_{1}x_{17}+x_{8}x_{18}-x_{11}x_{19}+x_{10}x_{20}-x_{17}x_{20}+x_{16}x_{21},\\
&&x_{5}x_{10}^{2}-2\,x_{4}x_{10}x_{11}-x_{4}x_{5}x_{12}-2\,x_{4}x_{11}x_{17}-x_{5}x_{17}^{2}-x_{5}^{2}x_{18}+x_{2}x_{10}-2\,x_{1}x_{11}+x_{8}x_{12}\\
&&\quad+x_{11}x_{13}-x_{10}x_{14}-x_{15}x_{16}-x_{2}x_{17}+x_{14}x_{17}+x_{9}x_{18},\\
&&x_{4}^{2}x_{5}x_{10}-x_{4}^{3}x_{11}+x_{4}x_{5}^{2}x_{16}-x_{4}^{2}x_{5}x_{17}-x_{2}x_{4}^{2}+x_{1}x_{4}x_{5}+x_{5}x_{7}x_{10}-2\,x_{4}x_{8}x_{10}+x_{4}x_{7}x_{11}\\
&&\quad-x_{4}x_{5}x_{13}+x_{4}^{2}x_{14}-x_{4}x_{9}x_{16}+x_{5}x_{7}x_{17}-x_{5}^{2}x_{19}+x_{4}x_{5}x_{20}+x_{2}x_{7}-x_{1}x_{8}+x_{8}x_{13}\\
&&\quad-x_{7}x_{14}+x_{9}x_{19}-x_{8}x_{20},\\
&&-2\,x_{4}x_{5}x_{6}x_{10}+x_{4}^{2}x_{6}x_{11}+x_{4}^{2}x_{5}x_{12}-x_{5}^{2}x_{6}x_{16}+x_{4}x_{5}x_{11}x_{16}-2\,x_{4}x_{5}x_{10}x_{17}+x_{4}^{2}x_{11}x_{17}\\
&&\quad-x_{5}^{2}x_{16}x_{17}+x_{4}x_{5}x_{17}^{2}+x_{4}x_{5}^{2}x_{18}-x_{3}x_{4}x_{5}-2\,x_{2}x_{4}x_{10}-2\,x_{1}x_{5}x_{10}+2\,x_{6}x_{8}x_{10}-3\,x_{8}x_{10}^{2}\\
&&\quad+2\,x_{1}x_{4}x_{11}-x_{6}x_{7}x_{11}+2\,x_{7}x_{10}x_{11}+x_{5}x_{7}x_{12}-2\,x_{4}x_{8}x_{12}-x_{4}x_{11}x_{13}+2\,x_{4}x_{10}x_{14}\\
&&\quad-2\,x_{2}x_{5}x_{16}+x_{6}x_{9}x_{16}-2\,x_{9}x_{10}x_{16}-x_{8}x_{11}x_{16}+x_{4}x_{15}x_{16}+x_{7}x_{11}x_{17}-x_{9}x_{16}x_{17}\\
&&\quad-x_{4}x_{9}x_{18}-x_{5}x_{11}x_{19}+2\,x_{5}x_{10}x_{20}+x_{5}x_{16}x_{21}-x_{1}x_{2}+x_{3}x_{8}+x_{1}x_{14}+x_{15}x_{19}\\
&&\quad+x_{2}x_{20}-x_{14}x_{20},\\
&&x_{4}^{2}x_{6}x_{10}-x_{4}^{3}x_{12}-x_{4}^{2}x_{11}x_{16}+x_{4}^{2}x_{6}x_{17}+x_{4}^{2}x_{10}x_{17}-x_{4}^{2}x_{5}x_{18}+x_{3}x_{4}^{2}+2\,x_{1}x_{4}x_{10}\\
&&\quad-x_{6}x_{7}x_{10}+x_{7}x_{10}^{2}+x_{4}x_{7}x_{12}-x_{4}x_{10}x_{13}+x_{7}x_{11}x_{16}+2\,x_{1}x_{4}x_{17}-x_{6}x_{7}x_{17}+x_{7}x_{10}x_{17}\\
&&\quad-x_{4}x_{13}x_{17}+x_{7}x_{17}^{2}+x_{5}x_{7}x_{18}-x_{5}x_{10}x_{19}-x_{5}x_{17}x_{19}+x_{1}^{2}-x_{3}x_{7}-x_{1}x_{13}-x_{2}x_{19}\\
&&\quad-x_{14}x_{19}+x_{13}x_{20}-x_{20}^{2}+x_{19}x_{21},\\
&&-x_{2}x_{4}x_{5}+x_{1}x_{5}^{2}+x_{5}x_{8}x_{10}-x_{4}x_{9}x_{10}-x_{4}x_{5}x_{14}+x_{4}^{2}x_{15}+x_{5}x_{8}x_{17}-x_{4}x_{9}x_{17}-x_{5}^{2}x_{20}\\
&&\quad+x_{4}x_{5}x_{21}+x_{2}x_{8}-x_{1}x_{9}+x_{8}x_{14}-x_{7}x_{15}+x_{9}x_{20}-x_{8}x_{21},\\
&&-x_{5}^{2}x_{6}x_{10}+x_{4}x_{5}^{2}x_{12}+x_{5}^{2}x_{11}x_{16}-x_{5}^{2}x_{6}x_{17}-x_{5}^{2}x_{10}x_{17}+x_{5}^{3}x_{18}-x_{3}x_{5}^{2}-2\,x_{2}x_{5}x_{10}\\
&&\quad+x_{6}x_{9}x_{10}-x_{9}x_{10}^{2}-x_{4}x_{9}x_{12}+x_{4}x_{10}x_{15}-x_{9}x_{11}x_{16}-2\,x_{2}x_{5}x_{17}+x_{6}x_{9}x_{17}-x_{9}x_{10}x_{17}\\
&&\quad+x_{4}x_{15}x_{17}-x_{9}x_{17}^{2}-x_{5}x_{9}x_{18}+x_{5}x_{10}x_{21}+x_{5}x_{17}x_{21}-x_{2}^{2}+x_{3}x_{9}+x_{14}^{2}+x_{1}x_{15}-x_{13}x_{15}\\
&&\quad+x_{15}x_{20}+x_{2}x_{21}-x_{14}x_{21},\\
&&x_{4}x_{5}x_{6}x_{10}-x_{4}^{2}x_{5}x_{12}-x_{4}x_{5}x_{11}x_{16}+x_{4}x_{5}x_{6}x_{17}+x_{4}x_{5}x_{10}x_{17}-x_{4}x_{5}^{2}x_{18}+x_{3}x_{4}x_{5}\\
&&\quad+x_{2}x_{4}x_{10}+x_{1}x_{5}x_{10}-x_{6}x_{8}x_{10}+x_{8}x_{10}^{2}+x_{4}x_{8}x_{12}-x_{4}x_{10}x_{14}+x_{8}x_{11}x_{16}+x_{2}x_{4}x_{17}\\
&&\quad+x_{1}x_{5}x_{17}-x_{6}x_{8}x_{17}+x_{8}x_{10}x_{17}-x_{4}x_{14}x_{17}+x_{8}x_{17}^{2}+x_{5}x_{8}x_{18}-x_{5}x_{10}x_{20}-x_{5}x_{17}x_{20}\\
&&\quad+x_{1}x_{2}-x_{3}x_{8}-x_{1}x_{14}-x_{15}x_{19}-x_{2}x_{20}+x_{14}x_{20}.
\end{eqnarray*}
There are also minDegree at least $3$ equations. A typical one is
\begin{eqnarray*}
&& 2\,x_{4}x_{5}x_{6}x_{10}^{2}-2\,x_{4}^{2}x_{6}x_{10}x_{11}-2\,x_{4}^{2}x_{5}x_{10}x_{12}+2\,x_{4}^{3}x_{11}x_{12}+2\,x_{5}^{2}x_{6}x_{10}x_{16}-x_{4}x_{5}x_{10}x_{11}x_{16}\\
&&\quad+x_{4}^{2}x_{11}^{2}x_{16}-2\,x_{4}x_{5}^{2}x_{12}x_{16}-x_{5}^{2}x_{11}x_{16}^{2}-x_{4}x_{5}x_{6}x_{10}x_{17}+x_{4}x_{5}x_{10}^{2}x_{17}-x_{4}^{2}x_{6}x_{11}x_{17}\\
&&\quad-x_{4}^{2}x_{10}x_{11}x_{17}+2\,x_{4}^{2}x_{5}x_{12}x_{17}+x_{5}^{2}x_{6}x_{16}x_{17}+x_{5}^{2}x_{10}x_{16}x_{17}+x_{4}x_{5}x_{11}x_{16}x_{17}-x_{4}x_{5}x_{6}x_{17}^{2}\\
&&\quad-2\,x_{4}x_{5}x_{10}x_{17}^{2}+x_{4}^{2}x_{11}x_{17}^{2}-x_{5}^{2}x_{16}x_{17}^{2}+x_{4}x_{5}x_{17}^{3}-x_{4}x_{5}^{2}x_{10}x_{18}+x_{4}^{2}x_{5}x_{11}x_{18}-x_{5}^{3}x_{16}x_{18}\\
&&\quad+x_{4}x_{5}^{2}x_{17}x_{18}+2\,x_{3}x_{4}x_{5}x_{10}+x_{2}x_{4}x_{10}^{2}+x_{1}x_{5}x_{10}^{2}-x_{6}x_{8}x_{10}^{2}+x_{8}x_{10}^{3}-2\,x_{3}x_{4}^{2}x_{11}\\
&&\quad-2\,x_{1}x_{4}x_{10}x_{11}+x_{6}x_{7}x_{10}x_{11}-x_{7}x_{10}^{2}x_{11}+x_{2}x_{4}^{2}x_{12}-x_{1}x_{4}x_{5}x_{12}-x_{5}x_{7}x_{10}x_{12}+2\,x_{4}x_{8}x_{10}x_{12}\\
&&\quad-x_{4}x_{7}x_{11}x_{12}+x_{4}x_{10}x_{11}x_{13}+x_{4}x_{5}x_{12}x_{13}-x_{4}x_{10}^{2}x_{14}-x_{4}^{2}x_{12}x_{14}+2\,x_{3}x_{5}^{2}x_{16}+2\,x_{2}x_{5}x_{10}x_{16}\\
&&\quad-x_{6}x_{9}x_{10}x_{16}+x_{9}x_{10}^{2}x_{16}+x_{4}x_{9}x_{12}x_{16}-x_{4}x_{10}x_{15}x_{16}-2\,x_{3}x_{4}x_{5}x_{17}-x_{2}x_{4}x_{10}x_{17}\\
&&\quad-x_{1}x_{5}x_{10}x_{17}+x_{6}x_{8}x_{10}x_{17}-x_{8}x_{10}^{2}x_{17}-x_{4}x_{8}x_{12}x_{17}+x_{4}x_{10}x_{14}x_{17}+x_{2}x_{4}x_{5}x_{18}-x_{1}x_{5}^{2}x_{18}\\
&&\quad-x_{5}x_{8}x_{10}x_{18}+x_{4}x_{9}x_{10}x_{18}+x_{4}x_{5}x_{14}x_{18}-x_{4}^{2}x_{15}x_{18}+x_{5}x_{10}x_{11}x_{19}+x_{5}^{2}x_{12}x_{19}-x_{5}x_{10}^{2}x_{20}\\
&&\quad-x_{4}x_{5}x_{12}x_{20}+x_{5}x_{10}x_{17}x_{20}+x_{5}^{2}x_{18}x_{20}-x_{5}x_{10}x_{16}x_{21}-x_{4}x_{5}x_{18}x_{21}-x_{3}x_{8}x_{10}+x_{3}x_{7}x_{11}\\
&&\quad-x_{3}x_{9}x_{16}+x_{3}x_{8}x_{17}.
\end{eqnarray*}
We do not present all the equations in $\mathcal{H}'_{\lambda}$.\footnote{The reader can find them using https://github.com/huxw06/Hilbert-scheme-of-points.} In fact, it turns out that, for $\lambda_{131}$, we can avoid the use of equations with minDegree at least $3$, as we will see at the end of this subsubsection.
Now we start to make changes of variables, in several steps.

\begin{enumerate}[label={\it Step}~\arabic*:, ref=\arabic*]
\item\label{step1}
\begin{eqnarray*}
x_{1} &\longmapsto& x_{1}+x_{13},\\
x_{2} &\longmapsto& x_{2}+x_{14},\\
x_{10} &\longmapsto& x_{10}+x_{17},\\
x_{13} &\longmapsto& x_{1}+2 x_{13}+x_{20},\\
x_{20} &\longmapsto& x_{1}+x_{13}+x_{20},\\
x_{21} &\longmapsto& x_{2}+2 x_{14}+x_{21}.
\end{eqnarray*}
Then the  equations of minDegree 2 are transformed into
\allowdisplaybreaks 
\begin{eqnarray}\label{eq-5points-step1}
&&  -x_{4}x_{5}x_{10}+x_{4}^{2}x_{11}-x_{5}^{2}x_{16}+x_{8}x_{10}-x_{7}x_{11}+x_{9}x_{16},\notag\\
&&  x_{4}x_{10}^{2}+x_{4}^{2}x_{12}+x_{5}x_{10}x_{16}+2\,x_{4}x_{10}x_{17}+2\,x_{5}x_{16}x_{17}-x_{7}x_{12}-x_{10}x_{20}-x_{16}x_{21},\notag\\
&&  x_{4}x_{10}x_{16}+2\,x_{4}x_{16}x_{17}+x_{4}^{2}x_{18}+x_{1}x_{16}-x_{7}x_{18}+x_{10}x_{19},\notag\\
&&  x_{4}x_{10}x_{11}+x_{4}x_{5}x_{12}+2\,x_{4}x_{11}x_{17}-x_{8}x_{12}+x_{15}x_{16}-x_{11}x_{20},\notag\\
&&  x_{5}x_{10}x_{16}+2\,x_{5}x_{16}x_{17}+x_{4}x_{5}x_{18}+x_{2}x_{16}-x_{8}x_{18}+x_{11}x_{19},\notag\\
&&  x_{4}x_{5}x_{10}-x_{4}^{2}x_{11}+x_{5}^{2}x_{16}-x_{8}x_{10}+x_{7}x_{11}-x_{9}x_{16},\notag\\
&&  x_{4}x_{10}x_{11}+x_{4}x_{5}x_{12}+2\,x_{4}x_{11}x_{17}-x_{8}x_{12}+x_{15}x_{16}-x_{11}x_{20},\notag\\
&&  x_{5}x_{10}x_{16}+2\,x_{5}x_{16}x_{17}+x_{4}x_{5}x_{18}+x_{2}x_{16}-x_{8}x_{18}+x_{11}x_{19},\notag\\
&&  x_{5}x_{10}x_{11}+x_{5}^{2}x_{12}+2\,x_{5}x_{11}x_{17}-x_{9}x_{12}-x_{10}x_{15}-x_{11}x_{21},\notag\\
&&  -x_{5}x_{10}^{2}+x_{4}x_{10}x_{11}-2\,x_{5}x_{10}x_{17}+2\,x_{4}x_{11}x_{17}+x_{5}^{2}x_{18}-x_{2}x_{10}+x_{1}x_{11}-x_{9}x_{18},\notag\\
&&  x_{6}x_{10}x_{16}-x_{10}^{2}x_{16}-x_{4}x_{12}x_{16}-x_{11}x_{16}^{2}+2\,x_{6}x_{16}x_{17}-3\,x_{10}x_{16}x_{17}-3\,x_{16}x_{17}^{2}\notag\\
&&\quad  -x_{4}x_{10}x_{18}-x_{5}x_{16}x_{18}-2\,x_{4}x_{17}x_{18}+x_{3}x_{16}+x_{12}x_{19}+x_{18}x_{20},\notag\\
&&  x_{6}x_{10}x_{11}-x_{10}^{2}x_{11}-x_{5}x_{10}x_{12}-x_{4}x_{11}x_{12}-x_{11}^{2}x_{16}+2\,x_{6}x_{11}x_{17}-3\,x_{10}x_{11}x_{17}-2\,x_{5}x_{12}x_{17}\notag\\
&& \quad -3\,x_{11}x_{17}^{2}-x_{5}x_{11}x_{18}+x_{3}x_{11}-x_{2}x_{12}+x_{15}x_{18},\notag\\
&&  -x_{6}x_{10}^{2}+x_{10}^{3}+2\,x_{4}x_{10}x_{12}+x_{10}x_{11}x_{16}-2\,x_{6}x_{10}x_{17}+3\,x_{10}^{2}x_{17}+2\,x_{4}x_{12}x_{17}+3\,x_{10}x_{17}^{2}\notag\\
&& \quad -2\,x_{5}x_{17}x_{18}-x_{3}x_{10}+x_{1}x_{12}+x_{18}x_{21},\notag\\
&&  -x_{4}x_{10}^{2}-x_{4}^{2}x_{12}-2\,x_{5}x_{10}x_{16}-2\,x_{4}x_{10}x_{17}-4\,x_{5}x_{16}x_{17}-x_{4}x_{5}x_{18}+x_{7}x_{12}-x_{2}x_{16}+x_{8}x_{18}\notag\\
&& \quad -x_{11}x_{19}+x_{10}x_{20}+x_{16}x_{21},\notag\\
&&  x_{5}x_{10}^{2}-2\,x_{4}x_{10}x_{11}-x_{4}x_{5}x_{12}+2\,x_{5}x_{10}x_{17}-4\,x_{4}x_{11}x_{17}-x_{5}^{2}x_{18}+x_{2}x_{10}-x_{1}x_{11}+x_{8}x_{12}\notag\\
&&  \quad-x_{15}x_{16}+x_{9}x_{18}+x_{11}x_{20},\notag\\
&&  x_{4}^{2}x_{5}x_{10}-x_{4}^{3}x_{11}+x_{4}x_{5}^{2}x_{16}-x_{2}x_{4}^{2}+x_{1}x_{4}x_{5}+x_{5}x_{7}x_{10}-2\,x_{4}x_{8}x_{10}+x_{4}x_{7}x_{11}-x_{4}x_{9}x_{16}\notag\\
&& \quad +2\,x_{5}x_{7}x_{17}-2\,x_{4}x_{8}x_{17}-x_{5}^{2}x_{19}+x_{2}x_{7}-x_{1}x_{8}+x_{9}x_{19},\notag\\
&&  -2\,x_{4}x_{5}x_{6}x_{10}+x_{4}^{2}x_{6}x_{11}+x_{4}^{2}x_{5}x_{12}-x_{5}^{2}x_{6}x_{16}+x_{4}x_{5}x_{11}x_{16}-2\,x_{4}x_{5}x_{6}x_{17}-2\,x_{4}x_{5}x_{10}x_{17}\notag\\
&&  \quad+x_{4}^{2}x_{11}x_{17}-x_{5}^{2}x_{16}x_{17}-x_{4}x_{5}x_{17}^{2}+x_{4}x_{5}^{2}x_{18}-x_{3}x_{4}x_{5}-2\,x_{2}x_{4}x_{10}+2\,x_{6}x_{8}x_{10}-3\,x_{8}x_{10}^{2}\notag\\
&&  \quad+x_{1}x_{4}x_{11}-x_{6}x_{7}x_{11}+2\,x_{7}x_{10}x_{11}+x_{5}x_{7}x_{12}-2\,x_{4}x_{8}x_{12}-x_{2}x_{5}x_{16}+x_{6}x_{9}x_{16}-2\,x_{9}x_{10}x_{16}\notag\\
&&  \quad-x_{8}x_{11}x_{16}+x_{4}x_{15}x_{16}-2\,x_{2}x_{4}x_{17}+2\,x_{6}x_{8}x_{17}-6\,x_{8}x_{10}x_{17}+3\,x_{7}x_{11}x_{17}-3\,x_{9}x_{16}x_{17}\notag\\
&& \quad -3\,x_{8}x_{17}^{2}-x_{4}x_{9}x_{18}-x_{5}x_{11}x_{19}+2\,x_{5}x_{10}x_{20}-x_{4}x_{11}x_{20}+2\,x_{5}x_{17}x_{20}+x_{5}x_{16}x_{21}+x_{3}x_{8}\notag\\
&&  \quad+x_{15}x_{19}+x_{2}x_{20},\notag\\
&&  x_{4}^{2}x_{6}x_{10}-x_{4}^{3}x_{12}-x_{4}^{2}x_{11}x_{16}+2\,x_{4}^{2}x_{6}x_{17}+x_{4}^{2}x_{10}x_{17}+x_{4}^{2}x_{17}^{2}-x_{4}^{2}x_{5}x_{18}+x_{3}x_{4}^{2}+x_{1}x_{4}x_{10}\notag\\
&&  \quad-x_{6}x_{7}x_{10}+x_{7}x_{10}^{2}+x_{4}x_{7}x_{12}+x_{7}x_{11}x_{16}+2\,x_{1}x_{4}x_{17}-2\,x_{6}x_{7}x_{17}+3\,x_{7}x_{10}x_{17}+3\,x_{7}x_{17}^{2}\notag\\
&&  \quad+x_{5}x_{7}x_{18}-x_{5}x_{10}x_{19}-2\,x_{5}x_{17}x_{19}-x_{4}x_{10}x_{20}-2\,x_{4}x_{17}x_{20}-x_{3}x_{7}-x_{1}x_{20}+x_{19}x_{21},\notag\\
&&  x_{5}x_{8}x_{10}-x_{4}x_{9}x_{10}+x_{4}^{2}x_{15}+2\,x_{5}x_{8}x_{17}-2\,x_{4}x_{9}x_{17}-x_{5}^{2}x_{20}+x_{4}x_{5}x_{21}-x_{7}x_{15}+x_{9}x_{20}\notag\\
&&  \quad-x_{8}x_{21},\notag\\
&&  -x_{5}^{2}x_{6}x_{10}+x_{4}x_{5}^{2}x_{12}+x_{5}^{2}x_{11}x_{16}-2\,x_{5}^{2}x_{6}x_{17}-x_{5}^{2}x_{10}x_{17}-x_{5}^{2}x_{17}^{2}+x_{5}^{3}x_{18}-x_{3}x_{5}^{2}-x_{2}x_{5}x_{10}\notag\\
&&  \quad+x_{6}x_{9}x_{10}-x_{9}x_{10}^{2}-x_{4}x_{9}x_{12}+x_{4}x_{10}x_{15}-x_{9}x_{11}x_{16}-2\,x_{2}x_{5}x_{17}+2\,x_{6}x_{9}x_{17}-3\,x_{9}x_{10}x_{17}\notag\\
&&  \quad+2\,x_{4}x_{15}x_{17}-3\,x_{9}x_{17}^{2}-x_{5}x_{9}x_{18}+x_{5}x_{10}x_{21}+2\,x_{5}x_{17}x_{21}+x_{3}x_{9}+x_{1}x_{15}+x_{2}x_{21},\notag\\
&&  x_{4}x_{5}x_{6}x_{10}-x_{4}^{2}x_{5}x_{12}-x_{4}x_{5}x_{11}x_{16}+2\,x_{4}x_{5}x_{6}x_{17}+x_{4}x_{5}x_{10}x_{17}+x_{4}x_{5}x_{17}^{2}-x_{4}x_{5}^{2}x_{18}\notag\\
&&  \quad+x_{3}x_{4}x_{5}+x_{2}x_{4}x_{10}-x_{6}x_{8}x_{10}+x_{8}x_{10}^{2}+x_{4}x_{8}x_{12}+x_{8}x_{11}x_{16}+2\,x_{2}x_{4}x_{17}-2\,x_{6}x_{8}x_{17}\notag\\
&&  \quad+3\,x_{8}x_{10}x_{17}+3\,x_{8}x_{17}^{2}+x_{5}x_{8}x_{18}-x_{5}x_{10}x_{20}-2\,x_{5}x_{17}x_{20}-x_{3}x_{8}-x_{15}x_{19}-x_{2}x_{20}.
 \end{eqnarray}
 Now we are going to find a nonhomogeneous, but weighted homogeneous, change of variables, to \emph{absorb} the terms of degree at least $3$. 
The above \textit{Step}~\ref{step1} change of variables has greatly diminished the possible choices for such absorbing.

\item\label{step2}
\begin{eqnarray*}
x_{1} &\longmapsto& x_{1} - 2 x_{4} x_{17}, \\
x_{2} &\longmapsto& x_{2} - x_{5} x_{10} + x_{4} x_{11} - 2 x_{5} x_{17}, \\
x_{7} &\longmapsto&  x_{7} + x_{4}^2, \\
x_{8} &\longmapsto& x_{8} + x_{4} x_{5}, \\
x_{9} &\longmapsto& x_{9} +  x_{5}^2, \\
x_{19} &\longmapsto& x_{19} - x_{4} x_{16}, \\
x_{20} &\longmapsto& x_{20} + 2 x_{4} x_{17} + x_{4} x_{10}, \\
x_{21} &\longmapsto& x_{21} + x_{5} x_{10} + 2 x_{5} x_{17}.
\end{eqnarray*}
Then  Equations (\ref{eq-5points-step1}) are transformed into
\begin{eqnarray}\label{eq-5points-step2-1}
&&x_{8}x_{10}-x_{7}x_{11}+x_{9}x_{16},\notag\\
&&-x_{7}x_{12}-x_{10}x_{20}-x_{16}x_{21},\notag\\
&&x_{1}x_{16}-x_{7}x_{18}+x_{10}x_{19},\notag\\
&&-x_{8}x_{12}+x_{15}x_{16}-x_{11}x_{20},\notag\\
&&x_{2}x_{16}-x_{8}x_{18}+x_{11}x_{19},\notag\\
&&-x_{8}x_{10}+x_{7}x_{11}-x_{9}x_{16},\notag\\
&&-x_{8}x_{12}+x_{15}x_{16}-x_{11}x_{20},\notag\\
&&x_{2}x_{16}-x_{8}x_{18}+x_{11}x_{19},\notag\\
&&-x_{9}x_{12}-x_{10}x_{15}-x_{11}x_{21},\notag\\
&&-x_{2}x_{10}+x_{1}x_{11}-x_{9}x_{18},\notag\\
&&x_{6}x_{10}x_{16}-x_{10}^{2}x_{16}-2\,x_{4}x_{12}x_{16}-x_{11}x_{16}^{2}+2\,x_{6}x_{16}x_{17}-3\,x_{10}x_{16}x_{17}-3\,x_{16}x_{17}^{2}\notag\\
&&\quad-x_{5}x_{16}x_{18}+x_{3}x_{16}+x_{12}x_{19}+x_{18}x_{20},\notag\\
&&x_{6}x_{10}x_{11}-x_{10}^{2}x_{11}-2\,x_{4}x_{11}x_{12}-x_{11}^{2}x_{16}+2\,x_{6}x_{11}x_{17}-3\,x_{10}x_{11}x_{17}-3\,x_{11}x_{17}^{2}\notag\\
&&\quad -x_{5}x_{11}x_{18}+x_{3}x_{11}-x_{2}x_{12}+x_{15}x_{18},\notag\\
&&-x_{6}x_{10}^{2}+x_{10}^{3}+2\,x_{4}x_{10}x_{12}+x_{10}x_{11}x_{16}-2\,x_{6}x_{10}x_{17}+3\,x_{10}^{2}x_{17}+3\,x_{10}x_{17}^{2}\notag\\
&&\quad +x_{5}x_{10}x_{18}-x_{3}x_{10}+x_{1}x_{12}+x_{18}x_{21},\notag\\
&&x_{7}x_{12}-x_{2}x_{16}+x_{8}x_{18}-x_{11}x_{19}+x_{10}x_{20}+x_{16}x_{21},\notag\\
&&x_{2}x_{10}-x_{1}x_{11}+x_{8}x_{12}-x_{15}x_{16}+x_{9}x_{18}+x_{11}x_{20},\notag\\
&&-2\,x_{4}x_{8}x_{10}+2\,x_{4}x_{7}x_{11}-2\,x_{4}x_{9}x_{16}+x_{2}x_{7}-x_{1}x_{8}+x_{9}x_{19},\notag\\
&&-x_{2}x_{4}x_{10}+2\,x_{6}x_{8}x_{10}-3\,x_{8}x_{10}^{2}+x_{1}x_{4}x_{11}-x_{6}x_{7}x_{11}+2\,x_{7}x_{10}x_{11}+x_{5}x_{7}x_{12}\notag\\
&&\quad-2\,x_{4}x_{8}x_{12}-x_{2}x_{5}x_{16}+x_{6}x_{9}x_{16}-2\,x_{9}x_{10}x_{16}-x_{8}x_{11}x_{16}+2\,x_{6}x_{8}x_{17}-6\,x_{8}x_{10}x_{17}\notag\\
&&\quad+3\,x_{7}x_{11}x_{17}-3\,x_{9}x_{16}x_{17}-3\,x_{8}x_{17}^{2}-x_{4}x_{9}x_{18}-x_{5}x_{11}x_{19}+x_{5}x_{10}x_{20}+x_{5}x_{16}x_{21}\notag\\
&&+x_{3}x_{8}+x_{15}x_{19}+x_{2}x_{20},\notag\\
&&-x_{6}x_{7}x_{10}+x_{7}x_{10}^{2}+x_{4}x_{7}x_{12}+x_{7}x_{11}x_{16}-2\,x_{6}x_{7}x_{17}+3\,x_{7}x_{10}x_{17}+3\,x_{7}x_{17}^{2}\notag\\
&&\quad+x_{5}x_{7}x_{18}-x_{4}x_{10}x_{20}-x_{4}x_{16}x_{21}-x_{3}x_{7}-x_{1}x_{20}+x_{19}x_{21},\notag\\
&&-x_{7}x_{15}+x_{9}x_{20}-x_{8}x_{21},\notag\\
&&x_{6}x_{9}x_{10}-x_{9}x_{10}^{2}-x_{4}x_{9}x_{12}+x_{4}x_{10}x_{15}-x_{9}x_{11}x_{16}+2\,x_{6}x_{9}x_{17}-3\,x_{9}x_{10}x_{17}\notag\\
&&\quad-3\,x_{9}x_{17}^{2}-x_{5}x_{9}x_{18}+x_{4}x_{11}x_{21}+x_{3}x_{9}+x_{1}x_{15}+x_{2}x_{21},\notag\\
&&-x_{6}x_{8}x_{10}+x_{8}x_{10}^{2}+x_{4}x_{8}x_{12}+x_{8}x_{11}x_{16}+x_{4}x_{15}x_{16}-2\,x_{6}x_{8}x_{17}+3\,x_{8}x_{10}x_{17}\notag\\
&&\quad+3\,x_{8}x_{17}^{2}+x_{5}x_{8}x_{18}-x_{4}x_{11}x_{20}-x_{3}x_{8}-x_{15}x_{19}-x_{2}x_{20}.
\end{eqnarray}
We select the quadratic equations in (\ref{eq-5points-step2-1}), deleting the repeating or linear dependent  ones:
\begin{equation}\renewcommand{\arraystretch}{1.2}
\begin{array}{lll}\label{eq-5points-step2-2}
x_{8}x_{10}-x_{7}x_{11}+x_{9}x_{16},&\quad
-x_{7}x_{12}-x_{10}x_{20}-x_{16}x_{21},&\quad
x_{1}x_{16}-x_{7}x_{18}+x_{10}x_{19},\\
-x_{8}x_{12}+x_{15}x_{16}-x_{11}x_{20},&\quad
x_{2}x_{16}-x_{8}x_{18}+x_{11}x_{19},&\quad
-x_{9}x_{12}-x_{10}x_{15}-x_{11}x_{21},\\
-x_{2}x_{10}+x_{1}x_{11}-x_{9}x_{18},&\quad
-x_{7}x_{15}+x_{9}x_{20}-x_{8}x_{21}.
\end{array}
\end{equation}
Take the equations in (\ref{eq-5points-step2-1}) modulo these quadratic equations, removing the repeating ones. We are left with 
\begin{eqnarray}\label{eq-5points-step2-3}
&&x_{6}x_{10}x_{16}-x_{10}^{2}x_{16}-2\,x_{4}x_{12}x_{16}-x_{11}x_{16}^{2}+2\,x_{6}x_{16}x_{17}-3\,x_{10}x_{16}x_{17}-3\,x_{16}x_{17}^{2}\notag\\
&&\quad-x_{5}x_{16}x_{18}+x_{3}x_{16}+x_{12}x_{19}+x_{18}x_{20},\notag\\
&&x_{6}x_{10}x_{11}-x_{10}^{2}x_{11}-2\,x_{4}x_{11}x_{12}-x_{11}^{2}x_{16}+2\,x_{6}x_{11}x_{17}-3\,x_{10}x_{11}x_{17}-3\,x_{11}x_{17}^{2}\notag\\
&&\quad-x_{5}x_{11}x_{18}+x_{3}x_{11}-x_{2}x_{12}+x_{15}x_{18},\notag\\
&&-x_{6}x_{10}^{2}+x_{10}^{3}+2\,x_{4}x_{10}x_{12}+x_{10}x_{11}x_{16}-2\,x_{6}x_{10}x_{17}+3\,x_{10}^{2}x_{17}+3\,x_{10}x_{17}^{2}\notag\\
&&\quad+x_{5}x_{10}x_{18}-x_{3}x_{10}+x_{1}x_{12}+x_{18}x_{21},\notag\\
&&x_{2}x_{7}-x_{1}x_{8}+x_{9}x_{19},\notag\\
&&x_{6}x_{7}x_{11}-x_{7}x_{10}x_{11}-x_{6}x_{9}x_{16}+x_{9}x_{10}x_{16}-x_{8}x_{11}x_{16}-2\,x_{4}x_{15}x_{16}+2\,x_{6}x_{8}x_{17}-3\,x_{7}x_{11}x_{17}\notag\\
&&\quad+3\,x_{9}x_{16}x_{17}-3\,x_{8}x_{17}^{2}-x_{5}x_{8}x_{18}+2\,x_{4}x_{11}x_{20}+x_{3}x_{8}+x_{15}x_{19}+x_{2}x_{20},\notag\\
&&-x_{6}x_{7}x_{10}+x_{7}x_{10}^{2}+x_{7}x_{11}x_{16}-2\,x_{6}x_{7}x_{17}+3\,x_{7}x_{10}x_{17}+3\,x_{7}x_{17}^{2}+x_{5}x_{7}x_{18}-2\,x_{4}x_{10}x_{20}\notag\\
&&\quad-2\,x_{4}x_{16}x_{21}-x_{3}x_{7}-x_{1}x_{20}+x_{19}x_{21},\notag\\
&&x_{6}x_{9}x_{10}-x_{9}x_{10}^{2}+2\,x_{4}x_{10}x_{15}-x_{9}x_{11}x_{16}+2\,x_{6}x_{9}x_{17}-3\,x_{9}x_{10}x_{17}-3\,x_{9}x_{17}^{2}-x_{5}x_{9}x_{18}\notag\\
&&\quad+2\,x_{4}x_{11}x_{21}+x_{3}x_{9}+x_{1}x_{15}+x_{2}x_{21}.
\end{eqnarray}

\item\label{step3}
\begin{equation*}
	x_{3}\ \longmapsto \ x_{3} + x_{10}^2 + x_{11} x_{16} + 3 x_{10} x_{17} + 3 x_{17}^2 - x_{6} x_{10} - 2 x_{6} x_{17} + x_{5} x_{18} + 2 x_{4} x_{12}.
\end{equation*}
Then Equations (\ref{eq-5points-step2-2}) and (\ref{eq-5points-step2-3}) are transformed into 
\begin{align}\label{eq-5points-step3}\renewcommand{\arraystretch}{1.2}
&\begin{array}{lll}
 x_{8}x_{10}-x_{7}x_{11}+x_{9}x_{16},& -x_{7}x_{12}-x_{10}x_{20}-x_{16}x_{21},&
 x_{1}x_{16}-x_{7}x_{18}+x_{10}x_{19},\\
 -x_{8}x_{12}+x_{15}x_{16}-x_{11}x_{20},& x_{2}x_{16}-x_{8}x_{18}+x_{11}x_{19},&
 -x_{9}x_{12}-x_{10}x_{15}-x_{11}x_{21},\\
 -x_{2}x_{10}+x_{1}x_{11}-x_{9}x_{18},& -x_{7}x_{15}+x_{9}x_{20}-x_{8}x_{21},&
 x_{3}x_{16}+x_{12}x_{19}+x_{18}x_{20},\\ 
 x_{3}x_{11}-x_{2}x_{12}+x_{15}x_{18},& -x_{3}x_{10}+x_{1}x_{12}+x_{18}x_{21},&
 x_{2}x_{7}-x_{1}x_{8}+x_{9}x_{19},
 \end{array}\notag \\
&\;-x_{6}x_{8}x_{10}+x_{8}x_{10}^{2}+x_{6}x_{7}x_{11}-x_{7}x_{10}x_{11}+2\,x_{4}x_{8}x_{12}-x_{6}x_{9}x_{16}+x_{9}x_{10}x_{16}-2\,x_{4}x_{15}x_{16}\\
&\quad +3\,x_{8}x_{10}x_{17}-3\,x_{7}x_{11}x_{17}+3\,x_{9}x_{16}x_{17}+2\,x_{4}x_{11}x_{20}+x_{3}x_{8}+x_{15}x_{19}+x_{2}x_{20},\notag\\ 
&\;-2\,x_{4}x_{7}x_{12}-2\,x_{4}x_{10}x_{20}-2\,x_{4}x_{16}x_{21}-x_{3}x_{7}-x_{1}x_{20}+x_{19}x_{21},\notag\\ 
&\; 2\,x_{4}x_{9}x_{12}+2\,x_{4}x_{10}x_{15}+2\,x_{4}x_{11}x_{21}+x_{3}x_{9}+x_{1}x_{15}+x_{2}x_{21}.\notag
\end{align}
Select the quadratic equations
\begin{equation*}\renewcommand{\arraystretch}{1.2}
\begin{array}{lll}
x_{8}x_{10}-x_{7}x_{11}+x_{9}x_{16},&-x_{7}x_{12}-x_{10}x_{20}-x_{16}x_{21},& x_{1}x_{16}-x_{7}x_{18}+x_{10}x_{19},\\
-x_{8}x_{12}+x_{15}x_{16}-x_{11}x_{20},& x_{2}x_{16}-x_{8}x_{18}+x_{11}x_{19},& -x_{9}x_{12}-x_{10}x_{15}-x_{11}x_{21},\\
-x_{2}x_{10}+x_{1}x_{11}-x_{9}x_{18},& -x_{7}x_{15}+x_{9}x_{20}-x_{8}x_{21},& x_{3}x_{16}+x_{12}x_{19}+x_{18}x_{20},\\
x_{3}x_{11}-x_{2}x_{12}+x_{15}x_{18},& -x_{3}x_{10}+x_{1}x_{12}+x_{18}x_{21},& x_{2}x_{7}-x_{1}x_{8}+x_{9}x_{19}.
\end{array}
\end{equation*}
Taking the last three equations of (\ref{eq-5points-step3}) modulo these quadratic equations, we get the equations
\begin{equation*}
	x_{3}x_{8}+x_{15}x_{19}+x_{2}x_{20},\quad -x_{3}x_{7}-x_{1}x_{20}+x_{19}x_{21},\quad x_{3}x_{9}+x_{1}x_{15}+x_{2}x_{21}.
\end{equation*}
\end{enumerate}
One easily checks that \textit{Step}~\ref{step1} is a nonsingular linear transformation, and \textit{Steps}~\ref{step2} and~\ref{step3} are unipotent isomorphisms (in the sense of Definition~\ref{def-unipotent}).
We record the total change of variables (\textit{i.e.}, the composition \textit{Step}~\ref{step3} $\circ$ \textit{Step}~\ref{step2} $\circ$ \textit{Step}~\ref{step1}):
\begin{eqnarray}\label{eq-131-totalchangeofvariables}
\allowdisplaybreaks
{x}_{1} &\longmapsto&  -2\,{x}_{4}{x}_{17}+{x}_{1}+{x}_{13},\notag\\
{x}_{2} &\longmapsto&  -{x}_{5}{x}_{10}+{x}_{4}{x}_{11}-2\,{x}_{5}{x}_{17}+{x}_{2}{x}_{14},\notag\\
{x}_{3} &\longmapsto& -{x}_{6}{x}_{10}+{x}_{10}^{2}+2\,{x}_{4}{x}_{12}+{x}_{11}{x}_{16}-2\,{x}_{6}{x}_{17}+3\,{x}_{10}{x}_{17}+3\,{x}_{17}^{2}+{x}_{5}{x}_{18}+{x}_{3},\notag\\ 
{x}_{7} &\longmapsto & {x}_{4}^{2}+{x}_{7},\notag\\ 
{x}_{8} &\longmapsto& {x}_{4}{x}_{5}+{x}_{8},\notag\\ 
{x}_{9} &\longmapsto&  {x}_{5}^{2}+{x}_{9},\\ 
{x}_{10} &\longmapsto&  {x}_{10}+{x}_{17},\notag\\ 
{x}_{13} &\longmapsto& {x}_{4}{x}_{10}+{x}_{1}+2\,{x}_{13}+{x}_{20},\notag\\ 
{x}_{19} &\longmapsto&  -{x}_{4}{x}_{16}+{x}_{19},\notag\\ 
{x}_{20} &\longmapsto&  {x}_{4}{x}_{10}+{x}_{1}+{x}_{13}+{x}_{20},\notag\\ 
{x}_{21} &\longmapsto& {x}_{4}{x}_{11}+{x}_{2}+2\,{x}_{14}+{x}_{21}.\notag
\end{eqnarray}
Therefore, after this change of variables, the subideal of $\mathcal{H}'_{\lambda}$ generated by its minDegree 2 equations is transformed into the ideal generated by
\begin{equation}\label{eq-5points-quadratic-equations-final}\renewcommand{\arraystretch}{1.2}
\begin{array}{lll}
x_{8}x_{10}-x_{7}x_{11}+x_{9}x_{16},& -x_{7}x_{12}-x_{10}x_{20}-x_{16}x_{21},& x_{1}x_{16}-x_{7}x_{18}+x_{10}x_{19},\\
-x_{8}x_{12}+x_{15}x_{16}-x_{11}x_{20},& x_{2}x_{16}-x_{8}x_{18}+x_{11}x_{19},& -x_{9}x_{12}-x_{10}x_{15}-x_{11}x_{21},\\
-x_{2}x_{10}+x_{1}x_{11}-x_{9}x_{18},& x_{3}x_{16}+x_{12}x_{19}+x_{18}x_{20},& x_{3}x_{11}-x_{2}x_{12}+x_{15}x_{18},\\
-x_{3}x_{10}+x_{1}x_{12}+x_{18}x_{21},& -x_{7}x_{15}+x_{9}x_{20}-x_{8}x_{21},& x_{2}x_{7}-x_{1}x_{8}+x_{9}x_{19},\\
x_{3}x_{8}+x_{15}x_{19}+x_{2}x_{20},& -x_{3}x_{7}-x_{1}x_{20}+x_{19}x_{21},& x_{3}x_{9}+x_{1}x_{15}+x_{2}x_{21}.
\end{array}
\end{equation}

The final quadratic equations (\ref{eq-5points-quadratic-equations-final}) are simple enough so that one can compute the Gr\"obner basis of the ideal generated by them. Then one finds that the  ideal $\mathcal{H}'_{\lambda}$, after the above change of variables (\ref{eq-131-totalchangeofvariables}),  is equal to the ideal generated by (\ref{eq-5points-quadratic-equations-final}). A more elegant way to confirm this, without the use of a Gr\"obner basis, is the following. The map
\begin{alignat*}{12}
x_{1}& \longmapsto  p_{0,2},&\quad x_{2} &\longmapsto  p_{0,3},&\quad x_{3} &\longmapsto  -p_{0,1},&\quad x_{7} &\longmapsto  p_{2,4},&\quad x_{8} &\longmapsto  p_{3,4},&\quad x_{9} &\longmapsto p_{2,3},\\
x_{10} &\longmapsto  p_{2,5},&\quad x_{11} &\longmapsto  p_{3,5},&\quad x_{12} &\longmapsto  -p_{1,5},&\quad x_{15} &\longmapsto  p_{1,3},&\quad x_{16} &\longmapsto  p_{4,5},&\quad x_{18} &\longmapsto  -p_{0,5},\\
 x_{19} &\longmapsto  -p_{0,4},&\quad x_{20} &\longmapsto  p_{1,4},&\quad x_{21} &\longmapsto  -p_{1,2}
\end{alignat*}
transforms  Equations (\ref{eq-5points-quadratic-equations-final}) into the Pl\"ucker equations (\ref{eq-pluckerideal}). So there is a closed immersion
\begin{equation}\label{eq-embedding-to-fibration-over-cone-1}
\Spec(A_{\lambda})\longhookrightarrow \widehat{G}(2,6)\times \mathbb{A}^6.
\end{equation}
But by Theorem~\ref{thm-monomialideal-smoothable}, the point corresponding to the ideal $I_{\lambda}$ lies in the main component, so $\dim A_{\lambda}\geq 15$. Note that 
$\widehat{G}(2,6)\times \mathbb{A}^6$ is an integral scheme.  Hence the equality of dimensions implies that  (\ref{eq-embedding-to-fibration-over-cone-1}) is  an isomorphism.\\

In Section~\ref{sec:3D-132} and Appendices~\ref{sec:3D-141}--\ref{sec:3D-232}, the Haiman neighborhoods turn out to be trivial affine fibrations over $\widehat{G}(2,6)$ as well. We present only the final change of variables, omitting the intermediate steps that lead to it. The relevant Macaulay2 codes and  an example for Appendix~\ref{sec:3D-141} illustrating the use of the codes, are given in the ancillary files.\footnote{See also \url{https://github.com/huxw06/Hilbert-scheme-of-points}}

\subsubsection{\texorpdfstring{$\lambda_{132}=\left((1)\subset (3,2)\right)$}{((1) in (3,2))}}
\label{sec:3D-132}
The partition $\lambda_{132}$ has the 3D diagram
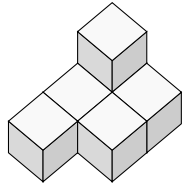
\begin{center}
\begin{tikzpicture}[x=(220:0.6cm), y=(-40:0.6cm), z=(90:0.42cm)]

\foreach \m [count=\y] in {{2,1,1},{1,1}}{
  \foreach \n [count=\x] in \m {
  \ifnum \n>0
      \foreach \z in {1,...,\n}{
        \draw [fill=gray!30] (\x+1,\y,\z) -- (\x+1,\y+1,\z) -- (\x+1, \y+1, \z-1) -- (\x+1, \y, \z-1) -- cycle;
        \draw [fill=gray!40] (\x,\y+1,\z) -- (\x+1,\y+1,\z) -- (\x+1, \y+1, \z-1) -- (\x, \y+1, \z-1) -- cycle;
        \draw [fill=gray!5] (\x,\y,\z)   -- (\x+1,\y,\z)   -- (\x+1, \y+1, \z)   -- (\x, \y+1, \z) -- cycle;  
      }
 \fi
 }
}    
\end{tikzpicture}
\end{center}
and corresponds to the ideal
\[
I_{\lambda_{132}}=\left(X_1^3, X_1^2 X_2, X_1 X_3, X_2^2,X_2 X_3, X_3^2\right).
\]
Algorithm~\ref{alg-step0-Haiman} gives
\begin{alignat*}{6}
c_{1, 0, 0}^{1, 0, 1} &\longmapsto   x_{1},&\quad c_{1, 0, 0}^{3, 0, 0} &\longmapsto  x_{2},&\quad c_{1, 0, 0}^{0, 2, 0} &\longmapsto   x_{3},&\quad c_{2, 0, 0}^{1, 0, 1} &\longmapsto  x_{4},&\quad c_{2, 0, 0}^{3, 0, 0} &\longmapsto   x_{5},&\quad c_{2, 0, 0}^{2, 1, 0} &\longmapsto  x_{6},\\
 c_{2, 0, 0}^{0, 2, 0} &\longmapsto   x_{7},&\quad c_{2, 0, 0}^{0, 1, 1} &\longmapsto  x_{8},&\quad c_{2, 0, 0}^{0, 0, 2} &\longmapsto   x_{9},&\quad c_{0, 1, 0}^{1, 0, 1} &\longmapsto  x_{10},&\quad c_{0, 1, 0}^{3, 0, 0} &\longmapsto   x_{11},&\quad c_{0, 1, 0}^{0, 2, 0} &\longmapsto  x_{12},\\
  c_{1, 1, 0}^{1, 0, 1} &\longmapsto   x_{13},&\quad c_{1, 1, 0}^{3, 0, 0} &\longmapsto  x_{14},&\quad c_{1, 1, 0}^{2, 1, 0} &\longmapsto   x_{15},&\quad c_{1, 1, 0}^{0, 2, 0} &\longmapsto  x_{16},&\quad c_{1, 1, 0}^{0, 1, 1} &\longmapsto   x_{17},&\quad c_{1, 1, 0}^{0, 0, 2} &\longmapsto  x_{18},\\
   c_{0, 0, 1}^{1, 0, 1} &\longmapsto   x_{19},&\quad c_{0, 0, 1}^{3, 0, 0} &\longmapsto  x_{20},&\quad c_{0, 0, 1}^{2, 1, 0} &\longmapsto   x_{21},&\quad c_{0, 0, 1}^{0, 2, 0} &\longmapsto  x_{22},&\quad c_{0, 0, 1}^{0, 1, 1} &\longmapsto   x_{23},&\quad c_{0, 0, 1}^{0, 0, 2} &\longmapsto  x_{24},
\end{alignat*}
and
\[
A_{\lambda_{132}}\cong \Bbbk[x_1,\dots,x_{24}]/\mathcal{H}'_{\lambda_{132}}.
\]
The unipotent isomorphism
\begin{eqnarray*}
x_{2} &\longmapsto& 2 x_{14}x_{16}x_{19}+x_{13}x_{14}x_{22}-x_{12}x_{14}-x_{5}x_{15}-x_{11}x_{16}+x_{5}x_{19}-2 x_{15}x_{19}+3 x_{19}^{2}\\
&&-x_{17}x_{20}-x_{14}x_{23}+x_{2},\\
x_{3} &\longmapsto& x_{13}x_{16}x_{22}-x_{5}x_{7}-x_{6}x_{16}-2 x_{7}x_{19}-x_{4}x_{22}-x_{17}x_{22}-x_{16}x_{23}+x_{3},\\
x_{4} &\longmapsto& -x_{13}x_{16}+x_{4}+x_{17},\\
x_{5} &\longmapsto& -x_{14}x_{16}+x_{5}+x_{15},\\
x_{6} &\longmapsto& x_{7}x_{14}+x_{6}+x_{23},\\
x_{8} &\longmapsto& x_{7}x_{13}+x_{8},\\
x_{9} &\longmapsto& x_{13}^{2}x_{16}^{2}+x_{7}x_{13}^{2}-2 x_{4}x_{13}x_{16}-2 x_{13}x_{16}x_{17}+x_{4}^{2}+2 x_{8}x_{13}+2 x_{4}x_{17}+x_{17}^{2}+x_{9},\\
x_{10} &\longmapsto& -x_{13}x_{19}+x_{10},\\
x_{11} &\longmapsto& -x_{14}x_{19}+x_{11},\\
x_{12} &\longmapsto& -x_{16}x_{19}-2 x_{13}x_{22}+x_{6}+x_{12}+2 x_{23},\\
x_{15} &\longmapsto& x_{15}-x_{19},\\
x_{18} &\longmapsto& -x_{13}^{2}x_{16}+x_{4}x_{13}+2 x_{13}x_{17}+x_{18},\\
x_{24} &\longmapsto& -4 x_{13}x_{16}x_{19}-x_{13}^{2}x_{22}+x_{12}x_{13}-x_{8}x_{14}+x_{4}x_{15}+x_{10}x_{16}+x_{4}x_{19}+4 x_{17}x_{19}\\
&&+2 x_{13}x_{23}+2 x_{1}+x_{24}.
\end{eqnarray*}
transforms the ideal $\mathcal{H}'_{\lambda_{132}}$ into
\begin{equation}\label{eq-132-quadratic-equations-final}\renewcommand{\arraystretch}{1.2}
\begin{array}{lllll}
  J_{132}&=&(-x_{8}x_{20}+x_{4}x_{21}+x_{10}x_{22},& -x_{9}x_{11}-x_{2}x_{18}-x_{10}x_{24},& -x_{6}x_{20}+x_{5}x_{21}+x_{11}x_{22},\notag\\
&& \hphantom{(} -x_{2}x_{5}-x_{11}x_{12}+x_{20}x_{24},& x_{2}x_{8}-x_{3}x_{10}+x_{9}x_{21},& -x_{2}x_{6}+x_{3}x_{11}+x_{21}x_{24},\notag\\
&& \hphantom{(} -x_{3}x_{20}-x_{12}x_{21}+x_{2}x_{22},& x_{2}x_{4}+x_{10}x_{12}+x_{9}x_{20},& x_{5}x_{10}-x_{4}x_{11}+x_{18}x_{20},\\
&& \hphantom{(} x_{3}x_{4}+x_{8}x_{12}+x_{9}x_{22},& -x_{4}x_{6}+x_{5}x_{8}+x_{18}x_{22},& -x_{6}x_{9}-x_{3}x_{18}-x_{8}x_{24},\notag\\
&&  \hphantom{(}    x_{5}x_{9}-x_{12}x_{18}+x_{4}x_{24},& -x_{3}x_{5}-x_{6}x_{12}+x_{22}x_{24},& x_{6}x_{10}-x_{8}x_{11}+x_{18}x_{21}).\notag
\end{array}
\end{equation}
Then the map
\begin{alignat*}{6}
x_{8} &\longmapsto  -p_{3,4},&\quad x_{20} &\longmapsto  p_{2,5},&\quad x_{4} &\longmapsto  -p_{2,4},&\quad x_{21} &\longmapsto  p_{3,5},&\quad x_{10} &\longmapsto  p_{4,5},&\quad x_{22} &\longmapsto p_{2,3},\\
 x_{2} &\longmapsto  -p_{1,5},&\quad x_{9} &\longmapsto  -p_{1,4},&\quad x_{3} &\longmapsto  -p_{1,3},&\quad x_{11} &\longmapsto  p_{0,5},&\quad x_{18} &\longmapsto -p_{0,4},&\quad x_{24} &\longmapsto  -p_{0,1},\\
 x_{6} &\longmapsto  p_{0,3},&\quad x_{5} &\longmapsto  p_{0,2},&\quad x_{12} &\longmapsto  p_{1,2}&&&&&&
\end{alignat*}
transforms  Equations (\ref{eq-132-quadratic-equations-final}) into the Pl\"ucker equations
(\ref{eq-pluckerideal}). So we obtain an equivariant isomorphism
\[
\Spec(A_{\lambda_{132}})\cong \widehat{G}(2,6)\times \mathbb{A}^9.
\]

\subsubsection{\texorpdfstring{$\lambda_{1321}=\left((1)\subset(3,2,1)\right)$}{((1) in (3,2,1))}
,  extra  dim =8}
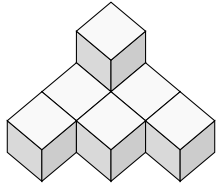
\label{sec:3D-1321}
\begin{center}
\begin{tikzpicture}[x=(220:0.6cm), y=(-40:0.6cm), z=(90:0.42cm)]

\foreach \m [count=\y] in {{2,1,1},{1,1},{1}}{
  \foreach \n [count=\x] in \m {
  \ifnum \n>0
      \foreach \z in {1,...,\n}{
        \draw [fill=gray!30] (\x+1,\y,\z) -- (\x+1,\y+1,\z) -- (\x+1, \y+1, \z-1) -- (\x+1, \y, \z-1) -- cycle;
        \draw [fill=gray!40] (\x,\y+1,\z) -- (\x+1,\y+1,\z) -- (\x+1, \y+1, \z-1) -- (\x, \y+1, \z-1) -- cycle;
        \draw [fill=gray!5] (\x,\y,\z)   -- (\x+1,\y,\z)   -- (\x+1, \y+1, \z)   -- (\x, \y+1, \z) -- cycle;  
      }
 \fi
 }
}    

\end{tikzpicture}
\end{center}
\[
I_{1321}=\left(X_1^3,X_1^2 X_2,X_1 X_2^2,X_1 X_3, X_2^3, X_2 X_3, X_3^2\right).
\]
The extra dimension of $\Hilb^7(\mathbb{A}^3)$ at $Z_{I_{1321}}$ is 8.
Algorithm~\ref{alg-step0-Haiman} gives
\begin{alignat}{6}\label{eq-table-1321-c-to-x}
 c_{1, 1, 0}^{0, 0, 2}&=x_{1},&\quad  c_{1, 1, 0}^{1, 0, 1}&=x_{2},&\quad  c_{2, 0, 0}^{0, 0, 2}&=x_{3},&\quad  c_{2, 0, 0}^{0, 1, 1}&=x_{4},&\quad  c_{2, 0, 0}^{3, 0, 0}&=x_{5},&\quad  c_{1, 1, 0}^{3, 0, 0}&=x_{6},\notag\\
 c_{2, 0, 0}^{1, 2, 0}&=x_{7},&\quad  c_{2, 0, 0}^{0, 3, 0}&=x_{8},&\quad  c_{2, 0, 0}^{1, 0, 1}&=x_{9},&\quad  c_{1, 1, 0}^{2, 1, 0}&=x_{10},&\quad  c_{0, 0, 1}^{0, 0, 2}&=x_{11},&\quad  c_{0, 2, 0}^{0, 0, 2}&=x_{12},\notag\\
 c_{0, 0, 1}^{0, 3, 0}&=x_{13},&\quad  c_{0, 2, 0}^{1, 0, 1}&=x_{14},&\quad  c_{1, 1, 0}^{0, 3, 0}&=x_{15},&\quad  c_{0, 0, 1}^{3, 0, 0}&=x_{16},&\quad  c_{0, 2, 0}^{3, 0, 0}&=x_{17},&\quad  c_{1, 1, 0}^{1, 2, 0}&=x_{18},\\
 c_{2, 0, 0}^{2, 1, 0}&=x_{19},&\quad  c_{0, 0, 1}^{1, 2, 0}&=x_{20},&\quad  c_{0, 0, 1}^{1, 0, 1}&=x_{21},&\quad  c_{0, 0, 1}^{0, 1, 1}&=x_{22},&\quad  c_{0, 2, 0}^{0, 1, 1}&=x_{23},&\quad  c_{0, 2, 0}^{2, 1, 0}&=x_{24},\notag\\
 c_{1, 1, 0}^{0, 1, 1}&=x_{25},&\quad  c_{0, 2, 0}^{1, 2, 0}&=x_{26},&\quad  c_{0, 2, 0}^{0, 3, 0}&=x_{27},&\quad  c_{1, 0, 0}^{1, 0, 1}&=x_{28},&\quad  c_{0, 0, 1}^{2, 1, 0}&=x_{29}.\notag
\end{alignat}
and 
\[
A_{\lambda_{1321}}\cong \Bbbk[x_1,\dots,x_{29}]/\mathcal{H}'_{\lambda_{1321}}.
\]
The unipotent homomorphism 
\begin{eqnarray}\label{eq-change-variables-1321}
  x_{1} & \longmapsto&   y_{1}+ 2 y_{2} y_{25}+ 2 y_{23} y_{25}+ y_{14} y_{4}+ y_{2} y_{9},\notag\\ 
x_{2} & \longmapsto&   y_{2}+y_{23},\notag\\ 
x_{3} & \longmapsto&   y_{3}+y_{25}^2 + y_{2} y_{4}+ 2 y_{23} y_{4}+ 2 y_{25} y_{9}+ y_{9}^2,\notag\\ 
x_{5} & \longmapsto&   y_{5}+y_{10}+ 2 y_{21}+ 3 y_{26},\notag\\ 
x_{6}  & \longmapsto&  y_{6}+y_{24},\notag\\
x_{9} & \longmapsto&   y_{9}+y_{25},\notag\\ 
x_{10} & \longmapsto&  y_{10}+ y_{21}+ 2 y_{26},\notag\\ 
x_{11} & \longmapsto&  y_{11}+ y_{14} y_{15}+ y_{18} y_{2}+ y_{19} y_{2}+ 2 y_{2} y_{22}+ 2 y_{22} y_{23}+ 4 y_{21} y_{25}\\
 & &+ 4 y_{25} y_{26}+ 2 y_{28} - y_{24} y_{4} - y_{4} y_{6}+ y_{14} y_{7}+ 
  y_{10} y_{9}+ 3 y_{21} y_{9}+ 4 y_{26} y_{9},\notag\\
x_{12} & \longmapsto&   y_{12}+ y_{23}^2+ 2 y_{14} y_{25}+ y_{14} y_{9},\notag\\ 
x_{15}  & \longmapsto&  y_{15}+ y_{7},\notag\\
x_{18} & \longmapsto&   y_{18}+ y_{19}+ 2 y_{22},\notag\\ 
x_{19} & \longmapsto&   y_{19}+y_{22},\notag\\ 
x_{21} & \longmapsto&   y_{21}+y_{26},\notag\\ 
x_{27} & \longmapsto&   y_{27}+ y_{18}+ y_{19}+ 3 y_{22},\notag\\
x_{i}  & \longmapsto&  y_{i}\ \mbox{for}\ i=4,7,8,13,14,16,17,20,22,23,24,25,26,28,29,\notag
\end{eqnarray}
transforms $\mathcal{H}'_{\lambda_{132}}$ into the Jacobian ideal of $F_{1321}$, 
where 
\begin{eqnarray*}
F_{1321}&=&-{y}_{5}{y}_{8}{y}_{10}{y}_{14}+{y}_{6}{y}_{7}{y}_{14}{y}_{15}+{y}_{8}{y}_{9}{y}_{10}{y}_{17}-{y}_{2}{y}_{7}{
     y}_{15}{y}_{17}+{y}_{4}{y}_{10}{y}_{15}{y}_{17}-{y}_{6}{y}_{8}{y}_{14}{y}_{18}\\
  &&   +{y}_{2}{y}_{8}{y}_{17}{y}_{18
     }-{y}_{10}{y}_{14}{y}_{15}{y}_{19}+{y}_{2}{y}_{5}{y}_{7}{y}_{21}-{y}_{6}{y}_{7}{y}_{9}{y}_{21}-{y}_{4}{y}_{5}{y}_{10}{y}_{21}-{y}_{4}{y}_{6}{y}_{18}{y}_{21}\\
  &&   +{y}_{9}{y}_{10}{y}_{19}{y}_{21}+{y}_{2}{y}_{18}{y}_{19}{y}_{21
     }+{y}_{2}{y}_{5}{y}_{8}{y}_{24}-{y}_{6}{y}_{8}{y}_{9}{y}_{24}-{y}_{4}{y}_{6}{y}_{15}{y}_{24}+{y}_{2}{y}_{15}{y}_{19}{y}_{24}\\
  &&   -{y}_{5}{y}_{7}{y}_{14}{y}_{27}+{y}_{7}{y}_{9}{y}_{17}{y}_{27}+{y}_{4}{y}_{17}{y}_{18}{y}_{27}-{y}_{14}{y}_{18}{y}_{19}{y}_{27}+{y}_{4}{y}_{5}{y}_{24}{y}_{27}-{y}_{9}{y}_{19}{y}_{24}{y}_{27}\\
  &&  -{y}_{11}{y}_{13}{y}_{14}+{y}_{1}{y}_{7}{y}_{16}+{y}_{4}{y}_{11}{y}_{16}+{y}_{8}{y}_{12}{y}_{16}-{y}_{3}{y}_{13}{y}_{17}+{y}_{3}{y}_{16}{y}_{19}-{y}_{3}{y}_{6}{y}_{20}\\
  &&   -{y}_{1}{y}_{10}{y}_{20}-{y}_{2}{y}_{11}{y}_{20}+{y}_{12}{y}_{13}{y
     }_{21}-{y}_{1}{y}_{13}{y}_{24}+{y}_{12}{y}_{20}{y}_{27}-{y}_{3}{y}_{5}{y}_{29}-{y}_{9}{y}_{11}{y}_{29}\\
  &&   +{y}_{12}{y}_{15}{y}_{29}+{y}_{1}{y}_{18}{y}_{29},
\end{eqnarray*}
and thus
\[
A_{\lambda_{1321}}\cong \Bbbk[y_1,\dots,y_{29}]/\Jac(F_{1321}).
\]

\subsection{Local equations at non-Borel ideals}\label{sec:3D-1311}
Let $\lambda_{1311}$ be the 3D partition $\left((1)\subset (3,1,1)\right)$. It corresponds to the 3D diagram and ideal
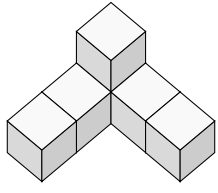
\begin{center}
\begin{tikzpicture}[x=(220:0.6cm), y=(-40:0.6cm), z=(90:0.42cm)]

\foreach \m [count=\y] in {{2,1,1},{1},{1}}{
  \foreach \n [count=\x] in \m {
  \ifnum \n>0
      \foreach \z in {1,...,\n}{
        \draw [fill=gray!30] (\x+1,\y,\z) -- (\x+1,\y+1,\z) -- (\x+1, \y+1, \z-1) -- (\x+1, \y, \z-1) -- cycle;
        \draw [fill=gray!40] (\x,\y+1,\z) -- (\x+1,\y+1,\z) -- (\x+1, \y+1, \z-1) -- (\x, \y+1, \z-1) -- cycle;
        \draw [fill=gray!5] (\x,\y,\z)   -- (\x+1,\y,\z)   -- (\x+1, \y+1, \z)   -- (\x, \y+1, \z) -- cycle;  
      }
 \fi
 }
}    
\end{tikzpicture}
\end{center}
\[
I_{1311}=\left(X_1^3, X_1X_2, X_1 X_3, X_2^3,X_2 X_3, X_3^2\right).
\]
This is the unique non-Borel monomial ideal of $\Bbbk[X_1,X_2,X_3]$ of colength $6$.

First we show that Lemma~\ref{lem-transitive-open-nbd} and the result in Section~\ref{sec:3D-132} give us the local structure of $\Spec(A_{\lambda_{1311}})$ at~$0$.  Let $\epsilon\in \Bbbk$ be a nonzero element. A linear transform $X_1\mapsto X_1+\epsilon X_2$ transforms $I$ into the ideal
\begin{eqnarray*}
J_{\epsilon}&:=&(X_1^3+3 \epsilon X_1^2 X_2+3 \epsilon^2 X_1 X_2^2+X_2^3, X_1X_2+\epsilon X_2^2, X_1 X_3+\epsilon X_2 X_3, X_2^3,X_2 X_3, X_3^2)\\
&=& (X_1^3, X_1X_2+\epsilon X_2^2, X_1 X_3, X_2^3,X_2 X_3, X_3^2).
\end{eqnarray*}
Since $\epsilon\neq 0$, the resulting ideal is equal to
\begin{eqnarray}\label{eq-transformed-I1311}
&&(X_1^3, X_2^2+\frac{1}{\epsilon} X_1 X_2, X_1 X_3, X_2^3,X_2 X_3, X_3^2)\notag\\
&&=(X_1^3, X_1 X_2^2,  X_1 X_3, X_2^2+\frac{1}{\epsilon} X_1 X_2, X_2 X_3, X_3^2)\notag\\
&&= (X_1^3, X_1^2 X_2,  X_1 X_3, X_2^2+\frac{1}{\epsilon} X_1 X_2, X_2 X_3, X_3^2).
\end{eqnarray}
So it lies in the Haiman neighborhood of $I_{132}$, with Haiman coordinate $c_{110}^{020}=-\frac{1}{\epsilon}$. Since the extra dimension of $\Hilb^6(\mathbb{A}^3)$ at $Z_{I_{1311}}$ is $6$, the extra dimension at $Z_{J_{\epsilon}}$ is also 6. But we have seen that the whole Haiman neighborhood of $I_{132}$ is a trivial affine fibration over the cone $\widehat{G}(2,6)$. 
So $J_{\epsilon}$ lies in the fiber ($\cong \mathbb{A}_{\Bbbk}^{6}$) over the vertex of the cone. Thus we obtain that there is an open neighborhood of  $Z_{I_{1311}}$ that is isomorphic to an open subset of $\widehat{G}(2,6)\times \mathbb{A}^9$. This is \textit{not} enough for the computation of the equivariant Hilbert function of $A_{\lambda_{1311}}$: The isomorphism obtained above is not equivariant.

We could not find a conceptual proof of such an equivariant isomorphism but only apply the algorithm and the trick of change of variables in the previous sections.  Algorithm~\ref{alg-step0-Haiman} does not perform as well for non-Borel monomial ideals as for the Borel ideals. For example, the Haiman equations imply, for $i\in \lambda_{1311}$,
\begin{align*}
c_{i}^{210}= \displaystyle\frac{1}{(1-c_{020}^{110}c_{200}^{110})}&\left(
c_{000}^{110}c_{i}^{100}+c_{100}^{110}c_{i}^{200}+c_{200}^{110}c_{i}^{300}
+c_{010}^{110}c_{i}^{110}+c_{001}^{110}c_{i}^{101}\right.\\
&\hphantom{\big)}\left.+c_{020}^{110}\left(c_{000}^{110}c_{i}^{010}+c_{010}^{110}c_{i}^{020}+c_{020}^{110}c_{i}^{030}+c_{100}^{110}c_{i}^{110}+c_{001}^{110}c_{i}^{011}\right)\right),
\end{align*}
and symmetrically
\begin{align*}
 c_{i}^{120}=\displaystyle\frac{1}{(1-c_{020}^{110}c_{200}^{110})}&\left(c_{000}^{110}c_{i}^{010}+c_{010}^{110}c_{i}^{020}+c_{020}^{110}c_{i}^{030}
+c_{100}^{110}c_{i}^{110}+c_{001}^{110}c_{i}^{011}\right.\\
&\left.\hphantom{\big)}+c_{200}^{110}(c_{000}^{110}c_{i}^{100}+c_{100}^{110}c_{i}^{200}+c_{200}^{110}c_{i}^{300}
+c_{010}^{110}c_{i}^{110}+c_{001}^{110}c_{i}^{101})\right).
\end{align*}

This is typical for non-Borel monomial ideals: To eliminate the variables, we need to allow the fractions. Algorithm~\ref{alg-step0-Haiman} gives a set of equations in $x_1,\dots,x_{29}$, where 
\begin{alignat*}{6}
 c_{2, 0, 0}^{2, 1, 0} &\longmapsto x_{1},&\quad  c_{0, 2, 0}^{2, 1, 0} &\longmapsto x_{2},&\quad  c_{0, 1, 0}^{2, 1, 0} &\longmapsto x_{3},&\quad  
c_{0, 0, 1}^{2, 1, 0} &\longmapsto x_{4},&\quad  c_{1, 0, 0}^{1, 1, 0} &\longmapsto x_{5},&\quad  c_{1, 0, 0}^{1, 0, 1} &\longmapsto x_{6},\\
 c_{1, 0, 0}^{3, 0, 0} &\longmapsto x_{7},&\quad  c_{1, 0, 0}^{0, 0, 2} &\longmapsto x_{8},&\quad  c_{2, 0, 0}^{1, 1, 0} &\longmapsto x_{9},&\quad  
 c_{2, 0, 0}^{1, 0, 1} &\longmapsto x_{10},&\quad  c_{2, 0, 0}^{3, 0, 0} &\longmapsto x_{11},&\quad  c_{2, 0, 0}^{0, 1, 1} &\longmapsto x_{12},\\ 
 c_{2, 0, 0}^{0, 3, 0} &\longmapsto x_{13},&\quad  c_{2, 0, 0}^{0, 0, 2} &\longmapsto x_{14},&\quad  c_{0, 1, 0}^{1, 1, 0} &\longmapsto x_{15},&\quad  
 c_{0, 1, 0}^{0, 1, 1} &\longmapsto x_{16},&\quad  c_{0, 1, 0}^{0, 3, 0} &\longmapsto x_{17},&\quad  c_{0, 2, 0}^{1, 1, 0} &\longmapsto x_{18},\\ 
 c_{0, 2, 0}^{1, 0, 1} &\longmapsto x_{19},&\quad  c_{0, 2, 0}^{3, 0, 0} &\longmapsto x_{20},&\quad  c_{0, 2, 0}^{0, 1, 1} &\longmapsto x_{21},&\quad  
 c_{0, 2, 0}^{0, 3, 0} &\longmapsto x_{22},&\quad  c_{0, 2, 0}^{0, 0, 2} &\longmapsto x_{23},&\quad  c_{0, 0, 1}^{1, 1, 0} &\longmapsto x_{24},\\
  c_{0, 0, 1}^{1, 0, 1} &\longmapsto x_{25},&\quad  c_{0, 0, 1}^{3, 0, 0} &\longmapsto x_{26},&\quad  c_{0, 0, 1}^{0, 1, 1} &\longmapsto x_{27},&\quad  
  c_{0, 0, 1}^{0, 3, 0} &\longmapsto x_{28},&\quad  c_{0, 0, 1}^{0, 0, 2} &\longmapsto x_{29}.
\end{alignat*}
The number of variables is larger than the embedded dimension $24$. There are equations of minDegree $1$. Using them, we eliminate $x_1,x_2,x_3,x_4,x_8$, allowing the inversion 
\[
(1-x_9 x_{18})(1-2x_9 x_{18}).
\]
For $\lambda_{1311}$ we  can find a fractional change of variables, more precisely an automorphism of the ring
\begin{equation*}
  \Bbbk[x_5,x_6,x_7,x_9,\dots,x_{29}]\left[\frac{1}{(1-x_{18}x_9)(1-2x_{18}x_9)
   (1-2 x_{18}x_9-x_{18}^2 x_9^2)
  (1-3 x_{18}x_9+x_{18}^2 x_9^2)}\right],
\end{equation*}
such that the Haiman ideal is transformed into the ideal
\begin{equation*}
\begin{array}{lllll}\renewcommand{\arraystretch}{1.2}
J_{1321}& =& (x_{12} x_{15}+ x_{13} x_{19}+ x_{14} x_{24},& x_{12} x_{20}+ x_{23} x_{24}+ x_{19} x_{5},& -x_{16} x_{24}- x_{12} x_{26}+ x_{19} x_{28},\\
&&\hphantom{(} -x_{14} x_{20}+ x_{15} x_{23}+ x_{19} x_{29},& x_{20} x_{28}+ x_{26} x_{5}+ x_{24} x_{7},& x_{17} x_{24}- x_{13} x_{26}- x_{15} x_{28},\\
&&\hphantom{(}  -x_{16} x_{20}+ x_{23} x_{26}- x_{19} x_{7},& -x_{13} x_{23}+ x_{12} x_{29}+ x_{14} x_{5},& x_{13} x_{20}+ x_{24} x_{29}- x_{15} x_{5},\\
&&\hphantom{(} x_{17} x_{20}+ x_{26} x_{29}+ x_{15} x_{7},& x_{13} x_{16}+ x_{12} x_{17}+ x_{14} x_{28},& x_{15} x_{16}- x_{17} x_{19}- x_{14} x_{26},\\
&&\hphantom{(}   x_{23} x_{28}+ x_{16} x_{5}- x_{12} x_{7},& x_{17} x_{23}+ x_{16} x_{29}+ x_{14} x_{7},& x_{28} x_{29}- x_{17} x_{5}- x_{13} x_{7}).
\end{array}
\end{equation*}
For the details, see Appendix~\ref{sec:change-var-1311}. 
Finally, the map 
\begin{alignat*}{6}
x_{5}&\longmapsto p_{1,3},&\quad x_{7}&\longmapsto p_{1,2},&\quad x_{12}&\longmapsto p_{3,4},&\quad x_{13}&\longmapsto p_{0,3},&\quad x_{14}&\longmapsto p_{0,4},&\quad x_{15}&\longmapsto p_{0,5},\\
x_{16}&\longmapsto p_{2,4},&\quad x_{17}&\longmapsto -p_{0,2},&\quad x_{19}&\longmapsto p_{4,5},&\quad x_{20}&\longmapsto p_{1,5},&\quad x_{23}&\longmapsto p_{1,4},&\quad x_{24}&\longmapsto -p_{3,5},\\
x_{26}&\longmapsto p_{2,5},&\quad x_{28}&\longmapsto -p_{2,3},&\quad x_{29}&\longmapsto p_{0,1}.
\end{alignat*}
transforms the ideal $J_{1321}$ into the  Pl\"ucker ideal (\ref{eq-pluckerideal}).

There are three non-Borel monomial ideals of $\Bbbk[X_1,X_2,X_3]$ of colength $7$; see Appendix~\ref{sec:appendix-changevariable-NonBorelideals}.  However, such an explicit fractional change of variables is too difficult to find in these cases. We can only state the following conjecture.

We say that a torus action on the cone $\widehat{G}(2,6)$ is \emph{standard} if it is induced by an action on the tangent space at the vertex that preserves the Pl\"ucker ideal.

\begin{conjecture}\label{conj-existence-equivariant-iso-exdim6}
Let $I_{\lambda}$ be a non-Borel monomial ideal of\, $\Bbbk[X_1,X_2,X_3]$ with extra dimension 6. Then there exist a $T$-stable open subset $U$ of\, $0\in \Spec(\mathbb{A}_{\lambda})$ and an equivariant open immersion $f\colon U\hookrightarrow \widehat{G}(2,6)\times \mathbb{A}^9$, where $\widehat{G}(2,6)$ is equipped with a standard $T$-action.
\end{conjecture}

Once we know the existence of such an equivariant isomorphism, to compute the equivariant Hilbert function of $\mathbb{A}_{\lambda}$, it suffices to determine the standard action on $\widehat{G}(2,6)$ and the action on $\mathbb{A}^9$ induced by $f$. We do this by considering only the degree at most $2$ terms of the Haiman equations when we eliminate the variables. That is, we modify Algorithm~\ref{alg-step0-Haiman} by \emph{cutting off} the monomials of degree at least $3$ appearing in the resulting equations in every step. In this way, assuming Conjecture~\ref{conj-existence-equivariant-iso-exdim6}, we compute the equivariant Hilbert functions of $A_{\lambda_{1411}}$, $A_{\lambda_{2311}}$, and $A_{\lambda_{11311}}$ in Appendix~\ref{sec:appendix-changevariable-NonBorelideals}.

\subsection{Phenomenology of the local structures}\label{sec:conjectures}
In this subsection, we discuss several conjectures observed from the computations in the previous sections and Appendices~\ref{sec:appendix-changevariable-Borelideals} and~\ref{sec:appendix-changevariable-NonBorelideals}.

\subsubsection{Critical locus}\label{sec:crit-locus}

\begin{conjecture}\label{conj-existence-superpotential}
Let $\lambda$ be a $3$-dimensional partition of $n$ such that $I_{\lambda}$ is a Borel ideal of\, $\Bbbk[X_1,X_2,X_3]$. Then there exists a regular function $F_{\lambda}$ on the tangent space of\, $\Hilb^n(\mathbb{A}^3)$ at $Z_I$  such that the Haiman neighborhood  $\Spec(A_\lambda)$ is isomorphic to the critical locus of\, $F_\lambda$.
\end{conjecture}

It is known (see, \textit{e.g.}, \cite[Proposition 3.1]{BBS13}) that $\Spec(A_\lambda)$ is a critical locus in a regular scheme of  dimension larger than the embedded dimension of $A_\lambda$.

By the results in Sections~\ref{sec:3D-121}--\ref{sec:3D-1321} and \ref{sec:3D-1311}, and Appendices~\ref{sec:appendix-changevariable-Borelideals} and~\ref{sec:change-var-1311}, the conjecture is true for $n\leq 7$. By Theorem~\ref{prop-potential-pyramid}, it holds for pyramids.  By Lemma~\ref{lem-transitive-open-nbd}, the conjecture implies that for any ideal $I$ of colength $n$ of $\Bbbk[X_1,X_2,X_3]$, there exists a regular function $F$ on the tangent space of $\Hilb^n(\mathbb{A}^3)$ at $Z_I$ such that an open neighborhood of $Z_I$ is isomorphic to an open neighborhood of the critical locus of $F$.

Such statements are not true in dimension greater than $3$.

\subsubsection{Monomial ideals with extra dimension 6}\label{sec:tripodideals}
By abuse of notation, we say a monomial ideal $I$ of $\Bbbk[X_1,X_2,X_3]$ of finite colength $n$ is a monomial ideal with  extra dimension $d$ if the extra dimension of $\Hilb^n(\mathbb{A}^3)$ at $Z_I$ equals $d$. As we have noted after Theorem~\ref{thm-monomialideal-smoothable}, if $I$ is a monomial ideal with extra dimension 0, then $\Hilb^n(\mathbb{A}^3)$ is smooth at $Z_I$. Inspired by the results in the previous sections and the appendices, and the verifications on computers, we have the following conjecture.

\begin{conjecture}\label{conj:smallest-extra-dim}
For $n\geq 4$,
the smallest nonzero extra dimension of points on $\Hilb^n(\mathbb{A}^3)$  is $6$.
\end{conjecture}

As one sees in Sections~\ref{sec:3D-121}--\ref{sec:3D-132} and~\ref{sec:3D-1311} and the appendices, the monomial ideals with extra dimension equal to $6$ have similar shapes. A typical case of ideals of this shape is
\[
I_1=\left(X_1^{n_1},X_1 X_2,X_1 X_3, X_2^{n_2},X_2 X_3,X_3^{n_3}\right),
\]
where $n_1,n_2,n_3\geq 2$. Using Corollary~\ref{cor-description-cotangentspace}, an easy counting shows that the extra dimension of $I_1$ is $6$. Graphically, its corresponding $3$-dimensional partition looks like
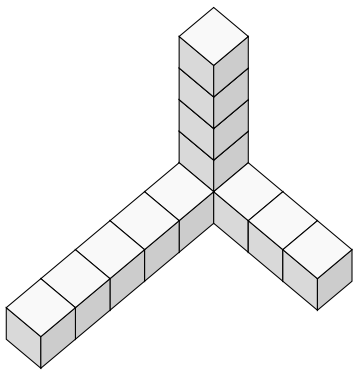
\begin{center}
\begin{tikzpicture}[x=(220:0.6cm), y=(-40:0.6cm), z=(90:0.42cm)]

\foreach \m [count=\y] in {{5,1,1,1,1,1},{1},{1},{1}}{
  \foreach \n [count=\x] in \m {
  \ifnum \n>0
      \foreach \z in {1,...,\n}{
        \draw [fill=gray!30] (\x+1,\y,\z) -- (\x+1,\y+1,\z) -- (\x+1, \y+1, \z-1) -- (\x+1, \y, \z-1) -- cycle;
        \draw [fill=gray!40] (\x,\y+1,\z) -- (\x+1,\y+1,\z) -- (\x+1, \y+1, \z-1) -- (\x, \y+1, \z-1) -- cycle;
        \draw [fill=gray!5] (\x,\y,\z)   -- (\x+1,\y,\z)   -- (\x+1, \y+1, \z)   -- (\x, \y+1, \z) -- cycle;  
      }
 \fi
 }
}    

\end{tikzpicture}
\end{center}
This picture inspires the following name.

\begin{definition}
A monomial ideal $I$ of $\Bbbk[X_1,X_2,X_3]$ is called a \emph{tripod} ideal if $I$ has a set of minimal generators of the form
\[
X_1^{a},X_1^{b} X_2^{c},X_1^{d} X_3^{e}, X_2^{f},X_2^{g} X_3^{h},X_3^{i}.
\]
\end{definition}

Note that the condition \emph{minimal} puts restrictions on these exponents.

There exist monomial ideals $I$ of $\Bbbk[X_1,X_2,X_3]$ with extra dimension $6$ that are not  tripod ideals, for example, 
\[
I_2=\left(X_1^{3},X_1^{2}X_2,X_1 X_2^{2},X_1 X_2 X_3,X_2^{3},X_2^{2}X_3, X_3^{2}\right),\
Z_{I_2} \in \Hilb^{10}(\mathbb{A}^3).
\]
The corresponding $3$-dimensional partition is
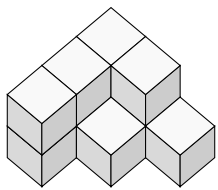
\begin{center}
\begin{tikzpicture}[x=(220:0.6cm), y=(-40:0.6cm), z=(90:0.42cm)]

\foreach \m [count=\y] in {{2,2,2},{2,1},{1}}{
  \foreach \n [count=\x] in \m {
  \ifnum \n>0
      \foreach \z in {1,...,\n}{
        \draw [fill=gray!30] (\x+1,\y,\z) -- (\x+1,\y+1,\z) -- (\x+1, \y+1, \z-1) -- (\x+1, \y, \z-1) -- cycle;
        \draw [fill=gray!40] (\x,\y+1,\z) -- (\x+1,\y+1,\z) -- (\x+1, \y+1, \z-1) -- (\x, \y+1, \z-1) -- cycle;
        \draw [fill=gray!5] (\x,\y,\z)   -- (\x+1,\y,\z)   -- (\x+1, \y+1, \z)   -- (\x, \y+1, \z) -- cycle;  
      }
 \fi
 }
}    

\end{tikzpicture}
\end{center}
But $I_2$ is not a Borel ideal. In fact, after checking monomial ideals of colength at most $25$,  we make the following conjecture.

\begin{conjecture}\label{conj:exta-dim-6-implies-tripod}
A Borel ideal of\, $\Bbbk[X_1,X_2,X_3]$ with extra dimension $6$ is a tripod ideal.
\end{conjecture}

Conversely, let us consider which tripod ideals have extra dimension equal to $6$. First we select the Borel tripod ideals.

\begin{lemma}
A tripod ideal $I=(X_1^{a},X_1^{b} X_2^{c},X_1^{d} X_3^{e}, X_2^{f},X_2^{g} X_3^{h},X_3^{i})$ is Borel fixed in the lexicographic order $X_1\succ X_2\succ X_3$ if and only if 
\[
c=e=g=h=1,\quad f=i=2.
\]
\end{lemma}

The proof is straightforward from Definition~\ref{def-Borelideal}\eqref{dB-2}. Then we make the following conjecture. 
\begin{conjecture}\label{conj:borelIdeal-extra-dim}
A Borel ideal of the form $I=(X_1^{a},X_1^{b} X_2,X_1^{d} X_3, X_2^{2},X_2 X_3,X_3^2)$ has extra dimension $6$ if and only if at least one of\, $b$ and $d$ is equal to $1$ or $a-1$.
\end{conjecture}

We have checked it for the ideals of this form of colength at most $100$.

\subsubsection{Types of singularities}\label{sec:similarity-conj}

\begin{proposition}\label{prop--localstructure-extradim6}
Let $z$ be a point on $\Hilb^n(\mathbb{A}^3)$. For $n\leq 7$, if the embedded dimension at $z$ is $3n+6$, then there exist an open neighborhood $U$ of $z$ and an open immersion $U\hookrightarrow \widehat{G}(2,6)\times \mathbb{A}^{3n-9}$.
\end{proposition}

\begin{proof}
By Lemma~\ref{lem-transitive-open-nbd}, we need only consider the Haiman neighborhoods of $Z_I$ where $I$ runs over the Borel ideals of $\Bbbk[X_1,X_2,X_3]$ of colength at most $7$. Then by Example~\ref{example-list-Borelideals}, the results follow from   Sections~\ref{sec:3D-121}--\ref{sec:3D-132}, Appendices~\ref{sec:3D-141}--\ref{sec:3D-232},  and the proof of Proposition~\ref{prop-jacF1321-isolated} (also Appendix~\ref{sec:appendix-change-var-jacF1321}). 
\end{proof}

\begin{conjecture}\label{conj-localstructure-extradim6}
Let $z$ be a point on the main component of\, $\Hilb^n(\mathbb{A}^3)$. If the embedded dimension at $z$ is $3n+6$, then $z$ has an open neighborhood which is isomorphic to an open subset of a trivial affine fibration over the cone $\widehat{G}(2,6)$.
\end{conjecture}

In short, these singular points are of the same type; \textit{i.e.}, their germs can be transformed into each other by a chain of smooth morphisms. Note that when $n$ is sufficiently large, \textit{e.g.}, $n\geq 78$, the Hilbert scheme $\Hilb^n(\mathbb{A}^3)$ is reducible.  Conjecture~\ref{conj-localstructure-extradim6} might need to be modified for the points not lying in the main component.  But since the monomial ideals all lie on the main component (see Theorem~\ref{thm-monomialideal-smoothable}), Conjecture~\ref{conj-localstructure-extradim6} should be true for them.

Also, the monomial ideals $I_{1321}$ in Section~\ref{sec:3D-1321} and $I_{2321}$ in Appendix~\ref{sec:3D-2321} are ideals with extra dimension~$8$ and have similar partition shapes. By the results of Section~\ref{sec:3D-1321} and Appendix~\ref{sec:3D-2321}, they correspond to the same type of singularities on $\Hilb^7(\mathbb{A}^3)$ and $\Hilb^8(\mathbb{A}^3)$, respectively.

In general, we make the following conjecture. 

\begin{conjecture}\label{conj-similarity}
Let $\lambda_1$, $\lambda_2$ be two $r$-dimensional partitions of length $l_1$, $l_2$, respectively. Suppose $\lambda_1$ and $\lambda_2$ have similar shapes and
\[
\mathrm{ex.dim}_{I_{\lambda_1}}\Hilb^{l_1}(\mathbb{A}^r)=\mathrm{ex.dim}_{I_{\lambda_2}}\Hilb^{l_2}(\mathbb{A}^r). 
\]
Then $\Hilb^{l_1}(\mathbb{A}^3)_{Z_{I_{\lambda_1}}}$ and $\Hilb^{l_2}(\mathbb{A}^3)_{Z_{I_{\lambda_2}}}$ are singularities of the same type.
\end{conjecture}

The \emph{shape} of a partition $\lambda$ roughly means the relative positions among the minimal lattices of the glove of~$\lambda$, or equivalently among the lattices of the exponents of the minimal monomial generators of the ideal~$I_{\lambda}$.  For example, the minimal lattices of the glove of a tripod partition look like
\begin{gather*}
\begin{tikzpicture}[x=(0:0.8cm), y=(90:0.8cm)]
\foreach \m [count=\y] in {{3,2},{2.5}}{
    \foreach \n [count=\x] in \m {
        \draw[densely dashed] (\n,\y) -- (\n+1,\y) -- (\n+0.5,\y+1)  -- cycle;
        \filldraw (\n,\y) circle (1pt);
        \filldraw (\n+1,\y) circle (1pt);
        \filldraw (\n+0.5,\y+1) circle (1pt);
 }
}.
\end{tikzpicture}
\end{gather*}
On the other hand, as we have checked on $\Hilb^n(\mathbb{A}^3)$ for $n\leq 20$, the minimal lattices of the glove of a partition which corresponds to Borel ideals with extra dimension $8$ look like
\begin{gather*}
\begin{tikzpicture}[x=(0:0.8cm), y=(90:0.8cm)]
    \draw[densely dashed] (1,2) --(2,2) --(3.35/2,3)--cycle;
    \draw[densely dashed] (2,2) --(2.7,2);
    \draw[densely dashed] (2.7,2) --(3.7,2) --(6.05/2,3)--cycle;
    \draw[densely dashed] (3.35/2,3) --(6.05/2,3) --(2.35,4)--cycle;
    \filldraw (1,2) circle (1pt);
    \filldraw (2,2) circle (1pt);
    \filldraw (3.35/2,3) circle (1pt);
    \filldraw (2.7,2) circle (1pt);
    \filldraw (3.7,2) circle (1pt);
    \filldraw (6.05/2,3) circle (1pt);
    \filldraw (2.35,4) circle (1pt);
\end{tikzpicture}
\end{gather*}
(Note that in both graphs, collinear vertices do not indicate that the corresponding lattices in $\mathbb{Z}^3$ are collinear.)

So for two partitions to have similar shapes, it is necessary that the corresponding monomial ideals have the same number of minimal monomial generators. But this condition is of course not sufficient. For example, plane partitions embedded into $\mathbb{Z}^3$ always correspond to smooth points (see Lemma~\ref{lem-lower-dim-partition}), but the corresponding ideals can have an arbitrary number of minimal monomial generators. We expected a stronger necessary condition for two 3D partitions $\lambda_1$ and $\lambda_2$ to have similar shapes: There exists a 3D partition $\lambda$ such that $\lambda\subset \lambda_1$ and $\lambda\subset \lambda_2$, and $I_{\lambda}$, $I_{\lambda_1}$, $I_{\lambda_2}$ have the same number of minimal monomial generators.

We cannot formulate Conjecture~\ref{conj-similarity} in a precise way because we do not have enough examples to make precise what \emph{similar shapes} mean. In fact, in our examples above, having the same extra dimension suffices. But we still feel that in the general cases, we need a certain condition on the shapes of the $r$-dimensional partitions. A reason is that, as we have seen in Section~\ref{sec:tripodideals}, when the extra dimension is small, the shape of the partition has few choices, or at least the Borel ones do. The first example of partitions that do not have similar shapes while the corresponding ideals have the same extra dimension is the following; these partitions correspond to Borel ideals with extra dimension $12$:

\begin{tabular}{cccccccccc}
&&&&&&&&&\\
&&&&&&&&\threeDYoung{{{2, 1}, {2, 1}, {1}, {1}}}
&
\threeDYoung{{{2, 1, 1}, {2, 1}, {1}, {1}}}
\end{tabular}

\begin{tabular}{cccccccccc}
&&&&&&&&&\\
&&&&&&&&\threeDYoung{{{2, 1, 1, 1}, {2, 1, 1}, {1, 1}, {1}}}
&
\threeDYoung{{{2, 1, 1, 1, 1}, {1, 1, 1, 1}, {1, 1, 1}, {1, 1}, {1}}}
\end{tabular}

\begin{question}\label{que-extradim12-points}
Are the singularities on $\Hilb^8(\mathbb{A}^3)$, $\Hilb^9(\mathbb{A}^3)$, $\Hilb^{12}(\mathbb{A}^3)$, and $\Hilb^{16}(\mathbb{A}^3)$ corresponding to the above $3$-dimensional partitions, respectively, of the same type?
\end{question}

\begin{remark}\label{rem:JRS24}
Recently, in 2024, Jelisiejew, Ramkumar, and Sammartano \cite{JRS24} proved Conjectures~\ref{conj:exta-dim-6-implies-tripod} and~\ref{conj:borelIdeal-extra-dim} completely and confirmed Conjectures~\ref{conj:smallest-extra-dim} and~\ref{conj-localstructure-extradim6} in many cases.
\end{remark}

\section{Euler characteristics of tautological sheaves}\label{sec:euler-char-taut-sheaves}

\subsection{Hilbert series of tripod singularities}
The singularities of $A_{\lambda}$ in Sections~\ref{sec:3D-121}--\ref{sec:3D-132} and \ref{sec:3D-1311} and Appendices~\ref{sec:3D-141}--\ref{sec:3D-232} and~\ref{sec:3D-1411}--\ref{sec:3D-11311} have the same singularity type (in the sense of modulo the smooth equivalence). We compute their equivariant Hilbert functions in a uniform way. Let 
$
R_{G(2,6)}=\Bbbk[p_{i,j}]_{0\leq i<j\leq 5}
$.
Then the ring
\[
S_{G(2,6)}=R_{G(2,6)}/I_{G(2,6)}
\]
is the coordinate ring of the cone $\widehat{G}(2,6)$, where $I_{G(2,6)}$ is the Pl\"ucker ideal (\ref{eq-pluckerideal}). There is an action of $T_1=\mathbb{G}_m^6$ on $S_{G(2,6)}$, with weights
\begin{equation}\label{eq-action-G(2,6)}
\mathbf{w}(p_{i,j})=\epsilon_i+\epsilon_j\in \mathbb{Z}^6,\quad 0\leq i\neq j\leq 5,
\end{equation}
where $\{\epsilon_i\}_{0\leq i\leq 5}$ is a basis of $\mathbb{Z}^6$, the character group of $T_1$.  We denote the generic point of $T_1$ by $\mathbf{u}=(u_0,u_1,u_2,u_3,u_4,u_5)$.  We are going to compute
\begin{eqnarray*}
H\left(S(G(2,6)\right);\mathbf{u}):=\frac{\sum_{i=0}^{15}(-1)^i\Tor_i^{R_{G(2,6)}}(S_{G(2,6)},\Bbbk)}{\prod_{0\leq i\neq j\leq 5}(1-u_i u_j)}
\end{eqnarray*}
 as a virtual representation of $T_1$.

\begin{lemma}
\allowdisplaybreaks 
  \begin{eqnarray}\label{eq-equihilb-tripod}
&&\sum_{i=0}^{15}(-1)^i\Tor_i^{R_{G(2,6)}}(S_{G(2,6)},\Bbbk)\notag\\
&&= 1-\sum_{\begin{subarray}{c}S\subset\{0,..5\}\\ |S|=4\end{subarray}}\prod_{i\in S} u_i
+\sum_{\begin{subarray}{c}S\subset\{0,..5\}\\ |S|=5\end{subarray}}\left(\prod_{i\in S} u_i\sum_{i\in S}u_i\right)
+5\prod_{i=0}^5 u_i-\prod_{i=0}^5 u_i\cdot \sum_{0\leq i\leq j\leq 5}u_i u_j\notag\\
&&\hphantom{=}\;
-\sum_{\begin{subarray}{c}S\subset\{0,..5\}\\ |S|=5\end{subarray}}\prod_{i\in S} u_i^2
-\sum_{\begin{subarray}{c}S\subset\{0,..5\}\\ |S|=4\end{subarray}}\left(\prod_{i\in S} u_i^2\cdot \prod_{j\not\in S} u_j\right)
+\prod_{i=0}^5 u_i\cdot \sum_{\begin{subarray}{c}S\subset\{0,..5\}\\ |S|=5\end{subarray}}\left(\prod_{i\in S} u_i\sum_{i\in S}u_i\right)+5\prod_{i=0}^5 u_i^2\notag\\
&&\hphantom{=}\;-\prod_{i=0}^{5}u_i^2\cdot\sum_{0\leq i<j\leq 5}u_iu_j
+\prod_{i=0}^{5}u_i^3.
\end{eqnarray}
\end{lemma}

\begin{proof}
We use an explicit resolution of the Pl\"ucker ideal of $G(m,n)$, which is known for $m=2$ and arbitrary~$n$. Let $E$ be a $\Bbbk$-vector space with a basis ${v_0,v_1,v_2,v_3,v_4,v_5}$, equipped with a $T_1$-action with weights $\mathbf{w}(v_i)=\epsilon_i\in \mathbb{Z}^6$ for $0\leq i\leq 5$. We identify the ring $R_{G(2,6)}$ with the symmetric algebra $\Sym^\bullet(\bigwedge^2 E)$ so that the induced torus action coincides with (\ref{eq-action-G(2,6)}). For a $2$-dimensional partition $\lambda$, let $L_{\lambda}E$ be the associated Schur module of $E$ (see \cite[Section~2.1]{Wey03}). Each $L_{\lambda}E$ is an irreducible $\GL(E)$-representation and has  highest weight $\lambda'$ (the conjugate partition of $\lambda$).

By \cite[Theorem 6.4.1]{Wey03}, the minimal resolution $P_\bullet$ of $S_{G(2,6)}$, as an $R_{G(2,6)}$-module, has the form $P_i=V_i\otimes_{\Bbbk}R_{G(2,6)}$, where $V_0=\Bbbk$ is the trivial $\mathrm{GL}(E)$-representation and
\begin{alignat}{3}\label{eq-SchurMod-1}
V_1&=L_{{\yng(4)}}E,&\quad V_2&=L_{\yng(5,1)}E,&\quad
V_3&=L_{\yng(6,1,1)}E\oplus L_{\yng(5,5)}E,\nonumber\\
V_{4}&=L_{\yng(6,5,1)}E,&\quad V_5&=L_{\yng(6,6,2)}E,&\quad 
V_{6}&=L_{\yng(6,6,6)}E
\end{alignat}
and $V_i=0$ for $i\geq 7$.
 The character of $L_{\lambda}E$ is equal to the Schur function $s_{\lambda'}(u_0,..,u_5)$. For example, 
\begin{align*}
\Char(L_{\yng(5,5)}E)&\ =\ s_{\yng(2,2,2,2,2)}(u_0,..,u_5)=\det\begin{pmatrix}e_5 & e_6\\ e_4 & e_5\end{pmatrix}\\
&\ =\ \sum_{\begin{subarray}{c}S\subset\{0,..5\}\\ |S|=5\end{subarray}}\prod_{i\in S} u_i^2
+\sum_{\begin{subarray}{c}S\subset\{0,..5\}\\ |S|=4\end{subarray}}\left(\prod_{i\in S} u_i^2\cdot \prod_{j\not\in S} u_j\right),
\end{align*}
where the $e_i=e_i(u_0,\ldots,u_5)$ are the elementary symmetric functions.
The characters of the other Schur modules in (\ref{eq-SchurMod-1}) can be computed directly from the definition: 
\begin{eqnarray*}
\Char(L_{\yng(4)}E)& =& \sum_{\begin{subarray}{c}S\subset\{0,..5\}\\ |S|=4\end{subarray}}\prod_{i\in S} u_i,\\
\Char(L_{\yng(5,1)}E)& =& \sum_{\begin{subarray}{c}S\subset\{0,..5\}\\ |S|=5\end{subarray}}\left(\prod_{i\in S} u_i\sum_{i\in S}u_i\right)
+5\prod_{i=0}^5 u_i,\\
\Char(L_{\yng(6,1,1)}E)& = &\prod_{i=0}^5 u_i\cdot \sum_{0\leq i\leq j\leq 5}u_i u_j,\\
 \Char(L_{\yng(6,5,1)}E)& =& 
\prod_{i=0}^5 u_i\cdot \sum_{\begin{subarray}{c}S\subset\{0,..5\}\\ |S|=5\end{subarray}}\left(\prod_{i\in S} u_i\sum_{i\in S}u_i\right)
+5\prod_{i=0}^5 u_i^2,\\
 \Char(L_{\yng(6,6,2)}E)& = &\prod_{i=0}^{5}u_i^2\cdot\sum_{0\leq i<j\leq 5}u_iu_j,\\
 \Char(L_{\yng(6,6,6)}E)& = &\prod_{i=0}^{5}u_i^3.
\end{eqnarray*}
Taking the signed summation $\sum_i (-1)^i \mathrm{char}(V_i)$, we are done.
\end{proof}

Denote the function (\ref{eq-equihilb-tripod}) by $K(u_0,u_1,u_2,u_3,u_4,u_5)$.

\begin{corollary}\label{cor-tripod-Hilbertfunctions}
\begin{eqnarray}\label{eq-equihilb-A121}
H(A_{{\lambda}_{121}};\mathbf{t})&=&K\left(-\frac{\sqrt{t_{2}} \sqrt{t_{3}}}{\sqrt{t_{1}}},-\frac{\sqrt{t_{1}} \sqrt{t_{3}}}{\sqrt{t_{2}}},-\frac{\sqrt{t_{1}} \sqrt{t_{2}}}{\sqrt{t_{3}}},-\frac{t_{2}^{3/2}}{\sqrt{t_{1}} \sqrt{t_{3}}},-\frac{t_{3}^{3/2}}{\sqrt{t_{1}} \sqrt{t_{2}}},-\frac{t_{1}^{3/2}}{\sqrt{t_{2}} \sqrt{t_{3}}}\right)\notag\\
&&\Big/\left((1-t_{1})^3(1-t_{2})^3(1-t_{3})^3\left(\frac{t_{1}-t_{2}^2}{t_{1}}\right)\left(\frac{t_{1}-t_{2} t_{3}}{t_{1}}\right)\left(\frac{t_{1}-t_{3}^2}{t_{1}}\right)\left(\frac{t_{2}-t_{1}^2}{t_{2}}\right)\left(\frac{t_{2}-t_{1} t_{3}}{t_{2}}\right)\right.\notag\\
  &&\hphantom{\Big/\Big(}\left.\left(\frac{t_{2}-t_{3}^2}{t_{2}}\right)\left(\frac{t_{3}-t_{1}^2}{t_{3}}\right)\left(\frac{t_{3}-t_{1} t_{2}}{t_{3}}\right)\left(\frac{t_{3}-t_{2}^2}{t_{3}}\right)\right),
\end{eqnarray}
\begin{eqnarray}\label{eq-equihilb-A131}
H(A_{{\lambda}_{131}};\mathbf{t})&=&K\left(-\frac{\sqrt{t_{2}} \sqrt{t_{3}}}{t_{1}},-\frac{t_{1} \sqrt{t_{3}}}{\sqrt{t_{2}}},-\frac{t_{1} \sqrt{t_{2}}}{\sqrt{t_{3}}},-\frac{t_{2}^{3/2}}{t_{1} \sqrt{t_{3}}},-\frac{t_{3}^{3/2}}{t_{1} \sqrt{t_{2}}},-\frac{t_{1}^2}{\sqrt{t_{2}} \sqrt{t_{3}}}\right)\notag\\
&&\Big/\left((1-t_{1})^3(1-t_{2})^3(1-t_{3})^3(1-t_{1}^2)\left(\frac{t_{2}-t_{1} t_{3}}{t_{2}}\right)\left(\frac{t_{2}-t_{3}^2}{t_{2}}\right)\left(\frac{t_{2}-t_{1}^3}{t_{2}}\right)\left(\frac{t_{3}-t_{1} t_{2}}{t_{3}}\right)\right.\notag\\
&&\hphantom{\Big/\Big(}\left.\left(\frac{t_{3}-t_{2}^2}{t_{3}}\right)\left(\frac{t_{3}-t_{1}^3}{t_{3}}\right)\left(\frac{t_{1}-t_{2}}{t_{1}}\right)\left(\frac{t_{1}-t_{3}}{t_{1}}\right)\left(\frac{t_{1}^2-t_{2}^2}{t_{1}^2}\right)\left(\frac{t_{1}^2-t_{2} t_{3}}{t_{1}^2}\right)\left(\frac{t_{1}^2-t_{3}^2}{t_{1}^2}\right)\right),
\end{eqnarray}
\begin{eqnarray}\label{eq-equihilb-A132}
H(A_{{\lambda}_{132}};\mathbf{t})&=&K\left(-\frac{\sqrt{t_{2}} \sqrt{t_{3}}}{\sqrt{t_{1}}},-\frac{\sqrt{t_{1}} \sqrt{t_{3}}}{\sqrt{t_{2}}},-\frac{\sqrt{t_{1}} \sqrt{t_{2}}}{\sqrt{t_{3}}},-\frac{t_{2}^{3/2}}{\sqrt{t_{1}} \sqrt{t_{3}}},-\frac{t_{3}^{3/2}}{t_{1}^{3/2} \sqrt{t_{2}}},-\frac{t_{1}^{5/2}}{\sqrt{t_{2}} \sqrt{t_{3}}}\right)\notag\\
&&\Big/\left((1-t_{1})^3(1-t_{2})^3(1-t_{3})^2(1-t_{1}^2)(\frac{t_{1}-t_{2}^2}{t_{1}})
\left(\frac{t_{1}-t_{2}}{t_{1}}\right)\left(\frac{t_{1}-t_{3}}{t_{1}}\right)^2
\left(\frac{t_{1}^2-t_{2}^2}{t_{1}^2}\right)\right.\notag\\
&&\hphantom{\Big/\Big(}\left(\frac{t_{1}^2-t_{2} t_{3}}{t_{1}^2}\right)\left(\frac{t_{1}^2-t_{3}^2}{t_{1}^2}\right)
\left(\frac{t_{2}-t_{1} t_{3}}{t_{2}}\right)\left(\frac{t_{2}-t_{1}^3}{t_{2}}\right)\left(\frac{t_{2}-t_{3}}{t_{2}}\right)\left(\frac{t_{2}-t_{1}^2}{t_{2}}\right)\left(\frac{t_{1} t_{2}-t_{3}^2}{t_{1} t_{2}}\right)\notag\\
&&\hphantom{\Big/\Big(}\left.\left(\frac{t_{3}-t_{1}^3}{t_{3}}\right)\left(\frac{t_{3}-t_{1}^2 t_{2}}{t_{3}}\right)\left(\frac{t_{3}-t_{2}^2}{t_{3}}\right)\right),
\end{eqnarray}
\begin{eqnarray}\label{eq-equihilb-A1311}
H(A_{{\lambda}_{1311}};\mathbf{t})&=&K\left(-\frac{t_{2} \sqrt{t_{3}}}{t_{1}},-\frac{t_{1} \sqrt{t_{3}}}{t_{2}},-\frac{t_{1} t_{2}}{\sqrt{t_{3}}},-\frac{t_{2}^2}{t_{1} \sqrt{t_{3}}},-\frac{t_{3}^{3/2}}{t_{1} t_{2}},-\frac{t_{1}^2}{t_{2} \sqrt{t_{3}}}\right)\notag\\
&&\Big/\left((1-t_{1})^3(1-t_{2})^3(1-t_{3})^3(1-t_{1}^2)(1-t_{2}^2)\left(\frac{t_{1}-t_{2}}{t_{1}}\right)\left(\frac{t_{1}-t_{3}}{t_{1}}\right)\left(\frac{t_{1}^2-t_{2} t_{3}}{t_{1}^2}\right)\right.\notag\\
&&\hphantom{\Big/\Big(}\left(\frac{t_{1}^2-t_{2}^3}{t_{1}^2}\right)
\left(\frac{t_{1}^2-t_{3}^2}{t_{1}^2}\right)\left(\frac{t_{2}-t_{1}}{t_{2}}\right)\left(\frac{t_{2}^2-t_{1} t_{3}}{t_{2}^2}\right)\left(\frac{t_{2}^2-t_{1}^3}{t_{2}^2}\right)\left(\frac{t_{2}-t_{3}}{t_{2}}\right)\left(\frac{t_{2}^2-t_{3}^2}{t_{2}^2}\right)\notag\\
&&\hphantom{\Big/\Big(}\left.\left(\frac{t_{3}-t_{1} t_{2}}{t_{3}}\right)
\left(\frac{t_{3}-t_{1}^3}{t_{3}}\right)\left(\frac{t_{3}-t_{2}^3}{t_{3}}\right)\right).
\end{eqnarray}
\end{corollary}

\begin{proof} We only show (\ref{eq-equihilb-A121}); the others are similar. Let  $T=\mathbb{G}_m^3$; the $T$-action on the Haiman coordinates transfers via (\ref{eq-step0-change-varialbes-A121}) and (\ref{eq-variablechange-plucker-A121}) to an action on the Pl\"ucker coordinates. One checks that this action coincides with the one induced by the map
\begin{alignat*}{3}
u_{0}&\longmapsto -\frac{\sqrt{t_{2}} \sqrt{t_{3}}}{t_{1}},&\quad
u_{1}&\longmapsto -\frac{t_{1} \sqrt{t_{3}}}{\sqrt{t_{2}}},&\quad
u_{2}&\longmapsto -\frac{t_{1} \sqrt{t_{2}}}{\sqrt{t_{3}}},\\
u_{3}&\longmapsto -\frac{t_{2}^{3/2}}{t_{1} \sqrt{t_{3}}},&\quad
u_{4}&\longmapsto -\frac{t_{3}^{3/2}}{t_{1} \sqrt{t_{2}}},&\quad
u_{5}&\longmapsto -\frac{t_{1}^2}{\sqrt{t_{2}} \sqrt{t_{3}}}.
\end{alignat*}
The appearance of square roots arises from that our choice of the $T_1$-action on $G(2,6)$, where $T_1=\mathbb{G}_m^6$, is not the primitive one. This gives the numerator of (\ref{eq-equihilb-A121}). The denominator is computed by the weights of the tangent spaces, via, \textit{e.g.}, the description of Corollary~\ref{cor-description-cotangentspace}. 
\end{proof}

In the same way, in Appendices~\ref{sec:appendix-changevariable-Borelideals} and~\ref{sec:appendix-changevariable-NonBorelideals},  we compute the other equivariant Hilbert functions that we need later.

\subsection{The equivariant Hilbert function of the local ring at \texorpdfstring{$((1)\subset (3,2,1))$}{((1) in (3,2,1))}}
We compute the equivariant Hilbert function of $A_{\lambda_{1321}}\cong \Bbbk[y_1,\dots,y_{29}]/\Jac(F_{1321})$ 
in the way  recalled in Section~\ref{sec:equivariant-Hilbertfunction}. Namely, we first compute the Gr\"obner basis of $\Jac(F_{1321})$ and then apply Proposition~\ref{prop-hilbertseries-gbbasis} and Lemma~\ref{lem-key}. We record the result in the following. 

\begin{proposition}\label{prop-Hilbertfunction-1321}
  \begin{align*}
H(A_{\lambda_{1321}};\mathbf{t})&\\
=-\left(t_{1}^8 t_{2}^8 t_{3}^2\right.\big(&-t_{2}^4 t_{3}^3 t_{1}^9-t_{2}^3 t_{3}^3 t_{1}^9-t_{2}^5 t_{3}^2 t_{1}^9+t_{2}^6 t_{3} t_{1}^9+t_{2}^5 t_{3} t_{1}^9+t_{2}^4 t_{3} t_{1}^9-t_{2}^8 t_{1}^8-t_{2}^7 t_{1}^8+t_{2}^3 t_{3}^5 t_{1}^8+t_{2}^5 t_{3}^4 t_{1}^8 \\
&+t_{2}^4 t_{3}^4 t_{1}^8+t_{2}^2 t_{3}^4 t_{1}^8+t_{2} t_{3}^4 t_{1}^8
-2 t_{2}^5 t_{3}^3 t_{1}^8-2 t_{2}^4 t_{3}^3 t_{1}^8+t_{2}^5 t_{3}^2 t_{1}^8-t_{2}^3 t_{3}^2 t_{1}^8-t_{2}^2 t_{3}^2 t_{1}^8+t_{2}^7 t_{3} t_{1}^8+t_{2}^6 t_{3} t_{1}^8\\
&+t_{2}^5 t_{3} t_{1}^8-t_{2}^8 t_{1}^7-t_{2}^7 t_{1}^7-t_{2} t_{3}^6 t_{1}^7+2 t_{2}^4 t_{3}^5 t_{1}^7-t_{2}^2 t_{3}^5 t_{1}^7+t_{2}^6 t_{3}^4 t_{1}^7+t_{2}^5 t_{3}^4 t_{1}^7-t_{2}^4 t_{3}^4 t_{1}^7+t_{2}^3 t_{3}^4 t_{1}^7+3 t_{2}^2 t_{3}^4 t_{1}^7\\
&+t_{2} t_{3}^4 t_{1}^7-3 t_{2}^6 t_{3}^3 t_{1}^7-3 t_{2}^5 t_{3}^3 t_{1}^7+t_{2}^4 t_{3}^3 t_{1}^7-t_{2}^2 t_{3}^3 t_{1}^7-t_{2}^7 t_{3}^2 t_{1}^7+2 t_{2}^6 t_{3}^2 t_{1}^7+t_{2}^5 t_{3}^2 t_{1}^7-2 t_{2}^4 t_{3}^2 t_{1}^7-2 t_{2}^3 t_{3}^2 t_{1}^7\\
&+t_{2}^8 t_{3} t_{1}^7+t_{2}^7 t_{3} t_{1}^7+t_{2}^6 t_{3} t_{1}^7+t_{2}^4 t_{3} t_{1}^7-t_{2}^8 t_{1}^6-t_{2}^6 t_{3}^6 t_{1}^6-t_{2}^5 t_{3}^6 t_{1}^6-t_{2}^4 t_{3}^6 t_{1}^6-2 t_{2}^2 t_{3}^6 t_{1}^6-t_{2} t_{3}^6 t_{1}^6+3 t_{2}^5 t_{3}^5 t_{1}^6\\
&-2 t_{2}^3 t_{3}^5 t_{1}^6-t_{3}^5 t_{1}^6+t_{2}^7 t_{3}^4 t_{1}^6+t_{2}^6 t_{3}^4 t_{1}^6-2 t_{2}^5 t_{3}^4 t_{1}^6+2 t_{2}^4 t_{3}^4 t_{1}^6+6 t_{2}^3 t_{3}^4 t_{1}^6+t_{2}^2 t_{3}^4 t_{1}^6-3 t_{2}^7 t_{3}^3 t_{1}^6-3 t_{2}^6 t_{3}^3 t_{1}^6\\
&+2 t_{2}^5 t_{3}^3 t_{1}^6-2 t_{2}^3 t_{3}^3 t_{1}^6+t_{2} t_{3}^3 t_{1}^6+2 t_{2}^7 t_{3}^2 t_{1}^6+t_{2}^6 t_{3}^2 t_{1}^6-3 t_{2}^5 t_{3}^2 t_{1}^6-3 t_{2}^4 t_{3}^2 t_{1}^6+t_{2}^9 t_{3} t_{1}^6+t_{2}^8 t_{3} t_{1}^6+t_{2}^7 t_{3} t_{1}^6\\
&+2 t_{2}^5 t_{3} t_{1}^6+t_{2}^4 t_{3} t_{1}^6+t_{2}^3 t_{3}^7 t_{1}^5+t_{2}^2 t_{3}^7 t_{1}^5+t_{3}^7 t_{1}^5-t_{2}^6 t_{3}^6 t_{1}^5+t_{2}^4 t_{3}^6 t_{1}^5-3 t_{2}^3 t_{3}^6 t_{1}^5-2 t_{2}^2 t_{3}^6 t_{1}^5+3 t_{2}^6 t_{3}^5 t_{1}^5\\
&-t_{2}^5 t_{3}^5 t_{1}^5-4 t_{2}^4 t_{3}^5 t_{1}^5+t_{2}^2 t_{3}^5 t_{1}^5-2 t_{2} t_{3}^5 t_{1}^5-t_{3}^5 t_{1}^5+t_{2}^8 t_{3}^4 t_{1}^5+t_{2}^7 t_{3}^4 t_{1}^5-2 t_{2}^6 t_{3}^4 t_{1}^5+t_{2}^5 t_{3}^4 t_{1}^5+8 t_{2}^4 t_{3}^4 t_{1}^5\\
&+2 t_{2}^3 t_{3}^4 t_{1}^5-t_{2}^2 t_{3}^4 t_{1}^5+t_{2} t_{3}^4 t_{1}^5-2 t_{2}^8 t_{3}^3 t_{1}^5-3 t_{2}^7 t_{3}^3 t_{1}^5+2 t_{2}^6 t_{3}^3 t_{1}^5-4 t_{2}^4 t_{3}^3 t_{1}^5+2 t_{2}^2 t_{3}^3 t_{1}^5-t_{2}^9 t_{3}^2 t_{1}^5+t_{2}^8 t_{3}^2 t_{1}^5\\
&+t_{2}^7 t_{3}^2 t_{1}^5-3 t_{2}^6 t_{3}^2 t_{1}^5-3 t_{2}^5 t_{3}^2 t_{1}^5+t_{2}^4 t_{3}^2 t_{1}^5-t_{2}^3 t_{3}^2 t_{1}^5+t_{2}^9 t_{3} t_{1}^5+t_{2}^8 t_{3} t_{1}^5+2 t_{2}^6 t_{3} t_{1}^5+t_{2}^5 t_{3} t_{1}^5+t_{2}^4 t_{3}^7 t_{1}^4\\
&+2 t_{2}^3 t_{3}^7 t_{1}^4+t_{2} t_{3}^7 t_{1}^4+t_{3}^7 t_{1}^4-t_{2}^6 t_{3}^6 t_{1}^4+t_{2}^5 t_{3}^6 t_{1}^4-3 t_{2}^4 t_{3}^6 t_{1}^4-3 t_{2}^3 t_{3}^6 t_{1}^4+t_{2}^2 t_{3}^6 t_{1}^4+t_{2} t_{3}^6 t_{1}^4-t_{3}^6 t_{1}^4\\
&+2 t_{2}^7 t_{3}^5 t_{1}^4-4 t_{2}^5 t_{3}^5 t_{1}^4+2 t_{2}^3 t_{3}^5 t_{1}^4-3 t_{2}^2 t_{3}^5 t_{1}^4-2 t_{2} t_{3}^5 t_{1}^4+t_{2}^8 t_{3}^4 t_{1}^4-t_{2}^7 t_{3}^4 t_{1}^4+2 t_{2}^6 t_{3}^4 t_{1}^4+8 t_{2}^5 t_{3}^4 t_{1}^4+t_{2}^4 t_{3}^4 t_{1}^4\\
&-2 t_{2}^3 t_{3}^4 t_{1}^4+t_{2}^2 t_{3}^4 t_{1}^4+t_{2} t_{3}^4 t_{1}^4-t_{2}^9 t_{3}^3 t_{1}^4-2 t_{2}^8 t_{3}^3 t_{1}^4+t_{2}^7 t_{3}^3 t_{1}^4-4 t_{2}^5 t_{3}^3 t_{1}^4-t_{2}^4 t_{3}^3 t_{1}^4+3 t_{2}^3 t_{3}^3 t_{1}^4-2 t_{2}^7 t_{3}^2 t_{1}^4\\
&-3 t_{2}^6 t_{3}^2 t_{1}^4+t_{2}^5 t_{3}^2 t_{1}^4-t_{2}^3 t_{3}^2 t_{1}^4+t_{2}^9 t_{3} t_{1}^4+t_{2}^7 t_{3} t_{1}^4+t_{2}^6 t_{3} t_{1}^4-t_{2} t_{3}^8 t_{1}^3+t_{2}^5 t_{3}^7 t_{1}^3+2 t_{2}^4 t_{3}^7 t_{1}^3+t_{2}^2 t_{3}^7 t_{1}^3\\
&+t_{2} t_{3}^7 t_{1}^3+t_{3}^7 t_{1}^3-3 t_{2}^5 t_{3}^6 t_{1}^3-3 t_{2}^4 t_{3}^6 t_{1}^3+t_{2}^3 t_{3}^6 t_{1}^3+2 t_{2}^2 t_{3}^6 t_{1}^3+t_{2}^8 t_{3}^5 t_{1}^3-2 t_{2}^6 t_{3}^5 t_{1}^3+2 t_{2}^4 t_{3}^5 t_{1}^3-3 t_{2}^3 t_{3}^5 t_{1}^3\\
&-3 t_{2}^2 t_{3}^5 t_{1}^3+t_{2}^7 t_{3}^4 t_{1}^3+6 t_{2}^6 t_{3}^4 t_{1}^3+2 t_{2}^5 t_{3}^4 t_{1}^3-2 t_{2}^4 t_{3}^4 t_{1}^3+t_{2}^3 t_{3}^4 t_{1}^3+t_{2}^2 t_{3}^4 t_{1}^3-t_{2}^9 t_{3}^3 t_{1}^3-2 t_{2}^6 t_{3}^3 t_{1}^3+3 t_{2}^4 t_{3}^3 t_{1}^3\\
&-t_{2}^8 t_{3}^2 t_{1}^3-2 t_{2}^7 t_{3}^2 t_{1}^3-t_{2}^5 t_{3}^2 t_{1}^3-t_{2}^4 t_{3}^2 t_{1}^3-t_{2}^3 t_{3}^2 t_{1}^3-t_{2}^2 t_{3}^8 t_{1}^2-t_{2} t_{3}^8 t_{1}^2+t_{2}^5 t_{3}^7 t_{1}^2+t_{2}^3 t_{3}^7 t_{1}^2+t_{2}^2 t_{3}^7 t_{1}^2\\
&+t_{2} t_{3}^7 t_{1}^2-2 t_{2}^6 t_{3}^6 t_{1}^2-2 t_{2}^5 t_{3}^6 t_{1}^2+t_{2}^4 t_{3}^6 t_{1}^2+2 t_{2}^3 t_{3}^6 t_{1}^2-t_{2}^2 t_{3}^6 t_{1}^2-t_{2}^7 t_{3}^5 t_{1}^2+t_{2}^5 t_{3}^5 t_{1}^2-3 t_{2}^4 t_{3}^5 t_{1}^2-3 t_{2}^3 t_{3}^5 t_{1}^2\\
&+t_{2}^8 t_{3}^4 t_{1}^2+3 t_{2}^7 t_{3}^4 t_{1}^2+t_{2}^6 t_{3}^4 t_{1}^2-t_{2}^5 t_{3}^4 t_{1}^2+t_{2}^4 t_{3}^4 t_{1}^2+t_{2}^3 t_{3}^4 t_{1}^2-t_{2}^7 t_{3}^3 t_{1}^2+2 t_{2}^5 t_{3}^3 t_{1}^2-t_{2}^8 t_{3}^2 t_{1}^2-t_{2}^3 t_{3}^8 t_{1}\\
&-t_{2}^2 t_{3}^8 t_{1}-t_{2} t_{3}^8 t_{1}+t_{2}^4 t_{3}^7 t_{1}+t_{2}^3 t_{3}^7 t_{1}+t_{2}^2 t_{3}^7 t_{1}-t_{2}^7 t_{3}^6 t_{1}-t_{2}^6 t_{3}^6 t_{1}+t_{2}^4 t_{3}^6 t_{1}-2 t_{2}^5 t_{3}^5 t_{1}-2 t_{2}^4 t_{3}^5 t_{1}\\
&\left.\left.+t_{2}^8 t_{3}^4 t_{1}+t_{2}^7 t_{3}^4 t_{1}+t_{2}^5 t_{3}^4 t_{1}+t_{2}^4 t_{3}^4 t_{1}+t_{2}^6 t_{3}^3 t_{1}+t_{2}^5 t_{3}^7+t_{2}^4 t_{3}^7+t_{2}^3 t_{3}^7-t_{1}^8 t_{2}^6-t_{2}^4 t_{3}^6-t_{2}^6 t_{3}^5-t_{2}^5 t_{3}^5\right)\right)\\
\Big/\big((t_{1}& -1)^3 \left(t_{1}^2-t_{2}\right)^2 (t_{2}-1)^3 \left(t_{1}-t_{2}^2\right)^2 \left(t_{1}^3-t_{2}^2\right) \left(t_{1}^2-t_{2}^3\right) (t_{1}-t_{3})^3 \left(t_{1}^3-t_{3}\right) (t_{2}-t_{3})^3\\
\hphantom{\Big(}\left(t_{1}^2\right. &\left. \left. t_{2}-t_{3}\right) \left(t_{1} t_{2}^2-t_{3}\right) \left(t_{2}^3-t_{3}\right) (t_{3}-1)^2 (t_{1}+t_{3}) (t_{2}+t_{3}) \left(t_{1} t_{3}-t_{2}^2\right) \left(t_{1}^2-t_{2} t_{3}\right) \left(t_{1} t_{2}-t_{3}^2\right)\right).
    \end{align*}
\end{proposition}

\subsection{The localization computation}

\begin{conjecture}\label{conj-eularchar-global}
Let $X$ be a smooth projective scheme. For line bundles $K,L$ on $X$,
\begin{eqnarray}\label{eq-eularchar-global}
1+\sum_{n=1}^{\infty}\chi\left(\Lambda_{-v}K^{[n]},\Lambda_{-u}L^{[n]}\right)Q^n=
\exp \left(\sum_{n=1}^{\infty}\chi\left(\Lambda_{-v^n}K, \Lambda_{-u^n}L\right)\frac{Q^n}{n}\right).
\end{eqnarray}
In particular,
\[
\sum_{n=1}^{\infty}\chi\left(L^{[n]}\right)Q^n=(1-Q)^{-\chi(\mathcal{O}_X)}\chi(L)Q.
\]
\end{conjecture}

\begin{conjecture}\label{conj-eularchar-local}
Denote by $\mathscr{P}_r$ the set of $r$-dimensional partitions $($an empty partition is allowed$)$. For  $\lambda\in \mathscr{P}_r$, recall that $A_{\lambda}$ is the coordinate ring of the Haiman neighborhood and $H(A_{\lambda};\theta_1,\dots,\theta_r)$ is the equivariant Hilbert function of $A_{\lambda}$. 
Then for $r\geq 2$,
\begin{eqnarray}\label{eq-sum-1}
&&\sum_{\lambda\in \mathscr{P}_r}\left(Q^{|\lambda|} H(A_{\lambda};\theta_1,\dots,\theta_r)
\prod_{\mathbf{i}=(i_1,\dots,i_r)\in \lambda}\left(1-u \theta_1^{i_1}\cdots\theta_r^{i_r}\right)\left(1-v \theta_1^{-i_1}\cdots\theta_r^{-i_r}\right)\right)\notag\\
&&=\exp\left(\sum_{n=1}^{\infty}\frac{(1-u^n )(1-v^n)Q^n}{n(1-\theta_1^n)\cdots(1-\theta_r^n)}\right).
\end{eqnarray}
\end{conjecture}

\begin{remark}
Conjecture~\ref{conj-eularchar-local} is a refinement of \cite[Conjecture 3]{WZ14}. Note that when $r=1$, (\ref{eq-sum-1}) is not true, but it is still true  if $v=0$.
\end{remark}

\begin{proposition}\label{prop-local-to-global}
Let $X$ be a smooth proper toric variety over $\Bbbk$ of dimension $r$. Let $T=\mathbb{G}_m^r$ be the open dense torus contained in $X$. Let $K$, $L$ be two $T$-line bundles on $X$. Then Conjecture~\ref{conj-eularchar-local} implies Conjecture~\ref{conj-eularchar-global} for $X$, $K$, and $L$. More precisely, if  Conjecture~\ref{conj-eularchar-local} holds modulo $Q^s$ for some $s>0$, Conjecture~\ref{conj-eularchar-global}  modulo $Q^s$ also holds for such triples $(X,K,L)$.
\end{proposition}

\begin{proof}
The $T$-action on $X$ has only isolated fixed points, denoted by $x_1,\dots,x_m$. For $1\leq b\leq m$, let $\mathbf{w}_{b,1},\dots,\mathbf{w}_{b,r}$ be the weights of the cotangent space of $X$ at $x_b$. Let  $\mathbf{k}_b$ (resp.~$\mathbf{l}_b$) be the weight of $K$ (resp.~$L$) at $x_b$. For $n\geq 1$, let $\alpha_n\colon T\rightarrow T$ be the homomorphism $\mathbf{t}=(t_1,\dots,t_r)\mapsto \mathbf{t}^n=(t_1^n,\dots,t_r^n)$. Recall our notation introduced before Example~\ref{example-Thomason-localization}. For each $n\geq 1$, applying (\ref{eq-localizationtheorem-withoutembedding-2}) to $X$, with the action on $X$, $K$, and $L$ by $T$ precomposed with $\alpha_n$,  we obtain
\begin{equation}\label{eq-local-to-global-1}
	\chi\left(\Lambda_{-v^n}K, \Lambda_{-u^n}L\right)
=\sum_{b=1}^m \frac{\left(1-u^n \mathbf{t}^{-n\mathbf{k_b}}\right)\left(1-v^n \mathbf{t}^{n \mathbf{l}_b}\right)}{\prod_{j=1}^r\left(1- \mathbf{t}^{n w_{b,j}}\right)}.
\end{equation}
Summing over $n$, we get an equality of series of virtual $T$-representations
\begin{equation}\label{eq-local-to-global-2}
	\exp \left(\sum_{n=1}^{\infty}\chi(\Lambda_{-v^n}K, \Lambda_{-u^n}L)\frac{Q^n}{n}\right)= \prod_{i=1}^m \exp\left(\frac{\left(1-u^n \mathbf{t}^{-n\mathbf{k_i}}\right)\left(1-v^n \mathbf{t}^{n \mathbf{l}_i}\right)}{n\prod_{j=1}^r\left(1- \mathbf{t}^{n w_{i,j}}\right)}Q^n\right).
\end{equation}
Now assume Conjecture~\ref{conj-eularchar-local} is valid. Replacing $u$  in (\ref{eq-sum-1}) by $u \mathbf{t}^{-\mathbf{k}_b}$, $v$ by $v \mathbf{t}^{\mathbf{l}_b}$, and $\theta_j$ by $\mathbf{t}^{\mathbf{w}_{b,j}}$ for $1\leq j\leq r$, we have
\begin{equation}\label{eq-local-to-global-3}
	  \exp\left(\frac{\left(1-u^n \mathbf{t}^{-n\mathbf{k_i}}\right)\left(1-v^n \mathbf{t}^{n \mathbf{l}_i}\right)}{n\prod_{j=1}^r\left(1- \mathbf{t}^{n w_{b,j}}\right)}Q^n\right)	
= \sum_{\lambda\in \mathscr{P}_r}\left(Q^{|\lambda|}H\left(A_{\lambda};\mathbf{t}^{w_{b,1}},\dots,\mathbf{t}^{w_{b,r}}\right)\prod_{\mathbf{i}=(i_1,\dots,i_r)\in \lambda}\left(1-u \mathbf{t}^{-\mathbf{k}_b}\cdot \prod_{j=1}^r \mathbf{t}^{i_j\cdot w_{b,j}}\right)\right).
\end{equation}
Consider the $T$-action on $\Hilb^n(X)$ induced by the action on $X$ (not involving $\alpha_n$). The equivariant local structure of $\Hilb^n(X)$ at a fixed subscheme $Z$ of length $l$  supported at $x_b$ is isomorphic to the equivariant local structure of $\Hilb^l(\mathbb{A}^r)$ at $Z_{I}$ supported at $0$, where $I$ is a monomial ideal, and where $T$ acts on $\mathbb{A}$ by weights $\mathbf{w}_{b,1},\dots,\mathbf{w}_{b,r}$ (see the proof of Proposition~\ref{prop-equivariant-immersion-toric}).  By (\ref{eq-localizationtheorem-withoutembedding-2}) or (\ref{eq-Thomason-reducedisolatefixed-2}), we obtain an equality in $R(T)$
\begin{equation}\label{eq-local-to-global-4}
\chi\left(\Lambda_{-v}K^{[n]},\Lambda_{-u}L^{[n]}\right)
= \sum_{\begin{subarray}{c}\lambda_1,\dots,\lambda_m\in \mathscr{P}_r \\  |\lambda_1|+\dots+|\lambda_m|=r
\end{subarray}} \sum_{b=1}^m \left(
H\left(A_{\lambda};\mathbf{t}^{w_{b,1}},\dots,\mathbf{t}^{w_{b,r}}\right)\prod_{\mathbf{i}=(i_1,\dots,i_r)\in \lambda_b}\left(1-u \mathbf{t}^{-\mathbf{k}_b}\cdot \prod_{j=1}^r \mathbf{t}^{i_j\cdot w_{b,j}} \right)\right).
\end{equation} 

Summing over $n\geq 0$, and combining (\ref{eq-local-to-global-2}), (\ref{eq-local-to-global-3}), and (\ref{eq-local-to-global-4}), we obtain the validity of (\ref{eq-eularchar-global}) for the equivariant triple $(X,K,L)$, as an equality of series of virtual representations of $T$. Taking the limit $\mathbf{t}\rightarrow 1$, we complete the proof of the first statement. The second statement follows by ignoring the higher-order terms of $Q$ in the above proof.
\end{proof}

\begin{proposition}\label{prop-verify-local}
 Conjecture~\ref{conj-eularchar-local} modulo $Q^7$  holds for smooth proper toric 3-folds $X$ and equivariant line bundles $K,L$ on $X$. Assume that  Conjecture~\ref{conj-existence-equivariant-iso-exdim6} is true. Then Conjecture~\ref{conj-eularchar-local} modulo $Q^8$  holds for smooth proper toric 3-folds $X$ and equivariant line bundles $K,L$ on $X$.
\end{proposition}

\begin{proof}
The computations of equivariant Hilbert functions $H(A_{\lambda};\mathbf{t})$ at singular points are done in Corollary~\ref{cor-tripod-Hilbertfunctions}, Proposition~\ref{prop-Hilbertfunction-1321}, and Appendices~\ref{sec:appendix-changevariable-Borelideals} and~\ref{sec:appendix-changevariable-NonBorelideals}. The contribution of smooth points is computed using the description of the cotangent spaces in Corollary~\ref{cor-description-cotangentspace}, and its implementation in Macaulay2. Then we verify (\ref{eq-sum-1}) by  brute force.\footnote{The verification using Mathematica is given in the ancillary files. See also
\url{https://github.com/huxw06/Hilbert-scheme-of-points}.}
\end{proof}

\begin{remark}
The right-hand side of (\ref{eq-sum-1}) is manifestly symmetric in $u$ and $v$, while the left-hand side seems not to be. Replacing $v$ by $v^{-1}$, and $Q$ by $v Q$,  we obtain a formula equivalent to (\ref{eq-sum-1}), which is symmetric in $u$ and $v$: 
\begin{eqnarray}\label{eq-sum-1'}
&&\sum_{\lambda}\left((-1)^{|\lambda|}Q^{|\lambda|}H(A_{\lambda};\theta_1,\dots,\theta_r)\prod_{i\in \lambda}\theta^{-i}\cdot
\prod_{(i_1,\dots,i_r)\in \lambda}\left(1-u \theta_1^{i_1}\cdots\theta_r^{i_r}\right)\left(1-v \theta_1^{i_1}\cdots\theta_r^{i_r}\right)\right)\notag\\
&&=\exp\left(-\sum_{n=1}^{\infty}\frac{(1-u^n )(1-v^n)Q^n}{n(1-\theta_1^n)\cdots(1-\theta_r^n)}\right).
\end{eqnarray}
\end{remark}

\begin{remark}
It is very desirable to have an at least conjectural formula for every single $H(A_{\lambda};\mathbf{t})$. Then it should be possible to show (\ref{eq-sum-1}), or (\ref{eq-sum-1'}),  as a combinatorial identity and thus avoid the cumbersome verifications  in the final step of the proof of  Proposition~\ref{prop-verify-local}, once and for all.
\end{remark}

\begin{remark}\label{rem-reciprocal-law}
The equivariant Hilbert functions $H(A_{\lambda};\theta_1,\dots,\theta_r)$ in Corollary~\ref{cor-tripod-Hilbertfunctions} and Proposition~\ref{prop-Hilbertfunction-1321} satisfy the \emph{self-reciprocal law}
\begin{equation}\label{eq-reciprocity-1}
H(A_{\lambda};\theta_1,\dots,\theta_r)=(-1)^{|\lambda|}(\theta_1\cdots \theta_r)^{-|\lambda|}
\prod_{(i_1,\dots,i_r)\in \lambda}\theta_1^{i_1}\cdots \theta_r^{i_r}\cdot
H\left(A_{\lambda};\theta_1^{-1},\dots,\theta_r^{-1}\right).
\end{equation}
This is related to the Gorenstein property by Theorem~\ref{thm-Stanley-Gorenstein}.
\end{remark}

By Propositions~\ref{prop-local-to-global} and~\ref{prop-verify-local}, we obtain the following. 

\begin{corollary}
 Conjecture~\ref{conj-eularchar-global} modulo $Q^7$ holds for smooth proper toric 3-folds $X$ and $T$-line bundles $K,L$ on $X$, where $T$ is the dense open torus in $X$.
Assume that  Conjecture~\ref{conj-existence-equivariant-iso-exdim6} is true. Then Conjecture~\ref{conj-eularchar-global} modulo $Q^8$  holds for smooth proper toric 3-folds $X$ and equivariant line bundles $K,L$ on $X$.
\end{corollary}

\begin{remark}
In \cite{Hu21}, we show that Conjecture~\ref{conj-eularchar-global} can be reduced to the cases where $X$ is a product of projective spaces and $K$, $L$ are exterior tensor products of the line bundles of the form $\mathcal{O}(k)$. Thus it follows that Conjecture~\ref{conj-eularchar-global} modulo  $Q^7$ holds for all smooth proper 3-fold. 
\end{remark}

\subsection{A McKay correspondence}\label{sec:McKay-correspondence}
Denote the symmetric group of cardinality $n!$ by $\mathfrak{S}_n$.
Denote by $X^{(n)}=X^n/\mathfrak{S}_n$ the $\supth{n}$ symmetric product of $X$. There is the Hilbert--Chow morphism $\rho\colon X^{[n]}\rightarrow X^{(n)}$ (see \cite[Section~2.2]{Ber12}); set-theoretically, it sends a $0$-dimensional subscheme to the associated 0-cycle.

\begin{conjecture}[K-theoretical pushforward]\label{conj-McKay-1}
We have $\rho_* \mathcal{O}_{X^{[n]}}=\mathcal{O}_{X^{(n)}}$ in $K_0(X^{(n)})$.
\end{conjecture}

For a smooth projective surface, this conjecture follows from \cite[Propositions~1.3.2 and~1.3.3]{Sca09}.

Let $L$ be a line bundle on $X$. Denote by $p_i\colon  X^n\rightarrow X$ the $\supth{i}$ projection. Let $L_i=p_i^* L$. Consider the quotient stack $[X/\mathfrak{S}_n]$. There is an obvious action of $\mathfrak{S}_n$ on the vector bundle $\oplus_{i=1}^n L_i$, rendering $\oplus_{i=1}^n L_i\rightarrow X^n$ equivariant, which gives a vector bundle 
$[\oplus_{i=1}^n L_i/\mathfrak{S}_n]$  on  $[X^n/\mathfrak{S}_n]$. Denote by $\pi\colon [X/\mathfrak{S}_n]\rightarrow X^{(n)} $ the projection to the coarse moduli space. Thus $\pi_* \mathcal{O}_{[X^n/\mathfrak{S}_n]}=\mathcal{O}_{ X^{(n)}}$. 

\begin{conjecture}\label{conj-McKay-2}
Suppose $\dim X\geq 2$.
Let $K,L$ be line bundles on $X$. For the vector bundles  $U=\oplus_{i=1}^n p_i^* K$ and $V=\oplus_{i=1}^n p_i^* L$ on $X^n$, we have
the K-theoretical pushforward
\[
\rho_* \left(\Lambda_{-u}\left(\left(K^{[n]}\right)^*\right)\otimes\Lambda_{-v}\left(L^{[n]}\right)\right)\cong \pi_* [(\Lambda_{-u}U^*\otimes \Lambda_{-v}V)/\mathfrak{S}_n]. 
\]
\end{conjecture}

For $u=0$ and $\dim X=2$, this conjecture follows from \cite[Proposition~2.4.5]{Sca09}. We regard this conjecture as a sort of McKay correspondence: 
\[
\xymatrix{
	& [X^n/\mathfrak{S}_n] \ar[d]^{\pi} \\
X^{[n]} \ar[r]^{\rho} & X^{(n)}\rlap{.}	
}
\]
Although $X^{[n]}$ is not smooth in general, it is a modular substitute of the quotient scheme $X^{(n)}$.  In the remainder of this subsection, we show Theorem~\ref{thm-McKaycorrespondence-symmetricrpoduct}, from which one sees that Conjecture~\ref{conj-McKay-2} implies Conjecture~\ref{conj-eularchar-global}.

\begin{remark}\label{rem:proof_from_Krug_paper}
As pointed out by a referee, Theorem~\ref{thm-McKaycorrespondence-symmetricrpoduct} follows directly from \cite[Proposition 4.1]{Kru18}.
In fact, the Euler characteristic of $\pi_* [(\wedge^e U^*\otimes \wedge^fV)/\mathfrak{S}_n]$ can be computed as the Euler characteristic of  the $\mathfrak{S}_n$-invariant Ext-space $\Ext_{\mathfrak{S}_n}^*(\wedge^e U, \wedge^fV)$, and the proof of \cite[Proposition 4.1]{Kru18} works in all dimensions. The related combinatorial identities are given in \cite[Appendix]{Kru18}. 
\end{remark}

We give another proof of Theorem~\ref{thm-McKaycorrespondence-symmetricrpoduct} using the Riemann--Roch theorem for smooth Deligne--Mumford quotient stacks (see \cite{EG05}).  Let us recall some notation. We follow the presentation of \cite{Edi13}.

Let $Y$ be a smooth projective scheme over $\Bbbk$ of characteristic zero, with an action by a finite group $G$. Let $\Psi_1,\dots,\Psi_m$ be the conjugacy classes of $G$. Choose a representative $g_k\in \Psi_i$ for each $1\leq k\leq m$. For any element $g\in G$, let $Z_g$ be the centralizer of $g$ in $G$, and let $H_g$ be the subgroup generated by $g$. Let $Y^{g_k}$ be the fixed subscheme of $Y$ and $\iota_k\colon Y^{g_k}\hookrightarrow Y$ be the closed immersion. Since the subgroup $H_{g_k}$ is diagonalizable, $Y^{g_k}$ is regular. Let $N_{\iota_k}$ be the normal bundle of $Y^{g_k}$ in $Y$.

Let $E$ be a $G$-equivariant vector bundle on $Y$. Now assume that $\Bbbk$ is algebraically closed. The restriction of $E$ on $Y^{g_k}$ decomposes into a sum of $g_k$-eigenbundles $\bigoplus_{\xi\in \mathbb{X}(H_{g_k})}E_{\xi}$, where $\mathbb{X}(H)$ is the group of characters of an abelian group $H$. Define 
\begin{equation}\label{eq-definition-twisting-operator}
t_{h_k}(E)=\sum_{\xi\in \mathbb{X}(H_{g_k})}\xi(g_k)E_{\xi}. 
\end{equation}
Here $t$ means \emph{twisting} and does not indicate any relation to a torus.
Then the Riemann--Roch theorem for the Deligne--Mumford stack $\mathcal{Y}=[Y/G]$ says
\begin{equation}\label{eq-RiemannRoch}
	\chi(\mathcal{Y},E)=\sum_{k=1}^{m}\int_{[Y^{g_k}/Z_{g_k}]}\ch\left(
	t_{g_k}\left(
	\frac{{\iota_k}^* E}{\lambda_{-1}(N^*_{\iota_k})}	\right)
	\right)\Td\left(\left[Y^{g_k}/Z_{g_k}\right]\right).
\end{equation}

\begin{remark}
In \cite{EG05}, the base field is assumed to be $\mathbb{C}$, and the group $G$ is a general algebraic group. In our case, $G$ is a finite abstract group, so the formula is simpler. And in this case, the assumption that the base field is algebraically closed is used only in the decomposition into eigenbundles. The characteristic zero assumption is essential for the sufficiently higher cohomology to vanish so that the left-hand side of (\ref{eq-RiemannRoch}) can be defined.
\end{remark}

For a vector bundle $F$, by the splitting principle, we may assume 
$F=\oplus_{i=1}^{r}W_i$, where the $W_i$ are line bundles, and define
\[
\Upsilon_m(F)=\prod_{i=1}^r\frac{e^{-mc_1(W_i)}-1}{e^{-c_1(W_i)}-1}.
\]

\begin{lemma}\label{lem-Upsilon-class}
\begin{i-enumerate}
	\item\label{lUc-1} Let $\zeta_m=e^{\frac{2\pi \sqrt{-1}}{m}}$. Then  
\begin{equation}\label{eq-Upsilon-class-1}
\Upsilon_m(F)=\prod_{k=1}^{m-1}\sum_{i=1}^r(-1)^i \zeta_m^k \ch(\wedge^i (F^{*})).
\end{equation}
	\item\label{lUc-2} Let $T_X$ be the tangent bundle of\, $X$. Then
	\begin{gather}\label{eq-Upsilon-class-2}
	\int_{X}\frac{\ch(L)^{m}\Td(X)}{\Upsilon_{m}(T_X)}
	=\int_X \frac{e^{m c_1(L)}\Td(X)}{\Upsilon_m(T_X)}
	=\chi(L),\\
\label{eq-Upsilon-class-3}
	\int_{X}\frac{\ch(K^*)^m\ch(L)^{m}\Td(X)}{\Upsilon_{m}(T_X)}
	=\int_X \frac{e^{-m c_1(K)}
	e^{m c_1(L)}\Td(X)}{\Upsilon_m(T_X)}
	=\chi(K^*\otimes L).
	\end{gather}
\end{i-enumerate}
\end{lemma}

\begin{proof}
\eqref{lUc-1}~ 
\begin{equation*}
\prod_{k=1}^{m-1}\sum_{i=1}^r(-1)^i \zeta_m^k \ch(\wedge^i (F^{*}))
= \prod_{k=1}^{m-1} \ch\left((1-\zeta_m^k W_1^*)\cdots (1-\zeta_m^k W_r^*)\right)
= \prod_{i=1}^r \frac{\ch(W_i^*)^m-1}{\ch(W_i^*)-1}
=\prod_{i=1}^r\frac{e^{-mc_1(W_i)}-1}{e^{-c_1(W_i)}-1}.
\end{equation*}

\eqref{lUc-2}~ 
Let $x_1,\dots,x_{\dim X}$ be the Chern roots of $TX$. Then
\begin{align*}
\int_X \frac{\ch(L)^{m}\Td(X)}{\Upsilon_m(TX)}
&=\int_{X} e^{m c_1(L)}\prod_{i=1}^{\dim X}\frac{x_i}{1-e^{-mx_i}}
=\frac{1}{m^{\dim X}}\int_{X}e^{m c_1(L)}\prod_{i=1}^{\dim X}\frac{mx_i}{1-e^{-mx_i}}\\
&=\frac{1}{m^{\dim X}}\cdot m^{\dim X}\int_{X}e^{ c_1(L)}\prod_{i=1}^{\dim X}\frac{x_i}{1-e^{-x_i}}
=\int_X \ch(L) \Td(X)=\chi(L).
\end{align*}
The proof of (\ref{eq-Upsilon-class-3}) is similar.
\end{proof}

For a (usual) partition $\mu$, let $z_{\mu}=\prod_{i\geq 1}i^{k_i} k_i!$, where $k_i$ is the number of parts of $\mu$ equal to $i$. The following identity is obtained by directly expanding the right-hand side.

\begin{lemma}\label{lem-exponential-identity}
For indeterminates $Y_1,Y_2,\dots$, we have the identity
\begin{equation}\label{eq-exponential-identity}
1+\sum_{\mu=(m_1,\dots,m_{l})}\frac{1}{z_{\mu}}\left(\prod_{j=1}^l Y_{m_j}Q^{|\mu|}\right)
=\exp\left(\sum_{r=1}^{\infty}Y_r\frac{Q^r}{r}
\right),
\end{equation}
where the sum on the left-hand side runs over all partitions of natural numbers.
\end{lemma}

\begin{theorem}\label{thm-McKaycorrespondence-symmetricrpoduct}
Let $X$ be a smooth projective scheme over a field of characteristic zero. For line bundles $K,L$ on $X$, let $U=\oplus_{i=1}^n p_i^*K$ and $V=\oplus_{i=1}^n p_i^*L$. Then
\begin{eqnarray*}
1+\sum_{n=1}^{\infty}\chi\left(X^{(n)},\pi_* \left[\left(\Lambda_{-u}U^*\otimes \Lambda_{-v}V\right)/\mathfrak{S}_n\right]\right)Q^n=
\exp \left(\sum_{n=1}^{\infty}\chi\left(\Lambda_{-v^n}K, \Lambda_{-u^n}L\right)\frac{Q^n}{n}\right).
\end{eqnarray*}
\end{theorem}

\begin{proof}
The formation of the quotient $X^{(n)}$ commutes with flat base changes, so we can assume  that $\Bbbk$ is algebraically closed.
The conjugacy classes of $\mathfrak{S}_n$ correspond to the partitions of $n$. 
Let $\mu=(m_1,\dots,m_l)$ be a partition of $n$, and also use $\mu$ to represent the conjugacy class corresponding to $\mu$. Let $h_{\mu}$ be the element of~$\mathfrak{S}_n$ which preserves the $l$ parts of $\mu$ and in the $\supth{i}$ part is the addition by $1$ mod $ m_i$. Formally, $h_{\mu}$ is uniquely determined by the following requirement: For $1\leq k\leq l$ and $1\leq i\leq m_k$,
\[
 \sum_{j=1}^{k-1}m_j<
 h_{\mu}\left(i+\sum_{j=1}^{k}m_j\right)\leq \sum_{j=1}^{k}m_j
 \]
 and 
\[
h_{\mu}\left(i+\sum_{j=1}^{k-1}m_j\right)\equiv i+1+\sum_{j=1}^{k-1}m_j \mod (m_k).
\]
  The fixed locus of $h_{\mu}$ is $(X^n)^{h_{\mu}}= (\Delta_{X})_1\times\cdots\times (\Delta_{X})_l$, where $(\Delta_{X})_k\cong X$ is the small diagonal in $X^{m_k}$.  Denote by $\iota_\mu\colon (X^n)^{h_\mu}\hookrightarrow X^n$ the obvious closed immersion. For $1\leq j\leq l$, denote by
\[
q_j\colon (X^n)^{h_{\mu}}=(\Delta_{X})_1\times\cdots\times (\Delta_{X})_l \longrightarrow (\Delta_{X})_j
\]
the $\supth{j}$ projection. 
The order of $h_\mu$ is $M_\mu=\mathrm{gcd}(m_1,\dots,m_l)$.
The center of $h_{\mu}$ is denoted by $Z_{\mu}$, and its order is  $z_{\mu}$. 
The twisting operator (\ref{eq-definition-twisting-operator}) according to $h_{\mu}$ is denoted by $t_{\mu}$.
We have
$i_\mu^* V= \oplus_{j=1}^l q_j^*(L^{\oplus m_j})$, 
\[
t_{\mu}(i_{\mu}^* V)=\sum_{j=1}^l p_j^*\left(\sum_{k=0}^{m_j-1}e^{\frac{2k\pi \sqrt{-1}}{m_j}}L\right), 
\]
and
\[
t_{\mu}\left(\lambda_{-1}\left(N_{X^\mu}^*\right)\right)=\sum_{j=1}^l p_j^*\left(
\prod_{k=1}^{m_j-1}\left(\sum_{i=0}^{\dim X}(-1)^i e^{\frac{2ki\pi \sqrt{-1}}{m_j}} \wedge^i T^{*}_X \right)
\right).
\]
So by (\ref{eq-Upsilon-class-1}), we have
\begin{eqnarray*}
\ch\left(t_{\mu}\left(\frac{i_{\mu}^*(\Lambda_{-u}U^*\otimes\Lambda_{-v}V)}{\lambda_{-1}(N_{X^\mu}^*)}\right)\right)
&=&\frac{\prod_{j=1}^l p_j^*\prod_{k=0}^{m_j-1}\left(1-ue^{\frac{2k\pi \sqrt{-1}}{m_j}}\ch(K^*)\right)\left(1-ve^{\frac{2k\pi \sqrt{-1}}{m_j}}\ch(L)\right)}
{\prod_{j=1}^{l}p_j^* \Upsilon_{m_j}(T_X)
}\\
&=&\frac{\prod_{j=1}^l p_j^*\left(1-u^{m_j}\ch(K^*)^m\right)\left(1-v^{m_j}\ch(L)^m\right)}
{\prod_{j=1}^{l}p_j^* \Upsilon_{m_j}(T_X)
}.
\end{eqnarray*}
Then using (\ref{eq-Upsilon-class-2}) and (\ref{eq-Upsilon-class-3}), we obtain
\begin{align*}
&\int_{[X^{\mu}/Z_{\mu}]}
\ch\left(t_{\mu}\left(\frac{i_{\mu}^*(\Lambda_{-u}U^*\otimes\Lambda_{-v}V)}{\lambda_{-1}(N_{X^\mu}^*)}\right)\right)
\Td([X^{\mu}/Z_{\mu}])\\
&=\frac{1}{z_{\mu}}\prod_{j=1}^l\left(\chi(\mathcal{O}_X)-u^{m_j}\chi(K^*)-v^{m_j}\chi(L)+u^{m_j}v^{m_j}\chi(K^*\otimes L)\right).
\end{align*}
 In (\ref{eq-exponential-identity}), we set 
\[
Y_m=\chi(\mathcal{O}_X)-u^m \chi(K^*)-v^{m}\chi(L)+u^m v^m \chi(K^*\otimes L).
\]
Then we obtain
\begin{align*}\pushQED{\qed}
& 1+\sum_{n=1}^{\infty}\chi\left([(\Lambda_{-u}U^*\otimes \Lambda_{-v}V)/\mathfrak{S}_n]\right)Q^n\\
&=1+\frac{1}{z_{\mu}}\prod_{j=1}^l\left(\chi(\mathcal{O}_X)-u^m \chi(K^*)-v^{m}\chi(L)+u^m v^m \chi(K^*\otimes L)\right)\\
&=\exp\left(\sum_{r=1}^{\infty}
(\chi(\mathcal{O}_X)-u^{r}\chi(K^*)-v^{r}\chi(L)+u^r v^r \chi(K^*\otimes L))\frac{Q^r}{r}
\right)\\
&= \exp \left(\sum_{r=1}^{\infty}\chi(\Lambda_{-u^r}K, \Lambda_{-v^r}L)\frac{Q^r}{r}\right).\qedhere \popQED
\end{align*}
\renewcommand{\qed}{}     
\end{proof}

\section{Local properties of  Hilbert schemes}\label{sec:local-properties-hilb-schemes}
In this section, we study certain properties of the singularities encountered in this paper.

\begin{proposition}\label{prop-local-properties-A121}
The scheme $\Spec(A_{\lambda_{121}})$ is normal and Gorenstein, and it has only rational singularities.
\end{proposition}

\begin{proof} The first two properties follow from the general fact that Grassmannians with respect to Pl\"ucker embeddings are arithmetically normal and arithmetically Gorenstein (see, \textit{e.g.}, \cite[Sections~6.3 and~7.5]{LB15}).
More explicitly, we have a regular sequence $a,c,d,e,g,h,f+m+q,i+p+r,j+l+n$ for the ring (see~\eqref{eq-ideal-J121})
\begin{equation}\label{eq-ring-modulo-J121}
A=\Bbbk[a, c, d, e, f, g, h, i, j,  l, m, n,  p, q, r]/J_{121}.
\end{equation}
The annihilator of the maximal ideal $P=(a, c, d, e, f, g, h, i, j,  l, m, n,  p, q, r)$ in 
\[
\Bbbk[a, c, d, e, f, g, h, i, j,  l, m, n,  p, q, r]/\left(J_{121}+(a,c,d,e,g,h,f+m+q,i+p+r,j+l+n)\right)
\] 
is a principal ideal generated by $r^3$. So by definition, (\ref{eq-ring-modulo-J121}) is Gorenstein.  By Serre's criterion, Cohen--Macaulay plus regularity in codimension $1$ implies normal. So $A_{\lambda_{121}}$ is normal because it has an isolated singularity at~$P$.

For the third property, when we blow up the cone (\ref{eq-ring-modulo-J121})  at the origin, the exceptional divisor $E$ is isomorphic to $G(2,6)$, with  normal sheaf $\mathcal{O}(-1)$. Denote the total space of the normal bundle by $N$. We only need  to show $H^i(N,\mathcal{O}_N)=0$ for $i>0$. Since $\pi_* \mathcal{O}_N=\Sym^\bullet(\mathcal{N}^{\vee})$, where $\pi\colon N\rightarrow E$ is the projection, we are left to show 
\[
H^i(G(2,6),\mathcal{O}(j))=0
\]
for $i,j>0$. This follows from the Kodaira vanishing theorem for the canonical bundle $K_{G(m,n)}\cong \mathcal{O}(-n)$. 
\end{proof}

In the rest of this section, we study $\Spec(A_{\lambda_{1321}})$. Look at the superpotential
\begin{eqnarray*}
F_{1321}&=&-{y}_{5}{y}_{8}{y}_{10}{y}_{14}+{y}_{6}{y}_{7}{y}_{14}{y}_{15}+{y}_{8}{y}_{9}{y}_{10}{y}_{17}-{y}_{2}{y}_{7}{
     y}_{15}{y}_{17}+{y}_{4}{y}_{10}{y}_{15}{y}_{17}-{y}_{6}{y}_{8}{y}_{14}{y}_{18}\\
  &&   +{y}_{2}{y}_{8}{y}_{17}{y}_{18
     }-{y}_{10}{y}_{14}{y}_{15}{y}_{19}+{y}_{2}{y}_{5}{y}_{7}{y}_{21}-{y}_{6}{y}_{7}{y}_{9}{y}_{21}-{y}_{4}{y}_{5}{y}_{10}{y}_{21}-{y}_{4}{y}_{6}{y}_{18}{y}_{21}\\
  &&   +{y}_{9}{y}_{10}{y}_{19}{y}_{21}+{y}_{2}{y}_{18}{y}_{19}{y}_{21
     }+{y}_{2}{y}_{5}{y}_{8}{y}_{24}-{y}_{6}{y}_{8}{y}_{9}{y}_{24}-{y}_{4}{y}_{6}{y}_{15}{y}_{24}+{y}_{2}{y}_{15}{y}_{19}{y}_{24}\\
  &&   -{y}_{5}{y}_{7}{y}_{14}{y}_{27}+{y}_{7}{y}_{9}{y}_{17}{y}_{27}+{y}_{4}{y}_{17}{y}_{18}{y}_{27}-{y}_{14}{y}_{18}{y}_{19}{y}_{27}+{y}_{4}{y}_{5}{y}_{24}{y}_{27}-{y}_{9}{y}_{19}{y}_{24}{y}_{27}\\
  &&  -{y}_{11}{y}_{13}{y}_{14}+{y}_{1}{y}_{7}{y}_{16}+{y}_{4}{y}_{11}{y}_{16}+{y}_{8}{y}_{12}{y}_{16}-{y}_{3}{y}_{13}{y}_{17}+{y}_{3}{y}_{16}{y}_{19}-{y}_{3}{y}_{6}{y}_{20}\\
  &&   -{y}_{1}{y}_{10}{y}_{20}-{y}_{2}{y}_{11}{y}_{20}+{y}_{12}{y}_{13}{y
     }_{21}-{y}_{1}{y}_{13}{y}_{24}+{y}_{12}{y}_{20}{y}_{27}-{y}_{3}{y}_{5}{y}_{29}-{y}_{9}{y}_{11}{y}_{29}\\
  &&   +{y}_{12}{y}_{15}{y}_{29}+{y}_{1}{y}_{18}{y}_{29}.
\end{eqnarray*}
The variables $y_{22},y_{23},y_{25},y_{26},y_{28}$ do not appear in the above expression.
It has the $\mathbb{Z}/2$ symmetry
\begin{align}\label{eq-symmetry-F1321}
&F_{1321}({y}_{1},{y}_{2},{y}_{3},{y}_{4},{y}_{5},{y}_{6},{y}_{7},{y}_{8},{y}_{9},{y}_{10},{y}_{11},{y}_{12},{y
     }_{13},{y}_{14},{y}_{15},\notag\\
& \hphantom{ F_{1321}(}\,  {y}_{16},{y}_{17},{y}_{18},{y}_{19},{y}_{20},{y}_{21},{y}_{22},{y}_{23},{y}_{24},{y
     }_{25},{y}_{26},{y}_{27},{y}_{28},{y}_{29})\notag\\
&=\ F_{1321}({y}_{1},{-{y}_{9}},{y}_{12},{y}_{14},{y}_{27},{y}_{15},{y}_{24},{y}_{17},{-{y}_{2}},{y}_{18},{y}_{11},{
     y}_{3},{y}_{16},{y}_{4},{y}_{6},\\
 &\hphantom{=\ F_{1321}(}\;   {y}_{13},{y}_{8},{y}_{10},{-{y}_{21}},{y}_{29},{-{y}_{19}},{y}_{22},{y}_{23},{
     y}_{7},{y}_{25},{y}_{26},{y}_{5},{y}_{28},{y}_{20}).\notag
\end{align}
In this section, we set $R=\mathbb{Q}[y_1,\dots,y_{29}]/\Jac(F_{1321})$  and $S=R/(y_{22},y_{23},y_{25},y_{26},y_{28})$.

\begin{proposition}\label{prop-jacF1321-isolated}
Away from $0$, $\Spec(S)$ has extra dimension equal to $0$ or $6$ and has at most rational singularities. More precisely, every point away from $0$ is a smooth point or has an open neighborhood that is isomorphic to an open subset of a trivial affine bundle over the cone  $\widehat{G}(2,6)$.
\end{proposition}

\begin{proof}
We use the notation of Section~\ref{sec:3D-1321}.  Let $I$ be an ideal of $\Bbbk[X_1,X_2,X_3]$ such that $Z_I$ lies in the Haiman neighborhood $\Spec(A_{\lambda_{1321}})$. By Theorem~\ref{thm-fundamental-haiman}\eqref{tfh-3}, this implies that 
\begin{equation}
	1,X_1,X_1^2,X_2,X_1 X_2,X_2^2,X_3
\end{equation}
form a basis of  $\Bbbk[X_1,X_2,X_3]/I$. Suppose $Z_I$ lies the open subset $\{x_4\neq 0\}$. Recall $x_4=c_{2,0,0}^{0,1,1}$. Then 
\begin{equation}
	1,X_1,X_2,X_3, X_1 X_2, X_2^2, X_2 X_3
\end{equation}
form a basis of  $\Bbbk[X_1,X_2,X_3]/I$. By Theorem~\ref{thm-fundamental-haiman}\eqref{tfh-3}, this implies that $I$ lies in the Haiman neighborhood of the ideal $J=(X_1^2, X_1 X_2^2, X_1 X_3, X_2^3, X_2^2 X_3, X_3^2)$. After the permutation $X_1\leftrightarrow X_2$, $J$ is transformed into the ideal $I_{\lambda_{232}}$ in Appendix~\ref{sec:3D-232}. The ideal $I_{\lambda_{232}}$ is a Borel tripod ideal with extra dimension $6$. By the result of Appendix~\ref{sec:3D-232} and Proposition~\ref{prop-local-properties-A121}, the Hilbert scheme $\Hilb^7(\mathbb{A}^3)$ is smooth or has a rational singularity at~$Z_I$. 

Similarly, we have
\begin{subequations}\label{eq-prop-jacF1321-isolated-HaimanNeighborhoods}
\begin{eqnarray}
&&\{x_8=c_{2,0,0}^{0,3,0}\neq 0\} \subset 
\Spec(A_{\lambda_{(X_1^2,X_1 X_2^2,X_1 X_3,X_2^4,X_2 X_3,X_3^2)}})\cong \Spec(A_{\lambda_{142}}),\notag\\
&&\{x_{13}=c_{0,0,1}^{0,3,0}\neq 0\} \subset 
\Spec(A_{\lambda_{(X_1^3,X_1^2 X_2,X_1 X_2^2,X_2^4,X_3)}}),\label{eq-prop-jacF1321-isolated-HaimanNeighborhoods-a}\\
&& \{x_{14}=c_{0,2,0}^{1,0,1}\neq 0\} \subset \Spec(A_{\lambda_{232}}),\notag\\
&& \{x_{16}=c_{0, 0, 1}^{3, 0, 0}\neq 0\} \subset 
\Spec(A_{\lambda_{(X_1^4,X_1^2 X_2,X_1 X_2^2,X_2^3,X_3)}}), \label{eq-prop-jacF1321-isolated-HaimanNeighborhoods-b}\\
&& \{x_{17}=c_{0, 2, 0}^{3, 0, 0}\neq 0\} \subset \Spec(A_{\lambda_{142}}),\notag\\
&&\{x_{20}=c_{0, 0, 1}^{1, 2, 0}\neq 0\} \subset \Spec(A_{\lambda_{(X_1^3,X_1^2 X_2,X_2^3,X_3)}}),\label{eq-prop-jacF1321-isolated-HaimanNeighborhoods-c}\\
&& \{x_{29}=c_{0, 0, 1}^{2, 1, 0}\neq 0\} \subset \Spec(A_{\lambda_{(X_1^3,X_1 X_2^2,X_2^3,X_3)}}).\label{eq-prop-jacF1321-isolated-HaimanNeighborhoods-d}
\end{eqnarray}
\end{subequations}
Thus by the results of Appendices~\ref{sec:3D-232} and~\ref{sec:3D-142}, when one of the coordinates $x_8,x_{14},x_{17}$ is not zero, the point is smooth or a rational singularity. By Lemma~\ref{lem-lower-dim-partition}, each of the  Haiman neighborhoods (\ref{eq-prop-jacF1321-isolated-HaimanNeighborhoods-a})--(\ref{eq-prop-jacF1321-isolated-HaimanNeighborhoods-d}) is isomorphic to the product of $\mathbb{A}^7$ with a Haiman neighborhood in $\Hilb^7(\mathbb{A}^2)$. 
Thus when one of the coordinates $x_{13},x_{16},x_{20},x_{29}$ is not zero, the point is smooth. 

Note that on $S$ we have $y_{22}=y_{23}=y_{25}=y_{26}=y_{28}=0$, which is equivalent to $x_{22}=x_{23}=x_{25}=x_{26}=x_{28}=0$. It remains to show that if one of the coordinates 
$$x_1,x_2,x_3,x_5,x_6,x_7,x_9,x_{10},x_{11},x_{12},x_{15},x_{18},x_{19},x_{21},x_{24},x_{27}$$
is nonzero, then the point is smooth or a rational singularity. By  (\ref{eq-change-variables-1321}), we may consider the $y$-coordinates.  By the $\mathbb{Z}/2$ symmetry (\ref{eq-symmetry-F1321}), we only need  to consider $y_1,y_2,y_3,y_5,y_6,y_7,y_{10},y_{11},y_{19}$.

If $y_1\neq 0$, the relations in $\Jac(F_{1321})$ imply 
\begin{eqnarray*}
y_{24}&=& \frac{-y_{11} y_{14}+y_{12} y_{21}-y_{17} y_{3}}{y_{1}},\\
y_{7}&=& \frac{-y_{11} y_{4}-y_{12} y_{8}-y_{19} y_{3}}{y_{1}},\\
y_{10}&=& \frac{-y_{11} y_{2}+y_{12} y_{27}-y_{3} y_{6}}{y_{1}},\\
y_{18}&=& \frac{y_{11} y_{9}-y_{12} y_{15}+y_{3} y_{5}}{y_{1}},\\
y_{16}&=& \frac{-y_{14} y_{15} y_{6}+y_{14} y_{27} y_{5}+y_{15} y_{17} y_{2}-y_{17} y_{27} y_{9}-y_{2} y_{21} y_{5}+y_{21} y_{6} y_{9}}{y_{1}},\\
y_{20}&=& \frac{-y_{14} y_{15} y_{19}-y_{14} y_{5} y_{8}+y_{15} y_{17} y_{4}+y_{17} y_{8} y_{9}+y_{19} y_{21} y_{9}-y_{21} y_{4} y_{5}}{y_{1}},\\
y_{29}&=& \frac{y_{14} y_{19} y_{27}+y_{14} y_{6} y_{8}-y_{17} y_{2} y_{8}-y_{17} y_{27} y_{4}-y_{19} y_{2} y_{21}+y_{21} y_{4} y_{6}}{y_{1}},\\
y_{13}&=& \frac{y_{15} y_{19} y_{2}-y_{15} y_{4} y_{6}-y_{19} y_{27} y_{9}+y_{2} y_{5} y_{8}+y_{27} y_{4} y_{5}-y_{6} y_{8} y_{9}}{y_{1}}.
\end{eqnarray*}
Substituting these equations into $\Jac(F_{1321})$, we obtain a zero ideal. 
So $\Spec(S)$ is an open subset of an affine space near such points. Similarly, one can show this  on $\{y_3\neq 0\}$ and $\{y_{11}\neq 0\}$.

If $y_2\neq 0$, the relations in $\Jac(F_{1321})$ imply 
\begin{eqnarray*}
y_{20}&=& \frac{-y_{13} y_{14}+y_{16} y_{4}-y_{29} y_{9}}{y_{2}},\\
y_{11}&=& \frac{-y_{1} y_{10}+y_{12} y_{27}-y_{3} y_{6}}{y_{2}}.
\end{eqnarray*}
When we substitute these equations into $\Jac(F_{1231})$ and then make the change of variables
\begin{alignat}{3}\label{eq-jacF1231-invert-y2-changeofvariables}
y_{5}&\longmapsto y_{5}+\frac{y_{6} y_{9}}{y_{2}},&\quad
y_{7}&\longmapsto y_{7}+ \frac{y_{10} y_{4}}{y_{2}},&\quad
y_{8}& \longmapsto y_{8}-\frac{y_{27} y_{4}}{y_{2}},\notag\\
y_{15}& \longmapsto y_{15}+\frac{y_{27} y_{9}}{y_{2}},&\quad
y_{17}& \longmapsto y_{17}+\frac{y_{14} y_{6}}{y_{2}},&\quad
y_{18}& \longmapsto y_{18}-\frac{y_{10} y_{9}}{y_{2}},\\
y_{19}& \longmapsto y_{19}+\frac{y_{4} y_{6}}{y_{2}},&\quad
y_{21}& \longmapsto y_{21}+\frac{y_{14} y_{27}}{y_{2}},&\quad
y_{24}& \longmapsto y_{24}+\frac{y_{10} y_{14}}{y_{2}},\notag
\end{alignat}
the ideal $\Jac(F_{1231})$ is transformed into the ideal 
\begin{equation}\label{eq-jacF1231-invert-y2}
  \renewcommand{\arraystretch}{1.2}
  \begin{array}{lll}
(-y_{13} y_{24}+y_{16} y_{7}+y_{18} y_{29},&
  -y_{13} y_{17}+y_{16} y_{19}-y_{29} y_{5},&
  y_{2} y_{21} y_{7}+y_{2} y_{24} y_{8}-y_{29} y_{3},\notag\\
\hphantom{(}y_{1} y_{16}-y_{15} y_{17} y_{2}+y_{2} y_{21} y_{5},&
y_{12} y_{16}+y_{17} y_{18} y_{2}+y_{2} y_{24} y_{5},&
y_{13} y_{21}+y_{15} y_{29}+y_{16} y_{8},\notag\\
\hphantom{(}-y_{1} y_{24}+y_{12} y_{21}-y_{17} y_{3},&
y_{12} y_{29}-y_{17} y_{2} y_{7}+y_{19} y_{2} y_{24},&
y_{1} y_{7}+y_{12} y_{8}+y_{19} y_{3},\\
\hphantom{(}-y_{13} y_{3}-y_{15} y_{2} y_{7}+y_{18} y_{2} y_{8},&
y_{1} y_{29}+y_{17} y_{2} y_{8}+y_{19} y_{2} y_{21},&
y_{15} y_{2} y_{24}+y_{16} y_{3}+y_{18} y_{2} y_{21},\notag\\
\hphantom{(}y_{12} y_{13}+y_{18} y_{19} y_{2}+y_{2} y_{5} y_{7},&
-y_{1} y_{13}+y_{15} y_{19} y_{2}+y_{2} y_{5} y_{8},&
y_{1} y_{18}+y_{12} y_{15}-y_{3} y_{5}).\notag
\end{array}
\end{equation}
Finally, the map
\begin{alignat}{6}\label{eq-jacF1231-chart-y2-to-plucker}
y_1&\longmapsto -p_{0,1},&\quad y_3&\longmapsto p_{0,4},&\quad y_5&\longmapsto -p_{1,2},&\quad y_7&\longmapsto p_{4,5},&\quad 
y_8&\longmapsto \frac{p_{0,5}}{y_2},&\quad y_{12}&\longmapsto-p_{1,4} y_2,\notag\\
y_{13}&\longmapsto -p_{2,5},&\quad y_{15}&\longmapsto \frac{p_{0,2}}{y_2},&\quad y_{16}&\longmapsto p_{2,3}, &\quad
y_{17}&\longmapsto -p_{1,3},&\quad y_{18}&\longmapsto -p_{2,4},&\quad y_{19}&\longmapsto p_{1,5},\\ 
y_{21}&\longmapsto \frac{p_{0,3}}{y_2},&\quad y_{24}&\longmapsto p_{3,4},&\quad  y_{29}&\longmapsto p_{3,5}\notag
\end{alignat}
transforms the ideal $(\ref{eq-jacF1231-invert-y2})$ into the Pl\"ucker ideal (\ref{eq-pluckerideal}).  Similarly, we can show that if one of the coordinates  $y_5,y_6,y_7,y_{10},y_{19}$ does not vanish, then the corresponding open locus is an open subset of a trivial affine fibration over the cone $\widehat{G}(2,6)$. 

The proof is completed.
\end{proof}

\begin{proposition}\label{prop-local-properties-A1321-1}
The ring $S$ is normal and Gorenstein.
\end{proposition}

\begin{proof}
The dimension of $S$  is $16$.
With the help of Macaulay2, we find a regular sequence of length $16$
\begin{gather*}
{y}_{1}, {y}_{2}, {y}_{3}, {y}_{5}+{y}_{21}, {y}_{6}+{y}_{24}, {y}_{7}, {y}_{8}, {y}_{10}+{y}_{21}, {y}_{13}, {y}_{17}, {y}_{18}+{y}_{19}+{y}_{27}, {y}_{9}+{y}_{11}+{y}_{14},
 {y}_{20}+{y}_{24},\\ 
  {y}_{4}+{y}_{7}+{y}_{15}+{y}_{16}, {y}_{12}+{y}_{14}+{y}_{15}+{y}_{18}+{y}_{19}, -{y}_{7}+{y}_{11}+{y}_{14}-{y}_{15}-2 {y}_{19}-2 {y}_{29}.
\end{gather*}
One easily checks that the vector $(-1,-1,-2)$ has positive inner products with each weight of the Haiman coordinates $x_1,\dots,x_{29}$ (see (\ref{eq-table-1321-c-to-x})). So these weights lie in a strictly convex cone. 
Applying Theorem~\ref{thm-Stanley-Gorenstein} and Remark~\ref{rem-reciprocal-law}, we see that $R$ is Gorenstein, and so is $S$.

By Proposition~\ref{prop-jacF1321-isolated}, the ring $S$ is regular in codimension $1$. Thus $S$ is normal. The proof is completed.
\end{proof}

\begin{theorem}\label{thm-local-property-lessorequal-7points}
Let $X$ be the smooth quasi-projective 3-fold. Then 
\begin{i-enumerate}
	\item\label{tlpl7-1} $\Hilb^n(X)$ is normal, Gorenstein for $n\leq 7$, and  has only rational singularities for $n\leq 6$; 
	\item\label{tlpl7-2} let $\rho\colon \Hilb^n(X)\rightarrow X^{(n)}$ be the Hilbert--Chow morphism; then for $n\leq 6$, $R^0\rho_* \mathcal{O}_{\Hilb^n(X)}=\mathcal{O}_{X^{(n)}}$ and  $R^i \rho_* \mathcal{O}_{\Hilb^n(X)}=0$ for $i>0$.
\end{i-enumerate}
\end{theorem}

\begin{proof}
Being Gorenstein and having only rational singularities are both \'etale-local properties. Being simultaneously normal and Cohen--Macaulay is also an \'etale-local property by Serre's criterion. Since $\Hilb^n(X)$ has the same \'etale-local structure as $\Hilb^n(\mathbb{A}^3)$,  conclusion \eqref{tlpl7-1} for arbitrary smooth quasi-projective 3-folds follows from Propositions~\ref{prop-local-properties-A121}
and~\ref{prop-local-properties-A1321-1}. Since $X^{(7)}$ has only rational singularities, \eqref{tlpl7-2} follows from \eqref{tlpl7-1}.
\end{proof}

This theorem suggests that if  Conjecture~\ref{conj-McKay-1} is true, one expects that the equality  moreover holds in the derived category, so that the local  results glue. It also suggests the following question.

\begin{question}
Does $\Hilb^7(\mathbb{A}^3)$  have only rational singularities?
\end{question}

We make the following remarks related to this question.

\begin{remark}\label{rmk-locus-extradim9-points-trivial-bundle}
The locus of points of extra dimension $8$  in $\Hilb^7(\mathbb{A}^3)$ is a trivial $\mathbb{P}^2$-bundle over  the diagonal $\delta\colon\Delta\cong\mathbb{A}^3\to (\mathbb{A}^3)^{(7)}$ via the Hilbert--Chow morphism $\rho$, where $\delta$ is the diagonal embedding. This can be shown by an explicit comparison of $\Spec(\Bbbk[y_{22},y_{23},y_{25},y_{26},y_{28}])$ and the corresponding fiber induced by permuting the coordinates $X_1,X_2,X_3$ on $\mathbb{A}^3$. We leave the details to the reader.
\end{remark}

\begin{remark}
According to the discussions in Section~\ref{sec:conjectures}, the points with extra dimension $8$ on $\Hilb^n(\mathbb{A}^3)$ are expected to be the second-simplest singularities. They deserve a detailed study. We do not know whether the point $0\in \Spec(S)$ is a rational singularity, though we expect this. Let $f\colon Y\rightarrow \Spec(S)$ be a resolution of singularities. Since $S$ is Cohen--Macaulay and the points away from $0$ are at most rational singularities, by \cite[Lemma 3.3]{Kov99}, the point $0\in \Spec(S)$ is a rational singularity if and only if $R^{15}f_* \mathcal{O}_Y=0$.  Since the locus of such singularities in $\Hilb^7(\mathbb{P}^3)$ is a trivial $\mathbb{P}^2$-bundle over the diagonal in $(\mathbb{P}^3)^{(7)}$, it seems plausible that the Euler characteristic $\chi(\mathcal{O}_{\Hilb^7(\mathbb{P}^3)})$ determines $R^{15}f_* \mathcal{O}_Y$. But to implement this idea, we need to know not only the locus of points of extra dimension $8$, but also the
structure of an open neighborhood of this locus.  Proposition~\ref{prop-jacF1321-isolated} and its proof give us a way to resolve the singularity of $\Spec(S)$.  Consider the blow-up of $\Spec(S)$ at $0$, \textit{i.e.}, along $\mathfrak{m}=(y_1,\dots,y_{21},y_{24},y_{27},y_{29})$.  Denote the homogeneous coordinates corresponding to the generators $y_1,\dots,y_{21},y_{24},y_{27},y_{29}$ of $\mathfrak{m}$ by $z_1,\dots,z_{21},z_{24},z_{27},z_{29}$.  We are going to see that the proof of Proposition~\ref{prop-jacF1321-isolated} gives all the data of this blow-up. For example, the chart $\{z_1\neq 0\}$ is smooth. The chart $\{z_2\neq 0\}$, after the change of variables (compare to (\ref{eq-jacF1231-invert-y2-changeofvariables}), setting $z_2=1$)
\begin{alignat*}{3}
z_{5}&\longmapsto z_{5}+z_{6} z_{9},&\quad
z_{7}&\longmapsto z_{7}+ z_{10} z_{4},&\quad
z_{8}& \longmapsto z_{8}-z_{27} z_{4},\\
z_{15}& \longmapsto z_{15}+z_{27} z_{9},&\quad
z_{17}& \longmapsto z_{17}+z_{14} z_{6},&\quad
z_{18}& \longmapsto z_{18}-z_{10} z_{9},\\
z_{19}& \longmapsto z_{19}+z_{4} z_{6},&\quad
z_{21}& \longmapsto z_{21}+z_{14} z_{27},&\quad
z_{24}& \longmapsto z_{24}+z_{10} z_{14},
\end{alignat*}
becomes 
\[
\Spec\ \left(\Bbbk[z_4,z_6,z_9,z_{10},z_{14},z_{27}]\otimes 
\Bbbk[z_1,y_2,z_3,z_5,z_7,z_8,z_{12},z_{13},z_{15},z_{16},z_{17},z_{18},z_{19},z_{21},z_{24},z_{29}]/J\right), 
\]
where $J$ is the ideal (compare to (\ref{eq-jacF1231-invert-y2}))
\begin{equation}\label{eq-jacF1231-blowup-invert-z2}
  \renewcommand{\arraystretch}{1.2}
  \begin{array}{lll}
(-z_{13} z_{24}+z_{16} z_{7}+z_{18} z_{29},&
-z_{13} z_{17}+z_{16} z_{19}-z_{29} z_{5},&
y_{2} z_{21} z_{7}+y_{2} z_{24} z_{8}-z_{29} z_{3},\\
\hphantom{(}z_{1} z_{16}- y_{2} z_{15} z_{17}+y_{2} z_{21} z_{5},&
z_{12} z_{16}+y_{2} z_{17} z_{18} +y_{2} z_{24} z_{5},&
z_{13} z_{21}+z_{15} z_{29}+z_{16} z_{8},\\
\hphantom{(}-z_{1} z_{24}+z_{12} z_{21}-z_{17} z_{3},&
z_{12} z_{29}-  y_{2}z_{17} z_{7}+ y_{2} z_{19} z_{24},&
z_{1} z_{7}+z_{12} z_{8}+z_{19} z_{3},\\
\hphantom{(}-z_{13} z_{3}-y_{2} z_{15} z_{7}+ y_{2}z_{18} z_{8},&
z_{1} z_{29}+  y_{2} z_{17} z_{8}+y_{2} z_{19} z_{21},&
 y_{2} z_{15}z_{24}+z_{16} z_{3}+ y_{2} z_{18}z_{21},\\
\hphantom{(}z_{12} z_{13}+ y_{2} z_{18} z_{19}+y_{2} z_{5} z_{7},&
-z_{1} z_{13}+ y_{2} z_{15} z_{19}+y_{2} z_{5} z_{8},&
z_{1} z_{18}+z_{12} z_{15}-z_{3} z_{5}).
  \end{array}
\end{equation}
In this computation, we need to note that the factors such as $z_1 z_7 z_{16} z_{21}+z_1 z_8 z_{16} z_{24}-z_3 z_{15} z_{17} z_{29}+z_3 z_5 z_{21}z_{29}$ of 
\[
(y_{2} z_{21} z_{7}+y_{2} z_{24} z_{8}-z_{29} z_{3})z_{1}z_{16}
+(z_{1} z_{16}- y_{2} z_{15} z_{17}+y_{2} z_{21} z_{5})z_{3}z_{29}
\]
lie in $J$.
Now consider the blow-up of 
$\widehat{G}(2,6)\times \mathbb{A}^1=\widehat{G}(2,6)\times\Spec(\Bbbk[t])$ along the subscheme
\[
\{t=0,\ p_{i,j}=0\ \mbox{for}\ (i,j)\neq (0,1),(0,4),(1,4)\}\cong \mathbb{A}^3. 
\]
Denote by $\widetilde{p}_{i,j}$  the homogeneous coordinates corresponding to $p_{i,j}$, for $(i,j)\neq (0,1),(0,4),(1,4)$, and by $\tilde{t}$ the homogeneous coordinate corresponding to $t$. Then the map (compare to (\ref{eq-jacF1231-chart-y2-to-plucker}))
\begin{alignat}{6}\label{eq-jacF1231-blowup-chart-z2-to-plucker}
z_1&\longmapsto -p_{0,1},&\quad z_3&\longmapsto p_{0,4},&\quad z_5&\longmapsto -\widetilde{p}_{1,2},&\quad z_7&\longmapsto \widetilde{p}_{4,5},&\quad 
z_8&\longmapsto \widetilde{p}_{0,5},&\quad z_{12}&\longmapsto-p_{1,4},\notag\\
z_{13}&\longmapsto -\widetilde{p}_{2,5},&\quad z_{15}&\longmapsto \widetilde{p}_{0,2},&\quad z_{16}&\longmapsto \widetilde{p}_{2,3},&\quad
z_{17}&\longmapsto -\widetilde{p}_{1,3},&\quad z_{18}&\longmapsto -\widetilde{p}_{2,4},&\quad z_{19}&\longmapsto \widetilde{p}_{1,5},\\ 
z_{21}&\longmapsto \widetilde{p}_{0,3},&\quad z_{24}&\longmapsto \widetilde{p}_{3,4},&\quad  z_{29}&\longmapsto \widetilde{p}_{3,5},&\quad
y_{2}&\longmapsto t\notag
\end{alignat}
gives an isomorphism between
\[
\Spec(\Bbbk[z_1,y_2,z_3,z_5,z_7,z_8,z_{12},z_{13},z_{15},z_{16},z_{17},z_{18},z_{19},z_{21},z_{24},z_{29}]/J)
\]
and the chart $\{\tilde{t}\neq 0\}$.  The other charts $\{z_i\neq 0\}$ are similar; one can use the formulae in Appendix~\ref{sec:appendix-change-var-jacF1321}.  A further blowing-up of the cone singularities will give a resolution of singularities and help understand the singularity $0\in \Spec(S)$.

\end{remark}

\begin{appendix}

\section{Change of variables for some Borel ideals}\label{sec:appendix-changevariable-Borelideals}
\subsection{\texorpdfstring{$\boldsymbol{\left((1)\subset (4,1)\right)}$}{((1) in (4,1))}, extra.dim \texorpdfstring{$\boldsymbol{=6}$}{=6}}
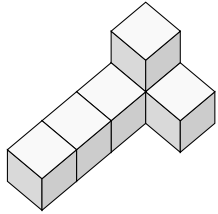
\label{sec:3D-141}
\begin{center}
\begin{tikzpicture}[x=(220:0.6cm), y=(-40:0.6cm), z=(90:0.42cm)]

\foreach \m [count=\y] in {{2,1,1,1},{1}}{
  \foreach \n [count=\x] in \m {
  \ifnum \n>0
      \foreach \z in {1,...,\n}{
        \draw [fill=gray!30] (\x+1,\y,\z) -- (\x+1,\y+1,\z) -- (\x+1, \y+1, \z-1) -- (\x+1, \y, \z-1) -- cycle;
        \draw [fill=gray!40] (\x,\y+1,\z) -- (\x+1,\y+1,\z) -- (\x+1, \y+1, \z-1) -- (\x, \y+1, \z-1) -- cycle;
        \draw [fill=gray!5] (\x,\y,\z)   -- (\x+1,\y,\z)   -- (\x+1, \y+1, \z)   -- (\x, \y+1, \z) -- cycle;  
      }
 \fi
 }
}    

\end{tikzpicture}
\end{center}

\[
I_{\lambda_{141}}=\left(X_1^4, X_1X_2, X_1 X_3, X_2^2,X_2 X_3, X_3^2\right).
\]
Algorithm~\ref{alg-step0-Haiman} gives
\begin{alignat*}{4}
 c_{0, 1, 0}^{0, 1, 1}&\longmapsto  x_{1},&\quad c_{1, 0, 0}^{1, 0, 1}&\longmapsto  x_{2},&\quad c_{0, 1, 0}^{1, 1, 0}&\longmapsto  x_{3},&\quad c_{0, 0, 1}^{1, 0, 1}&\longmapsto  x_{4},\\
 c_{0, 0, 1}^{4, 0, 0}&\longmapsto  x_{5},&\quad c_{3, 0, 0}^{4, 0, 0}&\longmapsto  x_{6},&\quad c_{0, 1, 0}^{4, 0, 0}&\longmapsto  x_{7},&\quad c_{0, 0, 1}^{0, 2, 0}&\longmapsto  x_{8},\\
 c_{3, 0, 0}^{1, 0, 1}&\longmapsto  x_{9},&\quad c_{1, 0, 0}^{1, 1, 0}&\longmapsto  x_{10},&\quad c_{0, 0, 1}^{0, 0, 2}&\longmapsto  x_{11},&\quad c_{0, 1, 0}^{0, 2, 0}&\longmapsto  x_{12},\\
 c_{3, 0, 0}^{1, 1, 0}&\longmapsto  x_{13},&\quad c_{0, 1, 0}^{1, 0, 1}&\longmapsto  x_{14},&\quad c_{0, 1, 0}^{0, 0, 2}&\longmapsto  x_{15},&\quad c_{0, 0, 1}^{1, 1, 0}&\longmapsto  x_{16},\\
 c_{3, 0, 0}^{0, 0, 2}&\longmapsto  x_{17},&\quad c_{2, 0, 0}^{1, 0, 1}&\longmapsto  x_{18},&\quad c_{2, 0, 0}^{4, 0, 0}&\longmapsto  x_{19},&\quad c_{3, 0, 0}^{0, 2, 0}&\longmapsto  x_{20},\\
 c_{2, 0, 0}^{1, 1, 0}&\longmapsto  x_{21},&\quad c_{3, 0, 0}^{0, 1, 1}&\longmapsto  x_{22},&\quad c_{1, 0, 0}^{4, 0, 0}&\longmapsto  x_{23},&\quad c_{0, 0, 1}^{0, 1, 1}&\longmapsto  x_{24},
\end{alignat*}
and 
\[
A_{l_{141}}\cong \Bbbk[x_1,\dots,x_{24}]/\mathcal{H}'_{\lambda_{141}}.
\]
The change of variables
\begin{eqnarray}\label{eq-141-totalchangeofvariables}
  x_{1}&\longmapsto& x_{1}+x_{2},\notag\\
  x_{2}&\longmapsto&  -x_{3}^{2} x_{9}-3 x_{3} x_{4} x_{9}-3 x_{4}^{2} x_{9}-x_{5} x_{9}^{2}-x_{7} x_{9} x_{13}-x_{9} x_{14} x_{16}-x_{3} x_{18}-2 x_{4} x_{18}+x_{2},\notag\\
  x_{3}&\longmapsto&  x_{3}+x_{4},\notag\\ 
  x_{10}&\longmapsto&  -x_{3}^{2} x_{13}-3 x_{3} x_{4} x_{13}-3 x_{4}^{2} x_{13}-x_{5} x_{9} x_{13}-x_{7} x_{13}^{2}-x_{13} x_{14} x_{16}-x_{3} x_{21}-2 x_{4} x_{21}+x_{10},\notag\\ 
  x_{11}&\longmapsto&  x_{1}+2 x_{2}+x_{11},\notag\\
  x_{12}&\longmapsto&  2 x_{10}+x_{12}+x_{24},\\
   x_{17}&\longmapsto&  x_{3} x_{9}^{2}+2 x_{4} x_{9}^{2}+x_{6} x_{9}^{2}+2 x_{9} x_{18}+x_{17},\notag\\ 
  x_{20}&\longmapsto&  x_{3} x_{13}^{2}+2 x_{4} x_{13}^{2}+x_{6} x_{13}^{2}+2 x_{13} x_{21}+x_{20},\notag\\
  x_{22}&\longmapsto&  x_{3} x_{9} x_{13}+2 x_{4} x_{9} x_{13}+x_{6} x_{9} x_{13}+x_{13} x_{18}+x_{9} x_{21}+x_{22},\notag\\
  x_{23}&\longmapsto&  x_{3}^{3}+4 x_{3}^{2} x_{4}+6 x_{3} x_{4}^{2}+4 x_{4}^{3}-x_{3}^{2} x_{6}-3 x_{3} x_{4} x_{6}-3 x_{4}^{2} x_{6}+x_{3} x_{5} x_{9}+3 x_{4} x_{5} x_{9}\notag\\
 		&&+2 x_{3} x_{7} x_{13}+3 x_{4} x_{7} x_{13}+x_{5} x_{13} x_{14}+x_{7} x_{9} x_{16}+2 x_{3} x_{14} x_{16}+4 x_{4} x_{14} x_{16}\notag\\
 		&&-x_{6} x_{14} x_{16}+x_{5} x_{18}-x_{3} x_{19}-2 x_{4} x_{19}+x_{7} x_{21}+x_{23},\notag\\ 
   x_{24}&\longmapsto&  x_{10}+x_{24},\notag
\end{eqnarray}
transforms the subideal of $\mathcal{H}'_{\lambda_{141}}$ generated by the minDegree 2 equations  into
\begin{equation}\label{eq-141-Haimanideal-transformed}
  \renewcommand{\arraystretch}{1.2}
  \begin{array}{lll}
   (x_{1} x_{7}+x_{5} x_{15}+x_{14} x_{23},& -x_{11} x_{14}-x_{3} x_{15}-x_{7} x_{17},& x_{7} x_{8}+x_{16} x_{23}+x_{5} x_{24},\\ 
  \hphantom{(} x_{15} x_{16}-x_{7} x_{22}-x_{14} x_{24},& x_{1} x_{3}-x_{12} x_{14}-x_{5} x_{17},& x_{16} x_{17}-x_{14} x_{20}+x_{3} x_{22},\\\
 \hphantom{(}   x_{8} x_{14}-x_{1} x_{16}-x_{5} x_{22},& x_{1} x_{11}+x_{12} x_{15}-x_{17} x_{23},& x_{3} x_{8}-x_{12} x_{16}-x_{5} x_{20},\\
\hphantom{(} -x_{11} x_{16}-x_{7} x_{20}-x_{3} x_{24},& -x_{5} x_{11}+x_{7} x_{12}+x_{3} x_{23},& -x_{8} x_{15}-x_{22} x_{23}+x_{1} x_{24},\\
\hphantom{(}   x_{8} x_{17}-x_{1} x_{20}+x_{12} x_{22},& x_{8} x_{11}-x_{20} x_{23}+x_{12} x_{24},& -x_{15} x_{20}-x_{11} x_{22}+x_{17} x_{24}).
  \end{array}
  \end{equation}
The map
\begin{alignat*}{6}
x_1&\longmapsto p_{2,4},&\quad x_3 &\longmapsto -p_{0,5},&\quad x_{5}&\longmapsto -p_{2,5},&\quad 
x_7&\longmapsto p_{1,5},&\quad x_8 &\longmapsto p_{2,3},&\quad x_{11}&\longmapsto -p_{0,1},\\
x_{12} &\longmapsto p_{0,2},&\quad x_{14}&\longmapsto p_{4,5},&\quad x_{15}&\longmapsto p_{1,4},&\quad 
x_{16}&\longmapsto p_{3,5},&\quad x_{17}&\longmapsto p_{0,4},&\quad x_{20}&\longmapsto p_{0,3},\\
x_{22}&\longmapsto p_{3,4},&\quad x_{23}&\longmapsto p_{1,2},&\quad x_{24}&\longmapsto p_{1,3} 
\end{alignat*}
transforms (\ref{eq-141-Haimanideal-transformed}) into the Pl\"ucker ideal (\ref{eq-pluckerideal}). Using the dimension argument as in Section~\ref{sec:3D-131}, we obtain that  the change of variables (\ref{eq-141-totalchangeofvariables}) transform the  ideal $\mathcal{H}'_{\lambda_{141}}$ into   (\ref{eq-141-Haimanideal-transformed}). A computation similar to Corollary~\ref{cor-tripod-Hilbertfunctions} yields
\begin{eqnarray}\label{eq-equihilb-A141}
H(A_{{\lambda}_{141}};\mathbf{t})&=&K\left(\frac{\sqrt{t_{2}} \sqrt{t_{3}}}{t_{1}^{3/2}},\frac{t_{1}^{3/2} \sqrt{t_{3}}}{\sqrt{t_{2}}},\frac{t_{1}^{3/2} \sqrt{t_{2}}}{\sqrt{t_{3}}},\frac{t_{2}^{3/2}}{t_{1}^{3/2} \sqrt{t_{3}}},\frac{t_{3}^{3/2}}{t_{1}^{3/2} \sqrt{t_{2}}},\frac{t_{1}^{5/2}}{\sqrt{t_{2}} \sqrt{t_{3}}}\right)\notag\\
&&\Big/\left((1-t_{1})^3(1-t_{2})^3(1-t_{3})^3(1-t_{1}^3)(1-t_{1}^2)\left(\frac{t_{1}-t_{2}}{t_{1}}\right)\left(\frac{t_{1}-t_{3}}{t_{1}}\right)\left(\frac{t_{1}^2-t_{2}}{t_{1}^2}\right)\left(\frac{t_{1}^2-t_{3}}{t_{1}^2}\right)\right.\notag\\
&&\hphantom{\Big/\Big(}\left(\frac{t_{1}^3-t_{2}^2}{t_{1}^3}\right)\left(\frac{t_{1}^3-t_{2} t_{3}}{t_{1}^3}\right)\left(\frac{t_{1}^3-t_{3}^2}{t_{1}^3}\right)\left(\frac{t_{2}-t_{1} t_{3}}{t_{2}}\right)\left(\frac{t_{2}-t_{1}^4}{t_{2}}\right)
\left(\frac{t_{2}-t_{3}^2}{t_{2}}\right)\notag\\
&&\hphantom{\Big/\Big(}\left.\left(\frac{t_{3}-t_{1} t_{2}}{t_{3}}\right)
\left(\frac{t_{3}-t_{1}^4}{t_{3}}\right)\left(\frac{t_{3}-t_{2}^2}{t_{3}}\right)\right).
\end{eqnarray}
In the following sections, we omit the intermediate explanations.

\subsection{\texorpdfstring{$\boldsymbol{\left((1)\subset(5,1)\right)}$}{((1) in (5,1))}, extra.dim \texorpdfstring{$\boldsymbol{=6}$}{=6}}
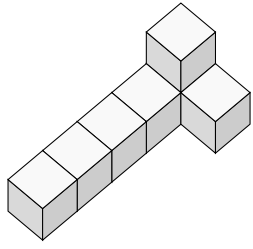
\label{sec:3D-151}
\begin{center}
\begin{tikzpicture}[x=(220:0.6cm), y=(-40:0.6cm), z=(90:0.42cm)]

\foreach \m [count=\y] in {{2,1,1,1,1},{1}}{
  \foreach \n [count=\x] in \m {
  \ifnum \n>0
      \foreach \z in {1,...,\n}{
        \draw [fill=gray!30] (\x+1,\y,\z) -- (\x+1,\y+1,\z) -- (\x+1, \y+1, \z-1) -- (\x+1, \y, \z-1) -- cycle;
        \draw [fill=gray!40] (\x,\y+1,\z) -- (\x+1,\y+1,\z) -- (\x+1, \y+1, \z-1) -- (\x, \y+1, \z-1) -- cycle;
        \draw [fill=gray!5] (\x,\y,\z)   -- (\x+1,\y,\z)   -- (\x+1, \y+1, \z)   -- (\x, \y+1, \z) -- cycle;  
      }
 \fi
 }
}    

\end{tikzpicture}
\end{center}

\[
I_{\lambda_{151}}=\left(X_1^5,X_1 X_2,X_1 X_3, X_2^2,X_2 X_3,X_3^2\right),\
\mathbf{h}=(1,3,1,1,1).
\]
\begin{alignat*}{6}
{c}_{1,0,0}^{1,1,0}  &\longmapsto  {x}_{1},&\quad {c}_{1,0,0}^{1,0,1}  &\longmapsto {x}_{2},&\quad 
{c}_{1,0,0}^{5,0,0}  &\longmapsto  {x}_{3},&\quad {c}_{2,0,0}^{1,1,0}  &\longmapsto {x}_{4},&\quad 
{c}_{2,0,0}^{1,0,1}  &\longmapsto  {x}_{5},&\quad {c}_{2,0,0}^{5,0,0}  &\longmapsto {x}_{6},\\
 {c}_{3,0,0}^{1,1,0}  &\longmapsto  {x}_{7},&\quad {c}_{3,0,0}^{1,0,1}  &\longmapsto {x}_{8},&\quad 
 {c}_{3,0,0}^{5,0,0}  &\longmapsto  {x}_{9},&\quad {c}_{4,0,0}^{1,1,0}  &\longmapsto {x}_{10},&\quad 
 {c}_{4,0,0}^{1,0,1}  &\longmapsto  {x}_{11},&\quad {c}_{4,0,0}^{5,0,0}  &\longmapsto  {x}_{12},\\
  {c}_{4,0,0}^{0,2,0}  &\longmapsto  {x}_{13},&\quad {c}_{4,0,0}^{0,1,1}  &\longmapsto {x}_{14},&\quad 
  {c}_{4,0,0}^{0,0,2}  &\longmapsto  {x}_{15},&\quad {c}_{0,1,0}^{1,1,0}  &\longmapsto {x}_{16},&\quad 
  {c}_{0,1,0}^{1,0,1}  &\longmapsto  {x}_{17},&\quad {c}_{0,1,0}^{5,0,0}  &\longmapsto {x}_{18},\\
   {c}_{0,1,0}^{0,2,0}  &\longmapsto  {x}_{19},&\quad {c}_{0,1,0}^{0,1,1}  &\longmapsto {x}_{20},&\quad 
   {c}_{0,1,0}^{0,0,2}  &\longmapsto  {x}_{21},&\quad {c}_{0,0,1}^{1,1,0}  &\longmapsto {x}_{22},&\quad 
   {c}_{0,0,1}^{1,0,1}  &\longmapsto  {x}_{23},&\quad {c}_{0,0,1}^{5,0,0}  &\longmapsto {x}_{24},\\
    {c}_{0,0,1}^{0,2,0}  &\longmapsto  {x}_{25},&\quad {c}_{0,0,1}^{0,1,1}  &\longmapsto {x}_{26},&\quad 
    {c}_{0,0,1}^{0,0,2}  &\longmapsto  {x}_{27}.
\end{alignat*}
\begin{eqnarray*}\label{eq-151-totalchangeofvariables}
{x}_{3} &\longmapsto&  -4 {x}_{12}{x}_{16}^{3}+5 {x}_{16}^{4}+6 {x}_{10}{x}_{16}^{2}{x}_{18}+{x}_{10}^{2}{x}_{18}^{2}-4 {x}_{12}{x}_{16}{x}_{17}{x}_{22}+10 {x}_{16}^{2}{x}_{17}{x}_{22}\notag\\
        &&+4 {x}_{11}{x}_{16}{x}_{18}{x}_{22}+2 {x}_{10}{x}_{
     17}{x}_{18}{x}_{22}+{x}_{17}^{2}{x}_{22}^{2}-6 {x}_{12}{x}_{16}^{2}{x}_{23}+10 {x}_{16}^{3}{x}_{23}\notag\\
     && +4 {x}_{10}{x}_{16}{x}_{18}{x}_{23}-2 {x}_{12}{x}_{17}{x}_{22}{x}_{23}+10 {x}_{16}{x}_{17}{x}_{22}{x}_{23}+2 {x}_{11}{x}_{18}{x}_{22}{x}_{23}\notag\\
     &&-4
      {x}_{12}{x}_{16}{x}_{23}^{2}+10 {x}_{16}^{2}{x}_{23}^{2}+{x}_{10}{x}_{18}{x}_{23}^{2}+3 {x}_{17}{x}_{22}{x}_{23}^{2}-{x}_{12}{x}_{23}^{3}+5 {x}_{16}{x}_{23}^{3}\notag\\
      &&+{x}_{23}^{4}+6 {x}_{11}{x}_{16}^{2}{x}_{24}+4 {x}_{10}{x}_{16}{x}_{17}{x}_{24}+{x
     }_{10}{x}_{11}{x}_{18}{x}_{24}+2 {x}_{11}{x}_{17}{x}_{22}{x}_{24}\notag\\
     &&+8 {x}_{11}{x}_{16}{x}_{23}{x}_{24}+2 {x}_{10}{x}_{17}{x}_{23}{x}_{24}+3 {x}_{11}{x}_{23}^{2}{x}_{24}+{x}_{11}^{2}{x}_{24}^{2}-3 {x}_{9}{x}_{16}^{2}+3 {x}_{7}{x}_{16}{x}_{18}\notag\\
     &&-{x
     }_{9}{x}_{17}{x}_{22}+{x}_{8}{x}_{18}{x}_{22}-3 {x}_{9}{x}_{16}{x}_{23}+{x}_{7}{x}_{18}{x}_{23}-{x}_{9}{x}_{23}^{2}+3 {x}_{8}{x}_{16}{x}_{24}+{x}_{7}{x}_{17}{x}_{24}\notag\\
     &&+2 {x}_{8}{x}_{23}{x}_{24}-2 {x}_{6}{x}_{16}+{x}_{4}{x}_{18}-{x}_{6}{x}_{23}+{x
     }_{5}{x}_{24}+{x}_{3},\notag\\
 {x}_{13} &\longmapsto&  {x}_{10}^{2}{x}_{12}^{2}+2 {x}_{10}^{2}{x}_{12}{x}_{16}+3 {x}_{10}^{2}{x}_{16}^{2}+{x}_{10}^{3}{x}_{18}+{x}_{10}^{2}{x}_{12}{x}_{23}+3 {x}_{10}^{2}{x}_{16}{x}_{23}+{x}_{10}^{2}{x}_{23}^{2}\notag\\
 	&&+{x}_{10}^{2}{x}_{11}{x}_{24}+{x}_{9}{x}_{10}^{2}+2 {x}_{7}{x}_{10}{x}_{12}+4 {x}_{7}{x}_{10}{x}_{16}+2 {x}_{7}{x}_{10}{x}_{23}+{x}_{7}^{2}+2 {x}_{4}{x}_{10}+{x}_{13}, \notag\\
 {x}_{14} &\longmapsto& 
     {x}_{10}{x}_{11}{x}_{12}^{2}+2 {x}_{10}{x}_{11}{x}_{12}{x}_{16}+3 {x}_{10}{x}_{11}{x}_{16}^{2}+{x}_{10}^{2}{x}_{11}{x}_{18}+{x}_{10}{x}_{11}{x}_{12}{x}_{23}\notag\\*
     &&+3 {x}_{10}{x}_{11}{x}_{16}{x}_{23}+{x}_{10}{x}_{11}{x}_{23}^{2}+{x}_{10}{x}_{11}^{2}{x}_{
     24}+{x}_{9}{x}_{10}{x}_{11}+{x}_{8}{x}_{10}{x}_{12}+{x}_{7}{x}_{11}{x}_{12}\notag\\
     &&+2 {x}_{8}{x}_{10}{x}_{16}+2 {x}_{7}{x}_{11}{x}_{16}+{x}_{8}{x}_{10}{x}_{23}+{x}_{7}{x}_{11}{x}_{23}+{x}_{7}{x}_{8}+{x}_{5}{x}_{10}+{x}_{4}{x}_{11}+{x}_{14}, \notag\\
 {x}_{15} &\longmapsto&  {x}_{11}^{2}{x}_{12}^{2}+2 {x}_{11}^{2}{x}_{12}{x}_{16}+3 {x}_{11}^{2}{x}_{16}^{2}-{x}_{10}^{2}{x}_{17}^{2}+{x}_{10}{x}_{11}^{2}{x}_{18}+{x}_{11}^{2}{x}_{17}{x}_{22}+{x}_{11}^{2}{x}_{12}{x}_{23}\notag\\
 	&&+3 {x}_{11}^{2}{x}_{16}{x}_{23}-{x}_{10
     }{x}_{11}{x}_{17}{x}_{23}+{x}_{11}^{2}{x}_{23}^{2}+{x}_{11}^{3}{x}_{24}+{x}_{9}{x}_{11}^{2}+2 {x}_{8}{x}_{11}{x}_{12}+4 {x}_{8}{x}_{11}{x}_{16}\notag\\
     &&+2 {x}_{8}{x}_{11}{x}_{23}+{x}_{8}^{2}+2 {x}_{5}{x}_{11}+{x}_{15},\notag\\
  {x}_{19} &\longmapsto& 
     8 {x}_{10}{x}_{16}^{3}+2 {x}_{10}^{2}{x}_{12}{x}_{18}+6 {x}_{10}^{2}{x}_{16}{x}_{18}+8 {x}_{10}{x}_{16}{x}_{17}{x}_{22}+3 {x}_{10}{x}_{11}{x}_{18}{x}_{22}\notag\\
     &&+12 {x}_{10}{x}_{16}^{2}{x}_{23}+2 {x}_{10}^{2}{x}_{18}{x}_{23}+4 {x}_{10}{x}_{17}{x}_{
     22}{x}_{23}+8 {x}_{10}{x}_{16}{x}_{23}^{2}+2 {x}_{10}{x}_{23}^{3}\notag\\
     &&+2 {x}_{10}{x}_{11}{x}_{12}{x}_{24}+6 {x}_{10}{x}_{11}{x}_{16}{x}_{24}+3 {x}_{10}^{2}{x}_{17}{x}_{24}+4 {x}_{10}{x}_{11}{x}_{23}{x}_{24}+6 {x}_{7}{x}_{16}^{2}\notag\\
     &&+4 {x}_{7}{x}_{10}{
     x}_{18}+2 {x}_{7}{x}_{17}{x}_{22}+6 {x}_{7}{x}_{16}{x}_{23}+2 {x}_{7}{x}_{23}^{2}+2 {x}_{8}{x}_{10}{x}_{24}+2 {x}_{7}{x}_{11}{x}_{24}\notag\\
     &&+4 {x}_{4}{x}_{16}+2 {x}_{4}{x}_{23}+2 {x}_{1}+{x}_{19}+{x}_{26},\notag\\
  {x}_{20} &\longmapsto& 
     4 {x}_{11}{x}_{16}^{3}+{x}_{10}{x}_{11}{x}_{12}{x}_{18}+3 {x}_{10}{x}_{11}{x}_{16}{x}_{18}+4 {x}_{11}{x}_{16}{x}_{17}{x}_{22}+{x}_{11}^{2}{x}_{18}{x}_{22}+6 {x}_{11}{x}_{16}^{2}{x}_{23}\notag\\
     &&+{x}_{10}{x}_{11}{x}_{18}{x}_{23}+2 {x}_{11}{x}_{17}{x}_{22
     }{x}_{23}+4 {x}_{11}{x}_{16}{x}_{23}^{2}+{x}_{11}{x}_{23}^{3}+{x}_{11}^{2}{x}_{12}{x}_{24}\notag\\
     &&+3 {x}_{11}^{2}{x}_{16}{x}_{24}+2 {x}_{10}{x}_{11}{x}_{17}{x}_{24}+2 {x}_{11}^{2}{x}_{23}{x}_{24}+3 {x}_{8}{x}_{16}^{2}+{x}_{8}{x}_{10}{x}_{18}+{x}_{7}{x
     }_{11}{x}_{18}\notag\\
     &&+{x}_{8}{x}_{17}{x}_{22}+3 {x}_{8}{x}_{16}{x}_{23}+{x}_{8}{x}_{23}^{2}+2 {x}_{8}{x}_{11}{x}_{24}+2 {x}_{5}{x}_{16}+{x}_{5}{x}_{23}+{x}_{2}+{x}_{20}, \notag\\
{x}_{23} &\longmapsto&  {x}_{16}+{x}_{23}, \notag\\
{x}_{26} &\longmapsto& 
     4 {x}_{10}{x}_{16}^{3}+{x}_{10}^{2}{x}_{12}{x}_{18}+3 {x}_{10}^{2}{x}_{16}{x}_{18}+4 {x}_{10}{x}_{16}{x}_{17}{x}_{22}+2 {x}_{10}{x}_{11}{x}_{18}{x}_{22}+6 {x}_{10}{x}_{16}^{2}{x}_{23}\notag\\
     &&+{x}_{10}^{2}{x}_{18}{x}_{23}+2 {x}_{10}{x}_{17}{x}_{22}{x}_{
     23}+4 {x}_{10}{x}_{16}{x}_{23}^{2}+{x}_{10}{x}_{23}^{3}+{x}_{10}{x}_{11}{x}_{12}{x}_{24}+3 {x}_{10}{x}_{11}{x}_{16}{x}_{24}\notag\\
     &&+{x}_{10}^{2}{x}_{17}{x}_{24}+2 {x}_{10}{x}_{11}{x}_{23}{x}_{24}+3 {x}_{7}{x}_{16}^{2}+2 {x}_{7}{x}_{10}{x}_{18}+{x}_{7}{x
     }_{17}{x}_{22}+3 {x}_{7}{x}_{16}{x}_{23}+{x}_{7}{x}_{23}^{2}\notag\\
     &&+{x}_{8}{x}_{10}{x}_{24}+{x}_{7}{x}_{11}{x}_{24}+2 {x}_{4}{x}_{16}+{x}_{4}{x}_{23}+{x}_{1}+{x}_{26}, \notag\\
 {x}_{27} &\longmapsto& 
     8 {x}_{11}{x}_{16}^{3}+2 {x}_{10}{x}_{11}{x}_{12}{x}_{18}+6 {x}_{10}{x}_{11}{x}_{16}{x}_{18}+{x}_{10}^{2}{x}_{17}{x}_{18}+8 {x}_{11}{x}_{16}{x}_{17}{x}_{22}+2 {x}_{11}^{2}{x}_{18}{x}_{22}\notag\\
     &&+12 {x}_{11}{x}_{16}^{2}{x}_{23}+3 {x}_{10}{x}_{11}{x}_{
     18}{x}_{23}+4 {x}_{11}{x}_{17}{x}_{22}{x}_{23}+8 {x}_{11}{x}_{16}{x}_{23}^{2}+2 {x}_{11}{x}_{23}^{3}+2 {x}_{11}^{2}{x}_{12}{x}_{24}\notag\\
     &&+6 {x}_{11}^{2}{x}_{16}{x}_{24}+3 {x}_{10}{x}_{11}{x}_{17}{x}_{24}+4 {x}_{11}^{2}{x}_{23}{x}_{24}+6 {x}_{8}{x
     }_{16}^{2}+2 {x}_{8}{x}_{10}{x}_{18}+2 {x}_{7}{x}_{11}{x}_{18}\notag\\
     &&+2 {x}_{8}{x}_{17}{x}_{22}+6 {x}_{8}{x}_{16}{x}_{23}+2 {x}_{8}{x}_{23}^{2}+4 {x}_{8}{x}_{11}{x}_{24}+4 {x}_{5}{x}_{16}+2 {x}_{5}{x}_{23}+2 {x}_{2}+{x}_{20}+{x}_{27}.\notag\\
\end{eqnarray*}

\begin{equation*}\label{eq-151-Haimanideal-transformed}
  \renewcommand{\arraystretch}{1.2}
  \begin{array}{lll}
  (-{x}_{13}{x}_{17}+{x}_{15}{x}_{22}-{x}_{14}{x}_{23},& -{x}_{13}{x}_{18}+{x}_{23}{x}_{26}-{x}_{22}{x}_{27},& -{x}_{19}{x}_{22}-{x}_{13}{x}_{24}-{x}_{23}{x}_{25},\\
\hphantom{(} -{x}_{14}{x}_{18}+{x}_{21}{x}_{22}-{x}_{17}{x}_{26},& -{x}_{20}{x}_{22}-{x}_{14}{x}_{24}+{x}_{17}{x}_{25},& -{x}_{15}{x}_{18}+{x}_{21}{x}_{23}-{x}_{17}{x}_{27},\\
\hphantom{(}-{x}_{17}{x}_{19}-{x}_{20}{x}_{23}-{x}_{15}{x}_{24},& {x}_{3}{x}_{22}+{x}_{18}{x}_{25}+{x}_{24}{x}_{26},& {x}_{3}{x}_{17}+{x}_{18}{x}_{20}+{x}_{21}{x}_{24},\\
\hphantom{(} -{x}_{18}{x}_{19}+{x}_{3}{x}_{23}+{x}_{24}{x}_{27},& {x}_{14}{x}_{19}-{x}_{13}{x}_{20}+{x}_{15}{x}_{25},& {x}_{3}{x}_{14}+{x}_{21}{x}_{25}-{x}_{20}{x}_{26},\\
  -{x}_{3}{x}_{13}+{x}_{19}{x}_{26}+{x}_{25}{x}_{27},& -{x}_{13}{x}_{21}+{x}_{15}{x}_{26}-{x}_{14}{x}_{27},& {x}_{3}{x}_{15}-{x}_{19}{x}_{21}-{x}_{20}{x}_{27}).\notag
  \end{array}
  \end{equation*}

\begin{alignat*}{6}
x_{3}&\longmapsto p_{1,2},&\quad x_{13}&\longmapsto p_{0,3},&\quad x_{14}&\longmapsto p_{3,4},&\quad x_{15}&\longmapsto p_{0,4},&\quad x_{17}&\longmapsto p_{4,5},&\quad x_{18}&\longmapsto p_{1,5},\\
x_{19}&\longmapsto p_{0,2},&\quad x_{20}&\longmapsto p_{2,4},&\quad x_{21}&\longmapsto p_{1,4},&\quad x_{22}&\longmapsto p_{3,5},&\quad x_{23}&\longmapsto p_{0,5},&\quad x_{24}&\longmapsto -p_{2,5},\\
 x_{25}&\longmapsto p_{2,3},&\quad x_{26}&\longmapsto p_{1,3},&\quad x_{27}&\longmapsto -p_{0,1}.
\end{alignat*}

\pagebreak 
\begin{eqnarray}\label{eq-equihilb-A151}
H(A_{{\lambda}_{151}};\mathbf{t})\!&=&\!\!K\left(\frac{\sqrt{t_{2}} \sqrt{t_{3}}}{t_{1}^2},\frac{t_{1}^2 \sqrt{t_{3}}}{\sqrt{t_{2}}},\frac{t_{1}^2 \sqrt{t_{2}}}{\sqrt{t_{3}}},\frac{t_{2}^{3/2}}{t_{1}^2 \sqrt{t_{3}}},\frac{t_{3}^{3/2}}{t_{1}^2 \sqrt{t_{2}}},\frac{t_{1}^3}{\sqrt{t_{2}} \sqrt{t_{3}}}\right)\notag\\
&&\!\!\Big/\left((1-t_{1})^3(1-t_{2})^3(1-t_{3})^3(1-t_{1}^4)(1-t_{1}^2)(1-t_{1}^3)\left(\frac{t_{1}-t_{2}}{t_{1}}\right)\left(\frac{t_{1}-t_{3}}{t_{1}}\right)\left(\frac{t_{1}^2-t_{2}}{t_{1}^2}\right)\left(\frac{t_{1}^2-t_{3}}{t_{1}^2}\right)\right.\notag\\
&&\!\!\hphantom{\Big/\Big(}\left(\frac{t_{1}^3-t_{2}}{t_{1}^3}\right)\left(\frac{t_{1}^3-t_{3}}{t_{1}^3}\right)
\left(\frac{t_{1}^4-t_{2}^2}{t_{1}^4}\right)\left(\frac{t_{1}^4-t_{2} t_{3}}{t_{1}^4}\right)\left(\frac{t_{1}^4-t_{3}^2}{t_{1}^4}\right)\left(\frac{t_{2}-t_{1} t_{3}}{t_{2}}\right)\left(\frac{t_{2}-t_{1}^5}{t_{2}}\right)
\left(\frac{t_{2}-t_{3}^2}{t_{2}}\right)\left(\frac{t_{3}-t_{1} t_{2}}{t_{3}}\right)\notag\\
&&\!\!\hphantom{\Big/\Big(}\left.\left(\frac{t_{3}-t_{1}^5}{t_{3}}\right)\left(\frac{t_{3}-t_{2}^2}{t_{3}}\right)\right).
\end{eqnarray}

\subsection{\texorpdfstring{$\boldsymbol{\left((1)\subset(4,2)\right)}$}{((1) in (4,2))},  extra.dim  \texorpdfstring{$\boldsymbol{=6}$}{=6}}
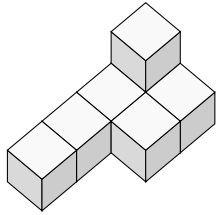
\label{sec:3D-142}
\begin{center}
\begin{tikzpicture}[x=(220:0.6cm), y=(-40:0.6cm), z=(90:0.42cm)]

\foreach \m [count=\y] in {{2,1,1,1},{1,1}}{
  \foreach \n [count=\x] in \m {
  \ifnum \n>0
      \foreach \z in {1,...,\n}{
        \draw [fill=gray!30] (\x+1,\y,\z) -- (\x+1,\y+1,\z) -- (\x+1, \y+1, \z-1) -- (\x+1, \y, \z-1) -- cycle;
        \draw [fill=gray!40] (\x,\y+1,\z) -- (\x+1,\y+1,\z) -- (\x+1, \y+1, \z-1) -- (\x, \y+1, \z-1) -- cycle;
        \draw [fill=gray!5] (\x,\y,\z)   -- (\x+1,\y,\z)   -- (\x+1, \y+1, \z)   -- (\x, \y+1, \z) -- cycle;  
      }
 \fi
 }
}    

\end{tikzpicture}
\end{center}
\[
I_{\lambda_{142}}=\left(X_1^4, X_1^2 X_2, X_1 X_3, X_2^2,X_2 X_3, X_3^2\right).
\]
\begin{alignat*}{6}
 c_{1, 0, 0}^{1, 0, 1} &\longmapsto x_{1},&\quad c_{2, 0, 0}^{1, 0, 1} &\longmapsto x_{2},&\quad c_{2, 0, 0}^{2, 1, 0} &\longmapsto x_{3},&\quad c_{2, 0, 0}^{4, 0, 0} &\longmapsto x_{4},&\quad c_{2, 0, 0}^{0, 2, 0} &\longmapsto x_{5},&\quad c_{3, 0, 0}^{1, 0, 1} &\longmapsto x_{6},\\
  c_{3, 0, 0}^{2, 1, 0} &\longmapsto x_{7},&\quad c_{3, 0, 0}^{4, 0, 0} &\longmapsto x_{8},&\quad c_{3, 0, 0}^{0, 2, 0} &\longmapsto x_{9},&\quad c_{3, 0, 0}^{0, 1, 1} &\longmapsto x_{10},&\quad c_{3, 0, 0}^{0, 0, 2} &\longmapsto x_{11},&\quad c_{0, 1, 0}^{1, 0, 1} &\longmapsto x_{12},\\
  c_{0, 1, 0}^{2, 1, 0} &\longmapsto x_{13},&\quad c_{0, 1, 0}^{4, 0, 0} &\longmapsto x_{14},&\quad c_{0, 1, 0}^{0, 2, 0} &\longmapsto x_{15},&\quad c_{1, 1, 0}^{1, 0, 1} &\longmapsto x_{16},&\quad c_{1, 1, 0}^{2, 1, 0} &\longmapsto x_{17},&\quad c_{1, 1, 0}^{4, 0, 0} &\longmapsto x_{18},\\
   c_{1, 1, 0}^{0, 2, 0} &\longmapsto x_{19},&\quad c_{1, 1, 0}^{0, 1, 1} &\longmapsto x_{20},&\quad c_{1, 1, 0}^{0, 0, 2} &\longmapsto x_{21},&\quad c_{0, 0, 1}^{1, 0, 1} &\longmapsto x_{22},&\quad c_{0, 0, 1}^{2, 1, 0} &\longmapsto x_{23},&\quad c_{0, 0, 1}^{4, 0, 0} &\longmapsto x_{24},\\
    c_{0, 0, 1}^{0, 2, 0} &\longmapsto x_{25},&\quad c_{0, 0, 1}^{0, 1, 1} &\longmapsto x_{26},&\quad c_{0, 0, 1}^{0, 0, 2} &\longmapsto x_{27}.
\end{alignat*}

\begin{eqnarray*}
x_{2} &\longmapsto&  -x_{16} x_{19} + x_{2} + x_{20} - x_{17} x_{6} - x_{22} x_{6},\\ 
x_{3} &\longmapsto&  x_{15} + x_{19} x_{22} + 2 x_{16} x_{25} + x_{26} + x_{3} - 2 x_{17} x_{7} - 2 x_{22} x_{7} + 2 x_{18} x_{9},\\ 
x_{4} &\longmapsto&  x_{13} + x_{17}^2 - x_{18} x_{19} + 2 x_{22}^2 + x_{16} x_{23} + x_{4} + x_{24} x_{6} + x_{18} x_{7} - x_{17} x_{8} - x_{22} x_{8},\\ 
x_{5} &\longmapsto&  x_{5} - x_{25} x_{6} - x_{19} x_{7} + x_{7}^2 + x_{17} x_{9} - x_{22} x_{9} - x_{8} x_{9},\\
x_{10} &\longmapsto&   x_{10} + x_{6} x_{7} + x_{16} x_{9},\\ 
x_{11} &\longmapsto&  x_{11} + x_{10} x_{16} - 2 x_{16} x_{19} x_{6} + 2 x_{2} x_{6} + 2 x_{20} x_{6} - 2 x_{17} x_{6}^2 + 2 x_{16} x_{6} x_{7} + x_{6}^2 x_{8} + x_{16}^2 x_{9},\\
x_{12} &\longmapsto&   x_{12} - x_{16} x_{22},\\ 
x_{13} &\longmapsto&  x_{13} - x_{17} x_{22} + x_{22}^2,\\ 
x_{14} &\longmapsto&   x_{14} - x_{18} x_{22} + x_{16} x_{24},\\ 
x_{15} &\longmapsto&  2 x_{15} + x_{19} x_{22} + 2 x_{16} x_{25} + x_{26} + x_{3} - 2 x_{17} x_{7} - 2 x_{22} x_{7} + 2 x_{18} x_{9},\\ 
x_{21} &\longmapsto&   -x_{16}^2 x_{19} + 2 x_{16} x_{2} + 2 x_{16} x_{20} + x_{21} + x_{18} x_{6}^2,\\ 
x_{26} &\longmapsto&  x_{15} + x_{19} x_{22} + x_{16} x_{25} + x_{26} - x_{17} x_{7} - x_{22} x_{7} + x_{18} x_{9},\\ 
x_{27} &\longmapsto&   2 x_{1} + 2 x_{15} x_{16} - x_{10} x_{18} + x_{12} x_{19} + x_{17} x_{2} - 2 x_{16} x_{19} x_{22} + 2 x_{2} x_{22} + 4 x_{20} x_{22}\\
&& + 3 x_{16}^2 x_{25} + x_{27} + x_{16} x_{3} + 2 x_{13} x_{6} - 4 x_{17} x_{22} x_{6} + 2 x_{22}^2 x_{6}\\ 
&&+  2 x_{16} x_{23} x_{6} + x_{24} x_{6}^2 - 2 x_{16} x_{17} x_{7} - 2 x_{16} x_{22} x_{7} + 2 x_{16} x_{18} x_{9}.
\end{eqnarray*}

\begin{equation*}
  \renewcommand{\arraystretch}{1.2}
  \begin{array}{lll}
(x_{2} x_{23}-x_{10} x_{24} +x_{12} x_{25},&  -x_{5} x_{12} +x_{10} x_{13} +x_{11} x_{23},&  -x_{5} x_{24} +x_{13} x_{25} +x_{23} x_{26},\\
\hphantom{(} x_{3} x_{12}-x_{10} x_{14} +x_{21} x_{23},&  x_{4} x_{23}-x_{3} x_{24} +x_{14} x_{25},&  x_{4} x_{12}-x_{2} x_{14} +x_{21} x_{24}, \\
\hphantom{(}       x_{2} x_{5} +x_{11} x_{25}-x_{10} x_{26},&  -x_{2} x_{3} +x_{4} x_{10} +x_{21} x_{25},&  -x_{2} x_{13}-x_{11} x_{24} +x_{12} x_{26}, \\
\hphantom{(}        x_{3} x_{13}-x_{5} x_{14}-x_{23} x_{27},&  -x_{4} x_{13} +x_{14} x_{26} +x_{24} x_{27},&  -x_{3} x_{11}-x_{5} x_{21}-x_{10} x_{27}, \\
\hphantom{(}       x_{4} x_{11} +x_{21} x_{26} +x_{2} x_{27},&  -x_{4} x_{5} +x_{3} x_{26} +x_{25} x_{27},&  -x_{11} x_{14}-x_{13} x_{21}-x_{12} x_{27}).
  \end{array}
  \end{equation*}

\begin{alignat*}{6}
x_{2}&\longmapsto p_{2,4},&\quad x_{3}&\longmapsto p_{1,3},&\quad x_{4}&\longmapsto p_{1,2},&\quad x_{5}&\longmapsto p_{0,3},&\quad x_{10}&\longmapsto p_{3,4},&\quad x_{11}&\longmapsto -p_{0,4},\\
x_{12}&\longmapsto -p_{4,5},&\quad x_{13}&\longmapsto p_{0,5},&\quad x_{14}&\longmapsto p_{1,5},&\quad x_{21}&\longmapsto p_{1,4},&\quad x_{23}&\longmapsto p_{3,5},&\quad x_{24}&\longmapsto p_{2,5},\\
 x_{25}&\longmapsto p_{2,3},&\quad x_{26}&\longmapsto p_{0,2},&\quad x_{27}&\longmapsto -p_{0,1}.
\end{alignat*}

\begin{eqnarray}\label{eq-equihilb-A142}
H(A_{{\lambda}_{142}};\mathbf{t})&=&K\left(\frac{\sqrt{t_{2}} \sqrt{t_{3}}}{t_{1}},\frac{t_{1} \sqrt{t_{3}}}{\sqrt{t_{2}}},\frac{t_{1} \sqrt{t_{2}}}{\sqrt{t_{3}}},\frac{t_{2}^{3/2}}{t_{1} \sqrt{t_{3}}},\frac{t_{3}^{3/2}}{t_{1}^2 \sqrt{t_{2}}},\frac{t_{1}^3}{\sqrt{t_{2}} \sqrt{t_{3}}}\right)\notag\\
&&\Big/\left((1-t_{1})^3 (1-t_{2})^3 (1-t_{3})^2(1-t_{1}^2)^2\left(\frac{t_{1}-t_{3}}{t_{1}}\right)^2\left(\frac{t_{1}^2-t_{2}^2}{t_{1}^2}\right)
\left(\frac{t_{1}^2-t_{3}}{t_{1}^2}\right)\left(\frac{t_{1}-t_{2}}{t_{1}}\right)^2\right.\notag\\
&& \hphantom{\Big/\Big(}\left(\frac{t_{1}^3-t_{2}^2}{t_{1}^3}\right)\left(\frac{t_{1}^3-t_{2} t_{3}}{t_{1}^3}\right)\left(\frac{t_{1}^3-t_{3}^2}{t_{1}^3}\right)\left(\frac{t_{2}-t_{1} t_{3}}{t_{2}}\right)\left(\frac{t_{2}-t_{1}^4}{t_{2}}\right)\left(\frac{t_{2}-t_{3}}{t_{2}}\right)
\left(\frac{t_{2}-t_{1}^3}{t_{2}}\right)\notag\\
&& \hphantom{\Big/\Big(}\left.\left(\frac{t_{1} t_{2}-t_{3}^2}{t_{1} t_{2}}\right)\left(\frac{t_{3}-t_{1}^2 t_{2}}{t_{3}}\right)\left(\frac{t_{3}-t_{1}^4}{t_{3}}\right)\left(\frac{t_{3}-t_{2}^2}{t_{3}}\right)\right).
\end{eqnarray}

\subsection{\texorpdfstring{$\boldsymbol{\left((2)\subset (3,2)\right)}$}{((2) in (3,2))},  extra.dim  \texorpdfstring{$\boldsymbol{=6}$}{=6}}
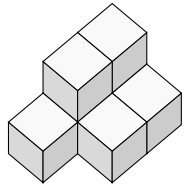
\label{sec:3D-232}
\begin{center}
\begin{tikzpicture}[x=(220:0.6cm), y=(-40:0.6cm), z=(90:0.42cm)]

\foreach \m [count=\y] in {{2,2,1},{1,1}}{
  \foreach \n [count=\x] in \m {
  \ifnum \n>0
      \foreach \z in {1,...,\n}{
        \draw [fill=gray!30] (\x+1,\y,\z) -- (\x+1,\y+1,\z) -- (\x+1, \y+1, \z-1) -- (\x+1, \y, \z-1) -- cycle;
        \draw [fill=gray!40] (\x,\y+1,\z) -- (\x+1,\y+1,\z) -- (\x+1, \y+1, \z-1) -- (\x, \y+1, \z-1) -- cycle;
        \draw [fill=gray!5] (\x,\y,\z)   -- (\x+1,\y,\z)   -- (\x+1, \y+1, \z)   -- (\x, \y+1, \z) -- cycle;  
      }
 \fi
 }
}    

\end{tikzpicture}
\end{center}
\[
I_{232}=\left(X_1^3, X_1^2 X_2,X_1^2 X_3,X_2^2,X_2 X_3,X_3^2\right).
\]
\begin{alignat*}{6}
 c_{1, 0, 0}^{3, 0, 0} &\longmapsto  x_{1},&\quad c_{2, 0, 0}^{3, 0, 0} &\longmapsto  x_{2},&\quad c_{2, 0, 0}^{2, 1, 0} &\longmapsto  x_{3},&\quad c_{2, 0, 0}^{2, 0, 1} &\longmapsto  x_{4},&\quad c_{2, 0, 0}^{0, 2, 0} &\longmapsto  x_{5},&\quad c_{2, 0, 0}^{0, 1, 1} &\longmapsto  x_{6},\\
 c_{2, 0, 0}^{0, 0, 2} &\longmapsto  x_{7},&\quad c_{0, 1, 0}^{3, 0, 0} &\longmapsto  x_{8},&\quad c_{0, 1, 0}^{0, 2, 0} &\longmapsto  x_{9},&\quad c_{0, 1, 0}^{0, 1, 1} &\longmapsto  x_{10},&\quad c_{0, 1, 0}^{0, 0, 2} &\longmapsto  x_{11},&\quad c_{1, 1, 0}^{3, 0, 0} &\longmapsto  x_{12},\\
  c_{1, 1, 0}^{2, 1, 0} &\longmapsto  x_{13},&\quad c_{1, 1, 0}^{2, 0, 1} &\longmapsto  x_{14},&\quad c_{1, 1, 0}^{0, 2, 0} &\longmapsto  x_{15},&\quad c_{1, 1, 0}^{0, 1, 1} &\longmapsto  x_{16},&\quad c_{1, 1, 0}^{0, 0, 2} &\longmapsto  x_{17},&\quad c_{0, 0, 1}^{3, 0, 0} &\longmapsto  x_{18},\\
   c_{0, 0, 1}^{0, 2, 0} &\longmapsto  x_{19},&\quad c_{0, 0, 1}^{0, 1, 1} &\longmapsto  x_{20},&\quad c_{0, 0, 1}^{0, 0, 2} &\longmapsto  x_{21},&\quad c_{1, 0, 1}^{3, 0, 0} &\longmapsto  x_{22},&\quad c_{1, 0, 1}^{2, 1, 0} &\longmapsto  x_{23},&\quad c_{1, 0, 1}^{2, 0, 1} &\longmapsto  x_{24},\\
    c_{1, 0, 1}^{0, 2, 0} &\longmapsto  x_{25},&\quad c_{1, 0, 1}^{0, 1, 1} &\longmapsto  x_{26},&\quad c_{1, 0, 1}^{0, 0, 2} &\longmapsto  x_{27}.
\end{alignat*}

\allowdisplaybreaks
\begin{eqnarray*}
{x}_{1}\ &\longmapsto& -{x}_{12}{x}_{16}{x}_{22}{x}_{26}+2{x}_{12}^{2}{x}_{26}^{2}-{x}_{16}{x}_{22}^{2}{x}_{27}+5{x}_{12}{x}_{22}{x}_{26}{x}_{27}+3{x}_{22}^{2}{x}_{27}^{2}-{x}_{6}{x}_{12}{x}_{
     22}\\
   &&  -{x}_{2}{x}_{16}{x}_{22}+{x}_{12}{x}_{16}{x}_{23}+{x}_{17}{x}_{22}{x}_{23}+{x}_{16}{x}_{22}{x}_{24}+{x}_{12}{x}_{14}{x}_{25}+3{x}_{2}{x}_{12}{x}_{26}\\
   && -3{x}_{12}{x}_{13}{x}_{26}-5{x}_{12}{x}_{24}{x}_{26}+4{x}_{2}{x}_{22}{x}_{27}-3{x}_{13}{x}_{22}{x}_{27}-6{x}_{22}{x}_{24}{x}_{27}+{x}_{2}^{2}+{x}_{3}{x}_{12}\\
   && -2{x}_{2}{x}_{13}+{x}_{13}^{2}-{x}_{8}{x}_{15}-{x}_{16}{x}_{18}-{x}_{12}{x}_{20}-{x}_{10}{x
     }_{22}-{x}_{21}{x}_{22}+{x}_{14}{x}_{23}-4{x}_{2}{x}_{24}\\
   && +3{x}_{13}{x}_{24}+3{x}_{24}^{2}+{x}_{1},\\
{x}_{3} &\longmapsto& 
     {x}_{12}{x}_{16}{x}_{25}+{x}_{17}{x}_{22}{x}_{25}-{x}_{12}{x}_{15}{x}_{26}-{x}_{16}{x}_{22}{x}_{26} +2{x}_{12}{x}_{26}^{2}-{x}_{12}{x}_{25}{x}_{27}\\*
   &&  +{x}_{22}{x}_{26}{x}_{27}+{x}_{5}{x}_{12}+{x}_{6}{x}_{22}+{
     x}_{16}{x}_{23}+{x}_{2}{x}_{26}-{x}_{13}{x}_{26}-{x}_{24}{x}_{26}-{x}_{23}{x}_{27}+{x}_{3}+{x}_{20},\\
{x}_{4} &\longmapsto& 
     {x}_{16}^{2}{x}_{22}-{x}_{15}{x}_{17}{x}_{22}+{x}_{12}{x}_{17}{x}_{25}+{x}_{17}{x}_{22}{x}_{26}+{x}_{2}{x}_{16}-{x}_{16}{x}_{24}+{x}_{4}+{x}_{10},\\
{x}_{5} &\longmapsto& 
     {x}_{16}{x}_{25}-{x}_{15}{x}_{26}+{x}_{26}^{2}-{x}_{25}{x}_{27}+{x}_{5},\\
{x}_{6} &\longmapsto& {x}_{17}{x}_{25}-{x}_{16}{x}_{26}+{x}_{6},\\
{x}_{7} &\longmapsto& {x}_{16}^{2}-{x}_{15}{x}_{17}+{x}_{17}{x}_{26}-{x}_{16}{x}_{27}+{x}_{7},\\
{x}_{8} &\longmapsto& -{x}_{12}^{2}{x}_{26}-{x}_{12}{x}_{22}{x}_{27}-{x}_{2}{x}_{12}+{x}_{12}{x}_{13}+{x}_{12}{x}_{24}+{x}_{8},\\
{x}_{9} &\longmapsto& -{x}_{12}{x}_{15}{x}_{26}+2{x}_{12}{x}_{26}^{2}-{x}_{15}{x}_{22}{x}_{27}+2{x}_{22}{x}_{26}{x}_{27}+{x}_{5}{x}_{12}-{x}_{2}{x}_{15}+{x}_{6}{x}_{22}\\
   && +2{x}_{16}{x}_{23}+{x}_{15}{x}_{24}-{x}_{14}{x}_{25}+2
     {x}_{2}{x}_{26}-{x}_{13}{x}_{26}-2{x}_{24}{x}_{26}-{x}_{23}{x}_{27}+{x}_{3}+{x}_{9}+2{x}_{20},\\
{x}_{11} &\longmapsto& -{x}_{12}{x}_{17}{x}_{26}-{x}_{17}{x}_{22}{x}_{27}-{x}_{7}{x}_{12}-{x}_{2}{x}_{17}+{x}_{17}{x}_{24}+{x}_{11},\\
{x}_{13} &\longmapsto& {x}_{12}{x}_{15}+{x}_{16}{x}_{22}-{x}_{12}{x}_{26}-{x}_{22}{x}_{27}+{x}_{13}+{x}_{24},\\
{x}_{14} &\longmapsto& {x}_{12}{x}_{16}+{x}_{17}{x}_{22}+{x}_{14},\\
{x}_{18} &\longmapsto& -{x}_{12}{x}_{22}{x}_{26}-{x}_{22}^{2}{x}_{27}-{x}_{2}{x}_{22}+{x}_{13}{x}_{22}-{x}_{12}{x}_{23}+{x}_{22}{x}_{24}+{x}_{18},\\
{x}_{19} &\longmapsto& -{x}_{12}{x}_{25}{x}_{26}-{x}_{22}{x}_{25}{x}_{27}-{x}_{15}{x}_{23}-{x}_{2}{x}_{25}+{x}_{13}{x}_{25}+{x}_{24}{x}_{25}+{x}_{23}{x}_{26}+{x}_{19},\\
{x}_{21} &\longmapsto& 2{x}_{12}{x}_{16}{x}_{26}+2{x}_{16}{x}_{22}{x}_{27}-{x}_{12}{x}_{26}{x}_{27}-{x}_{22}{x}_{27}^{2}+2{x}_{2}{x}_{16}-{x}_{7}{x}_{22}-{x}_{17}{x}_{23}-2{x}_{16}{x}_{24}\\
   && +{x}_{14}{x}_{26}-{x}_{2}{x}_{27}+{x}_{24}{x}_{27}+{x}_{4}+2{x}_{10}+{x}_{21},\\
{x}_{23} &\longmapsto&  {x}_{12}{x}_{25}+{x}_{22}{x}_{26}+{x}_{23}.
\end{eqnarray*}

\begin{equation*}
  \renewcommand{\arraystretch}{1.2}
  \begin{array}{lll}
(-{x}_{6}{x}_{13}+{x}_{5}{x}_{14}-{x}_{7}{x}_{23},& {x}_{6}{x}_{18}-{x}_{14}{x}_{19}-{x}_{4}{x}_{23},& -{x}_{5}{x}_{18}+{x}_{13}{x}_{19}-{x}_{9}{x}_{23},\\
 -{x}_{6}{x}_{8}+{x}_{3}{x}_{14}+{x}_{11}{x}_{23},& {x}_{5}{x}_{8}-{x}_{3}{x}_{13}+{x}_{21}{x}_{23},& -{x}_{7}{x}_{8}-{x}_{11}{x}_{13}-{x}_{14}{x}_{21},\\
  -{x}_{3}{x}_{18}+{x}_{8}{x}_{19}-{x}_{1}{x}_{23},& -{x}_{8}{x}_{9}+{x}_{1}{x}_{13}+{x}_{18}{x}_{21},& -{x}_{4}{x}_{13}-{x}_{9}{x}_{14}-{x}_{7}{x}_{18},\\
 {x}_{4}{x}_{5}+{x}_{6}{x}_{9}+{x}_{7}{x}_{19},& -{x}_{3}{x}_{7}-{x}_{5}{x}_{11}-{x}_{6}{x}_{21},& {x}_{1}{x}_{5}-{x}_{3}{x}_{9}+{x}_{19}{x}_{21},\\
    {x}_{1}{x}_{7}+{x}_{9}{x}_{11}-{x}_{4}{x}_{21},& -{x}_{4}{x}_{8}-{x}_{1}{x}_{14}+{x}_{11}{x}_{18},& -{x}_{3}{x}_{4}-{x}_{1}{x}_{6}+{x}_{11}{x}_{19}).
  \end{array}
  \end{equation*}

\begin{alignat*}{6}
x_{1}&\longmapsto -p_{1,2},&\quad x_{3}&\longmapsto p_{1,3},&\quad x_{4}&\longmapsto -p_{2,4},&\quad x_{5}&\longmapsto p_{0,3},&\quad x_{6}&\longmapsto -p_{3,4},&\quad x_{7}&\longmapsto p_{0,4},\\
x_{8}&\longmapsto p_{1,5},&\quad x_{9}&\longmapsto -p_{0,2},&\quad x_{11}&\longmapsto -p_{1,4},&\quad x_{13}&\longmapsto p_{0,5},&\quad x_{14}&\longmapsto p_{4,5},&\quad x_{18}&\longmapsto p_{2,5},\\
x_{19}&\longmapsto p_{2,3},&\quad x_{21}&\longmapsto -p_{0,1},&\quad x_{23}&\longmapsto p_{3,5}.
\end{alignat*}

\begin{eqnarray}\label{eq-equihilb-A232}
H(A_{{\lambda}_{232}};\mathbf{t})&=&K\left(\frac{\sqrt{t_{2}} \sqrt{t_{3}}}{t_{1}},\frac{t_{1} \sqrt{t_{3}}}{\sqrt{t_{2}}},\frac{t_{1} \sqrt{t_{2}}}{\sqrt{t_{3}}},\frac{t_{2}^{3/2}}{t_{1} \sqrt{t_{3}}},\frac{t_{3}^{3/2}}{t_{1} \sqrt{t_{2}}},\frac{t_{1}^2}{\sqrt{t_{2}} \sqrt{t_{3}}}\right)\notag\\
&&\Big/\left((1-t_{1})^3(1-t_{2})^3(1-t_{3})^3(1-t_{1}^2)\left(\frac{t_{1}^2-t_{2}^2}{t_{1}^2}\right)\left(\frac{t_{1}^2-t_{2} t_{3}}{t_{1}^2}\right)\left(\frac{t_{1}^2-t_{3}^2}{t_{1}^2}\right)\left(\frac{t_{2}-t_{1}^3}{t_{2}}\right)\right.\notag\\
&&\hphantom{\Big/\Big(}\left(\frac{t_{2}-t_{3}^2}{t_{2}}\right)\left(\frac{t_{2}-t_{1}^2}{t_{2}}\right)\left(\frac{t_{2}-t_{1} t_{3}}{t_{2}}\right)\left(\frac{t_{1}-t_{2}}{t_{1}}\right)^2\left(\frac{t_{1}-t_{3}}{t_{1}}\right)^2\left(\frac{t_{1} t_{2}-t_{3}^2}{t_{1} t_{2}}\right)\left(\frac{t_{3}-t_{1}^3}{t_{3}}\right)\left(\frac{t_{3}-t_{2}^2}{t_{3}}\right)\notag\\
&&\hphantom{\Big/\Big(}\left.\left(\frac{t_{3}-t_{1}^2}{t_{3}}\right)\left(\frac{t_{3}-t_{1} t_{2}}{t_{3}}\right)\left(\frac{t_{1} t_{3}-t_{2}^2}{t_{1} t_{3}}\right)\right).
\end{eqnarray}

\subsection{\texorpdfstring{$\boldsymbol{\left((2)\subset (3,2,1)\right)}$}{((2) in (3,2,1))}, extra.dim \texorpdfstring{$\boldsymbol{=8}$}{=8}}
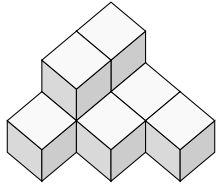
\label{sec:3D-2321}
\begin{center}
\begin{tikzpicture}[x=(220:0.6cm), y=(-40:0.6cm), z=(90:0.42cm)]

\foreach \m [count=\y] in {{2,2,1},{1,1},{1}}{
  \foreach \n [count=\x] in \m {
  \ifnum \n>0
      \foreach \z in {1,...,\n}{
        \draw [fill=gray!30] (\x+1,\y,\z) -- (\x+1,\y+1,\z) -- (\x+1, \y+1, \z-1) -- (\x+1, \y, \z-1) -- cycle;
        \draw [fill=gray!40] (\x,\y+1,\z) -- (\x+1,\y+1,\z) -- (\x+1, \y+1, \z-1) -- (\x, \y+1, \z-1) -- cycle;
        \draw [fill=gray!5] (\x,\y,\z)   -- (\x+1,\y,\z)   -- (\x+1, \y+1, \z)   -- (\x, \y+1, \z) -- cycle;  
      }
 \fi
 }
}    

\end{tikzpicture}
\end{center}
\[
I_{\lambda_{2321}}=\left(X_1^3,X_1^2 X_2, X_1 X_2^2,X_1^2 X_3, X_2^3,X_2 X_3,X_3^2\right).
\]

\begin{alignat*}{6}
c_{2, 0, 0}^{3, 0, 0} &\longmapsto  {x}_{1},&\quad c_{2, 0, 0}^{2, 1, 0} &\longmapsto   {x}_{2},&\quad c_{2, 0, 0}^{2, 0, 1} &\longmapsto {x}_{3},&\quad c_{2, 0, 0}^{0, 1, 1} &\longmapsto  {x}_{4},&\quad c_{2, 0, 0}^{1, 2, 0} &\longmapsto   {x}_{5},&\quad c_{2, 0, 0}^{0, 3, 0} &\longmapsto {x}_{6},\\
 c_{2, 0, 0}^{0, 0, 2} &\longmapsto  {x}_{7},&\quad c_{0, 1, 0}^{0, 1, 1} &\longmapsto  {x}_{8},&\quad c_{1, 1, 0}^{3, 0, 0} &\longmapsto {x}_{9},&\quad c_{1, 1, 0}^{2, 1, 0} &\longmapsto  {x}_{10},&\quad c_{1, 1, 0}^{2, 0, 1} &\longmapsto  {x}_{11},&\quad c_{1, 1, 0}^{0, 1, 1} &\longmapsto {x}_{12},\\
  c_{1, 1, 0}^{1, 2, 0} &\longmapsto  {x}_{13},&\quad c_{1, 1, 0}^{0, 3, 0} &\longmapsto  {x}_{14},&\quad c_{1, 1, 0}^{0, 0, 2} &\longmapsto {x}_{15},&\quad c_{0, 2, 0}^{3, 0, 0} &\longmapsto  {x}_{16},&\quad c_{0, 2, 0}^{2, 1, 0} &\longmapsto   {x}_{17},&\quad c_{0, 2, 0}^{2, 0, 1} &\longmapsto {x}_{18},\\
   c_{0, 2, 0}^{0, 1, 1} &\longmapsto  {x}_{19},&\quad c_{0, 2, 0}^{1, 2, 0} &\longmapsto  {x}_{20},&\quad c_{0, 2, 0}^{0, 3, 0} &\longmapsto {x}_{21},&\quad c_{0, 2, 0}^{0, 0, 2} &\longmapsto  {x}_{22},&\quad c_{0, 0, 1}^{3, 0, 0} &\longmapsto   {x}_{23},&\quad c_{0, 0, 1}^{0, 1, 1} &\longmapsto {x}_{24},\\
    c_{0, 0, 1}^{0, 0, 2} &\longmapsto  {x}_{25},&\quad c_{1, 0, 1}^{3, 0, 0} &\longmapsto  {x}_{26},&\quad c_{1, 0, 1}^{2, 1, 0} &\longmapsto {x}_{27},&\quad c_{1, 0, 1}^{2, 0, 1} &\longmapsto  {x}_{28},&\quad c_{1, 0, 1}^{0, 1, 1} &\longmapsto   {x}_{29},&\quad c_{1, 0, 1}^{1, 2, 0} &\longmapsto {x}_{30},\\
   c_{1, 0, 1}^{0, 3, 0} &\longmapsto  {x}_{31},&\quad c_{1, 0, 1}^{0, 0, 2} &\longmapsto  {x}_{32}.
\end{alignat*}
After the change of variables 
\begin{eqnarray*}
{x}_{1} &\longmapsto&   -{x}_{19}{x}_{26}{x}_{29}-{x}_{16}{x}_{29}^{2}-{x}_{12}{x}_{26}-{x}_{19}{x}_{27}-{x}_{9}{x}_{29}+2 {x}_{17}{x}_{29}+{x}_{1}+{x}_{10}+3 {x}_{20}-2 {x}_{28},\\
{x}_{3} &\longmapsto&   4 {x}_{19}^{2}{x}_{26}{x}_{29}^{2}+5 {x}_{12}{x}_{19}{x}_{26}{x}_{29}-2 {x}_{19}^{2}{x}_{27}{x}_{29}-2 {x}_{9}{x}_{19}{x}_{29}^{2}-{x}_{17}{x}_{19}{x}_{29}^{2}-{x}_{19}{x}_{26}{x}_{29}{x}_{32}\\
  	&&+{x}_{12}^{2}{x}_{26}-2 {x}_{4}{x}_{19}{x}_{26}-{x}_{12}{x}_{19}{x}_{27}-{x}_{9}{x}_{12}{x}_{29}-{x}_{1}{x}_{19}{x}_{29}+{x}_{15}{x}_{26}{x}_{29}+{x}_{18}{x}_{29}^{2}\\
  	&&-{x}_{12}{x}_{26}{x}_{32}+{x}_{19}{x}_{27}{x}_{32}+{x}_{4}{x}_{9}+{x}_{4}{x}_{17}+2 {x}_{2}{x}_{19}+{x}_{12}{x}_{20}+{x}_{19}{x}_{24}+{x}_{7}{x}_{26}\\
  	&&-{x}_{12}{x}_{28}+{x}_{11}{x}_{29}+{x}_{3}+{x}_{8},\\
{x}_{4} &\longmapsto&   -{x}_{19}{x}_{29}^{2}-{x}_{12}{x}_{29}+{x}_{4},\\
{x}_{5} &\longmapsto&  -{x}_{19}{x}_{26}{x}_{29}^{3}-{x}_{12}{x}_{26}{x}_{29}^{2}-3 {x}_{19}{x}_{27}{x}_{29}^{2}-{x}_{9}{x}_{29}^{3}+{x}_{4}{x}_{26}{x}_{29}-2 {x}_{12}{x}_{27}{x}_{29}-{x}_{1}{x}_{29}^{2}\\
   &&+{x}_{10}{x}_{29}^{2}+2 {x}_{19}{x}_{29}{x}_{30}+{x}_{27}{x}_{29}{x}_{32}+{x}_{4}{x}_{27}-{x}_{13}{x}_{29}+{x}_{24}{x}_{29}+{x}_{12}{x}_{30}-{x}_{30}{x}_{32}+{x}_{5},\\
{x}_{6} &\longmapsto&  -{x}_{19}{x}_{26}{x}_{29}^{4}-{x}_{12}{x}_{26}{x}_{29}^{3}+{x}_{4}{x}_{26}{x}_{29}^{2}-{x}_{10}{x}_{29}^{3}-3 {x}_{19}{x}_{29}^{2}{x}_{30}+{x}_{13}{x}_{29}^{2}-{x}_{21}{x}_{29}^{2}\\
	&&-{x}_{24}{x}_{29}^{2}-2 {x}_{12}{x}_{29}{x}_{30}+2 {x}_{19}{x}_{29}{x}_{31}+{x}_{29}{x}_{30}{x}_{32}-{x}_{5}{x}_{29}-{x}_{14}{x}_{29}+{x}_{4}{x}_{30}\\
	&&+{x}_{12}{x}_{31}-{x}_{31}{x}_{32}+{x}_{6},\\
{x}_{7} &\longmapsto&   {x}_{19}^{2}{x}_{29}^{2}+2 {x}_{12}{x}_{19}{x}_{29}-{x}_{19}{x}_{29}{x}_{32}+{x}_{12}^{2}+{x}_{4}{x}_{19}+{x}_{15}{x}_{29}-{x}_{12}{x}_{32}+{x}_{7},\\
{x}_{9} &\longmapsto& -{x}_{19}{x}_{26}-{x}_{16}{x}_{29}+{x}_{9}+{x}_{17},\\
{x}_{10} &\longmapsto&  -{x}_{19}{x}_{26}{x}_{29}-{x}_{19}{x}_{27}-{x}_{9}{x}_{29}+{x}_{17}{x}_{29}+{x}_{10}+2 {x}_{20}-{x}_{28},\\
{x}_{11} &\longmapsto&  2 {x}_{19}^{2}{x}_{26}{x}_{29}+{x}_{12}{x}_{19}{x}_{26}-{x}_{12}{x}_{16}{x}_{29}+{x}_{22}{x}_{26}{x}_{29}-{x}_{19}{x}_{26}{x}_{32}+{x}_{9}{x}_{12}+{x}_{4}{x}_{16}\\
   &&+{x}_{12}{x}_{17}-{x}_{1}{x}_{19}+{x}_{15}{x}_{26}+{x}_{18}{x}_{29}+{x}_{11},\\
{x}_{13} &\longmapsto&  {x}_{19}{x}_{26}{x}_{29}^{2}-{x}_{16}{x}_{29}^{3}+2 {x}_{12}{x}_{26}{x}_{29}-2 {x}_{19}{x}_{27}{x}_{29}-{x}_{9}{x}_{29}^{2}+{x}_{17}{x}_{29}^{2}-{x}_{4}{x}_{26}\\
  &&+{x}_{27}{x}_{32}+2 {x}_{2}+{x}_{13}+{x}_{24},\\
{x}_{14} &\longmapsto&   -{x}_{19}{x}_{26}{x}_{29}^{3}-3 {x}_{19}{x}_{27}{x}_{29}^{2}-{x}_{9}{x}_{29}^{3}-{x}_{17}{x}_{29}^{3}+{x}_{4}{x}_{26}{x}_{29}
	-2 {x}_{12}{x}_{27}{x}_{29}-{x}_{1}{x}_{29}^{2}\\
	&&+{x}_{10}{x}_{29}^{2}+{x}_{28}{x}_{29}^{2}+2 {x}_{19}{x}_{29}{x}_{30}+{x}_{27}{x}_{29}{x}_{32}+{x}_{4}{x}_{27}-2 {x}_{13}{x}_{29}+{x}_{24}{x}_{29}\\
	&&+2 {x}_{12}{x}_{30}-{x}_{30}{x}_{32}+{x}_{5}+{x}_{14},\\
{x}_{15} &\longmapsto&   2 {x}_{19}^{2}{x}_{29}+2 {x}_{12}{x}_{19}+{x}_{22}{x}_{29}-{x}_{19}{x}_{32}+{x}_{15},\\
{x}_{17} &\longmapsto&  -{x}_{16}{x}_{29}+{x}_{17},\\
{x}_{18} &\longmapsto&   {x}_{12}{x}_{16}+{x}_{17}{x}_{19}+{x}_{22}{x}_{26}+{x}_{18},\\
{x}_{21} &\longmapsto&  3 {x}_{19}{x}_{26}{x}_{29}^{2}-{x}_{16}{x}_{29}^{3}+3 {x}_{12}{x}_{26}{x}_{29}-4 {x}_{19}{x}_{27}{x}_{29}-2 {x}_{9}{x}_{29}^{2}+{x}_{17}{x}_{29}^{2}-2 {x}_{4}{x}_{26}\\
	&&-{x}_{12}{x}_{27}-{x}_{1}{x}_{29}+{x}_{10}{x}_{29}+{x}_{19}{x}_{30}+2 {x}_{27}{x}_{32}+3 {x}_{2}+{x}_{13}+{x}_{21}+2 {x}_{24},\\ 
{x}_{22} &\longmapsto&   {x}_{19}^{2}+{x}_{22},\\
{x}_{23} &\longmapsto&  -{x}_{17}{x}_{26}{x}_{29}-{x}_{20}{x}_{26}+{x}_{26}{x}_{28}+{x}_{23},\\
{x}_{24} &\longmapsto&  {x}_{19}{x}_{26}{x}_{29}^{2}+{x}_{12}{x}_{26}{x}_{29}-2 {x}_{19}{x}_{27}{x}_{29}-{x}_{9}{x}_{29}^{2}-{x}_{17}{x}_{29}^{2}-{x}_{4}{x}_{26}
	-{x}_{12}{x}_{27}\\
	&&-{x}_{1}{x}_{29}-{x}_{20}{x}_{29}+{x}_{28}{x}_{29}+{x}_{27}{x}_{32}+{x}_{2}+{x}_{24},\\
{x}_{25} &\longmapsto&  4 {x}_{19}^{2}{x}_{26}{x}_{29}^{2}+4 {x}_{12}{x}_{19}{x}_{26}{x}_{29}-8 {x}_{19}^{2}{x}_{27}{x}_{29}-4 {x}_{9}{x}_{19}{x}_{29}^{2}-4 {x}_{4}{x}_{19}{x}_{26}
	-4 {x}_{12}{x}_{19}{x}_{27}\\
	&&+{x}_{4}{x}_{16}{x}_{29}+2 {x}_{12}{x}_{17}{x}_{29}-4 {x}_{1}{x}_{19}{x}_{29}-{x}_{22}{x}_{27}{x}_{29}+3 {x}_{18}{x}_{29}^{2}+4 {x}_{19}{x}_{27}{x}_{32}\\
	&&-{x}_{17}{x}_{29}{x}_{32}+{x}_{4}{x}_{9}+{x}_{4}{x}_{17}+4 {x}_{2}{x}_{19}+2 {x}_{12}{x}_{20}+3 {x}_{19}{x}_{24}-{x}_{15}{x}_{27}-2 {x}_{12}{x}_{28}\\
	&&+2 {x}_{11}{x}_{29}-{x}_{22}{x}_{30}-{x}_{20}{x}_{32}+{x}_{28}{x}_{32}+{x}_{3}+2 {x}_{8}+{x}_{25},\\
{x}_{27} &\longmapsto&   {x}_{26}{x}_{29}+{x}_{27},\\
{x}_{28} &\longmapsto&  -2 {x}_{19}{x}_{26}{x}_{29}-{x}_{16}{x}_{29}^{2}-{x}_{12}{x}_{26}+2 {x}_{17}{x}_{29}+{x}_{26}{x}_{32}+{x}_{1}+{x}_{10}+2 {x}_{20}-{x}_{28},\\
{x}_{30} &\longmapsto&  {x}_{26}{x}_{29}^{2}+{x}_{30},\\
{x}_{31} &\longmapsto&   {x}_{26}{x}_{29}^{3}+{x}_{31}, 
\end{eqnarray*}
the Haiman ideal becomes
\begin{align}\label{eq-generators-haiman-8points}
  (&\begin{aligned}[t]
 -{x}_{3}{x}_{16}-{x}_{11}{x}_{17}+{x}_{22}{x}_{23}-{x}_{18}{x}_{28},
    &\quad -{x}_{5}{x}_{15}-{x}_{6}{x}_{22}+{x}_{7}{x}_{24}-{x}_{4}{x}_{25},\nonumber\\
 {x}_{4}{x}_{23}-{x}_{3}{x}_{27}-{x}_{11}{x}_{30}-{x}_{18}{x}_{31},
  &\quad{x}_{5}{x}_{9}+{x}_{6}{x}_{16}-{x}_{1}{x}_{24}+{x}_{25}{x}_{27},\nonumber\\
 {x}_{1}{x}_{3}-{x}_{11}{x}_{13}-{x}_{14}{x}_{18}-{x}_{7}{x}_{23},
  &\quad-{x}_{3}{x}_{9}-{x}_{10}{x}_{11}+{x}_{18}{x}_{21}+{x}_{15}{x}_{23},\nonumber\\
 {x}_{5}{x}_{10}+{x}_{6}{x}_{17}+{x}_{13}{x}_{24}+{x}_{25}{x}_{30},
  &\quad-{x}_{5}{x}_{21}+{x}_{14}{x}_{24}+{x}_{6}{x}_{28}+{x}_{25}{x}_{31},\end{aligned}\\
& {x}_{17}{x}_{21}{x}_{27}+{x}_{10}{x}_{27}{x}_{28}-{x}_{16}{x}_{21}{x}_{30}-{x}_{9}{x}_{28}{x}_{30}-{x}_{10}{x}_{16}{x}_{31}+{x}_{9}{x}_{17}{x}_{31}+{x}_{23}{x}_{24},\nonumber\\
& {x}_{4}{x}_{10}{x}_{16}-{x}_{4}{x}_{9}{x}_{17}+{x}_{15}{x}_{17}{x}_{27}-{x}_{10}{x}_{22}{x}_{27}-{x}_{15}{x}_{16}{x}_{30}+{x}_{9}{x}_{22}{x}_{30}-{x}_{18}{x}_{24},\nonumber\\
& {x}_{7}{x}_{10}{x}_{16}+{x}_{13}{x}_{15}{x}_{16}-{x}_{7}{x}_{9}{x}_{17}+{x}_{1}{x}_{15}{x}_{17}-{x}_{1}{x}_{10}{x}_{22}-{x}_{9}{x}_{13}{x}_{22}-{x}_{18}{x}_{25},\nonumber\\
& {x}_{4}{x}_{13}{x}_{16}+{x}_{1}{x}_{4}{x}_{17}-{x}_{7}{x}_{17}{x}_{27}-{x}_{13}{x}_{22}{x}_{27}+{x}_{7}{x}_{16}{x}_{30}-{x}_{1}{x}_{22}{x}_{30}+{x}_{5}{x}_{18},\nonumber\\
& -{x}_{14}{x}_{17}{x}_{27}+{x}_{13}{x}_{27}{x}_{28}+{x}_{14}{x}_{16}{x}_{30}+{x}_{1}{x}_{28}{x}_{30}-{x}_{13}{x}_{16}{x}_{31}-{x}_{1}{x}_{17}{x}_{31}-{x}_{5}{x}_{23},\nonumber\\
& -{x}_{4}{x}_{17}{x}_{21}-{x}_{4}{x}_{10}{x}_{28}+{x}_{21}{x}_{22}{x}_{30}+{x}_{15}{x}_{28}{x}_{30}-{x}_{15}{x}_{17}{x}_{31}+{x}_{10}{x}_{22}{x}_{31}-{x}_{3}{x}_{24},\nonumber\\
& {x}_{4}{x}_{16}{x}_{21}-{x}_{21}{x}_{22}{x}_{27}+{x}_{4}{x}_{9}{x}_{28}-{x}_{15}{x}_{27}{x}_{28}+{x}_{15}{x}_{16}{x}_{31}-{x}_{9}{x}_{22}{x}_{31}-{x}_{11}{x}_{24},\nonumber\\
& {x}_{10}{x}_{14}{x}_{16}-{x}_{9}{x}_{14}{x}_{17}+{x}_{13}{x}_{16}{x}_{21}+{x}_{1}{x}_{17}{x}_{21}+{x}_{1}{x}_{10}{x}_{28}+{x}_{9}{x}_{13}{x}_{28}+{x}_{23}{x}_{25},\nonumber\\
& {x}_{10}{x}_{14}{x}_{27}+{x}_{13}{x}_{21}{x}_{27}-{x}_{9}{x}_{14}{x}_{30}+{x}_{1}{x}_{21}{x}_{30}+{x}_{1}{x}_{10}{x}_{31}+{x}_{9}{x}_{13}{x}_{31}-{x}_{6}{x}_{23},\nonumber\\
& {x}_{14}{x}_{15}{x}_{16}-{x}_{7}{x}_{16}{x}_{21}-{x}_{9}{x}_{14}{x}_{22}+{x}_{1}{x}_{21}{x}_{22}-{x}_{7}{x}_{9}{x}_{28}+{x}_{1}{x}_{15}{x}_{28}+{x}_{11}{x}_{25},\nonumber\\
& -{x}_{4}{x}_{14}{x}_{16}+{x}_{14}{x}_{22}{x}_{27}-{x}_{1}{x}_{4}{x}_{28}+{x}_{7}{x}_{27}{x}_{28}-{x}_{7}{x}_{16}{x}_{31}+{x}_{1}{x}_{22}{x}_{31}+{x}_{5}{x}_{11},\nonumber\\
& -{x}_{1}{x}_{4}{x}_{10}-{x}_{4}{x}_{9}{x}_{13}+{x}_{7}{x}_{10}{x}_{27}+{x}_{13}{x}_{15}{x}_{27}-{x}_{7}{x}_{9}{x}_{30}+{x}_{1}{x}_{15}{x}_{30}+{x}_{6}{x}_{18},\nonumber\\
& {x}_{4}{x}_{14}{x}_{17}-{x}_{4}{x}_{13}{x}_{28}-{x}_{14}{x}_{22}{x}_{30}-{x}_{7}{x}_{28}{x}_{30}+{x}_{7}{x}_{17}{x}_{31}+{x}_{13}{x}_{22}{x}_{31}+{x}_{3}{x}_{5},\nonumber\\
& -{x}_{14}{x}_{15}{x}_{17}+{x}_{7}{x}_{17}{x}_{21}+{x}_{10}{x}_{14}{x}_{22}+{x}_{13}{x}_{21}{x}_{22}+{x}_{7}{x}_{10}{x}_{28}+{x}_{13}{x}_{15}{x}_{28}+{x}_{3}{x}_{25},\nonumber\\
& {x}_{4}{x}_{9}{x}_{14}-{x}_{1}{x}_{4}{x}_{21}-{x}_{14}{x}_{15}{x}_{27}+{x}_{7}{x}_{21}{x}_{27}+{x}_{7}{x}_{9}{x}_{31}-{x}_{1}{x}_{15}{x}_{31}+{x}_{6}{x}_{11},\nonumber\\
& -{x}_{4}{x}_{10}{x}_{14}-{x}_{4}{x}_{13}{x}_{21}+{x}_{14}{x}_{15}{x}_{30}-{x}_{7}{x}_{21}{x}_{30}-{x}_{7}{x}_{10}{x}_{31}-{x}_{13}{x}_{15}{x}_{31}+{x}_{3}{x}_{6}).
\end{align}

After the change of variables
\begin{alignat*}{6}
y_1&\longmapsto x_5,&\quad y_2&\longmapsto -x_{14},&\quad y_3&\longmapsto x_6,&\quad y_4&\longmapsto x_{13},&\quad y_5&\longmapsto -x_{22},&\quad y_6&\longmapsto -x_{28},\\
y_7&\longmapsto  x_{10},&\quad y_8&\longmapsto x_{30},&\quad y_9&\longmapsto x_7,&\quad y_{10}&\longmapsto x_{21},&\quad y_{11}&\longmapsto x_{24},&\quad y_{12}&\longmapsto x_{25},\\
y_{13}&\longmapsto x_3,&\quad y_{14}&\longmapsto -x_1,&\quad y_{15}&\longmapsto x_4,&\quad y_{16}&\longmapsto -x_{11},&\quad y_{17}&\longmapsto x_{16},&\quad y_{18}&\longmapsto x_{15},\\
y_{19}&\longmapsto x_{17},&\quad y_{20}&\longmapsto -x_{18},&\quad y_{21}&\longmapsto -x_{27},&\quad y_{24}&\longmapsto x_9,&\quad y_{27}&\longmapsto x_{31},&\quad
y_{29}&\longmapsto x_{23},
\end{alignat*}
the function $F_{1321}(y_1,\dots,y_{29})$ in Section~\ref{sec:3D-1321} becomes a potential function for (\ref{eq-generators-haiman-8points}).

\section{The change of variables of \texorpdfstring{$\boldsymbol{A_{\lambda_{1311}}}$}{A\_\{lambda\_1311\}}}
\label{sec:change-var-1311}

There are three steps.

\begin{enumerate}[label={\it Step}~\arabic*:, ref=\arabic*]
\item Use the  equations of minimal degree 1 to solve $x_1,x_2,x_3,x_4,x_8$:
\begin{eqnarray*}
{x}_{1}& = & \frac{{x}_{5}{x}_{9}{x}_{18}+{x}_{13}{x}_{18}^{2}+{x}_{12}{x}_{18}{x}_{24}+{x}_{9}{x}_{11}+{x}_{9}{x}_{15}+{x}_{10}{x}_{24}+{x}_{5}}{1 - x_{18}x_{9}},\\[1ex]
{x}_{2}& = &\frac{{x}_{5}{x}_{18}^{2}+{x}_{18}^{2}{x}_{22}+{x}_{18}{x}_{21}{x}_{24}+2\,{x}_{15}{x}_{18}+{x}_{9}{x}_{20}+{x}_{19}{x}_{24}}
{1 - x_{18}x_{9}},\\[1ex]
{x}_{3}
& =&\left(-{x}_{9}^{2}{x}_{11}{x}_{12}{x}_{18}^{4}{x}_{24}-{x}_{9}{x}_{10}{x}_{13}{x}_{18}^{5}{x}_{24}+{x}_{9}^{2}{x}_{13}{x}_{18}^{4}{x}_{19}{x}_{24}+{x}_{9}^{3}{x}_{12}{x}_{18}^{3}{x}_{20}{x}_{24}\right.\\
&&\hphantom{\big(}+{x}_{9}^{3}{x}_{11}{x}_{18}^{3}{x}_{21}{x}_{24}-{x}_{9}^{4}{x}_{18}^{2}{x}_{20}{x}_{21}{x}_{24}+{x}_{9}^{2}{x}_{10}{x}_{18}^{4}{x}_{22}{x}_{24}-{x}_{9}^{3}{x}_{18}^{3}{x}_{19}{x}_{22}{x}_{24}\\
&&\hphantom{\big(}-{x}_{9}{x}_{10}{x}_{12}{x}_{18}^{4}{x}_{24}^{2}+2 {x}_{9}^{2}{x}_{10}{x}_{18}^{3}{x}_{21}{x}_{24}^{2}-{x}_{9}^{3}{x}_{18}^{2}{x}_{19}{x}_{21}{x}_{24}^{2}+{x}_{9}^{2}{x}_{12}{x}_{18}^{4}{x}_{24}{x}_{25}\\
&&\hphantom{\big(}-{x}_{9}^{3}{x}_{18}^{3}{x}_{21}{x}_{24}{x}_{25}-{x}_{9}^{2}{x}_{10}{x}_{18}^{4}{x}_{24}{x}_{27}+{x}_{9}^{3}{x}_{18}^{3}{x}_{19}{x}_{24}{x}_{27}-{x}_{9}^{2}{x}_{10}{x}_{11}{x}_{18}^{3}{x}_{24}\\
&&\hphantom{\big(}+{x}_{9}^{2}{x}_{10}{x}_{15}{x}_{18}^{3}{x}_{24}-{x}_{5}{x}_{9}{x}_{12}{x}_{18}^{4}{x}_{24}-{x}_{9}^{3}{x}_{15}{x}_{18}^{2}{x}_{19}{x}_{24}+{x}_{9}^{3}{x}_{10}{x}_{18}^{2}{x}_{20}{x}_{24}\\
&&\hphantom{\big(}+{x}_{5}{x}_{9}^{2}{x}_{18}^{3}{x}_{21}{x}_{24}+{x}_{9}{x}_{13}{x}_{18}^{4}{x}_{21}{x}_{24}-{x}_{9}^{2}{x}_{18}^{3}{x}_{21}{x}_{22}{x}_{24}-{x}_{9}{x}_{10}^{2}{x}_{18}^{3}{x}_{24}^{2}\\
&&\hphantom{\big(}+{x}_{9}^{2}{x}_{10}{x}_{18}^{2}{x}_{19}{x}_{24}^{2}+{x}_{9}{x}_{12}{x}_{18}^{3}{x}_{21}{x}_{24}^{2}-{x}_{9}^{2}{x}_{18}^{2}{x}_{21}^{2}{x}_{24}^{2}-{x}_{9}^{3}{x}_{18}^{2}{x}_{19}{x}_{24}{x}_{25}\\
&&\hphantom{\big(}+{x}_{9}^{3}{x}_{18}^{3}{x}_{25}^{2}+{x}_{9}^{3}{x}_{10}{x}_{18}^{3}{x}_{26}-{x}_{9}^{4}{x}_{18}^{2}{x}_{19}{x}_{26}-{x}_{9}{x}_{12}{x}_{18}^{4}{x}_{24}{x}_{27}+{x}_{9}^{2}{x}_{18}^{4}{x}_{27}^{2}\\
&&\hphantom{\big(}-{x}_{9}{x}_{12}{x}_{18}^{5}{x}_{28}+{x}_{9}^{2}{x}_{18}^{4}{x}_{21}{x}_{28}+{x}_{7}{x}_{9}^{3}{x}_{18}^{3}-{x}_{5}{x}_{9}^{2}{x}_{11}{x}_{18}^{3}-2 {x}_{5}{x}_{9}^{2}{x}_{15}{x}_{18}^{3}\\
&&\hphantom{\big(}-{x}_{9}^{2}{x}_{17}{x}_{18}^{4}+{x}_{9}^{2}{x}_{13}{x}_{18}^{3}{x}_{20}-{x}_{9}^{2}{x}_{11}{x}_{18}^{3}{x}_{22}-{x}_{9}^{2}{x}_{15}{x}_{18}^{3}{x}_{22}-{x}_{5}{x}_{9}{x}_{10}{x}_{18}^{3}{x}_{24}\\
&&\hphantom{\big(}+2 {x}_{9}{x}_{11}{x}_{12}{x}_{18}^{3}{x}_{24}-{x}_{9}{x}_{12}{x}_{15}{x}_{18}^{3}{x}_{24}-2 {x}_{9}^{2}{x}_{16}{x}_{18}^{3}{x}_{24}+{x}_{10}{x}_{13}{x}_{18}^{4}{x}_{24}\\
&&\hphantom{\big(}-{x}_{5}{x}_{9}^{2}{x}_{18}^{2}{x}_{19}{x}_{24}-2 {x}_{9}{x}_{13}{x}_{18}^{3}{x}_{19}{x}_{24}-2 {x}_{9}^{2}{x}_{12}{x}_{18}^{2}{x}_{20}{x}_{24}-2 {x}_{9}^{2}{x}_{11}{x}_{18}^{2}{x}_{21}{x}_{24}\\
&&\hphantom{\big(}-{x}_{9}^{2}{x}_{15}{x}_{18}^{2}{x}_{21}{x}_{24}+{x}_{9}^{3}{x}_{18}{x}_{20}{x}_{21}{x}_{24}-2 {x}_{9}{x}_{10}{x}_{18}^{3}{x}_{22}{x}_{24}+2 {x}_{9}^{2}{x}_{18}^{2}{x}_{19}{x}_{22}{x}_{24}\\
&&\hphantom{\big(}+{x}_{10}{x}_{12}{x}_{18}^{3}{x}_{24}^{2}-3 {x}_{9}{x}_{10}{x}_{18}^{2}{x}_{21}{x}_{24}^{2}+{x}_{9}^{2}{x}_{18}{x}_{19}{x}_{21}{x}_{24}^{2}+{x}_{5}{x}_{9}^{2}{x}_{18}^{3}{x}_{25}\\
&&\hphantom{\big(}-3 {x}_{9}{x}_{12}{x}_{18}^{3}{x}_{24}{x}_{25}+3 {x}_{9}^{2}{x}_{18}^{2}{x}_{21}{x}_{24}{x}_{25}-{x}_{9}^{2}{x}_{12}{x}_{18}^{3}{x}_{26}+{x}_{9}^{3}{x}_{18}^{2}{x}_{21}{x}_{26}\\
&&\hphantom{\big(}+{x}_{9}^{2}{x}_{15}{x}_{18}^{3}{x}_{27}+3 {x}_{9}{x}_{10}{x}_{18}^{3}{x}_{24}{x}_{27}-3 {x}_{9}^{2}{x}_{18}^{2}{x}_{19}{x}_{24}{x}_{27}-{x}_{9}^{2}{x}_{18}^{3}{x}_{25}{x}_{27}\\
&&\hphantom{\big(}+{x}_{9}{x}_{10}{x}_{18}^{4}{x}_{28}-{x}_{9}^{2}{x}_{18}^{3}{x}_{19}{x}_{28}+{x}_{9}^{2}{x}_{18}^{3}{x}_{24}{x}_{29}-2 {x}_{9}^{2}{x}_{11}{x}_{15}{x}_{18}^{2}-3 {x}_{9}^{2}{x}_{15}^{2}{x}_{18}^{2}\\
&&\hphantom{\big(}-{x}_{5}^{2}{x}_{9}{x}_{18}^{3}-2 {x}_{9}^{3}{x}_{15}{x}_{18}{x}_{20}-{x}_{9}^{4}{x}_{20}^{2}-2 {x}_{5}{x}_{9}{x}_{18}^{3}{x}_{22}-{x}_{9}{x}_{18}^{3}{x}_{22}^{2}\\
&&\hphantom{\big(}+{x}_{9}{x}_{10}{x}_{11}{x}_{18}^{2}{x}_{24}-2 {x}_{9}{x}_{10}{x}_{15}{x}_{18}^{2}{x}_{24}+{x}_{5}{x}_{12}{x}_{18}^{3}{x}_{24}-{x}_{9}^{2}{x}_{11}{x}_{18}{x}_{19}{x}_{24}\\
&&\hphantom{\big(}-{x}_{9}^{2}{x}_{15}{x}_{18}{x}_{19}{x}_{24}-{x}_{9}^{2}{x}_{10}{x}_{18}{x}_{20}{x}_{24}-2 {x}_{9}^{3}{x}_{19}{x}_{20}{x}_{24}-2 {x}_{5}{x}_{9}{x}_{18}^{2}{x}_{21}{x}_{24}\\
&&\hphantom{\big(}-{x}_{13}{x}_{18}^{3}{x}_{21}{x}_{24}+{x}_{10}^{2}{x}_{18}^{2}{x}_{24}^{2}-2 {x}_{9}{x}_{10}{x}_{18}{x}_{19}{x}_{24}^{2}-{x}_{9}^{2}{x}_{19}^{2}{x}_{24}^{2}-{x}_{12}{x}_{18}^{2}{x}_{21}{x}_{24}^{2}\\
&&\hphantom{\big(}+{x}_{9}{x}_{18}{x}_{21}^{2}{x}_{24}^{2}+{x}_{9}{x}_{10}{x}_{18}^{2}{x}_{24}{x}_{25}+2 {x}_{9}^{2}{x}_{18}{x}_{19}{x}_{24}{x}_{25}-2 {x}_{9}^{2}{x}_{18}^{2}{x}_{25}^{2}-{x}_{9}^{2}{x}_{10}{x}_{18}^{2}{x}_{26}\\
&&\hphantom{\big(}+2 {x}_{9}^{3}{x}_{18}{x}_{19}{x}_{26}+{x}_{12}{x}_{18}^{3}{x}_{24}{x}_{27}-2 {x}_{9}{x}_{18}^{3}{x}_{27}^{2}+{x}_{12}{x}_{18}^{4}{x}_{28}-2 {x}_{9}{x}_{18}^{3}{x}_{21}{x}_{28}\\
&&\hphantom{\big(}-2 {x}_{7}{x}_{9}^{2}{x}_{18}^{2}+{x}_{5}{x}_{9}{x}_{11}{x}_{18}^{2}-{x}_{5}{x}_{9}{x}_{15}{x}_{18}^{2}+2 {x}_{9}{x}_{17}{x}_{18}^{3}-{x}_{5}{x}_{9}^{2}{x}_{18}{x}_{20}\\
&&\hphantom{\big(}-{x}_{9}{x}_{13}{x}_{18}^{2}{x}_{20}+{x}_{9}{x}_{11}{x}_{18}^{2}{x}_{22}-3 {x}_{9}{x}_{15}{x}_{18}^{2}{x}_{22}-3 {x}_{9}^{2}{x}_{18}{x}_{20}{x}_{22}+{x}_{5}{x}_{10}{x}_{18}^{2}{x}_{24}\\
&&\hphantom{\big(}-{x}_{11}{x}_{12}{x}_{18}^{2}{x}_{24}+{x}_{12}{x}_{15}{x}_{18}^{2}{x}_{24}+4 {x}_{9}{x}_{16}{x}_{18}^{2}{x}_{24}-{x}_{5}{x}_{9}{x}_{18}{x}_{19}{x}_{24}+{x}_{13}{x}_{18}^{2}{x}_{19}{x}_{24}\\
&&\hphantom{\big(}+2 {x}_{9}{x}_{12}{x}_{18}{x}_{20}{x}_{24}+{x}_{9}{x}_{11}{x}_{18}{x}_{21}{x}_{24}-2 {x}_{9}{x}_{15}{x}_{18}{x}_{21}{x}_{24}-2 {x}_{9}^{2}{x}_{20}{x}_{21}{x}_{24}\\
&&\hphantom{\big(}+{x}_{10}{x}_{18}^{2}{x}_{22}{x}_{24}-4 {x}_{9}{x}_{18}{x}_{19}{x}_{22}{x}_{24}+{x}_{10}{x}_{18}{x}_{21}{x}_{24}^{2}-2 {x}_{9}{x}_{19}{x}_{21}{x}_{24}^{2}\\
&&\hphantom{\big(}-2 {x}_{5}{x}_{9}{x}_{18}^{2}{x}_{25}+2 {x}_{12}{x}_{18}^{2}{x}_{24}{x}_{25}-3 {x}_{9}{x}_{18}{x}_{21}{x}_{24}{x}_{25}+{x}_{9}{x}_{12}{x}_{18}^{2}{x}_{26}-2 {x}_{9}^{2}{x}_{18}{x}_{21}{x}_{26}\\
&&\hphantom{\big(}-2 {x}_{9}{x}_{15}{x}_{18}^{2}{x}_{27}-2 {x}_{10}{x}_{18}^{2}{x}_{24}{x}_{27}+3 {x}_{9}{x}_{18}{x}_{19}{x}_{24}{x}_{27}+2 {x}_{9}{x}_{18}^{2}{x}_{25}{x}_{27}-{x}_{10}{x}_{18}^{3}{x}_{28}\\
&&\hphantom{\big(}+2 {x}_{9}{x}_{18}^{2}{x}_{19}{x}_{28}-2 {x}_{9}{x}_{18}^{2}{x}_{24}{x}_{29}+2 {x}_{9}{x}_{11}{x}_{15}{x}_{18}-2 {x}_{9}^{2}{x}_{15}{x}_{20}+{x}_{10}{x}_{15}{x}_{18}{x}_{24}\\
&&\hphantom{\big(}+{x}_{9}{x}_{11}{x}_{19}{x}_{24}-2 {x}_{9}{x}_{15}{x}_{19}{x}_{24}-{x}_{5}{x}_{18}{x}_{21}{x}_{24}-{x}_{18}{x}_{21}{x}_{22}{x}_{24}+{x}_{10}{x}_{19}{x}_{24}^{2}-{x}_{21}^{2}{x}_{24}^{2}\\
&&\hphantom{\big(}-{x}_{10}{x}_{18}{x}_{24}{x}_{25}-{x}_{9}{x}_{19}{x}_{24}{x}_{25}+{x}_{9}{x}_{18}{x}_{25}^{2}-{x}_{9}^{2}{x}_{19}{x}_{26}+{x}_{18}^{2}{x}_{27}^{2}+{x}_{18}^{2}{x}_{21}{x}_{28}\\
&&\hphantom{\big(}+{x}_{7}{x}_{9}{x}_{18}-{x}_{5}{x}_{15}{x}_{18}-{x}_{17}{x}_{18}^{2}-{x}_{5}{x}_{9}{x}_{20}+{x}_{9}{x}_{20}{x}_{22}-2 {x}_{16}{x}_{18}{x}_{24}-{x}_{12}{x}_{20}{x}_{24}\\
&&\hphantom{\big(}-{x}_{15}{x}_{21}{x}_{24}+{x}_{19}{x}_{22}{x}_{24}+{x}_{5}{x}_{18}{x}_{25}+{x}_{21}{x}_{24}{x}_{25}+{x}_{9}{x}_{21}{x}_{26}+{x}_{15}{x}_{18}{x}_{27}\\
&&\hphantom{\big(}\left.-{x}_{19}{x}_{24}{x}_{27}-{x}_{18}{x}_{25}{x}_{27}-{x}_{18}{x}_{19}{x}_{28}+{x}_{18}{x}_{24}{x}_{29}-{x}_{15}^{2}\right)\\
&&\big/\left((-1 + x_{18}x_{9})^2 (-1 + 2 x_{18} x_{9})\right),\\[1ex]
x_4&=&\frac{{x}_{5}{x}_{18}{x}_{24}+{x}_{18}{x}_{24}{x}_{27}+{x}_{18}^{2}{x}_{28}+{x}_{15}{x}_{24}+{x}_{24}{x}_{25}+{x
      }_{9}{x}_{26}}
{1 - x_{18} x_{9}},\\[1ex]
x_8&=&\left(-{x}_{6}{x}_{9}^{2}{x}_{12}{x}_{18}^{3}+{x}_{5}{x}_{9}{x}_{10}{x}_{12}{x}_{18}^{3}+{x}_{9}^{2}{x}_{12}{x}_{16}{x}_{18}^{3}+{x}_{9}^{2}{x}_{12}^{2}{x}_{18}^{2}{x}_{20}+{x}_{6}{x}_{9}^{3}{x}_{18}^{2}{x}_{21}\right.\\
&&\hphantom{\big(}-{x}_{5}{x}_{9}^{2}{x}_{10}{x}_{18}^{2}{x}_{21}-{x}_{9}^{2}{x}_{11}{x}_{12}{x}_{18}^{2}{x}_{21}-{x}_{9}^{2}{x}_{12}{x}_{15}{x}_{18}^{2}{x}_{21}-{x}_{9}^{3}{x}_{16}{x}_{18}^{2}{x}_{21}\\
&&\hphantom{\big(}-{x}_{9}{x}_{10}{x}_{13}{x}_{18}^{3}{x}_{21}+{x}_{9}^{2}{x}_{13}{x}_{18}^{2}{x}_{19}{x}_{21}-{x}_{9}^{3}{x}_{12}{x}_{18}{x}_{20}{x}_{21}+{x}_{9}^{3}{x}_{11}{x}_{18}{x}_{21}^{2}+{x}_{9}^{3}{x}_{15}{x}_{18}{x}_{21}^{2}\\
&&\hphantom{\big(}+{x}_{9}{x}_{10}{x}_{12}{x}_{18}^{3}{x}_{22}-{x}_{9}^{2}{x}_{12}{x}_{18}^{2}{x}_{19}{x}_{22}-{x}_{9}^{2}{x}_{12}{x}_{18}{x}_{19}{x}_{21}{x}_{24}+{x}_{9}^{2}{x}_{10}{x}_{18}{x}_{21}^{2}{x}_{24}\\
&&\hphantom{\big(}+{x}_{9}^{2}{x}_{12}{x}_{18}^{2}{x}_{19}{x}_{27}-{x}_{9}^{2}{x}_{10}{x}_{18}^{2}{x}_{21}{x}_{27}+{x}_{6}{x}_{9}^{2}{x}_{10}{x}_{18}^{2}-{x}_{5}{x}_{9}{x}_{10}^{2}{x}_{18}^{2}+2\,{x}_{9}{x}_{10}{x}_{12}{x}_{15}{x}_{18}^{2}\\
&&\hphantom{\big(}-{x}_{9}^{2}{x}_{10}{x}_{16}{x}_{18}^{2}-{x}_{9}^{2}{x}_{12}{x}_{15}{x}_{18}{x}_{19}+{x}_{9}^{2}{x}_{10}{x}_{12}{x}_{18}{x}_{20}-{x}_{9}^{3}{x}_{12}{x}_{19}{x}_{20}-{x}_{9}^{2}{x}_{10}{x}_{11}{x}_{18}{x}_{21}\\
&&\hphantom{\big(}-2\,{x}_{5}{x}_{9}{x}_{12}{x}_{18}^{2}{x}_{21}+{x}_{9}^{3}{x}_{10}{x}_{20}{x}_{21}+2\,{x}_{5}{x}_{9}^{2}{x}_{18}{x}_{21}^{2}+{x}_{9}{x}_{13}{x}_{18}^{2}{x}_{21}^{2}-{x}_{9}{x}_{10}^{2}{x}_{18}^{2}{x}_{22}\\
&&\hphantom{\big(}+{x}_{9}^{2}{x}_{10}{x}_{18}{x}_{19}{x}_{22}-2\,{x}_{9}{x}_{12}{x}_{18}^{2}{x}_{21}{x}_{22}+{x}_{9}^{2}{x}_{18}{x}_{21}^{2}{x}_{22}+{x}_{9}{x}_{10}{x}_{12}{x}_{18}{x}_{19}{x}_{24}\\
&&\hphantom{\big(}-{x}_{9}^{2}{x}_{12}{x}_{19}^{2}{x}_{24}-2\,{x}_{9}{x}_{10}^{2}{x}_{18}{x}_{21}{x}_{24}+{x}_{9}^{2}{x}_{10}{x}_{19}{x}_{21}{x}_{24}-{x}_{9}{x}_{12}{x}_{18}{x}_{21}^{2}{x}_{24}\\
&&\hphantom{\big(}+{x}_{9}^{2}{x}_{12}{x}_{18}{x}_{19}{x}_{25}+{x}_{9}^{2}{x}_{10}{x}_{18}{x}_{21}{x}_{25}
-{x}_{9}^{2}{x}_{10}{x}_{18}{x}_{19}{x}_{27}+{x}_{9}^{2}{x}_{18}{x}_{21}^{2}{x}_{27}-{x}_{9}{x}_{10}^{2}{x}_{11}{x}_{18}\\
&&\hphantom{\big(}-2\,{x}_{9}{x}_{10}^{2}{x}_{15}{x}_{18}+2\,{x}_{6}{x}_{9}{x}_{12}{x}_{18}^{2}-{x}_{5}{x}_{10}{x}_{12}{x}_{18}^{2}+{x}_{5}{x}_{9}{x}_{14}{x}_{18}^{2}-2\,{x}_{9}{x}_{12}{x}_{16}{x}_{18}^{2}\\
&&\hphantom{\big(}+{x}_{9}^{2}{x}_{10}{x}_{11}{x}_{19}+{x}_{9}^{2}{x}_{10}{x}_{15}{x}_{19}+{x}_{12}{x}_{13}{x}_{18}^{2}{x}_{19}-{x}_{9}^{2}{x}_{10}^{2}{x}_{20}-2\,{x}_{9}{x}_{12}^{2}{x}_{18}{x}_{20}\\
&&\hphantom{\big(}-{x}_{6}{x}_{9}^{2}{x}_{18}{x}_{21}+{x}_{9}{x}_{11}{x}_{12}{x}_{18}{x}_{21}-2\,{x}_{9}{x}_{12}{x}_{15}{x}_{18}{x}_{21}+3\,{x}_{9}^{2}{x}_{16}{x}_{18}{x}_{21}\\
&&\hphantom{\big(}-2\,{x}_{9}{x}_{13}{x}_{18}{x}_{19}{x}_{21}-{x}_{10}{x}_{12}{x}_{18}^{2}{x}_{22}+{x}_{9}{x}_{14}{x}_{18}^{2}{x}_{22}+{x}_{9}{x}_{12}{x}_{18}{x}_{19}{x}_{22}\\
&&\hphantom{\big(}+{x}_{9}{x}_{10}{x}_{18}{x}_{21}{x}_{22}-{x}_{9}^{2}{x}_{18}{x}_{22}{x}_{23}+{x}_{12}^{2}{x}_{18}{x}_{19}{x}_{24}-{x}_{10}{x}_{12}{x}_{18}{x}_{21}{x}_{24}\\
&&\hphantom{\big(}+{x}_{9}{x}_{14}{x}_{18}{x}_{21}{x}_{24}-{x}_{9}{x}_{12}{x}_{19}{x}_{21}{x}_{24}+{x}_{9}{x}_{10}{x}_{21}^{2}{x}_{24}-{x}_{9}{x}_{10}^{2}{x}_{18}{x}_{25}+{x}_{9}{x}_{12}{x}_{18}{x}_{21}{x}_{25}\\
&&\hphantom{\big(}-3\,{x}_{9}{x}_{12}{x}_{18}{x}_{19}{x}_{27}+{x}_{9}{x}_{10}{x}_{18}{x}_{21}{x}_{27}+{x}_{9}^{2}{x}_{18}{x}_{23}{x}_{27}-{x}_{9}^{2}{x}_{18}{x}_{21}{x}_{29}-3\,{x}_{6}{x}_{9}{x}_{10}{x}_{18}\\
&&\hphantom{\big(}+{x}_{9}{x}_{11}{x}_{14}{x}_{18}-2\,{x}_{10}{x}_{12}{x}_{15}{x}_{18}+2\,{x}_{9}{x}_{14}{x}_{15}{x}_{18}+{x}_{9}{x}_{10}{x}_{16}{x}_{18}+2\,{x}_{5}{x}_{9}{x}_{10}{x}_{19}\\
&&\hphantom{\big(}+{x}_{9}{x}_{11}{x}_{12}{x}_{19}+{x}_{10}{x}_{13}{x}_{18}{x}_{19}-{x}_{9}{x}_{10}{x}_{12}{x}_{20}+{x}_{9}^{2}{x}_{14}{x}_{20}+{x}_{9}{x}_{10}{x}_{15}{x}_{21}\\
&&\hphantom{\big(}+{x}_{5}{x}_{12}{x}_{18}{x}_{21}-{x}_{9}{x}_{10}{x}_{19}{x}_{22}+{x}_{12}{x}_{18}{x}_{21}{x}_{22}-{x}_{9}{x}_{21}^{2}{x}_{22}-{x}_{9}^{2}{x}_{11}{x}_{23}-{x}_{9}^{2}{x}_{15}{x}_{23}\\
&&\hphantom{\big(}+{x}_{10}{x}_{12}{x}_{19}{x}_{24}+{x}_{9}{x}_{14}{x}_{19}{x}_{24}+{x}_{12}{x}_{21}^{2}{x}_{24}-{x}_{9}{x}_{10}{x}_{23}{x}_{24}-{x}_{9}{x}_{14}{x}_{18}{x}_{25}\\
&&\hphantom{\big(}-{x}_{9}{x}_{12}{x}_{19}{x}_{25}-{x}_{9}{x}_{10}{x}_{21}{x}_{25}+{x}_{9}{x}_{10}{x}_{19}{x}_{27}-{x}_{9}{x}_{21}^{2}{x}_{27}+{x}_{9}{x}_{10}{x}_{18}{x}_{29}+{x}_{10}^{2}{x}_{11}\\
&&\hphantom{\big(}-{x}_{6}{x}_{12}{x}_{18}+{x}_{12}{x}_{16}{x}_{18}+{x}_{5}{x}_{12}{x}_{19}+{x}_{12}^{2}{x}_{20}+{x}_{12}{x}_{15}{x}_{21}-2\,{x}_{9}{x}_{16}{x}_{21}+{x}_{13}{x}_{19}{x}_{21}\\
&&\hphantom{\big(}-{x}_{12}{x}_{19}{x}_{22}-2\,{x}_{5}{x}_{9}{x}_{23}-{x}_{13}{x}_{18}{x}_{23}+{x}_{9}{x}_{22}{x}_{23}-{x}_{12}{x}_{23}{x}_{24}+{x}_{10}^{2}{x}_{25}-{x}_{12}{x}_{21}{x}_{25}\\
&&\hphantom{\big(}\left.+2\,{x}_{12}{x}_{19}{x}_{27}-{x}_{9}{x}_{23}{x}_{27}+{x}_{9}{x}_{21}{x}_{29}+2\,{x}_{6}{x}_{10}-{x}_{11}{x}_{14}+{x}_{14}{x}_{25}-{x}_{10}{x}_{29}
\right)\\
&&\big/(1 - x_{18}x_{9})^2.
\end{eqnarray*}

\item Substitute the above formulae of $x_1,x_2,x_3,x_4,x_8$ into the equations of minimal degree 2. Then make an invertible linear transformation (the  variables that are not displayed remain unchanged):
\begin{eqnarray*}
x_5&\longmapsto & x_5+x_{27},\\ 
x_{15}&\longmapsto & x_{15}+x_{25},\\
x_{16}&\longmapsto& x_6+x_{16},\\
x_{29}&\longmapsto& 2 x_6+x_{16}+x_{29}.
\end{eqnarray*}

\item\label{step2-3} Finally, make an invertible fractional change of variables. We write the formulae as partial fractions to present them in a way more consistent with the Borel case; this does not indicate the way we found them. The  variables that are not displayed remain unchanged.
\begin{eqnarray*}
x_5 &\longmapsto& 
\frac{x_{10} x_{18}^2 x_{24} x_{9}^2}{(x_{18} x_{9}-1)^2}-\frac{x_{10} x_{18}^2 x_{24} x_{9}^2}{(x_{18} x_{9}-1)^3}+\frac{x_{12} x_{18}^2 x_{24} x_{9}}{(x_{18} x_{9}-1)^2}-\frac{x_{13} x_{18}^2}{(x_{18} x_{9}-1)^3}+\frac{x_{15} x_{9}}{(x_{18} x_{9}-1)^3}\\
&&+\frac{x_{18}^2 x_{21} x_{24} x_{9}^3}{(x_{18} x_{9}-1)^3}+\frac{x_{5}}{(x_{18} x_{9}-1)^3 \left(x_{18}^2 x_{9}^2+2 x_{18} x_{9}-1\right)}+\frac{2 x_{19} x_{24} x_{9}^2}{x_{18} x_{9}-1}\\
&&+x_{21} x_{24} x_{9}-2 x_{25} x_{9},\\
x_7& \longmapsto&
-\frac{x_{10} x_{13} x_{24} (2 x_{18} x_{9}-1) x_{18}^4}{(x_{18} x_{9}-1)^5 \left(x_{18}^2 x_{9}^2-3 x_{18} x_{9}+1\right)}+\frac{x_{13} x_{21} x_{24} x_{9} (2 x_{18} x_{9}-1) x_{18}^4}{(x_{18} x_{9}-1)^5 \left(x_{18}^2 x_{9}^2-3 x_{18} x_{9}+1\right)}\\
&&+\frac{x_{10} x_{12} x_{24}^2 \left(x_{18}^3 x_{9}^3+2 x_{18}^2 x_{9}^2-3 x_{18} x_{9}+1\right) x_{18}^3}{(x_{18} x_{9}-1)^4 \left(x_{18}^2 x_{9}^2-3 x_{18} x_{9}+1\right)}-\frac{x_{10} x_{28} x_{18}^3}{(x_{18} x_{9}-1)^2}+\frac{x_{21} x_{28} x_{9} x_{18}^3}{(x_{18} x_{9}-1)^2}\\
&&-\frac{x_{21} x_{24} x_{5} x_{9}^2 (3 x_{18} x_{9}-2) x_{18}^3}{(x_{18} x_{9}-1)^6 \left(x_{18}^2 x_{9}^2-3 x_{18} x_{9}+1\right) \left(x_{18}^2 x_{9}^2+2 x_{18} x_{9}-1\right)}\\
&&+\frac{x_{10} x_{24} x_{5} x_{9} (3 x_{18} x_{9}-2) x_{18}^3}{(x_{18} x_{9}-1)^6 \left(x_{18}^2 x_{9}^2-3 x_{18} x_{9}+1\right) \left(x_{18}^2 x_{9}^2+2 x_{18} x_{9}-1\right)}\\
&&-\frac{x_{21} x_{22} x_{24} x_{9}^2 x_{18}^3}{(x_{18} x_{9}-1) \left(x_{18}^2 x_{9}^2-3 x_{18} x_{9}+1\right)}+\frac{3 x_{21} x_{24} x_{27} x_{9}^2 x_{18}^3}{(x_{18} x_{9}-1) \left(x_{18}^2 x_{9}^2-3 x_{18} x_{9}+1\right)}\\
&&+\frac{x_{10} x_{22} x_{24} x_{9} x_{18}^3}{(x_{18} x_{9}-1) \left(x_{18}^2 x_{9}^2-3 x_{18} x_{9}+1\right)}-\frac{3 x_{10} x_{24} x_{27} x_{9} x_{18}^3}{(x_{18} x_{9}-1) \left(x_{18}^2 x_{9}^2-3 x_{18} x_{9}+1\right)}\\
&&-\frac{x_{12} x_{21} x_{24}^2 x_{9} \left(x_{18}^3 x_{9}^3+2 x_{18}^2 x_{9}^2-3 x_{18} x_{9}+1\right) x_{18}^3}{(x_{18} x_{9}-1)^4 \left(x_{18}^2 x_{9}^2-3 x_{18} x_{9}+1\right)}-\frac{x_{13} x_{15} (2 x_{18} x_{9}-1) x_{18}^3}{(x_{18} x_{9}-1)^6 \left(x_{18}^2 x_{9}^2-3 x_{18} x_{9}+1\right)}\\
&&+\frac{x_{12} x_{19} x_{24}^2 x_{9} \left(5 x_{18}^2 x_{9}^2-4 x_{18} x_{9}+1\right) x_{18}^2}{(x_{18} x_{9}-1)^3 \left(x_{18}^2 x_{9}^2-3 x_{18} x_{9}+1\right)}\\
&&+\frac{x_{12} x_{15} x_{24} \left(x_{18}^3 x_{9}^3+2 x_{18}^2 x_{9}^2-3 x_{18} x_{9}+1\right) x_{18}^2}{(x_{18} x_{9}-1)^5 \left(x_{18}^2 x_{9}^2-3 x_{18} x_{9}+1\right)}+\frac{x_{19} x_{21} x_{24}^2 x_{9}^3 \left(x_{18}^3 x_{9}^3-x_{18}^2 x_{9}^2+1\right) x_{18}^2}{(x_{18} x_{9}-1)^4 \left(x_{18}^2 x_{9}^2-3 x_{18} x_{9}+1\right)}\\
&&+\frac{x_{10}^2 x_{24}^2 \left(x_{18}^5 x_{9}^5-4 x_{18}^4 x_{9}^4+5 x_{18}^3 x_{9}^3-6 x_{18}^2 x_{9}^2+4 x_{18} x_{9}-1\right) x_{18}^2}{(x_{18} x_{9}-1)^5 \left(x_{18}^2 x_{9}^2-3 x_{18} x_{9}+1\right)}-\frac{x_{17} x_{18}^2}{(x_{18} x_{9}-1)^3}\\
&&-\frac{x_{23} x_{24}^2 x_{9}^2 (x_{18} x_{9}-2) x_{18}^2}{(x_{18} x_{9}-1)^6 \left(x_{18}^2 x_{9}^2+2 x_{18} x_{9}-1\right)}+\frac{x_{13} x_{20} x_{9} x_{18}^2}{(x_{18} x_{9}-1)^5 \left(x_{18}^2 x_{9}^2-3 x_{18} x_{9}+1\right) \left(x_{18}^2 x_{9}^2+2 x_{18} x_{9}-1\right)}\\
&&+\frac{x_{10} x_{11} x_{24} x_{9} x_{18}^2}{x_{18}^2 x_{9}^2-3 x_{18} x_{9}+1}-\frac{3 x_{10} x_{24} x_{25} x_{9} x_{18}^2}{x_{18}^2 x_{9}^2-3 x_{18} x_{9}+1}+\frac{x_{15} x_{22} x_{9} x_{18}^2}{(x_{18} x_{9}-1)^2 \left(x_{18}^2 x_{9}^2-3 x_{18} x_{9}+1\right)}\\
&&-\frac{3 x_{15} x_{27} x_{9} x_{18}^2}{(x_{18} x_{9}-1)^2 \left(x_{18}^2 x_{9}^2-3 x_{18} x_{9}+1\right)}-\frac{x_{13} x_{19} x_{24} \left(x_{18}^2 x_{9}^2+2 x_{18} x_{9}-1\right) x_{18}^2}{(x_{18} x_{9}-1)^4 \left(x_{18}^2 x_{9}^2-3 x_{18} x_{9}+1\right)}\\
&&+x_{16} x_{24} (x_{18} x_{9}+1) x_{18}+\frac{x_{12} x_{26} \left(x_{18}^3 x_{9}^3-x_{18}^2 x_{9}^2+1\right) x_{18}}{(x_{18} x_{9}-1)^2}\\
&&+\frac{x_{10} x_{15} x_{24} \left(3 x_{18}^4 x_{9}^4-7 x_{18}^3 x_{9}^3+2 x_{18}^2 x_{9}^2+2 x_{18} x_{9}-1\right) x_{18}}{(x_{18} x_{9}-1)^6 \left(x_{18}^2 x_{9}^2-3 x_{18} x_{9}+1\right)}\\
&&+\frac{x_{10} x_{21} x_{24}^2 \left(11 x_{18}^4 x_{9}^4-19 x_{18}^3 x_{9}^3+15 x_{18}^2 x_{9}^2-6 x_{18} x_{9}+1\right) x_{18}}{(x_{18} x_{9}-1)^5 \left(x_{18}^2 x_{9}^2-3 x_{18} x_{9}+1\right)}\\
&&-\frac{x_{19} x_{28} \left(x_{18}^2 x_{9}^2+1\right) x_{18}}{x_{18} x_{9}-1}+\frac{x_{24} x_{29} x_{18}}{(x_{18} x_{9}-1)^6 \left(x_{18}^2 x_{9}^2+2 x_{18} x_{9}-1\right)}\\
&&-\frac{x_{15} x_{5} \left(x_{18}^3 x_{9}^3-7 x_{18}^2 x_{9}^2+6 x_{18} x_{9}-1\right) x_{18}}{(x_{18} x_{9}-1)^7 \left(x_{18}^2 x_{9}^2-3 x_{18} x_{9}+1\right) \left(x_{18}^2 x_{9}^2+2 x_{18} x_{9}-1\right)}+3 x_{25}^2-2 x_{11} x_{25}\\
&&+\frac{x_{21} x_{26} x_{9} \left(x_{18}^3 x_{9}^3+x_{18}^2 x_{9}^2-2 x_{18} x_{9}+1\right)}{(x_{18} x_{9}-1)^3}+\frac{x_{15}^2 \left(x_{18}^3 x_{9}^3+2 x_{18}^2 x_{9}^2-3 x_{18} x_{9}+1\right)}{(x_{18} x_{9}-1)^6 \left(x_{18}^2 x_{9}^2-3 x_{18} x_{9}+1\right)}\\
&&+\frac{x_{19}^2 x_{24}^2 x_{9}^2 \left(3 x_{18}^3 x_{9}^3+3 x_{18}^2 x_{9}^2-5 x_{18} x_{9}+1\right)}{(x_{18} x_{9}-1)^3 \left(x_{18}^2 x_{9}^2-3 x_{18} x_{9}+1\right)}+\frac{x_{15} x_{19} x_{24} x_{9} \left(11 x_{18}^3 x_{9}^3-17 x_{18}^2 x_{9}^2+9 x_{18} x_{9}-2\right)}{(x_{18} x_{9}-1)^5 \left(x_{18}^2 x_{9}^2-3 x_{18} x_{9}+1\right)}\\
&&+\frac{x_{10} x_{19} x_{24}^2 \left(3 x_{18}^5 x_{9}^5-7 x_{18}^4 x_{9}^4+9 x_{18}^3 x_{9}^3-10 x_{18}^2 x_{9}^2+5 x_{18} x_{9}-1\right)}{(x_{18} x_{9}-1)^4 \left(x_{18}^2 x_{9}^2-3 x_{18} x_{9}+1\right)}\\
&&+\frac{x_{15} x_{21} x_{24} \left(3 x_{18}^5 x_{9}^5-10 x_{18}^4 x_{9}^4+19 x_{18}^3 x_{9}^3-17 x_{18}^2 x_{9}^2+7 x_{18} x_{9}-1\right)}{(x_{18} x_{9}-1)^6 \left(x_{18}^2 x_{9}^2-3 x_{18} x_{9}+1\right)}\\
&&-\frac{x_{10} x_{26} \left(2 x_{18}^2 x_{9}^2-2 x_{18} x_{9}+1\right)}{(x_{18} x_{9}-1)^3}-\frac{x_{7}}{(x_{18} x_{9}-1)^3 \left(x_{18}^2 x_{9}^2+2 x_{18} x_{9}-1\right)}\\
&&-\frac{x_{20} x_{22} x_{9}}{(x_{18} x_{9}-1)^2 \left(x_{18}^2 x_{9}^2-3 x_{18} x_{9}+1\right) \left(x_{18}^2 x_{9}^2+2 x_{18} x_{9}-1\right)}\\
&&+\frac{3 x_{20} x_{27} x_{9}}{(x_{18} x_{9}-1)^2 \left(x_{18}^2 x_{9}^2-3 x_{18} x_{9}+1\right) \left(x_{18}^2 x_{9}^2+2 x_{18} x_{9}-1\right)}\\
&&-\frac{x_{11} x_{20} x_{9}^2}{(x_{18} x_{9}-1)^3 \left(x_{18}^2 x_{9}^2-3 x_{18} x_{9}+1\right) \left(x_{18}^2 x_{9}^2+2 x_{18} x_{9}-1\right)}\\
&&+\frac{3 x_{20} x_{25} x_{9}^2}{(x_{18} x_{9}-1)^3 \left(x_{18}^2 x_{9}^2-3 x_{18} x_{9}+1\right) \left(x_{18}^2 x_{9}^2+2 x_{18} x_{9}-1\right)}\\
&&-\frac{x_{20} x_{21} x_{24} x_{9}^2 (3 x_{18} x_{9}-1)}{(x_{18} x_{9}-1)^4 \left(x_{18}^2 x_{9}^2-3 x_{18} x_{9}+1\right) \left(x_{18}^2 x_{9}^2+2 x_{18} x_{9}-1\right)}\\
&&-\frac{x_{19} x_{20} x_{24} x_{9}^3 (8 x_{18} x_{9}-5)}{(x_{18} x_{9}-1)^5 \left(x_{18}^2 x_{9}^2-3 x_{18} x_{9}+1\right) \left(x_{18}^2 x_{9}^2+2 x_{18} x_{9}-1\right)}\\
&&-\frac{x_{10} x_{20} x_{24} x_{9} \left(2 x_{18}^2 x_{9}^2-2 x_{18} x_{9}+1\right)}{(x_{18} x_{9}-1)^5 \left(x_{18}^2 x_{9}^2-3 x_{18} x_{9}+1\right) \left(x_{18}^2 x_{9}^2+2 x_{18} x_{9}-1\right)}\\
&&-\frac{x_{12} x_{20} x_{24} \left(x_{18}^5 x_{9}^5-4 x_{18}^4 x_{9}^4+8 x_{18}^3 x_{9}^3-7 x_{18}^2 x_{9}^2+4 x_{18} x_{9}-1\right)}{(x_{18} x_{9}-1)^6 \left(x_{18}^2 x_{9}^2-3 x_{18} x_{9}+1\right) \left(x_{18}^2 x_{9}^2+2 x_{18} x_{9}-1\right)}\\
&&-\frac{x_{19} x_{24} x_{5} \left(x_{18}^5 x_{9}^5-6 x_{18}^4 x_{9}^4+3 x_{18}^3 x_{9}^3+8 x_{18}^2 x_{9}^2-6 x_{18} x_{9}+1\right)}{(x_{18} x_{9}-1)^6 \left(x_{18}^2 x_{9}^2-3 x_{18} x_{9}+1\right) \left(x_{18}^2 x_{9}^2+2 x_{18} x_{9}-1\right)}\\
&&-\frac{x_{15} x_{20} x_{9}^2 \left(2 x_{18}^2 x_{9}^2-1\right)}{(x_{18} x_{9}-1)^7 \left(x_{18}^2 x_{9}^2-3 x_{18} x_{9}+1\right) \left(x_{18}^2 x_{9}^2+2 x_{18} x_{9}-1\right)}\\
&&+\frac{x_{20}^2 x_{9}^4}{(x_{18} x_{9}-1)^7 \left(x_{18}^2 x_{9}^2-3 x_{18} x_{9}+1\right) \left(x_{18}^2 x_{9}^2+2 x_{18} x_{9}-1\right)^2}\\
&&-\frac{x_{20} x_{5} x_{9} (3 x_{18} x_{9}-2)}{(x_{18} x_{9}-1)^7 \left(x_{18}^2 x_{9}^2-3 x_{18} x_{9}+1\right) \left(x_{18}^2 x_{9}^2+2 x_{18} x_{9}-1\right)^2}+\frac{x_{11} x_{19} x_{24} x_{9} (2 x_{18} x_{9}-1)}{x_{18}^2 x_{9}^2-3 x_{18} x_{9}+1}\\
&&-\frac{3 x_{19} x_{24} x_{25} x_{9} (2 x_{18} x_{9}-1)}{x_{18}^2 x_{9}^2-3 x_{18} x_{9}+1}-\frac{x_{11} x_{21} x_{24} (3 x_{18} x_{9}-1)}{x_{18}^2 x_{9}^2-3 x_{18} x_{9}+1}+\frac{x_{19} x_{22} x_{24} (3 x_{18} x_{9}-1)}{x_{18}^2 x_{9}^2-3 x_{18} x_{9}+1}\\
&&+\frac{3 x_{21} x_{24} x_{25} (3 x_{18} x_{9}-1)}{x_{18}^2 x_{9}^2-3 x_{18} x_{9}+1}-\frac{3 x_{19} x_{24} x_{27} (3 x_{18} x_{9}-1)}{x_{18}^2 x_{9}^2-3 x_{18} x_{9}+1}+\frac{x_{11} x_{15} (2 x_{18} x_{9}-1)}{(x_{18} x_{9}-1)^2 \left(x_{18}^2 x_{9}^2-3 x_{18} x_{9}+1\right)}\\
&&-\frac{3 x_{15} x_{25} (2 x_{18} x_{9}-1)}{(x_{18} x_{9}-1)^2 \left(x_{18}^2 x_{9}^2-3 x_{18} x_{9}+1\right)}\\
&&-\frac{x_{21}^2 x_{24}^2 (2 x_{18} x_{9}-1) \left(2 x_{18}^5 x_{9}^5-4 x_{18}^4 x_{9}^4+8 x_{18}^3 x_{9}^3-9 x_{18}^2 x_{9}^2+5 x_{18} x_{9}-1\right)}{(x_{18} x_{9}-1)^5 \left(x_{18}^2 x_{9}^2-3 x_{18} x_{9}+1\right)},\\
  x_{12} &\longmapsto&
  \frac{x_{10} x_{9}}{1-x_{18} x_{9}}+x_{12}-\frac{x_{21} x_{9}^2}{1-x_{18} x_{9}},\\
x_{13} &\longmapsto&
\frac{x_{10} x_{18}^2 x_{24} x_{9}^4 (x_{18} x_{9}-3)}{(x_{18} x_{9}-1)^3 \left(x_{18}^2 x_{9}^2-3 x_{18} x_{9}+1\right)}+\frac{x_{11} x_{9}^3}{x_{18}^2 x_{9}^2-3 x_{18} x_{9}+1}\\
&&+\frac{x_{12} x_{18}^2 x_{24} x_{9}^3 (x_{18} x_{9}+1)}{(x_{18} x_{9}-1)^2 \left(x_{18}^2 x_{9}^2-3 x_{18} x_{9}+1\right)}-\frac{x_{13} \left(x_{18}^2 x_{9}^2+2 x_{18} x_{9}-1\right)}{(x_{18} x_{9}-1)^3 \left(x_{18}^2 x_{9}^2-3 x_{18} x_{9}+1\right)}\\
&&+\frac{x_{15} x_{9}^3 (x_{18} x_{9}+1)}{(x_{18} x_{9}-1)^3 \left(x_{18}^2 x_{9}^2-3 x_{18} x_{9}+1\right)}+\frac{x_{19} x_{24} x_{9}^4 (5 x_{18} x_{9}-4)}{(x_{18} x_{9}-1)^2 \left(x_{18}^2 x_{9}^2-3 x_{18} x_{9}+1\right)}\\
&&-\frac{x_{20} x_{9}^5}{(x_{18} x_{9}-1)^4 \left(x_{18}^2 x_{9}^2-3 x_{18} x_{9}+1\right) \left(x_{18}^2 x_{9}^2+2 x_{18} x_{9}-1\right)}\\
&&+\frac{x_{21} x_{24} x_{9}^3 (3 x_{18} x_{9}-1) \left(x_{18}^2 x_{9}^2-x_{18} x_{9}+1\right)}{(x_{18} x_{9}-1)^3 \left(x_{18}^2 x_{9}^2-3 x_{18} x_{9}+1\right)}+\frac{x_{22} x_{9}^2 (x_{18} x_{9}-1)}{x_{18}^2 x_{9}^2-3 x_{18} x_{9}+1}\\
&&-\frac{3 x_{25} x_{9}^3}{x_{18}^2 x_{9}^2-3 x_{18} x_{9}+1}-\frac{3 x_{27} x_{9}^2 (x_{18} x_{9}-1)}{x_{18}^2 x_{9}^2-3 x_{18} x_{9}+1}\\
&&+\frac{3 x_{5} x_{9}^2}{(x_{18} x_{9}-1)^3 \left(x_{18}^2 x_{9}^2-3 x_{18} x_{9}+1\right) \left(x_{18}^2 x_{9}^2+2 x_{18} x_{9}-1\right)}\\
&&+\frac{x_{5} x_{9}^2}{(x_{18} x_{9}-1)^4 \left(x_{18}^2 x_{9}^2-3 x_{18} x_{9}+1\right) \left(x_{18}^2 x_{9}^2+2 x_{18} x_{9}-1\right)},\\
x_{14} &\longmapsto&
-\frac{2 x_{10}^2 x_{18} x_{9}}{(x_{18} x_{9}-1)^2}+\frac{x_{10}^2}{(x_{18} x_{9}-1)^2}+\frac{x_{10} x_{12} x_{18}^2 x_{9}}{x_{18} x_{9}-1}-\frac{x_{10} x_{19} x_{9}^2}{(x_{18} x_{9}-1)^2}+\frac{2 x_{10} x_{18} x_{21} x_{9}^2}{(x_{18} x_{9}-1)^2}\\
&&+\frac{x_{12} x_{19} x_{9}}{x_{18} x_{9}-1}-\frac{x_{12} x_{21}}{x_{18} x_{9}-1}+\frac{x_{14}}{(x_{18} x_{9}-1)^2}+\frac{x_{18} x_{19} x_{21} x_{9}^4}{(x_{18} x_{9}-1)^2}-\frac{x_{21}^2 x_{9}^2}{(x_{18} x_{9}-1)^2},\\
x_{15} &\longmapsto&
-\frac{x_{10} x_{18}^2 x_{24} x_{9}}{(x_{18} x_{9}-1)^2}+\frac{x_{12} x_{18}^2 x_{24}}{x_{18} x_{9}-1}+\frac{x_{15}}{(x_{18} x_{9}-1)^2}-\frac{x_{20} x_{9}^2}{(x_{18} x_{9}-1)^4 \left(x_{18}^2 x_{9}^2+2 x_{18} x_{9}-1\right)}\\
&&+\frac{x_{21} x_{24} \left(x_{18}^2 x_{9}^2+x_{18} x_{9}-1\right)}{(x_{18} x_{9}-1)^2}+\frac{x_{18} x_{5}}{(x_{18} x_{9}-1)^4 \left(x_{18}^2 x_{9}^2+2 x_{18} x_{9}-1\right)}+\frac{x_{18} x_{19} x_{24} x_{9}^2}{(x_{18} x_{9}-1)^2}\\
&&-2 x_{18} x_{27},\\
x_{16} &\longmapsto&
-\frac{x_{10}^2 x_{18}^2 x_{24} x_{9} \left(x_{18}^2 x_{9}^2-3 x_{18} x_{9}+1\right)}{(x_{18} x_{9}-1)^4}-\frac{x_{12} x_{18}^2 x_{24} \left(2 x_{10} x_{18} x_{9}-x_{10}-x_{18} x_{21} x_{9}^2\right)}{(x_{18} x_{9}-1)^3}\\
&&+\frac{x_{10} x_{13} x_{18}^3}{(x_{18} x_{9}-1)^4}-\frac{x_{15} \left(2 x_{10} x_{18} x_{9}-x_{10}-x_{18} x_{21} x_{9}^2\right)}{(x_{18} x_{9}-1)^4}\\
&&-\frac{x_{19} x_{24} x_{9}^2 \left(2 x_{10} x_{18}^2 x_{9}-x_{10} x_{18}-3 x_{18} x_{21} x_{9}+2 x_{21}\right)}{(x_{18} x_{9}-1)^3}\\
&&+\frac{x_{20} \left(x_{10} x_{9}^2-x_{21} x_{9}^3\right)}{(x_{18} x_{9}-1)^5 \left(x_{18}^2 x_{9}^2+2 x_{18} x_{9}-1\right)}+\frac{x_{5} \left(-x_{10} x_{18}^2 x_{9}+2 x_{18} x_{21} x_{9}-x_{21}\right)}{(x_{18} x_{9}-1)^5 \left(x_{18}^2 x_{9}^2+2 x_{18} x_{9}-1\right)}\\
&&-\frac{x_{10} x_{21} x_{24} \left(x_{18}^4 x_{9}^4-x_{18}^3 x_{9}^3+4 x_{18}^2 x_{9}^2-3 x_{18} x_{9}+1\right)}{(x_{18} x_{9}-1)^4}+2 x_{10} x_{25}-\frac{x_{13} x_{18}^2 x_{21}}{(x_{18} x_{9}-1)^4}+x_{16}\\
&&+\frac{x_{18} x_{21}^2 x_{24} x_{9}^2 \left(2 x_{18}^2 x_{9}^2-2 x_{18} x_{9}+1\right)}{(x_{18} x_{9}-1)^4}-2 x_{21} x_{27},\\
x_{17}& \longmapsto&
-\frac{x_{19} x_{20} x_{24} (2 x_{18} x_{9}-1) x_{9}^5}{(x_{18} x_{9}-1)^6 \left(x_{18}^2 x_{9}^2-3 x_{18} x_{9}+1\right) \left(x_{18}^2 x_{9}^2+2 x_{18} x_{9}-1\right)}\\
&&+\frac{x_{19}^2 x_{24}^2 \left(2 x_{18}^3 x_{9}^3+2 x_{18}^2 x_{9}^2-5 x_{18} x_{9}+2\right) x_{9}^4}{(x_{18} x_{9}-1)^4 \left(x_{18}^2 x_{9}^2-3 x_{18} x_{9}+1\right)}-\frac{x_{19} x_{26} x_{9}^4}{(x_{18} x_{9}-1)^2}\\
&&-\frac{x_{20} x_{21} x_{24} (x_{18} x_{9}+2) (2 x_{18} x_{9}-1) x_{9}^4}{(x_{18} x_{9}-1)^6 \left(x_{18}^2 x_{9}^2-3 x_{18} x_{9}+1\right) \left(x_{18}^2 x_{9}^2+2 x_{18} x_{9}-1\right)}\\
&&+\frac{x_{11} x_{18} x_{19} x_{24} x_{9}^4}{(x_{18} x_{9}-1) \left(x_{18}^2 x_{9}^2-3 x_{18} x_{9}+1\right)}-\frac{3 x_{18} x_{19} x_{24} x_{25} x_{9}^4}{(x_{18} x_{9}-1) \left(x_{18}^2 x_{9}^2-3 x_{18} x_{9}+1\right)}\\
&&+\frac{x_{15} x_{18} x_{19} x_{24} (3 x_{18} x_{9}-2) x_{9}^4}{(x_{18} x_{9}-1)^5 \left(x_{18}^2 x_{9}^2-3 x_{18} x_{9}+1\right)}\\
&&+\frac{x_{19} x_{21} x_{24}^2 \left(3 x_{18}^5 x_{9}^5+6 x_{18}^4 x_{9}^4-13 x_{18}^3 x_{9}^3+x_{18}^2 x_{9}^2+7 x_{18} x_{9}-3\right) x_{9}^3}{(x_{18} x_{9}-1)^5 \left(x_{18}^2 x_{9}^2-3 x_{18} x_{9}+1\right)}\\
&&-\frac{x_{21} x_{26} (x_{18} x_{9}+2) x_{9}^3}{(x_{18} x_{9}-1)^2}+\frac{x_{18} x_{23} x_{24}^2 (x_{18} x_{9}-2) x_{9}^3}{(x_{18} x_{9}-1)^6 \left(x_{18}^2 x_{9}^2+2 x_{18} x_{9}-1\right)}\\
&&+\frac{3 x_{10} x_{20} x_{24} (2 x_{18} x_{9}-1) x_{9}^3}{(x_{18} x_{9}-1)^6 \left(x_{18}^2 x_{9}^2-3 x_{18} x_{9}+1\right) \left(x_{18}^2 x_{9}^2+2 x_{18} x_{9}-1\right)}\\
&&-\frac{x_{20} x_{5} (2 x_{18} x_{9}-1) x_{9}^3}{(x_{18} x_{9}-1)^8 \left(x_{18}^2 x_{9}^2-3 x_{18} x_{9}+1\right) \left(x_{18}^2 x_{9}^2+2 x_{18} x_{9}-1\right)^2}\\
&&+\frac{x_{11} x_{18} x_{21} x_{24} (x_{18} x_{9}+2) x_{9}^3}{(x_{18} x_{9}-1) \left(x_{18}^2 x_{9}^2-3 x_{18} x_{9}+1\right)}-\frac{3 x_{18} x_{21} x_{24} x_{25} (x_{18} x_{9}+2) x_{9}^3}{(x_{18} x_{9}-1) \left(x_{18}^2 x_{9}^2-3 x_{18} x_{9}+1\right)}\\
&&+\frac{x_{15} x_{18} x_{21} x_{24} (x_{18} x_{9}+2) (3 x_{18} x_{9}-2) x_{9}^3}{(x_{18} x_{9}-1)^5 \left(x_{18}^2 x_{9}^2-3 x_{18} x_{9}+1\right)}-x_{16} x_{18} x_{24} x_{9}^2+x_{18} x_{19} x_{28} x_{9}^2\\
&&+\frac{x_{18} x_{21} x_{22} x_{24} \left(3 x_{18}^2 x_{9}^2-4 x_{18} x_{9}+2\right) x_{9}^2}{(x_{18} x_{9}-1)^2 \left(x_{18}^2 x_{9}^2-3 x_{18} x_{9}+1\right)}\\
&&+\frac{x_{19} x_{24} x_{5} \left(x_{18}^4 x_{9}^4-2 x_{18}^3 x_{9}^3+11 x_{18}^2 x_{9}^2-10 x_{18} x_{9}+3\right) x_{9}^2}{(x_{18} x_{9}-1)^6 \left(x_{18}^2 x_{9}^2-3 x_{18} x_{9}+1\right) \left(x_{18}^2 x_{9}^2+2 x_{18} x_{9}-1\right)}\\
&&+\frac{x_{12} x_{20} x_{24} \left(x_{18}^4 x_{9}^4-6 x_{18}^3 x_{9}^3+4 x_{18}^2 x_{9}^2+4 x_{18} x_{9}-2\right) x_{9}^2}{(x_{18} x_{9}-1)^6 \left(x_{18}^2 x_{9}^2-3 x_{18} x_{9}+1\right) \left(x_{18}^2 x_{9}^2+2 x_{18} x_{9}-1\right)}\\
&&+\frac{x_{10} x_{19} x_{24}^2 \left(2 x_{18}^5 x_{9}^5-16 x_{18}^4 x_{9}^4+18 x_{18}^3 x_{9}^3+x_{18}^2 x_{9}^2-10 x_{18} x_{9}+4\right) x_{9}^2}{(x_{18} x_{9}-1)^5 \left(x_{18}^2 x_{9}^2-3 x_{18} x_{9}+1\right)}\\
&&+\frac{x_{21}^2 x_{24}^2 \left(x_{18}^6 x_{9}^6+5 x_{18}^5 x_{9}^5-9 x_{18}^4 x_{9}^4+15 x_{18}^3 x_{9}^3-14 x_{18}^2 x_{9}^2+6 x_{18} x_{9}-1\right) x_{9}^2}{(x_{18} x_{9}-1)^5 \left(x_{18}^2 x_{9}^2-3 x_{18} x_{9}+1\right)}\\
&&+\frac{3 x_{10} x_{26} x_{9}^2}{(x_{18} x_{9}-1)^2}-\frac{x_{18}^2 x_{21} x_{28} x_{9}^2}{(x_{18} x_{9}-1)^2}+\frac{x_{7} x_{9}^2}{(x_{18} x_{9}-1)^4 \left(x_{18}^2 x_{9}^2+2 x_{18} x_{9}-1\right)}\\
&&+\frac{x_{11} x_{18} x_{5} x_{9}^2}{(x_{18} x_{9}-1)^3 \left(x_{18}^2 x_{9}^2-3 x_{18} x_{9}+1\right) \left(x_{18}^2 x_{9}^2+2 x_{18} x_{9}-1\right)}\\
&&-\frac{3 x_{18} x_{25} x_{5} x_{9}^2}{(x_{18} x_{9}-1)^3 \left(x_{18}^2 x_{9}^2-3 x_{18} x_{9}+1\right) \left(x_{18}^2 x_{9}^2+2 x_{18} x_{9}-1\right)}\\
&&-\frac{3 x_{10} x_{11} x_{18} x_{24} x_{9}^2}{(x_{18} x_{9}-1) \left(x_{18}^2 x_{9}^2-3 x_{18} x_{9}+1\right)}+\frac{9 x_{10} x_{18} x_{24} x_{25} x_{9}^2}{(x_{18} x_{9}-1) \left(x_{18}^2 x_{9}^2-3 x_{18} x_{9}+1\right)}\\
&&+\frac{x_{19} x_{22} x_{24} (2 x_{18} x_{9}-1) x_{9}^2}{(x_{18} x_{9}-1) \left(x_{18}^2 x_{9}^2-3 x_{18} x_{9}+1\right)}-\frac{3 x_{19} x_{24} x_{27} (2 x_{18} x_{9}-1) x_{9}^2}{(x_{18} x_{9}-1) \left(x_{18}^2 x_{9}^2-3 x_{18} x_{9}+1\right)}\\
&&-\frac{3 x_{18} x_{21} x_{24} x_{27} \left(3 x_{18}^2 x_{9}^2-4 x_{18} x_{9}+2\right) x_{9}^2}{(x_{18} x_{9}-1)^2 \left(x_{18}^2 x_{9}^2-3 x_{18} x_{9}+1\right)}-\frac{3 x_{10} x_{15} x_{18} x_{24} (3 x_{18} x_{9}-2) x_{9}^2}{(x_{18} x_{9}-1)^5 \left(x_{18}^2 x_{9}^2-3 x_{18} x_{9}+1\right)}\\
&&+\frac{x_{12}^2 x_{18}^3 x_{24}^2 \left(2 x_{18}^2 x_{9}^2+4 x_{18} x_{9}-3\right) x_{9}}{(x_{18} x_{9}-1)^3 \left(x_{18}^2 x_{9}^2-3 x_{18} x_{9}+1\right)}+\frac{3 x_{10} x_{18} x_{24} x_{27} \left(2 x_{18}^2 x_{9}^2-3 x_{18} x_{9}+2\right) x_{9}}{(x_{18} x_{9}-1)^2 \left(x_{18}^2 x_{9}^2-3 x_{18} x_{9}+1\right)}\\
&&+\frac{x_{15} x_{5} \left(x_{18}^3 x_{9}^3-x_{18}^2 x_{9}^2+2 x_{18} x_{9}-1\right) x_{9}}{(x_{18} x_{9}-1)^7 \left(x_{18}^2 x_{9}^2-3 x_{18} x_{9}+1\right) \left(x_{18}^2 x_{9}^2+2 x_{18} x_{9}-1\right)}\\
&&+\frac{x_{12} x_{19} x_{24}^2 \left(15 x_{18}^4 x_{9}^4-26 x_{18}^3 x_{9}^3+14 x_{18}^2 x_{9}^2-x_{18} x_{9}-1\right) x_{9}}{(x_{18} x_{9}-1)^4 \left(x_{18}^2 x_{9}^2-3 x_{18} x_{9}+1\right)}\\
&&+\frac{x_{21} x_{24} x_{5} \left(2 x_{18}^5 x_{9}^5+5 x_{18}^4 x_{9}^4-9 x_{18}^3 x_{9}^3+2 x_{18}^2 x_{9}^2+2 x_{18} x_{9}-1\right) x_{9}}{(x_{18} x_{9}-1)^7 \left(x_{18}^2 x_{9}^2-3 x_{18} x_{9}+1\right) \left(x_{18}^2 x_{9}^2+2 x_{18} x_{9}-1\right)}\\
&&+\frac{x_{10} x_{21} x_{24}^2 \left(2 x_{18}^6 x_{9}^6-12 x_{18}^5 x_{9}^5+14 x_{18}^4 x_{9}^4-27 x_{18}^3 x_{9}^3+26 x_{18}^2 x_{9}^2-11 x_{18} x_{9}+2\right) x_{9}}{(x_{18} x_{9}-1)^5 \left(x_{18}^2 x_{9}^2-3 x_{18} x_{9}+1\right)}\\
&&-\frac{x_{12} x_{26} \left(x_{18}^2 x_{9}^2-x_{18} x_{9}+3\right) x_{9}}{x_{18} x_{9}-1}-\frac{x_{24} x_{29} x_{9}}{(x_{18} x_{9}-1)^6 \left(x_{18}^2 x_{9}^2+2 x_{18} x_{9}-1\right)}\\
&&-\frac{x_{13} x_{20} \left(x_{18}^2 x_{9}^2-5 x_{18} x_{9}+2\right) x_{9}}{(x_{18} x_{9}-1)^6 \left(x_{18}^2 x_{9}^2-3 x_{18} x_{9}+1\right) \left(x_{18}^2 x_{9}^2+2 x_{18} x_{9}-1\right)}\\
&&-\frac{x_{10} x_{18} x_{22} x_{24} \left(2 x_{18}^2 x_{9}^2-3 x_{18} x_{9}+2\right) x_{9}}{(x_{18} x_{9}-1)^2 \left(x_{18}^2 x_{9}^2-3 x_{18} x_{9}+1\right)}-\frac{2 x_{10} x_{13} x_{18}^3 x_{24} \left(x_{18}^2 x_{9}^2-5 x_{18} x_{9}+2\right) x_{9}}{(x_{18} x_{9}-1)^5 \left(x_{18}^2 x_{9}^2-3 x_{18} x_{9}+1\right)}\\
&&+3 x_{27}^2-2 x_{22} x_{27}+\frac{x_{10} x_{18} x_{28} (2 x_{18} x_{9}-1)}{(x_{18} x_{9}-1)^2}+\frac{x_{12} x_{18} x_{22} x_{24} \left(2 x_{18}^2 x_{9}^2-2 x_{18} x_{9}+1\right)}{(x_{18} x_{9}-1) \left(x_{18}^2 x_{9}^2-3 x_{18} x_{9}+1\right)}\\
&&+\frac{x_{5}^2 \left(x_{18}^3 x_{9}^3+2 x_{18}^2 x_{9}^2-3 x_{18} x_{9}+1\right)}{(x_{18} x_{9}-1)^8 \left(x_{18}^2 x_{9}^2-3 x_{18} x_{9}+1\right) \left(x_{18}^2 x_{9}^2+2 x_{18} x_{9}-1\right)^2}\\
&&+\frac{x_{12} x_{15} x_{24} \left(x_{18}^3 x_{9}^3+8 x_{18}^2 x_{9}^2-7 x_{18} x_{9}+1\right)}{(x_{18} x_{9}-1)^4 \left(x_{18}^2 x_{9}^2-3 x_{18} x_{9}+1\right)}\\
&&+\frac{x_{12} x_{18} x_{24} x_{5} \left(10 x_{18}^3 x_{9}^3-15 x_{18}^2 x_{9}^2+8 x_{18} x_{9}-2\right)}{(x_{18} x_{9}-1)^6 \left(x_{18}^2 x_{9}^2-3 x_{18} x_{9}+1\right) \left(x_{18}^2 x_{9}^2+2 x_{18} x_{9}-1\right)}\\
&&+\frac{x_{10} x_{12} x_{18} x_{24}^2 \left(4 x_{18}^4 x_{9}^4-25 x_{18}^3 x_{9}^3+21 x_{18}^2 x_{9}^2-8 x_{18} x_{9}+2\right)}{(x_{18} x_{9}-1)^4 \left(x_{18}^2 x_{9}^2-3 x_{18} x_{9}+1\right)}\\
&&+\frac{x_{10} x_{24} x_{5} \left(2 x_{18}^5 x_{9}^5-13 x_{18}^4 x_{9}^4+15 x_{18}^3 x_{9}^3-5 x_{18}^2 x_{9}^2-x_{18} x_{9}+1\right)}{(x_{18} x_{9}-1)^7 \left(x_{18}^2 x_{9}^2-3 x_{18} x_{9}+1\right) \left(x_{18}^2 x_{9}^2+2 x_{18} x_{9}-1\right)}\\
&&+\frac{x_{12} x_{21} x_{24}^2 \left(12 x_{18}^5 x_{9}^5-17 x_{18}^4 x_{9}^4+29 x_{18}^3 x_{9}^3-28 x_{18}^2 x_{9}^2+12 x_{18} x_{9}-2\right)}{(x_{18} x_{9}-1)^4 \left(x_{18}^2 x_{9}^2-3 x_{18} x_{9}+1\right)}\\
&&-\frac{x_{12} x_{18}^2 x_{28}}{x_{18} x_{9}-1}+\frac{x_{17}}{(x_{18} x_{9}-1)^2}+\frac{x_{22} x_{5} (2 x_{18} x_{9}-1)}{(x_{18} x_{9}-1)^3 \left(x_{18}^2 x_{9}^2-3 x_{18} x_{9}+1\right) \left(x_{18}^2 x_{9}^2+2 x_{18} x_{9}-1\right)}\\
&&-\frac{3 x_{27} x_{5} (2 x_{18} x_{9}-1)}{(x_{18} x_{9}-1)^3 \left(x_{18}^2 x_{9}^2-3 x_{18} x_{9}+1\right) \left(x_{18}^2 x_{9}^2+2 x_{18} x_{9}-1\right)}\\
&&-\frac{x_{13} x_{18}^2 x_{5} \left(2 x_{18}^2 x_{9}^2-1\right)}{(x_{18} x_{9}-1)^7 \left(x_{18}^2 x_{9}^2-3 x_{18} x_{9}+1\right) \left(x_{18}^2 x_{9}^2+2 x_{18} x_{9}-1\right)}+\frac{x_{11} x_{12} x_{24} (4 x_{18} x_{9}-1)}{x_{18}^2 x_{9}^2-3 x_{18} x_{9}+1}\\
&&-\frac{3 x_{12} x_{24} x_{25} (4 x_{18} x_{9}-1)}{x_{18}^2 x_{9}^2-3 x_{18} x_{9}+1}-\frac{x_{11} x_{13} x_{18}}{(x_{18} x_{9}-1) \left(x_{18}^2 x_{9}^2-3 x_{18} x_{9}+1\right)}\\
&&+\frac{3 x_{13} x_{18} x_{25}}{(x_{18} x_{9}-1) \left(x_{18}^2 x_{9}^2-3 x_{18} x_{9}+1\right)}-\frac{3 x_{12} x_{18} x_{24} x_{27} \left(2 x_{18}^2 x_{9}^2-2 x_{18} x_{9}+1\right)}{(x_{18} x_{9}-1) \left(x_{18}^2 x_{9}^2-3 x_{18} x_{9}+1\right)}\\
&&-\frac{x_{13} x_{18}^2 x_{22}}{(x_{18} x_{9}-1)^2 \left(x_{18}^2 x_{9}^2-3 x_{18} x_{9}+1\right)}+\frac{3 x_{13} x_{18}^2 x_{27}}{(x_{18} x_{9}-1)^2 \left(x_{18}^2 x_{9}^2-3 x_{18} x_{9}+1\right)}\\
&&-\frac{x_{12} x_{13} x_{18}^3 x_{24} (7 x_{18} x_{9}-3)}{(x_{18} x_{9}-1)^4 \left(x_{18}^2 x_{9}^2-3 x_{18} x_{9}+1\right)}+\frac{x_{13}^2 x_{18}^4}{(x_{18} x_{9}-1)^5 \left(x_{18}^2 x_{9}^2-3 x_{18} x_{9}+1\right)}\\
&&-\frac{x_{13} x_{15} x_{18} (3 x_{18} x_{9}-2)}{(x_{18} x_{9}-1)^5 \left(x_{18}^2 x_{9}^2-3 x_{18} x_{9}+1\right)}\\
&&-\frac{x_{13} x_{18} x_{21} x_{24} \left(2 x_{18}^4 x_{9}^4+2 x_{18}^3 x_{9}^3+3 x_{18}^2 x_{9}^2-4 x_{18} x_{9}+1\right)}{(x_{18} x_{9}-1)^5 \left(x_{18}^2 x_{9}^2-3 x_{18} x_{9}+1\right)}\\
&&-\frac{x_{13} x_{19} x_{24} \left(3 x_{18}^4 x_{9}^4-x_{18}^3 x_{9}^3-3 x_{18}^2 x_{9}^2+3 x_{18} x_{9}-1\right)}{(x_{18} x_{9}-1)^5 \left(x_{18}^2 x_{9}^2-3 x_{18} x_{9}+1\right)}\\
&&-\frac{x_{10}^2 x_{24}^2 \left(3 x_{18}^5 x_{9}^5-9 x_{18}^4 x_{9}^4-x_{18}^3 x_{9}^3+7 x_{18}^2 x_{9}^2-4 x_{18} x_{9}+1\right)}{(x_{18} x_{9}-1)^5 \left(x_{18}^2 x_{9}^2-3 x_{18} x_{9}+1\right)},\\
x_{19} &\longmapsto&
-\frac{x_{10} x_{18}^2}{1-x_{18} x_{9}}+\frac{x_{18} x_{21}}{1-x_{18} x_{9}}+x_{19},\\
x_{20} &\longmapsto&\frac{x_{10} x_{18}^3 x_{24} \left(x_{18}^3 x_{9}^3-3 x_{18}^2 x_{9}^2+1\right)}{(x_{18} x_{9}-1)^3 \left(x_{18}^2 x_{9}^2-3 x_{18} x_{9}+1\right)}+\frac{x_{11} x_{18}^2 (x_{18} x_{9}-1)}{x_{18}^2 x_{9}^2-3 x_{18} x_{9}+1}\\
&&+\frac{x_{12} x_{18}^4 x_{24} (3 x_{18} x_{9}-2)}{(x_{18} x_{9}-1)^2 \left(x_{18}^2 x_{9}^2-3 x_{18} x_{9}+1\right)}-\frac{x_{13} x_{18}^5}{(x_{18} x_{9}-1)^3 \left(x_{18}^2 x_{9}^2-3 x_{18} x_{9}+1\right)}\\
&&+\frac{x_{15} x_{18}^2 (3 x_{18} x_{9}-2)}{(x_{18} x_{9}-1)^3 \left(x_{18}^2 x_{9}^2-3 x_{18} x_{9}+1\right)}-\frac{x_{20}}{(x_{18} x_{9}-1)^4 \left(x_{18}^2 x_{9}^2-3 x_{18} x_{9}+1\right)}\\
&&-\frac{3 x_{18}^2 x_{25} (x_{18} x_{9}-1)}{x_{18}^2 x_{9}^2-3 x_{18} x_{9}+1}+\frac{x_{18}^4 x_{21} x_{24} x_{9}^2 \left(x_{18}^2 x_{9}^2+x_{18} x_{9}-1\right)}{(x_{18} x_{9}-1)^3 \left(x_{18}^2 x_{9}^2-3 x_{18} x_{9}+1\right)}\\
&&+\frac{x_{18} x_{19} x_{24} \left(2 x_{18}^3 x_{9}^3+x_{18}^2 x_{9}^2-2 x_{18} x_{9}+1\right)}{(x_{18} x_{9}-1)^2 \left(x_{18}^2 x_{9}^2-3 x_{18} x_{9}+1\right)}+\frac{x_{18}^3 x_{22}}{x_{18}^2 x_{9}^2-3 x_{18} x_{9}+1}\\
&&-\frac{3 x_{18}^3 x_{27}}{x_{18}^2 x_{9}^2-3 x_{18} x_{9}+1}+\frac{x_{18}^3 x_{5} (x_{18} x_{9}+1)}{(x_{18} x_{9}-1)^4 \left(x_{18}^2 x_{9}^2-3 x_{18} x_{9}+1\right) \left(x_{18}^2 x_{9}^2+2 x_{18} x_{9}-1\right)},\\
x_{23} &\longmapsto&
-\frac{x_{10}^2 x_{18}^2}{(x_{18} x_{9}-1)^2}+\frac{x_{10} x_{12} x_{18}^3}{x_{18} x_{9}-1}+\frac{2 x_{10} x_{18}^2 x_{21} x_{9}}{(x_{18} x_{9}-1)^2}-\frac{x_{12} x_{18}^2 x_{21}}{x_{18} x_{9}-1}\\
&&+\frac{x_{23}}{(x_{18} x_{9}-1)^3 \left(x_{18}^2 x_{9}^2+2 x_{18} x_{9}-1\right)}-\frac{x_{19}^2 x_{9}^2}{x_{18} x_{9}-1}-\frac{x_{21}^2 (2 x_{18} x_{9}-1)}{(x_{18} x_{9}-1)^2},\\
x_{29} &\longmapsto&
-\frac{x_{10}^2 x_{18}^2 x_{24} x_{9}}{(x_{18} x_{9}-1)^2}+\frac{x_{10} x_{12} x_{18}^2 x_{24}}{x_{18} x_{9}-1}+\frac{x_{10} x_{15}}{(x_{18} x_{9}-1)^4}-\frac{x_{10} x_{20} x_{9}^2 (x_{18} x_{9}-2)}{(x_{18} x_{9}-1)^6 \left(x_{18}^2 x_{9}^2+2 x_{18} x_{9}-1\right)}\\
&&+\frac{x_{10} x_{21} x_{24} \left(3 x_{18}^3 x_{9}^3+5 x_{18}^2 x_{9}^2-5 x_{18} x_{9}+1\right)}{(x_{18} x_{9}-1)^2 \left(x_{18}^2 x_{9}^2+2 x_{18} x_{9}-1\right)}-\frac{x_{10} x_{19} x_{24} x_{9}}{x_{18} x_{9}-1}+\frac{x_{10} x_{19} x_{24} x_{9}}{(x_{18} x_{9}-1)^2}\\
&&+2 x_{10} x_{25}-\frac{x_{12} x_{15} x_{18}}{(x_{18} x_{9}-1)^3}+\frac{x_{12} x_{20} x_{9} \left(x_{18}^2 x_{9}^2-3 x_{18} x_{9}+1\right)}{(x_{18} x_{9}-1)^6 \left(x_{18}^2 x_{9}^2+2 x_{18} x_{9}-1\right)}-\frac{x_{12} x_{19} x_{24}}{x_{18} x_{9}-1}\\
&&-\frac{x_{12} x_{18} x_{21} x_{24}}{x_{18} x_{9}-1}-\frac{x_{15} x_{19} x_{9}^2}{(x_{18} x_{9}-1)^4}-\frac{x_{15} x_{18} x_{21} x_{9}^2}{(x_{18} x_{9}-1)^4}-\frac{x_{19} x_{5} x_{9}}{(x_{18} x_{9}-1)^6 \left(x_{18}^2 x_{9}^2+2 x_{18} x_{9}-1\right)}\\
&&-\frac{x_{20} x_{21} x_{9}^3}{(x_{18} x_{9}-1)^6 \left(x_{18}^2 x_{9}^2+2 x_{18} x_{9}-1\right)}+\frac{x_{18} x_{23} x_{24} x_{9}^2 (x_{18} x_{9}-2)}{(x_{18} x_{9}-1)^6 \left(x_{18}^2 x_{9}^2+2 x_{18} x_{9}-1\right)}\\
&&-\frac{x_{29}}{(x_{18} x_{9}-1)^6 \left(x_{18}^2 x_{9}^2+2 x_{18} x_{9}-1\right)}-\frac{x_{19}^2 x_{24} x_{9}^3}{(x_{18} x_{9}-1)^2}-\frac{x_{19} x_{21} x_{24} x_{9}^2}{(x_{18} x_{9}-1)^2}\\
&&-\frac{x_{18} x_{21}^2 x_{24} x_{9}^2}{(x_{18} x_{9}-1)^2}+2 x_{21} x_{27}.
\end{eqnarray*}
\end{enumerate}
We need to show that \textsc{Step}~\ref{step2-3} is an automorphism of the localized ring
\begin{equation}\label{eq-localized-ring-A1311}
  \Bbbk[x_5,x_6,x_7,x_9,\dots,x_{29}]\left[\frac{1}{(1-x_{18}x_9) (1-2x_{18}x_9)
   (1-2 x_{18}x_9-x_{18}^2 x_9^2)
  (1-3 x_{18}x_9+x_{18}^2 x_9^2)}\right].
\end{equation}
In fact, \textsc{Step}~\ref{step2-3} is block upper-triangular in the order
\begin{equation*}
 \{x_{7},x_{17}\}\prec x_{29}\prec x_{23}\prec x_{5}\prec x_{14} \prec \{ x_{13},x_{15}, x_{20}\} \prec x_{12}\prec x_{19}.
\end{equation*}
The transformation inside the block $\{ x_{13},x_{15}, x_{20}\}$ is the matrix
\[
\left(
\begin{array}{ccc}
 -\frac{x_{18}^2 x_{9}^2+2 x_{18} x_{9}-1}{(x_{18} x_{9}-1)^3 \left(x_{18}^2 x_{9}^2-3 x_{18} x_{9}+1\right)} & \frac{ x_{9}^3 (x_{18} x_{9}+1)}{(x_{18} x_{9}-1)^3 \left(x_{18}^2 x_{9}^2-3 x_{18} x_{9}+1\right)} & -\frac{x_{9}^5}{(x_{18} x_{9}-1)^4 \left(x_{18}^2 x_{9}^2-3 x_{18} x_{9}+1\right) \left(x_{18}^2 x_{9}^2+2 x_{18} x_{9}-1\right)} \\[1.5ex]
 0 & \frac{1}{(x_{18} x_{9}-1)^2} & -\frac{ x_{9}^2}{(x_{18} x_{9}-1)^4 \left(x_{18}^2 x_{9}^2+2 x_{18} x_{9}-1\right)} \\[1.5ex]
 -\frac{ x_{18}^5}{(x_{18} x_{9}-1)^3 \left(x_{18}^2 x_{9}^2-3 x_{18} x_{9}+1\right)} & \frac{ x_{18}^2 (3 x_{18} x_{9}-2)}{(x_{18} x_{9}-1)^3 \left(x_{18}^2 x_{9}^2-3 x_{18} x_{9}+1\right)} & -\frac{1}{(x_{18} x_{9}-1)^4 \left(x_{18}^2 x_{9}^2-3 x_{18} x_{9}+1\right)} 
\end{array}
\right)
\]
with determinant 
\[
-\frac{1}{(x_{18} x_{9}-1)^8 \left(x_{18}^2 x_{9}^2-3 x_{18} x_{9}+1\right) \left(x_{18}^2 x_{9}^2+2 x_{18} x_{9}-1\right)}.
\]
The transformation inside the block $\{ x_{7},x_{17}\}$ is the matrix
\[
\left(
\begin{array}{cc}
 -\frac{1}{(x_{18} x_{9}-1)^3 \left(x_{18}^2 x_{9}^2+2 x_{18} x_{9}-1\right)} & -\frac{x_{18}^2}{(x_{18} x_{9}-1)^3} \\[2ex]
 \frac{ x_{9}^2}{(x_{18} x_{9}-1)^4 \left(x_{18}^2 x_{9}^2+2 x_{18} x_{9}-1\right)} & \frac{1}{(x_{18} x_{9}-1)^2} 
\end{array}
\right)
\]
with determinant
\[
\frac{2 x_{18} x_{9}-1}{(x_{18} x_{9}-1)^7 \left(x_{18}^2 x_{9}^2+2 x_{18} x_{9}-1\right)}.
\]
Both determinants are invertible in (\ref{eq-localized-ring-A1311}). Hence \textsc{Step}~\ref{step2-3} is an automorphism.

After these three steps, the ideal generated by the equations of minimal degree at most $2$ is transformed into the ideal 
\begin{equation*}
  \renewcommand{\arraystretch}{1.2}
  \begin{array}{lll}
(x_{12} x_{15}+ x_{13} x_{19}+ x_{14} x_{24},& x_{12} x_{20}+ x_{23} x_{24}+ x_{19} x_{5},& -x_{16} x_{24}- x_{12} x_{26}+ x_{19} x_{28},\\
\hphantom{(} -x_{14} x_{20}+ x_{15} x_{23}+ x_{19} x_{29},& x_{20} x_{28}+ x_{26} x_{5}+ x_{24} x_{7},& x_{17} x_{24}- x_{13} x_{26}- x_{15} x_{28},\\
\hphantom{(}  -x_{16} x_{20}+ x_{23} x_{26}- x_{19} x_{7},& -x_{13} x_{23}+ x_{12} x_{29}+ x_{14} x_{5},& x_{13} x_{20}+ x_{24} x_{29}- x_{15} x_{5},\\
\hphantom{(}   x_{17} x_{20}+ x_{26} x_{29}+ x_{15} x_{7},& x_{13} x_{16}+ x_{12} x_{17}+ x_{14} x_{28},& x_{15} x_{16}- x_{17} x_{19}- x_{14} x_{26},\\
 \hphantom{(}   x_{23} x_{28}+ x_{16} x_{5}- x_{12} x_{7},& x_{17} x_{23}+ x_{16} x_{29}+ x_{14} x_{7},& x_{28} x_{29}- x_{17} x_{5}- x_{13} x_{7})
  \end{array}
  \end{equation*}
in the localized ring (\ref{eq-localized-ring-A1311}). 

\section{(First-order) change of variables for some non-Borel ideals}\label{sec:appendix-changevariable-NonBorelideals}

\subsection{\texorpdfstring{$\boldsymbol{\left((1)\subset (4,1,1)\right)}$}{((1) in (4,1,1))},  extra.dim \texorpdfstring{$\boldsymbol{=6}$}{=6}}
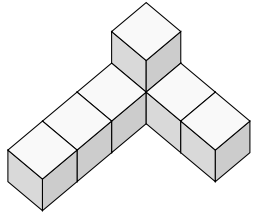
\label{sec:3D-1411}
\begin{center}
\begin{tikzpicture}[x=(220:0.6cm), y=(-40:0.6cm), z=(90:0.42cm)]

\foreach \m [count=\y] in {{2,1,1,1},{1},{1}}{
  \foreach \n [count=\x] in \m {
  \ifnum \n>0
      \foreach \z in {1,...,\n}{
        \draw [fill=gray!30] (\x+1,\y,\z) -- (\x+1,\y+1,\z) -- (\x+1, \y+1, \z-1) -- (\x+1, \y, \z-1) -- cycle;
        \draw [fill=gray!40] (\x,\y+1,\z) -- (\x+1,\y+1,\z) -- (\x+1, \y+1, \z-1) -- (\x, \y+1, \z-1) -- cycle;
        \draw [fill=gray!5] (\x,\y,\z)   -- (\x+1,\y,\z)   -- (\x+1, \y+1, \z)   -- (\x, \y+1, \z) -- cycle;  
      }
 \fi
 }
}    

\end{tikzpicture}
\end{center}
\[
I_{\lambda_{1411}}=\left(X_1^4,X_1 X_2, X_1 X_3, X_2^3,X_2 X_3, X_3^2\right).
\]
As we said in the last paragraph of Section~\ref{sec:3D-1311}, Algorithm~\ref{alg-step0-Haiman} and the cutting of terms of  degree at least $3$ give 
\begin{alignat*}{6}
  c_{1, 0, 0}^{1, 1, 0} &\longmapsto x_{1},&\quad c_{1, 0, 0}^{1, 0, 1} &\longmapsto x_{2},&\quad c_{1, 0, 0}^{4, 0, 0} &\longmapsto x_{3},&\quad c_{2, 0, 0}^{1, 1, 0} &\longmapsto x_{4},&\quad c_{2, 0, 0}^{1, 0, 1} &\longmapsto x_{5},&\quad c_{2, 0, 0}^{4, 0, 0} &\longmapsto x_{6},\\ 
  c_{3, 0, 0}^{1, 1, 0} &\longmapsto x_{7},&\quad c_{3, 0, 0}^{1, 0, 1} &\longmapsto x_{8},&\quad c_{3, 0, 0}^{4, 0, 0} &\longmapsto x_{9},&\quad c_{3, 0, 0}^{0, 1, 1} &\longmapsto x_{10},&\quad c_{3, 0, 0}^{0, 3, 0} &\longmapsto x_{11},&\quad c_{3, 0, 0}^{0, 0, 2} &\longmapsto x_{12},\\
   c_{0, 1, 0}^{1, 1, 0} &\longmapsto x_{13},&\quad c_{0, 1, 0}^{0, 1, 1} &\longmapsto x_{14},&\quad c_{0, 1, 0}^{0, 3, 0} &\longmapsto x_{15},&\quad c_{0, 2, 0}^{1, 1, 0} &\longmapsto x_{16},&\quad c_{0, 2, 0}^{1, 0, 1} &\longmapsto x_{17},&\quad c_{0, 2, 0}^{4, 0, 0} &\longmapsto x_{18},\\
    c_{0, 2,0}^{0, 1, 1} &\longmapsto x_{19},&\quad c_{0, 2, 0}^{0, 3, 0} &\longmapsto x_{20},&\quad c_{0, 2, 0}^{0, 0, 2} &\longmapsto x_{21},&\quad c_{0, 0, 1}^{1, 1, 0} &\longmapsto x_{22},&\quad c_{0, 0, 1}^{1, 0, 1} &\longmapsto x_{23},&\quad c_{0, 0, 1}^{4, 0, 0} &\longmapsto x_{24},\\ 
  c_{0, 0,1}^{0, 1, 1} &\longmapsto x_{25},&\quad c_{0, 0, 1}^{0, 3, 0} &\longmapsto x_{26},&\quad c_{0, 0, 1}^{0, 0, 2} &\longmapsto x_{27},
\end{alignat*}
and we denote the resulting ideal by $\mathcal{H}_{\lambda_{1411}}^{\mathrm{cutoff}}$.  
The change of variables
\[
x_{25}\longmapsto x_{1}+x_{25},\quad x_{23}\longmapsto x_{13}+x_{23},\quad x_{14}\longmapsto x_{14}+x_{2},\quad x_{27}\longmapsto x_{27}+2 x_{2}+x_{14},
\]
transforms $\mathcal{H}_{\lambda_{1411}}^{\mathrm{cutoff}}$ into
\begin{equation*}
\renewcommand{\arraystretch}{1.2}
\begin{array}{lllll}      
J_{1411}^{\mathrm{cutoff}}&=&({x}_{11}{x}_{17}-{x}_{12}{x}_{22}+{x}_{10}{x}_{23},& -{x}_{10}{x}_{18}+{x}_{21}{x}_{22}-{x}_{17}{x}_{25},& -{x}_{14}{x}_{22}-{x}_{10}{x}_{24}+{x}_{17}{x}_{26},\\
&&\hphantom{(}-{x}_{12}{x}_{18}+{x}_{21}{x}_{23}-{x}_{17}{x}_{27},& {x}_{3}{x}_{22}+{x}_{24}{x}_{25}+{x}_{18}{x}_{26},& {x}_{3}{x}_{17}+{x}_{14}{x}_{18}+{x}_{21}{x}_{24},\\
&&\hphantom{(}-{x}_{15}{x}_{18}+{x}_{3}{x}_{23}+{x}_{24}{x}_{27},& {x}_{15}{x}_{17}+{x}_{14}{x}_{23}+{x}_{12}{x}_{24},& -{x}_{11}{x}_{21}+{x}_{12}{x}_{25}-{x}_{10}{x}_{27},\\
&&\hphantom{(}{x}_{3}{x}_{12}-{x}_{15}{x}_{21}-{x}_{14}{x}_{27},& -{x}_{11}{x}_{18}+{x}_{23}{x}_{25}-{x}_{22}{x}_{27},& -{x}_{15}{x}_{22}-{x}_{11}{x}_{24}-{x}_{23}{x}_{26},\\
&\hphantom{(}&-{x}_{11}{x}_{14}+{x}_{10}{x}_{15}+{x}_{12}{x}_{26},& -{x}_{3}{x}_{11}+{x}_{15}{x}_{25}+{x}_{26}{x}_{27},& {x}_{3}{x}_{10}-{x}_{14}{x}_{25}+{x}_{21}{x}_{26}).
\end{array}
\end{equation*}
Finally, the map
\begin{alignat*}{6}
x_3&\longmapsto p_{1,2},&\quad x_{10}&\longmapsto p_{3,4},&\quad x_{11}&\longmapsto p_{0,3},&\quad x_{12}&\longmapsto p_{0,4},&\quad x_{14}&\longmapsto p_{2,4},&\quad x_{15}&\longmapsto p_{0,2},\\
 x_{17}&\longmapsto p_{4,5},&\quad x_{18}&\longmapsto p_{1,5},&\quad x_{21}&\longmapsto p_{1,4},&\quad x_{22}&\longmapsto p_{3,5},&\quad x_{23}&\longmapsto p_{0,5},&\quad x_{24}&\longmapsto -p_{2,5},\\
 x_{25}&\longmapsto p_{1,3},&\quad x_{26}&\longmapsto p_{2,3},&\quad x_{27}&\longmapsto -p_{0,1}
\end{alignat*}
 transforms the ideal $J_{1411}^{\mathrm{cutoff}}$ into the Pl\"ucker ideal~\eqref{eq-pluckerideal}. So assuming Conjecture~\ref{conj-existence-equivariant-iso-exdim6},  we obtain
\begin{eqnarray*}
H(A_{{\lambda}_{1411}};\mathbf{t})&=&K\left(\frac{t_{2} \sqrt{t_{3}}}{t_{1}^{3/2}},\frac{t_{1}^{3/2} \sqrt{t_{3}}}{t_{2}},\frac{t_{1}^{3/2} t_{2}}{\sqrt{t_{3}}},\frac{t_{2}^2}{t_{1}^{3/2} \sqrt{t_{3}}},\frac{t_{3}^{3/2}}{t_{1}^{3/2} t_{2}},\frac{t_{1}^{5/2}}{t_{2} \sqrt{t_{3}}}\right)\notag\\
&&\Big/\left((1-t_{1})^3(1-t_{2})^3(1-t_{3})^3(1-t_{1}^3)(1-t_{1}^2)(1-t_{2}^2)\left(\frac{t_{1}-t_{2}}{t_{1}}\right)\left(\frac{t_{1}-t_{3}}{t_{1}}\right)\right.\notag\\
&&\hphantom{\Big/\Big(}\left(\frac{t_{1}^2-t_{2}}{t_{1}^2}\right)\left(\frac{t_{1}^2-t_{3}}{t_{1}^2}\right)\left(\frac{t_{1}^3-t_{2} t_{3}}{t_{1}^3}\right)\left(\frac{t_{1}^3-t_{2}^3}{t_{1}^3}\right)\left(\frac{t_{1}^3-t_{3}^2}{t_{1}^3}\right)\left(\frac{t_{2}-t_{1}}{t_{2}}\right)\left(\frac{t_{2}-t_{3}}{t_{2}}\right)\notag\\
&&\hphantom{\Big/\Big(}\left.\left(\frac{t_{2}^2-t_{1} t_{3}}{t_{2}^2}\right)\left(\frac{t_{2}^2-t_{1}^4}{t_{2}^2}\right)
\left(\frac{t_{2}^2-t_{3}^2}{t_{2}^2}\right)\left(\frac{t_{3}-t_{1} t_{2}}{t_{3}}\right)\left(\frac{t_{3}-t_{1}^4}{t_{3}}\right)\left(\frac{t_{3}-t_{2}^3}{t_{3}}\right)\right).
\end{eqnarray*}
In the following two subsections, we omit the intermediate explanations.

\subsection{\texorpdfstring{$\boldsymbol{\left((2)\subset (3,1,1)\right)}$}{((2) in (3,1,1))},   extra.dim \texorpdfstring{$\boldsymbol{=6}$}{=6}}
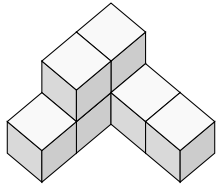
\label{sec:3D-2311}
\begin{center}
\begin{tikzpicture}[x=(220:0.6cm), y=(-40:0.6cm), z=(90:0.42cm)]

\foreach \m [count=\y] in {{2,2,1},{1},{1}}{
  \foreach \n [count=\x] in \m {
  \ifnum \n>0
      \foreach \z in {1,...,\n}{
        \draw [fill=gray!30] (\x+1,\y,\z) -- (\x+1,\y+1,\z) -- (\x+1, \y+1, \z-1) -- (\x+1, \y, \z-1) -- cycle;
        \draw [fill=gray!40] (\x,\y+1,\z) -- (\x+1,\y+1,\z) -- (\x+1, \y+1, \z-1) -- (\x, \y+1, \z-1) -- cycle;
        \draw [fill=gray!5] (\x,\y,\z)   -- (\x+1,\y,\z)   -- (\x+1, \y+1, \z)   -- (\x, \y+1, \z) -- cycle;  
      }
 \fi
 }
}    

\end{tikzpicture}
\end{center}

\medskip
\[
I_{\lambda_{2311}}=\left(X_1^3, X_1 X_2, X_1^2 X_3,X_2^3,X_2 X_3, X_3^2\right).
\]

\medskip
\begin{alignat*}{6}
{c}_{1, 0, 0}^{1, 1, 0} &\longmapsto  {x}_{1},&\quad {c}_{1, 0, 0}^{3, 0, 0} &\longmapsto {x}_{2},&\quad {c}_{1, 0, 0}^{0, 0, 2} &\longmapsto  {x}_{3},&\quad 
{c}_{2, 0, 0}^{1, 1, 0} &\longmapsto {x}_{4},&\quad {c}_{2, 0, 0}^{3, 0, 0} &\longmapsto  {x}_{5},&\quad {c}_{2, 0, 0}^{2, 0, 1} &\longmapsto {x}_{6},\\
{c}_{2, 0, 0}^{0, 1, 1} &\longmapsto  {x}_{7},&\quad {c}_{2, 0, 0}^{0, 3, 0} &\longmapsto {x}_{8},&\quad {c}_{2, 0, 0}^{0, 0, 2} &\longmapsto  {x}_{9},&\quad 
{c}_{0, 1, 0}^{1, 1, 0} &\longmapsto {x}_{10},&\quad {c}_{0, 1, 0}^{0, 1, 1} &\longmapsto  {x}_{11},&\quad {c}_{0, 1, 0}^{0, 3, 0} &\longmapsto {x}_{12},\\ 
{c}_{0, 2, 0}^{1, 1, 0} &\longmapsto  {x}_{13},&\quad {c}_{0, 2, 0}^{3, 0, 0} &\longmapsto {x}_{14},&\quad {c}_{0, 2, 0}^{2, 0, 1} &\longmapsto  {x}_{15},&\quad 
{c}_{0, 2, 0}^{0, 1, 1} &\longmapsto {x}_{16},&\quad {c}_{0, 2, 0}^{0, 3, 0} &\longmapsto  {x}_{17},&\quad {c}_{0, 2, 0}^{0, 0, 2} &\longmapsto {x}_{18},\\
{c}_{0, 0, 1}^{1, 1, 0} &\longmapsto  {x}_{19},&\quad {c}_{0, 0, 1}^{3, 0, 0} &\longmapsto {x}_{20},&\quad {c}_{0, 0, 1}^{0, 0, 2} &\longmapsto  {x}_{21},&\quad 
{c}_{1, 0, 1}^{1, 1, 0} &\longmapsto {x}_{22},&\quad {c}_{1, 0, 1}^{3, 0, 0} &\longmapsto  {x}_{23},&\quad {c}_{1, 0, 1}^{2, 0, 1} &\longmapsto {x}_{24},\\
 {c}_{1, 0, 1}^{0, 1, 1} &\longmapsto  {x}_{25},&\quad {c}_{1, 0, 1}^{0, 3, 0} &\longmapsto {x}_{26},&\quad {c}_{1, 0, 1}^{0, 0, 2} &\longmapsto  {x}_{27}.
\end{alignat*}
\[
x_{25}\longmapsto x_{4}+x_{25},\quad x_{11}\longmapsto x_{6}+x_{11},\quad x_{21}\longmapsto x_{21}+2 x_{6}+x_{11},\quad x_{5}\longmapsto x_{5}+x_{10}+x_{24},
\]
\begin{equation*}
\renewcommand{\arraystretch}{1.2}
\begin{array}{lll} 
(-{x}_{7}{x}_{14}+{x}_{18}{x}_{19}-{x}_{15}{x}_{25},&{x}_{11}{x}_{14}+{x}_{5}{x}_{15}+{x}_{18}{x}_{20},&{x}_{2}{x}_{5}-{x}_{12}{x}_{14}+{x}_{20}{x}_{21},\\
\hphantom{(}{x}_{8}{x}_{14}+{x}_{19}{x}_{21}-{x}_{2}{x}_{25},&{x}_{5}{x}_{19}+{x}_{20}{x}_{25}+{x}_{14}{x}_{26},&{x}_{2}{x}_{7}+{x}_{8}{x}_{15}-{x}_{3}{x}_{19},\\
\hphantom{(}{x}_{2}{x}_{11}+{x}_{12}{x}_{15}+{x}_{3}{x}_{20},&-{x}_{3}{x}_{14}+{x}_{2}{x}_{18}-{x}_{15}{x}_{21},&-{x}_{8}{x}_{18}-{x}_{7}{x}_{21}+{x}_{3}{x}_{25},\\
\hphantom{(}-{x}_{5}{x}_{7}+{x}_{11}{x}_{25}-{x}_{18}{x}_{26},&-{x}_{12}{x}_{19}-{x}_{8}{x}_{20}-{x}_{2}{x}_{26},&-{x}_{8}{x}_{11}+{x}_{7}{x}_{12}+{x}_{3}{x}_{26},\\
\hphantom{(}-{x}_{5}{x}_{8}+{x}_{12}{x}_{25}+{x}_{21}{x}_{26},&-{x}_{3}{x}_{5}+{x}_{12}{x}_{18}+{x}_{11}{x}_{21},&{x}_{11}{x}_{19}+{x}_{7}{x}_{20}-{x}_{15}{x}_{26}).
\end{array}
\end{equation*}

\begin{alignat*}{6}
x_2&\longmapsto p_{0,5},&\quad  x_3&\longmapsto p_{0,4},&\quad  x_5&\longmapsto p_{1,2},&\quad  x_7&\longmapsto p_{3,4},&\quad  x_8&\longmapsto p_{0,3},&\quad  x_{11}&\longmapsto p_{2,4},\\
x_{12}&\longmapsto p_{0,2},&\quad  x_{14}&\longmapsto p_{1,5},&\quad  x_{15}&\longmapsto p_{4,5},&\quad  x_{18}&\longmapsto p_{1,4},&\quad  x_{19}&\longmapsto p_{3,5},&\quad  x_{20}&\longmapsto -p_{2,5},\\
x_{21}&\longmapsto -p_{0,1},&\quad  x_{25}&\longmapsto p_{1,3},&\quad x_{26}&\longmapsto p_{2,3}.
\end{alignat*}

\begin{eqnarray*}
H(A_{{\lambda}_{2311}};\mathbf{t})&=&K\left(\frac{t_{2} \sqrt{t_{3}}}{\sqrt{t_{1}}},\frac{\sqrt{t_{1}} \sqrt{t_{3}}}{t_{2}},\frac{\sqrt{t_{1}} t_{2}}{\sqrt{t_{3}}},\frac{t_{2}^2}{t_{1}^{3/2} \sqrt{t_{3}}},\frac{t_{3}^{3/2}}{\sqrt{t_{1}} t_{2}},\frac{t_{1}^{5/2}}{t_{2} \sqrt{t_{3}}}\right)\notag\\
&&\Big/\left((1-t_{1})^3(1-t_{2})^2(1-t_{3})^3(1-t_{1}^2)(1-t_{2}^2)\left(\frac{t_{1}-t_{3}^2}{t_{1}}\right)\left(\frac{t_{1}-t_{2}}{t_{1}}\right)^2 \left(\frac{t_{1}-t_{3}}{t_{1}}\right)\right.\notag\\
&&\hphantom{\Big/\Big(}\left(\frac{t_{1}^2-t_{2} t_{3}}{t_{1}^2}\right)
\left(\frac{t_{1}^2-t_{2}^3}{t_{1}^2}\right)\left(\frac{t_{1}^2-t_{3}^2}{t_{1}^2}\right)
\left(\frac{t_{2}-t_{1}}{t_{2}}\right)\left(\frac{t_{2}^2-t_{1}^3}{t_{2}^2}\right)\left(\frac{t_{2}^2-t_{1}^2 t_{3}}{t_{2}^2}\right)\left(\frac{t_{2}-t_{3}}{t_{2}}\right)\left(\frac{t_{2}^2-t_{3}^2}{t_{2}^2}\right)\notag\\
&&\hphantom{\Big/\Big(}\left.\left(\frac{t_{3}-t_{1} t_{2}}{t_{3}}\right)\left(\frac{t_{3}-t_{1}^3}{t_{3}}\right)
\left(\frac{t_{3}-t_{2}}{t_{3}}\right)\left(\frac{t_{3}-t_{1}^2}{t_{3}}\right)\left(\frac{t_{1} t_{3}-t_{2}^3}{t_{1} t_{3}}\right)\right).
\end{eqnarray*}

\subsection{\texorpdfstring{$\boldsymbol{\left((1)\subset (1)\subset (3,1,1)\right)}$}{((1) in (1) in (3,1,1))},  extra.dim \texorpdfstring{$\boldsymbol{=6}$}{=6}}
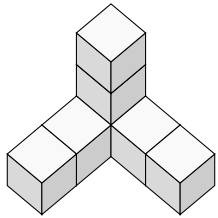
\label{sec:3D-11311}
\begin{center}
\begin{tikzpicture}[x=(220:0.6cm), y=(-40:0.6cm), z=(90:0.42cm)]

\foreach \m [count=\y] in {{3,1,1},{1},{1}}{
  \foreach \n [count=\x] in \m {
  \ifnum \n>0
      \foreach \z in {1,...,\n}{
        \draw [fill=gray!30] (\x+1,\y,\z) -- (\x+1,\y+1,\z) -- (\x+1, \y+1, \z-1) -- (\x+1, \y, \z-1) -- cycle;
        \draw [fill=gray!40] (\x,\y+1,\z) -- (\x+1,\y+1,\z) -- (\x+1, \y+1, \z-1) -- (\x, \y+1, \z-1) -- cycle;
        \draw [fill=gray!5] (\x,\y,\z)   -- (\x+1,\y,\z)   -- (\x+1, \y+1, \z)   -- (\x, \y+1, \z) -- cycle;  
      }
 \fi
 }
}    

\end{tikzpicture}
\end{center}
\[
I_{\lambda_{11311}}=\left(X_1^3, X_1 X_2,X_1X_3,X_2^3,X_2 X_3,X_3^3\right).
\]
\begin{alignat*}{6}
{c}_{1, 0, 0}^{1, 1, 0} &\longmapsto  {x}_{1},&\quad  {c}_{1, 0, 0}^{1, 0, 1} &\longmapsto {x}_{2},&\quad  {c}_{1, 0, 0}^{3, 0, 0} &\longmapsto  {x}_{3},&\quad  
{c}_{2, 0, 0}^{1, 1, 0} &\longmapsto  {x}_{4},&\quad  {c}_{2, 0, 0}^{1, 0, 1} &\longmapsto  {x}_{5},&\quad  {c}_{2, 0, 0}^{3, 0, 0} &\longmapsto {x}_{6},\\
{c}_{2, 0, 0}^{0, 1, 1} &\longmapsto  {x}_{7},&\quad  {c}_{2, 0, 0}^{0, 3, 0} &\longmapsto {x}_{8},&\quad  {c}_{2, 0, 0}^{0, 0, 3} &\longmapsto  {x}_{9},&\quad  
{c}_{0, 1, 0}^{1, 1, 0} &\longmapsto  {x}_{10},&\quad  {c}_{0, 1, 0}^{0, 1, 1} &\longmapsto  {x}_{11},&\quad  {c}_{0, 1, 0}^{0, 3, 0} &\longmapsto {x}_{12},\\
{c}_{0, 2, 0}^{1, 1, 0} &\longmapsto  {x}_{13},&\quad  {c}_{0, 2, 0}^{1, 0, 1} &\longmapsto {x}_{14},&\quad  {c}_{0, 2, 0}^{3, 0, 0} &\longmapsto  {x}_{15},&\quad  
{c}_{0, 2, 0}^{0, 1, 1} &\longmapsto {x}_{16},&\quad  {c}_{0, 2, 0}^{0, 3, 0} &\longmapsto  {x}_{17},&\quad  {c}_{0, 2, 0}^{0, 0, 3} &\longmapsto {x}_{18},\\
{c}_{0, 0, 1}^{1, 0, 1} &\longmapsto  {x}_{19},&\quad  {c}_{0, 0, 1}^{0, 1, 1} &\longmapsto {x}_{20},&\quad  {c}_{0, 0, 1}^{0, 0, 3} &\longmapsto  {x}_{21},&\quad  
{c}_{0, 0, 2}^{1, 1, 0} &\longmapsto {x}_{22},&\quad  {c}_{0, 0, 2}^{1, 0, 1} &\longmapsto  {x}_{23},&\quad  {c}_{0, 0, 2}^{3, 0, 0} &\longmapsto {x}_{24},\\
 {c}_{0, 0, 2}^{0, 1, 1} &\longmapsto  {x}_{25},&\quad  {c}_{0, 0, 2}^{0, 3, 0} &\longmapsto {x}_{26},&\quad  {c}_{0, 0, 2}^{0, 0, 3} &\longmapsto  {x}_{27}.
\end{alignat*}

\smallskip
\[
x_{19}\longmapsto x_{10}+x_{19},\quad x_{20}\longmapsto x_{1}+x_{20},\quad x_{11}\longmapsto x_{2}+x_{11}
\]

\begin{equation*}
\renewcommand{\arraystretch}{1.2}
\begin{array}{lll} 
(-{x}_{7}{x}_{15}-{x}_{14}{x}_{20}+{x}_{18}{x}_{22},&-{x}_{11}{x}_{22}-{x}_{7}{x}_{24}+{x}_{14}{x}_{26},&{x}_{8}{x}_{14}+{x}_{7}{x}_{19}-{x}_{9}{x}_{22},\\
\hphantom{(}{x}_{3}{x}_{22}+{x}_{20}{x}_{24}+{x}_{15}{x}_{26},&{x}_{3}{x}_{14}+{x}_{11}{x}_{15}+{x}_{18}{x}_{24},&-{x}_{12}{x}_{15}+{x}_{3}{x}_{19}+{x}_{21}{x}_{24},\\
\hphantom{(}{x}_{8}{x}_{15}-{x}_{19}{x}_{20}+{x}_{21}{x}_{22},&{x}_{12}{x}_{14}+{x}_{11}{x}_{19}+{x}_{9}{x}_{24},&-{x}_{12}{x}_{22}-{x}_{8}{x}_{24}-{x}_{19}{x}_{26},\\
\hphantom{(}-{x}_{8}{x}_{11}+{x}_{7}{x}_{12}+{x}_{9}{x}_{26},&-{x}_{3}{x}_{8}+{x}_{12}{x}_{20}+{x}_{21}{x}_{26},&{x}_{3}{x}_{7}-{x}_{11}{x}_{20}+{x}_{18}{x}_{26},\\
\hphantom{(}-{x}_{9}{x}_{15}+{x}_{18}{x}_{19}-{x}_{14}{x}_{21},&-{x}_{8}{x}_{18}+{x}_{9}{x}_{20}-{x}_{7}{x}_{21},&{x}_{3}{x}_{9}-{x}_{12}{x}_{18}-{x}_{11}{x}_{21}).
\end{array}\end{equation*}
\begin{alignat*}{6}
x_3&\longmapsto p_{1,2},&\quad x_7&\longmapsto p_{3,4},&\quad x_8&\longmapsto p_{0,3},&\quad x_9&\longmapsto p_{0,4},&\quad x_{11}&\longmapsto p_{2,4},&\quad x_{12}&\longmapsto p_{0,2},\\
x_{14}&\longmapsto p_{4,5},&\quad x_{15}&\longmapsto p_{1,5},&\quad x_{18}&\longmapsto p_{1,4},&\quad x_{19}&\longmapsto p_{0,5},&\quad x_{20}&\longmapsto p_{1,3},&\quad x_{21}&\longmapsto -p_{0,1},\\
x_{22}&\longmapsto p_{3,5},&\quad x_{24}&\longmapsto -p_{2,5},&\quad x_{26}&\longmapsto p_{2,3}.
\end{alignat*}
\begin{eqnarray}\label{eq-equihilb-A11311}
H(A_{{\lambda}_{11311}};\mathbf{t})&=&K\left(\frac{t_{2} t_{3}}{t_{1}},\frac{t_{1} t_{3}}{t_{2}},\frac{t_{1} t_{2}}{t_{3}},\frac{t_{2}^2}{t_{1} t_{3}},\frac{t_{3}^2}{t_{1} t_{2}},\frac{t_{1}^2}{t_{2} t_{3}}\right)\notag\\
&&\Big/\left((1-t_{1})^3(1-t_{2})^3(1-t_{3})^3(1-t_{1}^2)(1-t_{2}^2)(1-t_{3}^2)\left(\frac{t_{1}-t_{2}}{t_{1}}\right)\left(\frac{t_{1}-t_{3}}{t_{1}}\right)\left(\frac{t_{1}^2-t_{2} t_{3}}{t_{1}^2}\right)\right.\notag\\
&&\hphantom{\Big/\Big(}\left(\frac{t_{1}^2-t_{2}^3}{t_{1}^2}\right)\left(\frac{t_{1}^2-t_{3}^3}{t_{1}^2}\right)
\left(\frac{t_{2}-t_{1}}{t_{2}}\right)\left(\frac{t_{2}-t_{3}}{t_{2}}\right)\left(\frac{t_{2}^2-t_{1} t_{3}}{t_{2}^2}\right)\left(\frac{t_{2}^2-t_{1}^3}{t_{2}^2}\right)\left(\frac{t_{2}^2-t_{3}^3}{t_{2}^2}\right)\notag\\
&&\hphantom{\Big/\Big(}\left.\left(\frac{t_{3}^2-t_{1} t_{2}}{t_{3}^2}\right)
\left(\frac{t_{3}-t_{1}}{t_{3}}\right)\left(\frac{t_{3}^2-t_{1}^3}{t_{3}^2}\right)\left(\frac{t_{3}-t_{2}}{t_{3}}\right)\left(\frac{t_{3}^2-t_{2}^3}{t_{3}^2}\right)\right).
\end{eqnarray}

\section{Change of variables in the proof of Proposition~\ref{prop-jacF1321-isolated}}\label{sec:appendix-change-var-jacF1321}
\subsection{\texorpdfstring{$\boldsymbol{y_5\neq 0}$}{ y5 not 0}}
The relations in $\Jac(F_{1321})$ imply
\begin{equation*}
y_{3}= \frac{y_{1} y_{18}-y_{11} y_{9}+y_{12} y_{15}}{y_{5}},\quad
y_{29}= \frac{-y_{13} y_{17}+y_{16} y_{19}-y_{20} y_{6}}{y_{5}}.
\end{equation*}
After the change of variables
\begin{alignat*}{3}
y_{2}&\longmapsto y_{2} + \frac{y_{6} y_{9}}{y_{5}},&\quad 
y_{4}&\longmapsto y_{4} + \frac{y_{9} y_{19}}{y_{5}},&\quad
y_{7}&\longmapsto y_{7} - \frac{y_{18} y_{19}}{y_{5}}, \\
y_{8}&\longmapsto y_{8} -  \frac{y_{15} y_{19}}{y_{5}},&\quad
y_{10}&\longmapsto y_{10} - \frac{y_{6} y_{18}}{y_{5}},&\quad
y_{14}&\longmapsto y_{14} + \frac{y_{9} y_{17}}{y_{5}},\\
y_{21}&\longmapsto y_{21} + \frac{y_{15} y_{17}}{y_{5}},&\quad
y_{24}&\longmapsto y_{24} - \frac{y_{17} y_{18}}{y_{5}},&\quad
y_{27}&\longmapsto y_{27} + \frac{y_{15} y_{6}}{y_{5}}, 
\end{alignat*}
the ideal $\Jac(F_{1231})$ is transformed into the ideal
\begin{equation}\label{eq-jacF1231-invert-y5}
\renewcommand{\arraystretch}{1.2}
\begin{array}{lll} 
(-y_{10} y_{20}-y_{13} y_{24}+y_{16} y_{7},& 
-y_{11} y_{20}+y_{21} y_{5} y_{7}+y_{24} y_{5} y_{8},&
-y_{10} y_{21} y_{5}+y_{11} y_{16}+y_{24} y_{27} y_{5},\\
\hphantom{(}y_{1} y_{16}-y_{14} y_{27} y_{5}+y_{2} y_{21} y_{5},&
 -y_{10} y_{14} y_{5}+y_{12} y_{16}+y_{2} y_{24} y_{5},&
 -y_{1} y_{20}-y_{14} y_{5} y_{8}-y_{21} y_{4} y_{5},\\
\hphantom{(}  -y_{13} y_{14}+y_{16} y_{4}-y_{2} y_{20},&
  y_{13} y_{21}+y_{16} y_{8}+y_{20} y_{27},&
  -y_{1} y_{24}-y_{11} y_{14}+y_{12} y_{21},\\
\hphantom{(}   -y_{10} y_{5} y_{8}-y_{11} y_{13}-y_{27} y_{5} y_{7},& 
   y_{1} y_{7}+y_{11} y_{4}+y_{12} y_{8},&
   -y_{1} y_{10}-y_{11} y_{2}+y_{12} y_{27},\\
\hphantom{(}   -y_{10} y_{4} y_{5}+y_{12} y_{13}+y_{2} y_{5} y_{7},&
     -y_{1} y_{13}+y_{2} y_{5} y_{8}+y_{27} y_{4} y_{5},& 
     y_{12} y_{20}-y_{14} y_{5} y_{7}+y_{24} y_{4} y_{5}).
\end{array}
\end{equation}
Finally, the map
\begin{alignat*}{6}
 y_1&\longmapsto p_{0,1},&\quad y_2&\longmapsto -p_{0,2},&\quad  y_4&\longmapsto p_{0,4},&\quad  y_7&\longmapsto p_{4,5},&\quad
 y_8&\longmapsto -\frac{p_{1,4}}{y_5},&\quad
 y_{10}&\longmapsto -p_{2,5},\\
 y_{11}&\longmapsto -p_{1,5},&\quad y_{12} &\longmapsto -p_{0,5} y_5,&\quad 
 y_{13} &\longmapsto p_{2,4},&\quad y_{14} &\longmapsto p_{0,3},&\quad  y_{16} &\longmapsto p_{2,3},&\quad
 y_{20} &\longmapsto p_{3,4},\\
 y_{21} &\longmapsto \frac{p_{1,3}}{y_5},&\quad
y_{24} &\longmapsto p_{3,5},&\quad  y_{27} &\longmapsto -\frac{p_{1,2}}{y_5} 
\end{alignat*}
transforms the ideal $(\ref{eq-jacF1231-invert-y5})$ into the Pl\"ucker ideal (\ref{eq-pluckerideal}). In the following sections, we omit the intermediate explanations.

\subsection{\texorpdfstring{$\boldsymbol{y_6\neq 0}$}{y6 not 0}}
\vspace{-.5\baselineskip}
\begin{gather*}
y_{20}= \frac{-y_{13} y_{17}+y_{16} y_{19}-y_{29} y_{5}}{y_{6}},\quad
y_{3}=\frac{-y_{1} y_{10}-y_{11} y_{2}+y_{12} y_{27}}{y_{6}}.
\end{gather*}
\begin{alignat*}{3}
y_4&\longmapsto y_{4}+ \frac{y_{19} y_{2}}{y_{6}},&\quad
y_7&\longmapsto y_{7}+\frac{y_{10} y_{19}}{y_{6}},&\quad
y_8&\longmapsto y_{8}-\frac{y_{19} y_{27}}{y_{6}},\\
y_9&\longmapsto y_{9}+\frac{y_{2} y_{5}}{y_{6}},&\quad
y_{14}&\longmapsto y_{14}+\frac{y_{17} y_{2}}{y_{6}},&\quad
y_{15}&\longmapsto y_{15}+\frac{y_{27} y_{5}}{y_{6}},\\
y_{18}&\longmapsto y_{18}-\frac{y_{10} y_{5}}{y_{6}},&\quad
y_{21}&\longmapsto y_{21}+\frac{y_{17} y_{27}}{y_{6}},&\quad
y_{24}&\longmapsto y_{24}+\frac{y_{10} y_{17}}{y_{6}}.
\end{alignat*}

\vspace{-.5\baselineskip}
\begin{equation*}
\renewcommand{\arraystretch}{1.2}
\begin{array}{lll} 
(-y_{13} y_{24}+y_{16} y_{7}+y_{18} y_{29},& 
y_{11} y_{16}-y_{15} y_{24} y_{6}-y_{18} y_{21} y_{6},& 
y_{1} y_{16}+y_{14} y_{15} y_{6}-y_{21} y_{6} y_{9},\\
\hphantom{(} y_{12} y_{16}-y_{14} y_{18} y_{6}-y_{24} y_{6} y_{9},& 
 -y_{11} y_{29}-y_{21} y_{6} y_{7}-y_{24} y_{6} y_{8},& 
 -y_{13} y_{14}+y_{16} y_{4}-y_{29} y_{9},\notag\\
 \hphantom{(} y_{13} y_{21}+y_{15} y_{29}+y_{16} y_{8},& 
  -y_{1} y_{24}-y_{11} y_{14}+y_{12} y_{21},& 
  -y_{11} y_{13}+y_{15} y_{6} y_{7}-y_{18} y_{6} y_{8},\\
\hphantom{(}   y_{12} y_{29}+y_{14} y_{6} y_{7}-y_{24} y_{4} y_{6},&
    y_{1} y_{7}+y_{11} y_{4}+y_{12} y_{8},&
     y_{1} y_{29}-y_{14} y_{6} y_{8}-y_{21} y_{4} y_{6},\notag\\
      y_{12} y_{13}-y_{18} y_{4} y_{6}-y_{6} y_{7} y_{9},&
\hphantom{(}       -y_{1} y_{13}-y_{15} y_{4} y_{6}-y_{6} y_{8} y_{9},& 
       y_{1} y_{18}-y_{11} y_{9}+y_{12} y_{15}).\notag
\end{array}\end{equation*}
\begin{alignat*}{6}
y_1 &\longmapsto p_{0,1},&\quad
y_4&\longmapsto -p_{1,5},&\quad
y_7&\longmapsto p_{4,5},&\quad
y_8&\longmapsto -\frac{p_{0,5}}{y_6},&\quad
y_9&\longmapsto p_{1,2},&\quad
y_{11}&\longmapsto p_{0,4},\\
y_{12}&\longmapsto -p_{1,4} y_6,&\quad
y_{13}&\longmapsto -p_{2,5},&\quad
y_{14}&\longmapsto p_{1,3},&\quad
y_{15}&\longmapsto -\frac{p_{0,2}}{y_6},&\quad
y_{16}&\longmapsto p_{2,3},&\quad
y_{18}&\longmapsto -p_{2,4},\\
y_{21}&\longmapsto -\frac{p_{0,3}}{y_6},&\quad
y_{24}&\longmapsto p_{3,4},&\quad
y_{29}&\longmapsto p_{3,5}.
\end{alignat*}

\subsection{\texorpdfstring{$\boldsymbol{y_7\neq 0}$}{y7 not 0}}
\begin{eqnarray*}
y_{16}= \frac{y_{10} y_{20}+y_{13} y_{24}-y_{18} y_{29}}{y_{7}},\
y_{1}= \frac{-y_{11} y_{4}-y_{12} y_{8}-y_{19} y_{3}}{y_{7}}.
\end{eqnarray*}
\begin{alignat*}{3}
y_{2}&\longmapsto y_{2}+\frac{y_{10} y_{4}}{y_{7}},&\quad
y_{5}&\longmapsto y_{5}-\frac{y_{18} y_{19}}{y_{7}},&\quad
y_{6}&\longmapsto y_{6}+\frac{y_{10} y_{19}}{y_{7}},\\
y_{9}&\longmapsto y_{9}-\frac{y_{18} y_{4}}{y_{7}},&\quad
y_{14}&\longmapsto y_{14}+\frac{y_{24} y_{4}}{y_{7}},&\quad
y_{15}&\longmapsto y_{15}+\frac{y_{18} y_{8}}{y_{7}},\\
y_{17}&\longmapsto y_{17}+\frac{y_{19} y_{24}}{y_{7}},&\quad
y_{21}&\longmapsto y_{21}-\frac{y_{24} y_{8}}{y_{7}},&\quad
y_{27}&\longmapsto y_{27}-\frac{y_{10} y_{8}}{y_{7}}.
\end{alignat*}

\begin{equation*}
\renewcommand{\arraystretch}{1.2}
\begin{array}{lll} 
(-y_{11} y_{20}-y_{15} y_{17} y_{7}+y_{21} y_{5} y_{7},& 
-y_{13} y_{17}-y_{20} y_{6}-y_{29} y_{5},& 
-y_{14} y_{27} y_{7}+y_{2} y_{21} y_{7}-y_{29} y_{3},\\
\hphantom{(} y_{14} y_{15} y_{7}-y_{20} y_{3}-y_{21} y_{7} y_{9},&
  -y_{11} y_{29}+y_{17} y_{27} y_{7}-y_{21} y_{6} y_{7},&
   -y_{13} y_{14}-y_{2} y_{20}-y_{29} y_{9},\\
\hphantom{(}    y_{13} y_{21}+y_{15} y_{29}+y_{20} y_{27},& 
    -y_{11} y_{14}+y_{12} y_{21}-y_{17} y_{3},& 
    -y_{11} y_{13}+y_{15} y_{6} y_{7}-y_{27} y_{5} y_{7},\\
\hphantom{(}     y_{12} y_{29}+y_{14} y_{6} y_{7}-y_{17} y_{2} y_{7},& 
     -y_{13} y_{3}-y_{15} y_{2} y_{7}+y_{27} y_{7} y_{9},& 
     -y_{11} y_{2}+y_{12} y_{27}-y_{3} y_{6},\\
\hphantom{(}      y_{12} y_{13}+y_{2} y_{5} y_{7}-y_{6} y_{7} y_{9},& 
      y_{12} y_{20}-y_{14} y_{5} y_{7}+y_{17} y_{7} y_{9},& 
       -y_{11} y_{9}+y_{12} y_{15}-y_{3} y_{5}).
\end{array}\end{equation*}
\begin{alignat*}{6}
y_2&\longmapsto p_{0,1},&\quad
y_3&\longmapsto p_{0,2},&\quad
y_5&\longmapsto \frac{p_{4,5}}{y_7},&\quad
y_6&\longmapsto p_{1,5},&\quad
y_9&\longmapsto \frac{p_{0,4}}{y_7},&\quad
y_{11}&\longmapsto -p_{2,5},\\
y_{12}&\longmapsto -p_{0,5},&\quad
y_{13}&\longmapsto -p_{1,4},&\quad
y_{14}&\longmapsto p_{0,3},&\quad
y_{15}&\longmapsto \frac{p_{2,4}}{y_7},&\quad
y_{17}&\longmapsto p_{3,5},&\quad
y_{20}&\longmapsto p_{3,4},\\
y_{21}&\longmapsto p_{2,3},&\quad
y_{27}&\longmapsto -p_{1,2},&\quad
y_{28}&\longmapsto p_{1,3} y_7.
\end{alignat*}

\subsection{\texorpdfstring{$\boldsymbol{y_{10}\neq 0}$}{ y10 not 0}}
\vspace{-.5\baselineskip}
\begin{eqnarray*}
y_{20}= \frac{-y_{13} y_{24}+y_{16} y_{7}+y_{18} y_{29}}{y_{10}},\quad
y_{1}= \frac{-y_{11} y_{2}+y_{12} y_{27}-y_{3} y_{6}}{y_{10}}.
\end{eqnarray*}
\begin{alignat*}{3}
y_4&\longmapsto y_{4}+\frac{y_{2} y_{7}}{y_{10}},&\quad
y_5&\longmapsto y_{5}-\frac{y_{18} y_{6}}{y_{10}},&\quad
y_8&\longmapsto y_{8}-\frac{y_{27} y_{7}}{y_{10}},\\
y_9&\longmapsto y_{9}-\frac{y_{18} y_{2}}{y_{10}},&\quad
y_{14}&\longmapsto y_{14}+\frac{y_{2} y_{24}}{y_{10}},&\quad
y_{15}&\longmapsto y_{15}-\frac{y_{18} y_{27}}{y_{10}},\\
y_{17}&\longmapsto y_{17}+\frac{y_{24} y_{6}}{y_{10}},&\quad
y_{19}&\longmapsto y_{19}+\frac{y_{6} y_{7}}{y_{10}},&\quad
y_{21}&\longmapsto y_{21}+\frac{y_{24} y_{27}}{y_{10}}.
\end{alignat*}

\begin{equation*}
\renewcommand{\arraystretch}{1.2}
\begin{array}{lll} 
(-y_{13} y_{17}+y_{16} y_{19}-y_{29} y_{5},&  
y_{10} y_{15} y_{17}-y_{10} y_{21} y_{5}+y_{11} y_{16},& 
 -y_{10} y_{14} y_{8}-y_{10} y_{21} y_{4}-y_{29} y_{3},\\
\hphantom{(}   -y_{10} y_{14} y_{5}+y_{10} y_{17} y_{9}+y_{12} y_{16},&  
   y_{10} y_{17} y_{8}+y_{10} y_{19} y_{21}-y_{11} y_{29},&  
   -y_{13} y_{14}+y_{16} y_{4}-y_{29} y_{9},\\
\hphantom{(}    y_{13} y_{21}+y_{15} y_{29}+y_{16} y_{8},& 
     -y_{11} y_{14}+y_{12} y_{21}-y_{17} y_{3},&  
     -y_{10} y_{15} y_{19}-y_{10} y_{5} y_{8}-y_{11} y_{13},\\
 \hphantom{(}      -y_{10} y_{14} y_{19}+y_{10} y_{17} y_{4}+y_{12} y_{29},&  
       y_{11} y_{4}+y_{12} y_{8}+y_{19} y_{3},&  
       y_{10} y_{15} y_{4}+y_{10} y_{8} y_{9}-y_{13} y_{3},\\
\hphantom{(}         -y_{10} y_{14} y_{15}+y_{10} y_{21} y_{9}+y_{16} y_{3},& 
          y_{10} y_{19} y_{9}-y_{10} y_{4} y_{5}+y_{12} y_{13},&  
          -y_{11} y_{9}+y_{12} y_{15}-y_{3} y_{5}).
\end{array}\end{equation*}
\begin{alignat*}{6}
y_3&\longmapsto p_{0,1},&\quad  
y_4&\longmapsto p_{0,5},&\quad  
y_5&\longmapsto -p_{2,3},&\quad  
y_8&\longmapsto -\frac{p_{1,5}}{y_{10}},&\quad  
y_9&\longmapsto p_{0,2},&\quad  
y_{11}&\longmapsto p_{1,3},\\ 
y_{12}&\longmapsto p_{0,3} y_{10},&\quad  
y_{13}&\longmapsto -p_{2,5},&\quad  
y_{14}&\longmapsto p_{0,4},&\quad  
y_{15}&\longmapsto \frac{p_{1,2}}{y_{10}},&\quad  
y_{16}&\longmapsto -p_{2,4},&\quad  
y_{17}&\longmapsto p_{3,4},\\ 
y_{19}&\longmapsto p_{3,5},&\quad  
y_{21}&\longmapsto \frac{p_{1,4}}{y_{10}},&\quad  
y_{29}&\longmapsto p_{4,5}.
\end{alignat*}

\subsection{\texorpdfstring{$\boldsymbol{y_{19}\neq 0}$}{ y19 not 0}}
\vspace{-.5\baselineskip}
\begin{eqnarray*}
y_{16}\longrightarrow \frac{y_{13} y_{17}+y_{20} y_{6}+y_{29} y_{5}}{y_{19}},\quad
y_{3}\longrightarrow \frac{-y_{1} y_{7}-y_{11} y_{4}-y_{12} y_{8}}{y_{19}}.
\end{eqnarray*}
\begin{alignat*}{6}
y_2&\longmapsto y_{2}+\frac{y_{4} y_{6}}{y_{19}},&\quad
y_9&\longmapsto y_{9}+\frac{y_{4} y_{5}}{y_{19}},&\quad
y_{10}&\longmapsto y_{10}+\frac{y_{6} y_{7}}{y_{19}},\\
y_{14}&\longmapsto y_{14}+\frac{y_{17} y_{4}}{y_{19}},&\quad 
y_{15}&\longmapsto y_{15}-\frac{y_{5} y_{8}}{y_{19}},&\quad
y_{18}&\longmapsto y_{18}-\frac{y_{5} y_{7}}{y_{19}},\\
y_{21}&\longmapsto y_{21}-\frac{y_{17} y_{8}}{y_{19}},& \quad
y_{24}&\longmapsto y_{24}+\frac{y_{17} y_{7}}{y_{19}},& \quad
y_{27}&\longmapsto y_{27}-\frac{y_{6} y_{8}}{y_{19}}.
\end{alignat*}

\begin{equation*}
\renewcommand{\arraystretch}{1.2}
\begin{array}{lll} 
(-y_{10} y_{20}-y_{13} y_{24}+y_{18} y_{29},& 
-y_{11} y_{20}+y_{15} y_{19} y_{24}+y_{18} y_{19} y_{21},& 
y_{10} y_{19} y_{21}-y_{11} y_{29}-y_{19} y_{24} y_{27},\\
\hphantom{(} -y_{1} y_{20}-y_{14} y_{15} y_{19}+y_{19} y_{21} y_{9},& 
 -y_{13} y_{14}-y_{2} y_{20}-y_{29} y_{9},& 
 y_{13} y_{21}+y_{15} y_{29}+y_{20} y_{27},\\
\hphantom{(}  -y_{1} y_{24}-y_{11} y_{14}+y_{12} y_{21},& 
  -y_{10} y_{15} y_{19}-y_{11} y_{13}-y_{18} y_{19} y_{27},& 
  -y_{10} y_{14} y_{19}+y_{12} y_{29}+y_{19} y_{2} y_{24},\\
\hphantom{(}   y_{1} y_{29}-y_{14} y_{19} y_{27}+y_{19} y_{2} y_{21},& 
   -y_{1} y_{10}-y_{11} y_{2}+y_{12} y_{27},& 
   y_{10} y_{19} y_{9}+y_{12} y_{13}+y_{18} y_{19} y_{2},\\
 \hphantom{(}   -y_{1} y_{13}+y_{15} y_{19} y_{2}-y_{19} y_{27} y_{9},&
     y_{12} y_{20}-y_{14} y_{18} y_{19}-y_{19} y_{24} y_{9},& 
     y_{1} y_{18}-y_{11} y_{9}+y_{12} y_{15}).
\end{array}\end{equation*}
\begin{alignat*}{6}
y_1&\longmapsto p_{0,1},&\quad 
y_2&\longmapsto p_{0,5},&\quad 
y_9&\longmapsto -p_{0,3},&\quad 
y_{10}&\longmapsto -p_{2,5},&\quad 
y_{11}&\longmapsto -p_{1,2},&\quad 
y_{12}&\longmapsto p_{0,2} y_{19},\\ 
y_{13}&\longmapsto -p_{3,5},&\quad 
y_{14}&\longmapsto p_{0,4},&\quad 
y_{15}&\longmapsto \frac{p_{1,3}}{y_{19}},&\quad 
y_{18}&\longmapsto -p_{2,3},&\quad 
y_{20}&\longmapsto p_{3,4},&\quad 
y_{21}&\longmapsto -\frac{p_{1,4}}{y_{19}},\\ 
y_{24}&\longmapsto -p_{2,4},&\quad 
y_{27}&\longmapsto -\frac{p_{1,5}}{y_{19}},&\quad 
y_{29}&\longmapsto  -p_{4,5}.
\end{alignat*}

\end{appendix}


\end{document}